\documentclass[11pt]{amsart}
\usepackage{amssymb,mathrsfs}

\usepackage{amscd, amssymb, amsfonts}

\newtheorem{theorem}{Theorem}[section]

\newtheorem{lemma}[theorem]{Lemma}

\theoremstyle{definition}
\newtheorem{definition}[theorem]{Definition}

\theoremstyle{remark}
\newtheorem{remark}[theorem]{Remark}

\theoremstyle{proposition}
\newtheorem{proposition}[theorem]{Proposition}

\theoremstyle{corollary}
\newtheorem{corollary}[theorem]{Corollary}

\numberwithin{equation}{section}

\begin{document}

\title[\tiny {Heisenberg Groups, Theta Functions and the Weil Representation}]
{Heisenberg Groups, Theta Functions and\\ the Weil Representation}

\author{Jae-Hyun Yang}

\noindent
\address{Department of Mathematics \\ Inha University  \\ Incheon 402-751
\\ Korea}
\noindent
\email{jhyang@inha.ac.kr}

\thanks{
\endgraf Keywords and phrases\,: Heisenberg group, theta functions, Weil representation.\\
\indent This work was supported by Basic Science Program through the National Research\\
\indent Foundation of Korea(NRF) funded by the Ministry of Education, Science and Technol-\\
\indent ogy (41493-01) and 
partially supported by the Max-Planck-Institut f{\"u}r Mathematik in\\
\indent Bonn.}





\maketitle


\newcommand\CM{{\mathcal M}}
\def\tr{\triangleright}
\def\a{\alpha}
\def\be{\beta}
\def\g{\gamma}
\def\gh{\Cal G^J}
\def\G{\Gamma}
\def\de{\delta}
\def\e{\epsilon}
\def\z{\zeta}
\def\th{\theta}
\def\vth{\vartheta}
\def\vp{\varphi}
\def\r{\rho}
\def\om{\omega}
\def\p{\pi}
\def\la{\lambda}
\def\lb{\lbrace}
\def\lk{\lbrack}
\def\rb{\rbrace}
\def\rk{\rbrack}
\def\s{\sigma}
\def\w{\wedge}
\def\lrt{\longrightarrow}
\def\lmt{\longmapsto}
\def\lmk{(\lambda,\mu,\kappa)}
\def\Om{\Omega}
\def\k{\kappa}
\def\ba{\backslash}
\def\ph{\phi}
\def\M{{\mathcal M}}
\def\bA{\mathbf A}
\def\bH{\mathbf H}
\def\dett{det\,(4T-R{{\Cal M}^{-1}}{^t\!R})}
\def\dif{{{\partial}\over {\partial Y}} + {1\over {8\pi}}
^t\!\left({{\partial}\over {\partial V}}\right) \left(
{{\partial}\over {\partial V}}\right)}
\def\dtr{(4T-R{^t\!R})}
\def\Hom{\text{Hom}}
\def\cP{\mathcal P}
\def\cH{\mathcal H}
\def\BZ{\mathbb Z}
\def\BC{\mathbb C}
\def\BR{\mathbb R}
\def\BQ{\mathbb Q}
\def\pa{\partial}
\def\A{\left[\begin{matrix} A\\ 0\end{matrix}\right]}
\def\Bb{\left[\begin{matrix} A\\ B\end{matrix} \right]}
\def\Dd{\left[\begin{matrix} -A \\ -B \end{matrix} \right]}
\def\x{\left[\begin{matrix} A+\xi\\ B+\eta\end{matrix} \right]}
\def\lam{\left[\begin{matrix} A+\lambda\\ B+\mu\end{matrix}\right]}
\def\ze{\left[\begin{matrix} 0\\ 0\end{matrix}\right]}
\def\J{J\in {\Bbb Z}^{(m,n)}_{\geq 0}}
\def\N{N\in {\Bbb Z}^{(m,n)}}
\def\Dm{\left[\begin{matrix} -A\\ -B\end{matrix}\right]}
\def\dt{{{d}\over {dt}}\bigg|_{t=0}}
\def\lt{\lim_{t\to 0}}
\def\zhg{\BZ^{(m,n)}}
\def\bhg{\BR^{(m,n)}}
\def\ex{\par\smallpagebreak\noindent}
\def\Box{$\square$}
\def\pis{\pi i \sigma}
\def\sd{\,\,{\vartriangleright}\kern -1.0ex{<}\,}
\def\sc{\bf}
\def\wt{\widetilde}
\newcommand\BH{\mathbb H}

\vskip 0.15cm
\centerline{\large \bf Table of Contents}

\vskip 0.51cm $ \qquad\qquad\qquad\textsf{\large \ \ 1.
Introduction}$\vskip 0.215cm

$\qquad\qquad\qquad \textsf{\large\ \ 2. The Heisenberg Group $H_{\BR}^{(n,m)}$ }$
\vskip 0.2150cm

$ \qquad\qquad\qquad\textsf{\large \ \ 3. Theta Functions}$
\vskip 0.215cm

$ \qquad\qquad\qquad\textsf{\large \ \ 4.
Induced Representations}$
\vskip 0.215cm

$ \qquad\qquad\qquad\textsf{\large \ \ 5.
Schr{\"o}dinger Representations}$
\vskip 0.215cm

$ \qquad\qquad\qquad\textsf{\large \ \ 6.
Fock Representations}$
\vskip 0.215cm

$ \qquad\qquad\qquad\textsf{\large \ \ 7.
Lattice Representations}$
\vskip 0.215cm

$\qquad\qquad\qquad\textsf{\large \ \ 8.
The Coadjoint Orbits of $H_{\BR}^{(n,m)}$}$
\vskip 0.215cm

$\qquad\qquad\qquad\textsf{\large \ \ 9.
Hermite Operators}$
\vskip 0.215cm

$\qquad\qquad\qquad\textsf{\large \ 10.
Harmonic Analysis on $H_{\BZ}^{(n,m)}\backslash
H_{\BR}^{(n,m)}$}$
\vskip 0.215cm

$\qquad\qquad\qquad\textsf{\large \ 11.
The Symplectic Group}$
\vskip 0.215cm

$\qquad\qquad\qquad\textsf{\large \ 12.
Some Geometry on Siegel Space}$
\vskip 0.215cm

$\qquad\qquad\qquad\textsf{\large \ 13.
The Weil Representation}$
\vskip 0.215cm

$\qquad\qquad\qquad\textsf{\large \ 14.
Covariant Maps for the Weil Representation}$
\vskip 0.215cm

$\qquad\qquad\qquad\textsf{\large \ 15.
Theta Series with Quadratic Forms}$
\vskip 0.215cm

$\qquad\qquad\qquad\textsf{\large \ 16.
Theta Series in Spherical Harmonics}$
\vskip 0.215cm

$\qquad\qquad\qquad\textsf{\large \ 17.
Relation between Theta Series and the Weil Representation}$
\vskip 0.215cm

$\qquad\qquad\qquad\textsf{\large \ 18.
Spectral Theory on the Abelian Variety}$
\vskip 0.215cm

$ \qquad\qquad\qquad\textsf{\large\ References}$


\newpage


\begin{section}{{\large\bf Introduction}}
\setcounter{equation}{0}

A certain nilpotent Lie group plays an important role in the study of
the foundations of quantum mechanics (\,cf.\,\cite{Ne}\,and\,\cite{Wey}\,)
and the study of theta functions
(\,see\,\cite{C},\,\cite{D},\,\cite{I},\,\cite{Mum1},\,\cite{Mum2},\,\cite{S1},\,\cite{Wei},\,\cite{Y1}\,and\,\cite{Y2}).
\vskip2mm
For any positive integers $m$ and $n$, we consider the Heisenberg group
$$H_{\BR}^{(n,m)}:=\left\{\,(\la,\mu,\kappa)\,\vert\ \la,\mu\in \BR^{(m,n)},\
\kappa\in \BR^{(m,m)},\ \kappa+\mu\,^t\!\la\ \text{symmetric} \right\}$$
endowed with the following multiplication law
$$(\la,\mu,\kappa)\circ (\la',\mu',\kappa')=(\la+\la',\mu+\mu',\kappa+
\kappa'+\la\,^t\!\mu'-\mu\,^t\!\la').$$
The Heisenberg group $H_{\BR}^{(n,m)}$ is embedded in the symplectic group
$Sp(m+n,\BR)$ via the mapping
$$H_{\BR}^{(n,m)}\ni (\la,\mu,\kappa)\longmapsto
\begin{pmatrix} I_n & 0 & 0 & {}^t\mu \\
\la & I_m & \mu & \kappa \\
0 & 0 & I_n & -^t\la \\
0 & 0 & 0 & I_m \end{pmatrix} \in Sp(m+n,\BR).$$
This Heisenberg group is a 2-step nilpotent Lie group and is important in
the study of smooth compactification of the Siegel modular variety. In fact,
$H_{\BR}^{(n,m)}$ is obtained as the unipotent radical of the parabolic
subgroup of the rational boundary component
$F_n$(\,cf.\,\cite{F-C}\,pp.\,122-123,\,\cite{Na}\,p.\,21\,or\,\cite{YJH2}\,p.\,36).
In the case $m=1$, the study on this Heisenberg group was done by many
mathematicians, e.g., P. Cartier\,\cite{C}, J. Igusa\,\cite{I}, D. Mumford\,\cite{Mum1},
\cite{Mum2} and many analysts(cf.\,\cite{A-T}) explicitly.
For the case $m>1$, the multiplication
law is a little different from that of the Heisenberg group which is
usually known and needs much more complicated computation than the case
$m=1$.

The aim of this paper is to investigate the Heisenberg group $H_{\BR}^{(n,m)}$
in more detail. In the previous papers \cite{Y1} and \cite{Y2}, the author decomposed
the $L^2$-space $L^2\left( H_{\BZ}^{(n,m)}\ba H_{\BR}^{(n,m)}\right)$
with respect to the right regular representation of $H_{\BR}^{(n,m)}$
explicitly and related the study of $H_{\BR}^{(n,m)}$ to that of
theta functions, where $H_{\BZ}^{(n,m)}$ denotes the discrete subgroup of
$H_{\BR}^{(n,m)}$ consisting of integral elements.
We need to investigate
$H_{\BR}^{(n,m)}$ for the study of Jacobi forms\,(\,cf.\,\cite{YJH2},\,\cite{Z}),
degeneration of abelian varieties\,(\,cf.\,\cite{F-C}\,) and so on.

\vskip 0.2cm
This paper
is organized as follows. In Section 2, we introduce the new multiplication
on $H_{\BR}^{(n,m)}$ which will be useful in the subsequent sections.
And we find the Lie algebra of
$H_{\BR}^{(n,m)}$ and obtain the commutation relation for $H_{\BR}^{(n,m)}.$
In Section 3, we give an explicit description of theta functions due to
J. Igusa\,(\,cf.\, \cite{I} or \cite{Mum1}\,) and identify the theta functions with
the smooth functions on $H_{\BR}^{(n,m)}$ satisfying some conditions. The
results of this section will be used later. In Section 4, using the Mackey
decomposition of a locally compact group\,(\,cf.\,\cite{Mc}\,), we introduce
the induced representations of $H_{\BR}^{(n,m)}$ and compute the unitary
dual of $H_{\BR}^{(n,m)}.$ In Section 5, we realize the Schr{\"o}dinger
representation of $H_{\BR}^{(n,m)}$ as the the representation of
$H_{\BR}^{(n,m)}$ induced by the one-dimensional unitary character of
a certain subgroup of $H_{\BR}^{(n,m)}.$ In Section 6, we consider the
Fock representation $\left(\,U^{F,\M},{\mathcal H}_{F,\M}\,\right)$ of
$H_{\BR}^{(n,m)}$. We prove that for a positive definite symmetric
half-integral matrix $\M$ of degree $m,\ U^{F,\M}$ is unitarily equivalent
to the Schr{\"o}dinger representation $U^{S,\M}.$ We also find an
orthonormal basis for the representation space ${\mathcal H}_{F,\M}.$
This section is mainly based on the papers \cite{S1,S2,Y5}.
In Section 7, we prove that for any positive definite symmetric,
half-integral matrix of degree $m$, the lattice representation $\pi_{\M}$ of
$H_{\BR}^{(n,m)}$ is unitarily equivalent to the $(\,\det\,2\M\,)^n$-multiples
of the Schr{\"o}dinger representation $U^{\mathcal S,\M}.$
We give a relation between the lattice representation $\pi_{\M}$ and theta
functions. This section is based on the paper \cite{Y6}.
In Section 8, we find the coadjoint orbits of $H_{\BR}^{(n,m)}$.
And we describe explicitly the connection between the coadjoints orbits and
the irreducible unitary representations of $H_{\BR}^{(n,m)}$ following the
work of A. Kirillov\,(\,cf.\ \cite{K1},\,\cite{K2}\,and\,\cite{K3}\,). In Section 9,
considering the Schr{\"o}dinger representation $\left(\,U^{S,I_m},\,
L^2\big(\BR^{(m,n)},d\xi\big)\,\right),$ we study the Hermite operators and the
Hermite functions. We prove that Hermite functions defined
in this section form an orthonormal basis for $L^2\big(\BR^{(m,n)},d\xi\big)$ and
eigenfunctions for Hermite operators, the Fourier transform and the Fourier
cotransform. We mention that Hermitian functions are used to construct
non-holomorphic modular forms of half-integral weight\,(\,cf.\,\cite{Y2}\,).
Implicitly the study of the Heisenberg group $H_{\BR}^{(n,m)}$ implies that
the confluent hypergeometric equations (in this case, the Hermite equation)
are related to the study of automorphic forms. In Section 10, we investigate
the irreducible components of $L^2\left(H_{\BZ}^{(n,m)}\ba
H_{\BR}^{(n,m)}\right)$.
We describe the connection among these irreducible components, the
Schr{\"o}dinger representations, the Fock representations and the lattice
representations explicitly. We also provide the orthonormal bases for the
representation spaces respectively. A decomposition of $L^2(\Gamma\ba G)$
for a general nilpotent Lie group $G$ and a discrete subgroup $\Gamma$ of $G$
was dealt by C. C. Moore\,(\,cf.\,\cite{Mo}\,). In Section 11, we briefly review the symplectic group
and its action on the Siegel upper half plane to be needed in the subsequent sections. We construct the universal
covering group of the symplectic group. In Section 12, we present some properties of the geometry on the Siegel
upper half plane which are used in the subsequent sections.
In Section 13, we study the Weil representation associated to a positive definite symmetric real
matrix of degree $m$. We describe the explicit actions for the Weil representation. We describe the results
on the Weil representation which were obtained by Kashiwara and Vergne\,\cite{KV}.
In Section 14, we construct the covariant maps for the Weil representation. In Section 15,
we review various type of theta series associated to quadratic forms. In Section 16, we discuss the theta
series with harmonic coefficients. Pluriharmonic polynomials play an important role in the study of the
Weil representation. We prove that the theta series with pluriharmonic polynomials as coefficients are
a modular form for a suitable congruence subgroup of the Siegel modular group. This section is mainly based
on the book \cite{Mum2}. In Section 17, we investigate the relation between the Weil representation and
the theta series. We construct modular forms using the covariant maps for the Weil representation. In Section 18,
we discuss the spectral theory on the principally polarized abelian variety $A_\Omega$ attached to
an element of the Siegel upper half plane. We decompose the $L^2$-space of $A_\Omega$ into irreducibles explicitly.
We refer to \cite{Y7} for more detail.
\vskip 0.1cm
Finally I would like to mention that a Heisenberg group was paid to
an attention by some differential geometers, e.g., M. L. Gromov, in the
sense of a parabolic geometry. A Heisenberg group is regarded as a principal
fibre bundle over an Euclidean space with a vector space or a circle as
fibres and may be also regarded as the boundary of a complex ball. The
geometry of this group is quite different from that of an Euclidean space.

\vskip 0.2cm\noindent
{\bf Notations:}
We denote by $\BZ,\,\BR$ and $\BC$ the ring of integers, the field of real
numbers, and the field of complex numbers respectively. $\BC^*$ denotes the multiplicative group
consisting of all nonzero complex numbers.
$\BC_1^*$
denotes the multiplicative group consisting of all complex numbers $z$
with $\vert z\vert =1.$
$Sp(n,\BR)$ denotes the symplectic group of degree $n$. ${\mathbb H}_n$ denotes
the Siegel upper half plane of degree $n$.
The symbol ``:='' means that the expression on the
right is the definition of that on the left. We denote by $\BZ^+$
the set of all positive integers. $F^{(k,l)}$ denotes the set of
all $k\times l$ matrices with entries in a commutative ring $F$.
For any $M\in F^{(k,l)},\ ^t\!M$ denotes the transposed matrix of $M$.
For a complex matrix $A$,
${\overline A}$ denotes the complex {\it conjugate} of $A$. The diagonal matrix with entries
$a_1,\cdots,a_n$ on the diagonal position is denoted by ${\rm diag}(a_1,\cdots,a_n)$.
For $A\in F^{(k,k)},\ \sigma(A)$ denotes the trace of $A$. For
$A\in F^{(k,l)}$ and $B\in F^{(k,k)},$ we set $B[A]={^t\!A}BA$.
$I_k$ denotes
the identity matrix of degree $k$. For a positive integer $m,\  \textrm{Sym}\,(m,K)$
denotes the vector space consisting of all symmetric $m\times m$ matrices
with entries in a field $K.$ If $H$ is a complex matrix or a complex bilinear
form on a complex vector space, $\text{Re}\,H$ and
$\text{Im}\,H$ denote the real part of $H$ and the imaginary part of
$H$ respectively. If $X$ is a space, ${\mathcal S}(X),\
C(X)$ and $C_c^{\infty}(X)$ denotes
the Schwarz space of infinitely differentiable functions on $X$ that
are rapidly decreasing at infinity, the space of all continuous functions
on $X$ and the vector space consisting of all
compactly supported and infinitely differentiable functions on $X$
respectively.

$$\begin{aligned}
\BZ^{(m,n)}_{\geq 0}=&\left\{\,J=(J_{ka})\in \BZ^{(m,n)}\,
\vert\ J_{ka}\geq 0
\ \text{ for\ all}\ k,a\,\right\},\\
\vert J\vert=&\sum_{k,a}\,J_{k,a},\\
J\pm \epsilon_{ka}=&(J_{11},\cdots,J_{ka}\pm 1,\cdots,J_{mn}),\\
J!=&J_{11}!\cdots J_{ka}!\cdots J_{mn}!.
\end{aligned}$$
For $\xi=(\xi_{ka})\in \BR^{(m,n)}\ \text{ or}\ \BC^{(m,n)}\ \text{ and}\
J=(J_{ka})\in \BZ^{(m,n)}_{\geq 0},$ we denote
$$\xi^J=\xi_{11}^{J_{11}}\,\xi_{12}^{J_{12}}\cdots
\xi_{ka}^{J_{ka}}\cdots\xi_{mn}^{J_{mn}}.$$


\newpage

\centerline{\large \bf Table of Symbols}
\vskip 0.1cm
\def\Cal{\mathcal}
$$\begin{aligned}
{\textsc {Section\ 2\,:}}\ &\ {\mathcal {A}},\ {\widehat {\mathcal A}},\ {\mathcal S},\ \alpha_{\la},\ \alpha_{\la}^*,
\ {\widehat {\Cal O}}_{{\hat {\kappa}}},\ {\widehat {\Cal O}}_{{\hat y}},
\ S_{{\hat {\k}}},\ X_{kl}^0,\ X_{ka},\ {\widehat X}_{lb},\ D_{kl}^0,\ D_{ka},\\
                   &\ {\widehat D}_{lb},\ Z_{kl}^0,\ Y_{ka}^+,
\ Y_{lb}^-,\ E_{kl}^{\bullet},\ R_{kl},\ P_{ka},
\ Q_{lb}\\
{\textsc {Section\ 3\,:}}\ &\ \Omega,\ R^{\Omega}_{\M},\
\vth^{(S)}\left[\begin{matrix} A\\ B\end{matrix}\right]
(\Omega,W),\ Q_{\xi}(W),
\ J(\xi,W),\ l(\xi),\ {\rm {Her}}\,Q,\\
                   &\ \psi(\xi), \ L(Q,l,\psi),\ {\rm {Sym}}\,Q,
\ Th(H,\psi,L),\ \chi_{S,\Omega,A,B},\ q_{S,\Omega},\ H_{S,\Omega}, \\
                   &\ \psi_{S,\Omega},\ L_{\Omega},
\ A_{S,\Omega},\ R_S^{\Omega},\ \Theta,\
Th(H_{S,\Omega},\ \psi_{S,\Omega},L_{\Omega}),
\ R^{\Omega}_{S,A,B},\ \Theta_{A,B},\ \\
                   &\ J_{S,\Omega,A,B}, \
                   {\tilde J}_{S,\Omega,A,B}, \ {\Cal A}_{S,\Omega},
\ {\Cal L}_X,\ \varphi_f\\
{\textsc {Section\ 4\,:}}\ & \ U_\sigma,\ (\,\,,\,\,)_{\Cal H},\ {\Cal H}_{\sigma},\ k_g,
\ s_g,\ T_{{\hat {\k}}},\ T_{{\hat x},{\hat y}},
\chi_{\hat x}\\
{\textsc {Section\ 5\,:}}\ & \ G,K,\ k_g,\ s_g,\ U_{\s_c},\ U_c,\ {\Cal H}^{\s_c},\ {\mathcal H}^c,
\ {\Cal H}_{\s_c},\ {\mathcal H}_c,\ \Phi_c,\ dU_c(X),\ f_{c,J}\\
{\textsc {Section\ 6\,:}}\  & \ P_{ka},\ Q_{lb},\ {\bold A},\ J,\ J_{\BC},\ V^+,\ V^-,
\ T,\ {\bold H},\ V_*,\ G_{\BC},\ z^0,\ z^1,\\
                   &\ R^+,\ R^-,\ z^+,\ z^-,\ U^{F,c},\ {\Cal H}^{F,c},\ \delta_c,
\ {\Cal H}_{F,c},\ (\,\,,\,\,)_{F,c},\ \Lambda,\ \Lambda_f,\\
                     &\ \Delta,\ \Delta_{\psi},
\ d\mu(W),\ {\Cal H}_{m,n},\ \Phi_J(W), \kappa(W,W'), \parallel f\parallel_{\M}, \\
                     &\ (\,\,,\,\,)_{\M},\ d\mu_{\M}(W),
\ \Phi_{\M,J}(W),\ k(U,W),\ {\Cal I}(W,W'), \ k_{\M}(U,W),\\
                     &\ I_{\M}(W,W'),\ U^{S,\M},\ I_{\M},\ h_J,\ A_{\M}(U,W),
\ dU^{F,\M}(X)\\
{\textsc {Section\ 7\,:}}\ & \ L_B^*,\ \G_L,\ \G_{L_B^*},\ {\Cal Z}_0,\ \phi_{k,l},
\ \phi_{\M,q},\ \pi_{\M,q},
\ {\Cal H}_{\M},\ {\Cal T},\ \phi_{\M,\a},\\
                    &\ {\Cal H}_{\M,\a},\ \vth_{\M,\a},q_{\M},
\ \varphi_{\M,q_{\M}},\ {\Cal H}_{\M,q_{\M}},
\ \pi_{\M,q_{\M}},\ \pi_{\M,q_{\M}},\ {\bold H}_{\M,q_{\M}},\\
\ &\ E_{\phi},\ F_{\phi},
\ F_{\Omega,\phi},\ \vth_{\Omega,\phi}\\
{\textsc {Section\ 8\,:}}\ & \ \frak g \ ,{\frak g}^*,\ F(a,b,c),
\ Ad_G^*,\ \Omega_{a,b},
\ \Omega_c,\ {\Cal O}(G),\ {\widehat G},\ B_F,
\ ad_{\frak g}^*,\ G_F,\\ &\ {\frak g}_F,
                    \ ,\ \textrm{rad}\,B_F,\ \Omega_F,\ {\widetilde X},
\ B_{\Omega_F},\ \pi_{a,b},\ \frak k ,\ \chi_{c,\frak k},\ \pi_{c,\frak k},\ \pi_c,
\ \chi_c,\ \pi_c^1,\\
                   &\ {\Cal C}{\Cal F}_{\frak g},\ C_c^{\infty}(G),
\ C_c^{\infty}(\frak g ),\ C(\frak g^*),
\ {\Cal S}(G/{\Cal Z}),\ \pi_c^1,\ L^2(G/Z,\chi_c),\\
                   &\ TC\big(L^2\big(\BR^{(m,n)},d\xi\big)\big),
\ HS\big(L^2\big(\BR^{(m,n)},d\xi\big)\big),
\ Ad_K^*,\ \omega_{b,c},\ \chi_{b,c},\\
& \ p(\Omega_c)\\
{\textsc {Section\ 9\,:}}\ &\ dU_{I_m}(X),\ A_{ka}^+,\ A_{lb}^-,\ C_{kl},
\ f_0,\ f_J,\ h_J,\ H_{ka},\ P_J,\ \partial_{ka},\ U(X),\\
                   &\ A^+,\ A^-,\ c_{k,p},\ d_{k,p},\ b_{k,p}\\
{\textsc {Section\ 10\,:}}\ & \ d\xi_{\Omega,\M},\ {\Cal T},\ {\Cal L},
\ \Phi_J^{(\M)}\left[\begin{matrix} A_{\a}\\ 0\end{matrix}\right]
(\Omega\vert\,\cdot\,),\ \Gamma_G,
\ H_{\Omega}^{(\M)}\left[\begin{matrix} A_{\a}\\ 0\end{matrix}\right],
\ \rho,\ R(c),\\
                    &\ f_{\Omega,J}^{(\M)},\ \Phi_{\Omega,\a}^{(\M)},
\ \Delta_{\Omega,\M},\ H_J(\xi),\ \vth_{\M,\a,J},
\ H_J^{(\M)}\left[\begin{matrix} A_{\a}\\ 0\end{matrix}\right]
(\Omega\vert\,\cdot\,)\\
\end{aligned}$$

\newpage
$$\begin{aligned}
{\textsc {Section\ 11\,:}}\ & \ Sp(n,\BR),\ J_n,\ \Gamma_n,\ (\Gamma_n)_\Omega,\ \Gamma_n(q),
\ \Gamma_{\vartheta,n},\ \Gamma_{n,0}(q), \Om^*,\ X^*,\ Y^*,    \\
& \ dM_\Om,\ T_\Om(\BH_n),\ (V,B),\ L^{\perp},\ Sp(B),\ \tau(L_1,L_2,L_3),\ \Lambda, \ \widetilde{\Lambda},\ {\mathscr U},\\
& \ \tau(L_1,L_2,\cdots,L_k),\ U(L_1,u_1;{\mathscr U},L_2),\ \pi:\widetilde{\Lambda}\lrt\Lambda,\ L_*,\
\widetilde{Sp(B)_*},\\
& \ {\mathscr E},\ {\mathscr W}({\mathscr E},L_2),\ (V,\varepsilon),\ \delta(A),\ g_{M,L},\ \xi \big( (L_1,\varepsilon_1),(L_2,\varepsilon_2)\big),
\  L^+,\ s_L(g),\\
& \ s \big( (L_1,\varepsilon_1),(L_2,\varepsilon_2)\big),\ {\widetilde s}_*,\ \ Sp(B)_*,\ c_*(g_1,g_2),\ \varphi(g,n),\ Mp(B)_*      \\
{\textsc {Section\ 12\,:}}\ & \ ds^2,\ dv_n,\ R(\Om_0,\Om_1),\ \rho(\Om_0,\Om_1),\ {\mathbb D}_n,\ \Psi,\ T,
\ G_*,\ SU(n,n),\\
\ & \ P^+,\ K_*,\ ds_*^2,\ \Delta_*,\ K,\ {\mathfrak s}{\mathfrak p}(n,\BR),\ {\mathfrak k},\ {\mathfrak p},\
\psi,\ \delta,\ {\mathbb T}_n,\ {\rm {Pol}}({\mathbb T}_n),\ \Phi, \\
\ & \ f_Z,\ \Phi,\ {\mathcal P}_n,\ {\mathcal R}_n,\ {\partial}{\mathcal R}_n.\ d\mu_n, {\mathcal F}_n\\
{\textsc {Section\ 13\,:}}\ & \ U_c,\ G^J,\ U_c^M,\ R_c,\ \alpha_c(M_1,M_2),\ J(M,\Omega),\ J^*(M,\Omega),
\ Sp(n,\BR)_*,\\
& \ {\widetilde R}_c,
\ s_c(M),\ Mp(n,\BR),\ \omega_c,\ t_b,\ d_a,\ \sigma_n,\ T_c(t_b),\ A_c(d_a),\ B_c(\sigma_n),\\
& \ O(m),\ (\sigma,V_{\sigma}),\ L^2\big( \BR^{(m,n)};\sigma\big),\ \omega_c(\sigma),
\ {\widehat{O(m)}},\ \Sigma_m, \ K:=U(n),\\
& \ {\widehat K},\ {\mathcal O}({\mathbb H}_n,V_\tau),
\ T_\tau,\ \Phi_\tau,\ {\mathfrak H}, \ \tau(\sigma),\ {\mathfrak H}(\sigma),\ \sigma^*,\ {\mathscr F}_\sigma  \\
{\textsc {Section\ 14\,:}}\ & \ {\mathscr F}^{(c)},\ J_m(M,\Omega)\\
{\textsc {Section\ 15\,:}}\ & \ A(S,T),\ S,\ \vartheta_S(\Omega),\ h(\Omega),\ \vartheta_{S;A,B}(\Omega),\
\vartheta(\Omega;a,b),\ \{a,b\},\\
& \ \gamma\diamond \begin{pmatrix} a\\ b\end{pmatrix},\ {\mathscr C}^e,
\ \nu(\gamma),\ t_S,\ k_n,\ \Delta^{(n)}(\Omega),\ \nu_S(\gamma)\\
{\textsc {Section\ 16\,:}}\ & \ \vartheta_S \!\left[ \begin{array}{c} \alpha \\ \beta \end{array} \right]\!(\Omega,Z),
\ \chi\! \left[ \begin{array}{c} \alpha \\ \beta \end{array} \right]\!(N),\ {\mathfrak P}_{m,n},\
\vartheta_{S,P}\! \left[ \begin{array}{c} \alpha \\ \beta \end{array} \right]\!(\Omega,Z),\\
& \ \vartheta_{S,P}(\Omega,Z),
\ \vartheta_{S,P}\! \left[ \begin{array}{c} \alpha \\ \beta \end{array} \right]\!(\Omega),
 \ P(\partial),\ {\mathfrak P}_N,\ \langle P,Q\rangle,\ {\mathfrak H}(S),\\
& \ I,\ h_{ij},
\ T=(t_{kl}),\
 {\mathfrak P}_{m,n},
\ {\mathfrak H}(S)_\BR,\ I_\BR, \ f_{A,B},\ P_{A,B},\ {\mathfrak H}(S)^{\perp},\\
& \ O(S), \ {\widetilde P}(Z),
\ \vartheta_{W}\! \left[ \begin{array}{c} \alpha \\ \beta \end{array} \right],\ \widetilde{GL(n,\BC)},
\ {\mathcal L}^{\frac 12},\ {\widetilde M},\ {\mathcal L}^{\frac k2}, \ {\mathfrak H}_{m,n}(\rho)     \\
{\textsc {Section\ 17\,:}}\ & \ (\pi,V_\pi),\ {\mathscr F},\ \theta,\ \Theta(\Omega),\
\Theta_{\mathcal M}(\Omega),\ \vartheta, \ {\mathscr F}^{\mathcal M},\ \omega_{\mathcal M},\ {\widehat f} \\
{\textsc {Section\ 18\,:}}\ & \  {\mathbb H}_n\times \BC^{(m,n)},\ {\mathbb H}_{n,m},\ \Gamma_{n,m},\ E_{kj},\ F_{kj}(\Omega),
\ L_\Omega,\ \Omega_\flat,\ A_\Omega,\ \Delta_{n,m},\\
\ & \ {\mathcal F}_{n,m},\ \Delta_\Omega,\ L^2(A_\Om),
\ {\rm {Im}}\,\Omega,\ E_{\Omega;A,B}(Z),\ ds_\Omega^2,\ \Gamma^J,\ H_\BZ^{(n,m)},\ ||f||_\Omega,\\
\ & \ dv_\Omega,\ (f,g)_\Omega,\ L^2(T),\ E_{A,B}(W),\ \Delta_T,\ \Phi_\Omega    \\
\end{aligned}$$

\end{section}

\newpage
\def\MA{\mathcal A}
\def\MS{\mathcal S}

\newpage


\begin{section}{{\large\bf The Heisenberg Group}}
\setcounter{equation}{0}

For any two positive integer $m$ and $n$, we let
$$H_{\Bbb R}^{(n,m)}=\left\{\,(\lambda,\mu,\kappa)\,\Big|\ \lambda,\mu\in \Bbb
R^{(m,n)},
\ \kappa\in \Bbb R^{(m,m)},\ \kappa+\mu\,^t\!\lambda\ \textrm{symmetric} \ \right\}$$
the Heisenberg group endowed with the following multiplication law
\begin{equation}(\la,\mu,\kappa)\circ (\la_0,\mu_0,\kappa_0):=(\la+\la_0,\mu+\mu_0,\kappa+
\kappa_0+\la\,^t\!\mu_0-\mu\,^t\!\la_0).
\end{equation}
We observe that
$H_{\BR}^{(n,m)}$ is a 2-step nilpotent Lie group. It is easy to see that
the inverse of an element $(\la,\mu,\kappa)\in H_{\BR}^{(n,m)}$ is given by
$$(\la,\mu,\kappa)^{-1}=(-\la,-\mu,-\kappa+\la\,^t\!\mu-\mu\,^t\!\la).$$
Now we put
\begin{equation}
[\la,\mu,\kappa]=(0,\mu,\kappa)\circ (\la,0,0)=(\la,\mu,\kappa-\mu\,^t\!
\la).
\end{equation}
Then $H_{\BR}^{(n,m)}$ may be regarded as a group
equipped with the following multiplication
\begin{equation}
[\la,\mu,\kappa]\diamond [\la_0,\mu_0,\kappa_0]=[\la+\la_0,\mu+\mu_0,
\kappa+\kappa_0+\la\,^t\!\mu_0+\mu_0\,^t\!\la].
\end{equation}
The inverse of $[\la,\mu,\kappa]\in H_{\BR}^{(n,m)}$ is given by
$$[\la,\mu,\kappa]^{-1}=[-\la,-\mu,-\kappa+\la\,^t\!\mu+\mu\,^t\!\la].$$
We set
\begin{equation}
{\mathcal A}=\left\{\,[0,\mu,\kappa]\in H_{\BR}^{(n,m)} \,\Big|\ \mu\in \BR^{(m,n)},\
\kappa=\,^t\!\kappa\in \BR^{(m,m)}\ \right\}.
\end{equation}
Then ${\mathcal A}$ is a commutative normal subgroup of $H_{\BR}^{(n,m)}$. Let
${\widehat {\mathcal A}}$ be the Pontrajagin dual of $\mathcal A$, i.e., the commutative group
consisting of all unitary characters of $\MA$. Then ${\widehat {\MA}}$ is
isomorphic to the additive group $\BR^{(m,n)}\times  \textrm{Sym}\,(m,\BR)$ via
\begin{equation}
\langle a,{\hat {a}}\rangle:=e^{2\pi i\s({\hat {\mu}}\,^t\!\mu+{\hat {\kappa}}\kappa)},
\ \ \ a=[0,\mu,\kappa]\in \MA,\ {\hat {a}}=({\hat {\mu}},{\hat {\kappa}})
\in {\widehat {\MA}}.
\end{equation}
We put
\begin{equation}
\MS=\left\{\,[\la,0,0]\in H_{\BR}^{(n,m)}\,\Big|\ \la\in \BR^{(m,n)}\,
\right\}\cong \BR^{(m,n)}.
\end{equation}
Then $\MS$ acts on $\MA$ as follows:
\begin{equation}
\alpha_{\la}([0,\mu,\kappa]):=[0,\mu,\kappa+\la\,^t\!\mu+\mu\,^t\!\la],
\ \ \ \alpha_{\la}=[\la,0,0]\in \MS.
\end{equation}
It is easy to see that the Heisenberg group $\left( H_{\BR}^{(n,m)},
\diamond\right)$ is isomorphic to the semidirect product $G_H:=\MS\ltimes \MA$
of $\MA$ and $\MS$ whose multiplication is given by
$$(\la,a)\cdot (\la_0,a_0)=(\la+\la_0,a+\alpha_{\la}(a_0)),\ \ \la,\la_0\in \MS,\
a,a_0\in \MA.$$
On the other hand, $\MS$ acts on ${\widehat {\MA}}$ by
\begin{equation}
\alpha_{\la}^{*}({\hat {a}}):=({\hat {\mu}}+2{\hat {\kappa}}\la,
{\hat {\kappa}}),\ \ [\la,0,0]\in \MS,\ \ {\hat a}=({\hat {\mu}},{\hat {\kappa}})\in
{\widehat {\MA}}.
\end{equation}
Then we have the relation $\langle \alpha_{\la}(a),{\hat {a}}\rangle=\langle a,\alpha_{\la}^{*}
({\hat {a}})\rangle$ for all $a\in \MA$ and ${\hat {a}}\in {\widehat {\MA}}.$

\vskip 0.2cm
We have two types of $\MS$-orbits in ${\widehat {\MA}}.$
\vskip 0.2cm

\noindent
{\sc Type}
{\sc I.}\ \ Let ${\hat {\kappa}}\in  \textrm{Sym}\,(m,\BR)$ with ${\hat {\kappa}}
\neq 0.$ The $\MS$-orbit of ${\hat {a}}({\hat {\kappa}}):=(0,{\hat {\kappa}})
\in {\widehat {\MA}}$ is given by
\begin{equation}
{\widehat {\mathcal O}}_{\hat {\kappa}}:=\left\{\,
(2{\hat {\kappa}}\la,{\hat {\kappa}})
\in {\widehat {\MA}}\ \Big|\ \la\in \BR^{(m,n)}\,\right\}\,\cong\, \BR^{(m,n)}.
\end{equation}

\vskip 0.2cm
\noindent
{\sc Type} {\sc II.}\ \ Let ${\hat {y}}\in \BR^{(m,n)}.$ The $\MS$-orbit
${\widehat {\mathcal  O}}_{\hat {y}}$ of ${\hat {a}}({\hat {y}}):=({\hat {y}},0)$
is given by
\begin{equation}
{\widehat {\mathcal  O}}_{\hat {y}}:=\left\{\,({\hat {y}},0)\,\right\}={\hat {a}}
({\hat {y}}).
\end{equation}
We have
$${\widehat \MA}=\left(\bigcup_{{\hat {\kappa}}\in  \textrm{Sym}(m,\BR)}
{\widehat {\mathcal  O}}_{\hat {\kappa}}\right)\bigcup \left( \bigcup_{{\hat {y}}\in
\BR^{(m,n)}}{\widehat {\mathcal  O}}_{\hat {y}}\right) $$
as a set. The stabilizer $\MS_{\hat {\kappa}}$ of $\MS$ at ${\hat {a}}({\hat
{\kappa}})=(0,{\hat {\kappa}})$ is given by
\begin{equation}
\MS_{\hat {\kappa}}=\{0\}.
\end{equation}
And the stabilizer $\MS_{\hat {y}}$ of $\MS$ at ${\hat {a}}({\hat {y}})=
({\hat {y}},0)$ is given by
\begin{equation}
\MS_{\hat {y}}=\left\{\,[\la,0,0]\,\Big|\ \la\in \BR^{(m,n)}\,\right\}=\MS
\,\cong\,\BR^{(m,n)}.
\end{equation}
\ \ \ The following matrices
$$\begin{aligned}
X_{kl}^0:=&\begin{pmatrix} 0&0&0&0\\ 0&0&0&{\frac 12}(E_{kl}+E_{lk})\\
0&0&0&0\\ 0&0&0&0 \end{pmatrix},\ \ \ \ 1\leq k\leq l\leq m,\\
X_{ka}:=&\begin{pmatrix} 0&0&0&0\\ E_{ka}&0&0&0\\
0&0&0&-^t\!E_{ka}\\ 0&0&0&0 \end{pmatrix},\ \ \ 1\leq k\leq m,\ 1\leq a\leq n,\\
{\widehat X}_{lb}:=&\begin{pmatrix} 0&0&0&^t\!E_{lb}\\ 0&0&E_{lb}&0\\
0&0&0&0\\ 0&0&0&0 \end{pmatrix},\ \ \ 1\leq l\leq m,\ 1\leq b\leq n
\end{aligned}$$
form a basis of the Lie algebra ${\mathcal  H}_{\BR}^{(n,m)}$ of the real
Heisenberg group $H_{\BR}^{(n,m)}.$ Here $E_{kl}$ denotes the $m\times m$
matrix with entry $1$ where the $k$-th row and the $l$-th column meet,
all other entries $0$ and $E_{ka}$\,(resp. $E_{lb}$) denotes the $m\times n$
matrix with entry $1$ where the $k$-th\,(resp. the $l$-th) row and the
$a$-th\,(resp. the $b$-th) column meet, all other entires $0$. By an easy
calculation, we see that the following vector fields
$$\begin{aligned}
D_{kl}^0:=&{{\partial}\over {\partial\kappa_{kl}} },\ \ \ 1\leq k\leq m,\\
D_{ka}:=&{{\partial}\over {\partial\la_{ka}}}-\left(\,
\sum_{p=1}^k\,\mu_{pa}{{\partial}\over {\partial\k_{pk}}}+
\sum_{p=k+1}^m\,\mu_{pa}{{\partial}\over {\partial\k_{kp}}}\right),\ \
1\leq k\leq m,\ 1\leq a\leq n,\\
{\widehat D}_{lb}:=&{{\partial}\over {\partial\mu_{lb}}}+\left(\,
\sum_{p=1}^l\,\la_{pb}{{\partial}\over {\partial\k_{pl}}}+
\sum_{p=l+1}^m\,\la_{pb}{{\partial}\over {\partial\k_{lp}}}\right),\ \
1\leq k\leq m,\ 1\leq a\leq n
\end{aligned}$$
form a basis for the Lie algebra of left-invariant vector fields on the
Lie group $H_{\BR}^{(n,m)}.$

\vskip2mm

\begin{lemma}
We have the following
{\it Heisenberg\ commutation\ relations}
$$\begin{aligned}
\ \ \ [\,D_{kl}^0,D_{st}^0\,]=&[\,D_{kl}^0,D_{sa}\,]
=[\,D_{kl}^0,{\widehat D}_{sa}\,]=0,\\
\ \ \ [\,D_{ka},D_{lb}\,]=&[\,{\widehat D}_{ka},{\widehat D}_{lb}\,]=0,\\
\ \ \ [\,D_{ka},{\widehat D}_{lb}\,]=&\,2\,\delta_{ab}\,D_{kl}^0,
\end{aligned}$$
where $1\leq k,l,s,t\leq m,\ 1\leq a,b\leq n$ and $\delta_{ab}$ denotes
the Kronecker delta symbol.
\end{lemma}

\noindent
{\it Proof.} The proof follows from a straightforward calculation.
\hfill $\square$\vskip2mm
We put
$$\begin{aligned}
Z_{kl}^0&:=-\sqrt{-1}\, D_{kl}^0,\ \ \ \ 1\leq k\leq l\leq m,\\
Y_{ka}^{+}&:={\frac 12}\,(D_{ka}+\sqrt{-1}\,{\widehat D}_{ka}),\ \ \
1\leq k\leq m,\ 1\leq a\leq n,\\
Y_{lb}^{-}&:={\frac 12}\,(D_{lb}-\sqrt{-1}\,{\hat D}_{lb}),\ \ \
1\leq l\leq m,\ 1\leq b\leq n.
\end{aligned}$$
Then it is easy to see that the vector fields $Z_{kl}^0,\,Y_{ka}^{+},\,
Y_{lb}^{-}$ form a basis of the complexification of the real Lie algebra
${\mathcal  H}_{\BR}^{(n,m)}.$

\vskip2mm
\begin{lemma}
We have the following commutation relations
$$\begin{aligned}
\ \ \ [Z_{kl}^0,Z_{st}^0]&=[Z_{kl}^0,Y_{sa}^{+}]=[Z_{kl}^0,Y_{sa}^{-}]=0,\\
\ \ \ [Y_{ka}^+,Y_{lb}^+]&=[Y_{ka}^-,Y_{lb}^-]=0,\\
\ \ \ [Y_{ka}^+,Y_{lb}^-]&=\delta_{ab}\,Z_{kl}^0,
\end{aligned} $$
where $1\leq k,l,s,t\leq m$ and $1\leq a,b\leq n.$
\end{lemma}
\noindent
{\it Proof.} It follows immediately from Lemma 2.1.
\hfill $\square$

\vskip2mm
We let $E_{kl}^{\bullet}:=E_{kl}+E_{lk}$ for $1\leq k\leq l\leq m.$ We put
$$\begin{aligned}
R_{kl}(r):=&\,\exp \big(2rX_{kl}^0\big)=(0,0,rE_{kl}^{\bullet}),\ \ \ r\in \BR,\\
P_{sa}(x):=&\,\exp \big(xX_{sa}\big)=(xE_{sa},0,0),\ \ \ x\in \BR,\\
Q_{tb}(y):=&\,\exp \big(y{\widehat X}_{tb}\big)=(0,yE_{tb},0),\ \ \ y\in \BR,
\end{aligned}$$
where $1\leq k\leq l\leq m,\ 1\leq s,t\leq m$ and $1\leq a,b\leq n.$ Then
these one-parameter subgroups generate the Heisenberg group
$H_{\BR}^{(n,m)}.$ They satisfy the $\textsf{Weyl\ commutation\ relations}$\,:
$$P_{sa}(x)\circ Q_{sa}(y)=Q_{sa}(y)\circ P_{sa}(x)\circ R_{ss}(xy)
\ \ \ (\,\hbox{all\ others\ commute}\,),$$
where $1\leq s\leq m$ and $1\leq a\leq n.$

\vskip 0.2cm
J. von Neumann \cite{Ne} and M. Stone \cite{St} proved the following uniqueness theorem simultaneously and
independently.

\begin{theorem} Let $\pi_1$ and $\pi_2$ be two irreducible unitary representations of the Heisenberg
group $H_{\BR}^{(n,m)}$ such that
\begin{equation*}
\pi_1((0,0,\kappa))=\pi_2((0,0,\kappa))\qquad {\rm for\ all}\ \kappa=\,{}^t\kappa\in \BR^{(m,m)}.
\end{equation*}
Then $\pi_1$ is unitarily equivalent to $\pi_2.$
\end{theorem}

We omit the proof of the above theorem. We refer to \cite{LV} for the proof of Theorem 2.3 in the case $m=1$ and also to
\cite{C} and \cite{Mum2} for more detail.

\end{section}

\newpage


\begin{section}{{\large\bf Theta Functions}}
\setcounter{equation}{0}

 We fix an element $\Omega\in {\mathbb H}_n$ once and for all.
From now on, we put $i=\sqrt {-1}.$ Let $\M$ be a positive
definite, symmetric even integral matrix of degree $m$. A holomorphic
function $f:\BC^{(m,n)}\lrt \BC$ satisfying the following equation
\begin{equation}
f(W+\xi\Omega+\eta)=e^{-\pi i\, \s \{\M(\xi \Omega\,^t\!\xi\,+\,2\,\xi\,^t\!W)\}}f(W),
\ \ \ W\in \BC^{(m,n)}
\end{equation}
for all $\xi,\,\eta\in \BZ^{(m,n)}$ is called a $\textsf{theta function of level}$ $\M$
with respect to $\Omega$. The set $R^{\Omega}_{\M}$ of all theta functions of
level $\M$ with respect to $\Omega$ is a complex vector space of dimension
$\left( \det\,\M\right)^n$ with a basis consisting of theta functions
\begin{equation}
\vth^{(\M)}\A (\Omega,W):=\sum_{N\in \BZ^{(m,n)}}
e^{\pi i\,\s\{\M((N+A)\Omega\,^t\!(N+A)\,+\,2\,W\,^t\!(N+A))\}},
\end{equation}
where $A$ runs over a complete system of the cosets $\M^{-1}\BZ^{(m,n)}/
\BZ^{(m,n)}.$

\vskip2mm
\begin{definition}
Let $S$ be
a positive definite, symmetric real matrix
of degree $m$ and let $A,B\in \BR^{(m,n)}.$ We define the theta function
\begin{equation}
\vth^{(S)}\Bb (\Omega,W)=\sum_{N\in \BZ^{(m,n)}}
e^{\pi i\,\s\{ S((N+A)\Omega\,^t\!(N+A)\,+\,2\,(W+B)\,^t\!(N+A))\}}
\end{equation}
with characteristic $(A,B)$ converging normally on ${\mathbb H}_n\times \BC^{(m,n)}.$
\end{definition}

\vskip2mm

We have a general definition of theta functions.\vskip2mm

\begin{definition}
Let $V$ be a complex
vector space and let $L\subset
V$ be a lattice of $V$. A $\textsf{theta\ function}$ on $V$ $\textsf{relative\ to}$
$L$ is a nonzero holomorphic function $\vth$ on $V$ satisfying the following
condition
$$\vth(W+\xi)=e^{2\pi i(Q_{\xi}(W)+c_{\xi})}\vth(W),$$
where $Q_{\xi}$ is a $\BC$-linear form on $V$ and $c_{\xi}$ is an element of
$\BC$, for every $W\in V$ and $\xi\in L.$
\end{definition}

\vskip2mm

If $\vth$ is a theta function on $V$ relative to $L$, then
the mapping $J_{\vth}:L\times V\lrt \BC^*$ defined by
$$J_{\vth}(\xi,W):=e^{2\pi i(Q_{\xi}(W)+c_{\xi})},\ \ \ \xi\in L,\ \ W\in V$$
is easily seen to be an automorphic factor. This means that $J_{\vth}$ satisfies the
following condition
\begin{equation*}
J_{\vth}(\xi_1+\xi_2,W)= J_{\vth}(\xi_1, W+\xi_2)\, J_\vth (\xi_2,W)
\end{equation*}
for all $\xi_1,\xi_2\in L$ and $W\in V.$
We observe that for all $\xi_1,
\xi_2\in L$ and $W\in V,$
$$Q_{\xi_1+\xi_2}(W)+c_{\xi_1+\xi_2}\equiv Q_{\xi_1}(W+\xi_2)+Q_{\xi_2}(W)
+c_{\xi_1}+c_{\xi_2}\ \  \textrm{mod}\ \BZ.$$
$J_\vth$ is called the $\textsf{automorphic factor of the theta function}$ $\vth$ on $V$
$\textsf{relative to}$ $L$.

\vskip2mm
\begin{theorem}
$\text{ (Igusa\,\cite{I},\,p.67).}$ Let
$J:L\times V\lrt \BC^{\times}$ be the
automorphic factor of a theta function $\vth$ on $V$ relative to $L$. Then
there exists a unique triple $(Q,\ell,\psi)$ such that
\begin{equation}
J(\xi,W)=e^{\pi\{Q(W,\xi)+{\frac {1}{2}}Q(\xi,\xi)+2i\,\ell(\xi)\}}\psi(\xi),
\ \ \ \xi\in L,\ W\in V,
\end{equation}
where\vskip2mm
\begin{enumerate}
\item  \ \ $Q$ is a quasi-hermitian form on $V\times V,$
\item\ \  the hermitian form $H:={\rm Her}\,(Q)$ defined by
$$H(W_1,W_2)={1\over {2i}}\,\big\{ Q(iW_1,W_2)-Q(W_1,iW_2)\big\},\ \
W_1,\ W_2\in V$$
is a Riemann form with respect to $L$, that is, $H=\,^t{\overline H} > 0$ and
$({\rm {Im}}\,H)(L\times L)\subset \BZ,$
\item\ \ $\ell:V\lrt \BC$ is a $\BC$-linear form on $V$,
\item \ \ $\psi$ is a {\it second\ degree\ character} of $L$ which is associated
with $A:={\rm {Im}}\,H,$
\item \ \  $\psi$ is strongly associated with $A$.
\end{enumerate}
\end{theorem}

\vskip 0.2cm
\begin{remark} (4) means that
$\psi:L\lrt \BC_1^*$ is
a semi-character of $L$ satisfying
the functional equation
$$\psi(\xi_1+\xi_2)=e^{\pi iA(\xi_1,\xi_2)}\psi(\xi_1)\psi(\xi_2),
\ \ \xi_1,\xi_2\in L.\eqno (*)$$
\end{remark}

\begin{definition}
A theta function with the automorphic factor of the
form (3.4) is called a $\textsf{theta\ function\ of\ type}$
$(Q,\ell,\psi)$. We denote by
$L(Q,\ell,\psi)$ the union of theta functions of type $(Q,\ell,\psi)$ and the
constant 0. A theta function of type $(Q,\ell,\psi)$ is said to be $\textsf{normalized}$ 
if ${\rm Sym}\,Q=0$ and $\ell=0.$ Here ${\rm Sym}\,Q:V\times V\lrt \BC$ is a
symmetric $\BC$-linear form on $V\times V$ defined by
$$({\rm Sym}\,Q)(z,w)={1\over {2i}}\,\big\{ Q(iz,w)+Q(z,iw)\big\},\ \ z,w\in V.$$
\end{definition}

We observe that $Q={\rm Her}\,Q+{\rm Sym}\,Q.$ We note that ${\rm Sym}\,Q=0$ if and only if
$Q={\rm Her}\,Q=H.$ We denote by $Th(H,\psi,L)$ the union of the set of all
normalized theta functions of type $(H,0,\psi)$ and the constant 0. It is
easily seen that if $\vth\in Th(H,\psi,L),$ for all $W\in V,\ \xi\in L,$ we
have
\begin{equation}
\vth(W+\xi)=e^{\pi H(W+{1\over 2}\xi,\xi)}\psi(\xi)\,\vth(W).
\end{equation}
\vskip 0.2cm

\begin{theorem}
Let $S$ be a positive definite, symmetric integral matrix of
degree $m$ and let $A,B$ be two $m\times n$ real matrices. Then for
$\Omega\in {\mathbb H}_n$ and $W\in \BC^{(m,n)},$ we have\vskip2mm\noindent
\begin{enumerate}
\item[$(\theta.1)$]\ \ \ \ $\vth^{(S)}\Bb (\Omega,-W)=\vth^{(S)}\Dd (\Omega,W),$\vskip2mm
\noindent
\item[$(\theta.2)$]\ \ \ \ $\vth^{(S)}\Bb (\Omega,W+\la\Omega+\mu)$\vskip2mm\noindent
\hskip 1cm $= e^{-\pi i\,\s\{ S(\la\Omega\,^t\!\la\,+\,2\,(W+\mu)\,^t\!\la)\} }\,
e^{-2\pi i\,\s(SB\,^t\!\la)}\cdot \vth^{(S)}\left[\begin{matrix} A+\la\\ B+\mu\end{matrix}
\right]
(\Omega,W)$\vskip2mm
for all $\la,\mu\in \BR^{(m,n)}.$\vskip2mm
\noindent
\item[$(\theta.3)$]\ \ \ $\vth^{(S)}\Bb(\Omega,W)=e^{\pi i\,\s\{ S(A\Omega\,^t\!A)
\,+\, 2\, (W+B)\,^t\!A)\}}\,
\vth^{(S)}\left[\begin{matrix} 0\\ 0\end{matrix}\right](\Omega,W+A\Omega+B).$\vskip2mm
Moreover, if $S$ is a positive definite, symmetric integral matrix of degree
$m$, we have\vskip2mm\noindent
\item[$(\theta.4)$]\ \ \ $\vth^{(S)}\x(\Omega,W)=e^{2\pi i\,\s(SA^t\!\eta)}
\,\vth^{(S)}\Bb(\Omega,W).$\vskip2mm\noindent
for\ all\ $\xi,\eta\in \BZ^{(m,n)}.$\vskip2mm\noindent
\item[$(\theta.5)$]\ \ \
$\vth^{(S)}\Bb(\Omega,W+\xi\Omega+\eta)$
$$=e^{-\pi i\,\s\{ S(\xi\Omega\,^t\!\xi\,+\, 2\, W^t\!\xi)\}}\cdot
e^{2\pi i\,\s\{ S(A\,^t\eta-B\,^t\xi)\} }\cdot\vth^{(S)}\Bb(\Omega,W)
  $$
for all $\xi,\eta\in \BZ^{(m,n)}.$\vskip2mm
\end{enumerate}
\end{theorem}
\noindent
{\it Proof.} $(\theta.1)$ follows immediately from the definition (3.3).
$(\theta.2)$
follows immediately from the relation
$$\begin{aligned}
\ \ \ &  (N+A)\Omega\,^t\!(N+A)+ 2(W+\la \Omega+\mu+B)\,^t\!(N+A)\\
&= (N+A+\la)\,\Omega\,^t\!(N+A+\la)+2(W+\mu+B)\,^t\!(N+A+\la)
-(N+A)\Omega\,^t\!\la\\
& \ +\la\Omega\,^t\!(N+A)-\la\Omega\,^t\!\la
-2(W+\mu+B)\,^t\!\la.
\end{aligned}$$
If we put $A=B=0$ and replace $\la,\mu$ by $A,B$ in $(\theta.2)$, then we
obtain $(\theta.3)$. For $\xi,\eta\in \BZ^{(m,n)},$ we have
$$\begin{aligned}
 &  \vth^{(S)}\x(\Omega,W)\\
&=\sum_{N\in \BZ^{(m,n)}}e^{\pi i\, \s\{S((A+N+\xi)\Omega\,^t\!(A+N+\xi)\,+
\, 2\, (W+B)\,^t\!(A+N+\xi))\}}\\
&\ \ \ \ \times  e^{2\pi i\,\s\{S\eta\,^t\!(N+\xi)\}}\cdot e^{2\pi i\,\s(S\,^t\!\eta A)}\\
&=e^{2\pi i\,\s(SA\,^t\!\eta)}\cdot \vth^{(S)}\Bb(\Omega,W).
\end{aligned}$$
Here in the last equality we used the fact that $\s(S\eta\,^t\!(N+\xi))\in
\BZ$ because $S$ is integral. $(\theta.5)$ follows from $(\theta.2),\
(\theta.4)$ and the fact that $\s(S\eta\,^t\!\xi)$ is integral.
$\hfill \square$\vskip2mm

For a positve definite, symmetric real matrix $S$ of degree $m,\ \Omega\in
{\mathbb H}_n$ and $A,B\in \BR^{(m,n)},$ we put
\begin{equation}
\chi_{S,\Omega,A,B}(\xi\Omega+\eta):=\chi_{S,\Omega,A,B}(\xi,\eta):=
e^{2\pi i\,\s\{S(A\,^t\eta-B\,^t\xi)\}},
\end{equation}
where $\xi,\eta\in \BZ^{(m,n)}.$

\vskip 0.2cm
We define
\begin{equation}
q_{S,\Omega}(W)={\frac 12}\,\s\left( SW(\Omega-{\overline \Omega})^{-1}\,^tW
\right),
\ \ W\in \BC^{(m,n)}
\end{equation}
and also define
\begin{equation}
H_{S,\Omega}(W_1,W_2)=2\,i\,\s\left( SW_1(\Omega-{\overline \Omega})^{-1}
\,^t{\overline W}_2\right),\ \
W_1,W_2\in \BC^{(m,n)}.
\end{equation}
It is easy to check that $H_{S,\Omega}$ is a positive hermitian form on
$\BC^{(m,n)}.$
\vskip 3mm
\begin{lemma}
For $W\in \BC^{(m,n)}$ and $l\in \BZ^{(m,n)}\Omega
+\BZ^{(m,n)}$, we have
\begin{equation}
q_{S,\Omega}(W+l)=q_{S,\Omega}(W)+q_{S,\Omega}(l)+\s\left( Sl(\Omega-
{\overline \Omega})^{-1}\,^tW\right)
\end{equation}
and
\begin{eqnarray}
H_{S,\Omega}\left( W+{\frac l2},l \right)&=\s\left( S\left(W+{\frac l2}\right)
({\rm {Im}}\,\Omega)^{-1}
\,^tl\right)\\
& \ \  -2\,i\,\s\left( S\left(W+{\frac l2}\right)\,^t\!\xi\right),\nonumber
\end{eqnarray}
where $l=\xi\Omega+\eta,\ \xi,\eta\in \BZ^{(m,n)}.$
\end{lemma}
\noindent
{\it Proof.} It follows immediately from a straightforward computation.
\hfill \Box\vskip2mm

\begin{lemma}
Let $S$ be a positive definite, symmetric integral matrix
of degree $m$. For $\Omega\in {\mathbb H}_n,$ we let $L_{\Omega}:=\BZ^{(m,n)}\Omega+
\BZ^{(m,n)}$ be the lattice in $\BC^{(m,n)}.$ We define the mapping
$\psi_{S,\Omega}:L_{\Omega}\lrt \BC_1^*$ by
\begin{equation}
\psi_{S,\Omega}(\xi\Omega+\eta)=e^{\pi i\,\s(S\eta\,^t\!\xi)},\ \ \ \xi,\eta
\in \BZ^{(m,n)}.
\end{equation}
Then \vskip2mm
(a) $\psi_{S,\Omega}$ is a second-degree character of $L_{\Omega}$ associated
with {\rm {Im}}\,$H_{S,\Omega}.$\vskip2mm
(b) $\psi_{S,\Omega}\cdot \chi_{S,\Omega,A,B}$ is a second-degree character
of $L_{\Omega}$ associated with {\rm {Im}}\,$H_{S,\Omega}.$
\end{lemma}
\noindent
{\it Proof.} {\it (a)} We fix $l=\xi\Omega+\eta\in L_{\Omega}$ with $\xi,\eta\in
\BZ^{(m,n)}.$ We define $f_l:L_{\Omega}\lrt \BC_1^{\times}$ by
$$f_l(l_1):={{\psi_{S,\Omega}(l_1+l)}\over
{\psi_{S,\Omega}(l_1)\,\psi_{S,\Omega}(l)}}\ \,,\ \ \ l_1\in L_{\Omega}.$$
It is easy to see that $f_l$ is a character of $L_{\Omega}$ and hence to
see that the map from $L_{\Omega}\times L_{\Omega}$ to $\BC_1^*$
defined by
$$(l_1,l_2)\longmapsto {{\psi_{S,\Omega}(l_1+l_2)}\over
{\psi_{S,\Omega}(l_1)\,\psi_{S,\Omega}(l_2)}}$$
is a bicharacter of $L_{\Omega}$, i.e., a character of $L_{\Omega}$ in
$l_1$ and $l_2.$ Hence $\psi_{S,\Omega}$ is a second degree character of
$L_{\Omega}$. In order to show that $\psi_{S,\Omega}$ is associated with
$H_{S,\Omega},$ it is enough to prove that
\begin{equation}
\psi_{S,\Omega}(l_1+l_2)=e^{\pi iA_{S,\Omega}(l_1,l_2)}\,
\psi_{S,\Omega}(l_1)\,\psi_{S,\Omega}(l_2)
\end{equation}
for all $l_1,l_2\in L_{\Omega}.$ Here $A_{S,\Omega}$ denotes the imaginary
part of the positive hermitian form $H_{S,\Omega}.$ By an easy computation,
we have
\begin{equation}
A_{S,\Omega}(l_1,l_2)=\s\{ S(\xi_1\,^t\!\eta_2-\eta_1\,^t\!\xi_2)\},
\end{equation}
where $l_i=\xi_i\Omega+\eta_i\in L_{\Omega}\,(\,1\leq i\leq 2\,).$ Hence
(3.12) follows immediately from (3.13).

\vskip2mm
\noindent
{\it (b)} We fix $l=\xi\Omega+\eta\in L_{\Omega}$ with $\xi,\eta\in \BZ^{(m,n)}.$
We put ${\widetilde {\psi}}_{S,\Omega,A,B}:=\psi_{S,\Omega}\cdot
\chi_{S,\Omega,A,B}.$ Then the map ${\tilde f}_l:L_{\Omega}\lrt
\BC_1^*$ defined by
$${\tilde f}_l(l_1)={{{\widetilde {\psi}}_{S,\Omega,A,B}(l_1+l)}\over
{{\widetilde {\psi}}_{S,\Omega,A,B}(l_1)\,{\widetilde {\psi}}_{S,\Omega,A,B}(l_2)}}
\ \, ,\ \ \
l_1\in L_{\Omega}$$
is a character of $L_{\Omega}.$ So ${\widetilde {\psi}}_{S,\Omega,A,B}$ is a
second degree character of $L_{\Omega}.$ In order to show that
${\widetilde {\psi}}_{S,\Omega}$ is associated with $A_{S,\Omega},$ it suffices
to prove that
\begin{equation}
{\widetilde {\psi}}_{S,\Omega,A,B}(l_1+l_2)=e^{\pi iA_{S,\Omega}(l_1,l_2)}\,
{\widetilde {\psi}}_{S,\Omega,A,B}(l_1)\,{\widetilde {\psi}}_{S,\Omega,A,B}(l_2)
\end{equation}
for all $l_1,l_2\in L_{\Omega}.$ An easy calculation yields (3.14).
$\hfill \square$
\vskip2mm

\begin{theorem}
We assume that $S$ is a positive definite, symmetric
integral matrix of degree $m$. Let $\Omega\in {\mathbb H}_n.$
We denote by $R_S^{\Omega}$ the vector space of all holomorphic functions
$f:\BC^{(m,n)}\longrightarrow \BC$ satisfying the transformation behaviour
$$f(W+\xi\Omega+\eta)=e^{-\pi i\,\s\{S(\xi\Omega\,^t\!\xi\,+\,2\,\xi\,^tW)\} }
\,f(W),\ \ \ W\in \BC^{(m,n)}$$
for all $\xi,\eta\in \BZ^{(m,n)}.$
Then the mapping
$$\Theta: R^{\Omega}_S\lrt Th(H_{S,\Omega},\psi_{S,\Omega},L_{\Omega})$$
defined by
$$\left(\Theta(F)\right)(W):=e^{2\pi i\,q_{S,\Omega}(W)}F(W),\ \ \
F\in R_S^{\Omega},\ W\in \BC^{(m,n)}$$
is an isomorphism of vector spaces, where $L_{\Omega}$ and $\psi_{S,\Omega}$
are the same as in Lemma 3.8.
\end{theorem}
\noindent
{\it Proof.} First of all, we will show the image $\Theta(R_S^{\Omega})$ is
contained in $Th(H_{S,\Omega},\psi_{S,\Omega},L_{\Omega}).$ If
$F\in R_S^{\Omega},\ W\in \BC^{(m,n)}$ and $l=\xi\Omega+\eta\in L_{\Omega},$
then we have
$$\begin{aligned}
\Theta(F)(W+l)&=e^{2\pi i\,q_{S,\Omega}(W+l)}\,F(W+l)\\
&=e^{2\pi i\{ q_{S,\Omega}(W)\,+\, q_{S,\Omega}(l)\,+\, \s(Sl(\Omega-
{\overline {\Omega}}^{-1}\,^t\!W)\}}\\
& \ \ \ \ \times e^{-\pi i\,\s \{ S(\xi\Omega\,^t\!\xi\,+\,2\,W\,^t\!\xi)\}}\,F(W)
\ \ \ \ \ (\,\textrm{by\ Lemma}\ 3.7\,)\\
&=e^{2\pi i\,\s\{ S(W+{\frac l2})(\Omega-{\overline {\Omega}})^{-1}\,^t\!l\}}\\
&  \ \ \ \  \times e^{-\pi i\,\s\{ S(\xi\Omega\,^t\xi\,+\, 2\,W\,^t\xi\,+\,2\,W\,^t\xi)\}}\cdot
\Theta(F)(W)\\
&=e^{\pi H_{S,\Omega}(W+{\frac l2},\,l)}\cdot e^{-\pi i\,\s(S\eta\,^t\xi)}\,
\Theta(F)(W)\\
&=e^{\pi H_{S,\Omega}(W+{\frac l2},\,l)}\,\psi_{S,\Omega}(l)\,
\Theta(F)(W).
\end{aligned}$$
Thus $\Theta(F)$ is contained in the set $Th(H_{S,\Omega},\psi_{S,\Omega},
L_{\Omega}).$ It is easy to see that the mapping $\Theta$ is an
isomorphism.
\hfill \Box
\vskip2mm
\begin{proposition}
Let $S$ be as above in Theorem 3.9 and $A,B\in
\BR^{(m,n)}.$ We denote by $R_{S,A,B}^{\Omega}$ the union of the set of all
theta functions with characteristic $(A,B)$ with respect to $S$ and
$\Omega$ and the constant
$0$. Then we have an isomorphism
$$R_{S,A,B}^{\Omega}\cong Th(H_{S,\Omega},\psi_{S,\Omega}\cdot
\chi_{S,\Omega,A,B},L_{\Omega}).$$
\end{proposition}
\noindent
{\it Proof.} First we observe that $\psi_{S,\Omega}\cdot
\chi_{S,\Omega,A,B}$ is a second degree character of $L_{\Omega}$ associated
with $A_{S,\Omega}$\,(\,cf.\,Lemma 3.8 (B)\,). In a similar way in the
proof of Theorem 3.9, using $(\theta.5)$, we can show that the mapping
$$\Theta_{A,B}(f)(W):=e^{2\pi i\,q_{S,\Omega}(W)}\,f(W),\ \ \
f\in R_{S,A,B}^{\Omega},\ W\in \BC^{(m,n)}$$
has its image in $Th(H_{S,\Omega},\psi_{S,\Omega}\cdot \chi_{S,\Omega,A,B},
L_{\Omega}).$
\hfill \Box
\vskip2mm
\begin{proposition}
Let $S$ be as above in Theorem 3.9 and let $A,B
\in \BR^{(m,n)}.$ Then we have an isomorphism
$$Th(H_{S,\Omega},\psi_{S,\Omega},L_{\Omega})\cong
TH(H_{S,\Omega},\psi_{S,\Omega}\cdot \chi_{S,\Omega,A,B},L_{\Omega}).$$
\end{proposition}
\noindent
{\it Proof.} The proof follows from the fact that the dimension of the
complex vector space $Th(H_{S,\Omega},\psi_{S,\Omega},L_{\Omega})$ is equal
to that of $Th(H_{S,\Omega},\psi_{S,\Omega}\cdot \chi_{S,\Omega,A,B},
L_{\Omega})$. It is well known that the dimension of $Th(H_{S,\Omega},
\psi_{S,\Omega},L_{\Omega})$ is equal to the Pfaffian of $A_{S,\Omega}$
relative to $L_{\Omega}$\,(\,cf.\,\cite{I},\,p.72\,).
\hfill \Box
\vskip2mm

\begin{remark}
From Theorem 3.9, Proposition 3.10 and
Proposition 3.11, $R_S^{\Omega}$
is isomorphic to $R_{S,A,B}^{\Omega}$ for any $A,B\in \BR^{(m,n)}.$
\end{remark}
\vskip2mm

Now as before, we fix an element $\Omega\in {\mathbb H}_n$ and let $S$ be a positive
symmetric integral matrix of degree $m$. Then the lattice $L:=\BZ^{(m,n)}
\times \BZ^{(m,n)}$ acts on $\BC^{(m,n)}$ freely by
$$(\xi,\,\eta)\,\cdot \,W=W+\xi\Omega+\eta,\ \ \ \xi,\eta\in \BZ^{(m,n)},
\ \ W\in \BC^{(m,n)}.$$

\begin{lemma}
Let $A,\,B\in \BR^{(m,n)}.$ Then the mapping
$J_{S,\Omega,A,B}:L\times \BC^{(m,n)}\lrt \BC^*$ defined by
\begin{equation}
J_{S,\Omega,A,B}(l,W):=e^{\pi i\, \s\{ S(\xi\Omega\,^t\xi\,+\,2\,W\,^t\xi)\} }
\cdot e^{-2\pi i\,\s\{ S(A\,^t\eta-B\,^t\xi)\} },
\end{equation}
where $l=(\xi,\eta)\in L$ and $W\in \BC^{(m,n)}.$ Then
$J_{S,\Omega,A,B}$ is an automorphic factor for the lattice $L.$
\end{lemma}
\noindent
{\it Proof.} For brevity, we write $J:=J_{S,\Omega,A,B}.$ For any two
elements $l_i=(\xi_i,\,\eta_i)\,(\,i=1,2\,)$ of $L$ and $W\in \BC^{(m,n)},$
we must show that
\begin{equation}
J(l_1+l_2,\,W)=J(l_1,\,l_2+W)\,J(l_2,W).
\end{equation}
Using the fact that $\s(2S\eta_2\,^t\xi_1)$ is an even integer, an easy
computation yields (3.16).
\hfill \Box

\vskip2mm
The Heisenberg group $H_{\BR}^{(n,m)}$ with multiplication $\diamond$ acts
on $\BC^{(m,n)}$
$$[\la_0,\mu_0,\kappa_0]\,\cdot\,(\la\Omega+\mu):=(\la_0+\la)\Omega+(\mu_0+
\mu),\ \ \ \la,\mu\in \BR^{(m,n)}.$$
Since the center ${\mathcal  Z}=\left\{ \,[0,0,\kappa]\,\mid\ \kappa=\,^t\kappa
\in \BR^{(m,m)}\,\right\}$ of $H_{\BR}^{(n,m)}$ is the stabilizer of
$H_{\BR}^{(n,m)}$ at $0$, the homogeneous space
$H_{\BR}^{(n,m)}/ {\mathcal  Z}$ is identified with $\BC^{(m,n)}$ via
$$[\la,\mu,\kappa]\,\cdot\, {\mathcal  Z}\,\longmapsto\, [\la,\mu,\kappa]\cdot 0=
\la\Omega+\mu.$$
Thus the automorphic factor $J_{S,\Omega,A,B}$ for the lattice $L$ may be
lifted to the automorphic factor ${\tilde {J}}_{S,\Omega,A,B}:
H_{\BR}^{(n,m)}\times \BC^{(m,n)}\lrt \BC^*$ defined by
\begin{equation}
{\tilde {J}}_{S,\Omega,A,B}(g_0,W)=e^{ \pi i\,\s \{ S(\la\Omega\,^t\la\,+\,
2\,W\,^t\la\,+\, \kappa)\} }\cdot e^{ -\pi i\,\s\{ S(A\,^t\mu-B\,^t\la)\} },
\end{equation}
where $g_0=[\la,\mu,\kappa]\in H_{\BR}^{(n,m)}.$\vskip2mm

We denote by ${\mathcal  A}_{S,\Omega}$ be the complex vector space consisting of
$\BC$-valued smooth functions $\varphi$ on $H_{\BR}^{(n,m)}$ satisfying the
following conditions
\vskip2mm
\ \ $(a)\ \ \varphi([\xi,\eta,0]\diamond g_0)=\varphi(g_0)$ for all
$\xi,\eta\in \BZ^{(m,n)}$ and $g_0\in H_{\BR}^{(n,m)},$\vskip2mm
\ \ $(b)\ \ \varphi(g_0\diamond [0,0,\kappa])=e^{\pi i\, \s(S\kappa)}\,
\varphi(g_0)$ for all $\kappa=\,^t\kappa\in \BR^{(m,m)}$ and \par
\ \ \ \ $\ \ \ g_0\in
H_{\BR}^{(n,m)},
$\vskip2mm
\ \ $(c)\ \ \bigl( \,{\mathcal  L}_{X_{ka}}\,-\,\sum_{b=1}^n\,\Omega
{\mathcal  L}_{{\widehat X}_{kb}}\bigl)=0$ for all $1\leq k\leq m$ and $1\leq a\leq n.$\vskip2mm\noindent
Here if $X$ is an element of the Lie algebra of $H_{\BR}^{(n,m)},$
$$\bigl(\,{\mathcal  L}_X\varphi\,\bigl)(g_0)=\dt\varphi(g_0\diamond \exp tX),
\ \ \ g_0\in H_{\BR}^{(n,m)}.$$
\vskip 0.2cm
\begin{theorem}
Let $S$ and $\Omega$ be as before. Then the
vector space $R_S^{\Omega}$ is isomorphic to the vector space
${\mathcal  A}_{S,\Omega}$ via the mapping
$$f\,\longmapsto\,\varphi_f(g_0):={\tilde {J}}_{S,\Omega,0,0}(g_0,\,0)\,
f(g_0\cdot 0),$$
where $g_0\in H_{\BR}^{(n,m)}$ and $f\in R_S^{\Omega}.$
\vskip2mm\noindent
The inverse of the above isomorphism is given by
$$\varphi\,\longmapsto\, f_{\varphi}(W):={\tilde {J}}_{S,\Omega,0,0}
(g_0,\,0)^{-1}\,\varphi(g_0),\ \ \ \varphi\in {\mathcal  A}_{S,\Omega},$$
where $W=g_0\cdot 0.$ This definition does not depend on the choice of $g_0$
with $W=g_0\cdot 0.$
\end{theorem}
\noindent
{\it Proof.} For brevity, we write ${\tilde {J}}:={\tilde J}_{S,\Omega,0,0}.$
If $\gamma=[\xi,\eta,0]\in H_{\BR}^{(n,m)}$ with $\xi,\eta\in \BZ^{(m,n)},$
we have for all $g_0\in H_{\BR}^{(n,m)}$
$$\begin{aligned}
\varphi_f(\gamma\diamond g_0)=&{\tilde {J}}(\gamma\diamond g_0,\,
0)\, f((\gamma\diamond g_0)\cdot 0)\\
=&{\tilde {J}}(\gamma,\,g_0\cdot 0)\,{\tilde {J}}(g_0,\,0)\,f(g_0\cdot 0+
\xi\Omega+\eta)\\
=&{\tilde {J}}(\gamma,\,g_0\cdot 0)\,{\tilde {J}}(g_0,0)\,
J((\xi,\eta),\,g_0\cdot 0)^{-1}\,f(g_0\cdot 0)\\
=&{\tilde {J}}(\gamma,0)\,f(g_0\cdot 0)\\
=&\varphi_f(g_0).
\end{aligned}$$
And if $\,\kappa=\,^t\kappa\in \BR^{(m,m)},$ we have
$$\begin{aligned}
\varphi_f(g_0\diamond [0,0,\kappa])&={\tilde {J}}(g_0\diamond
[0,0,\kappa],\,0)\,f((g_0\diamond [0,0,\kappa])\cdot 0)\\
&={\tilde {J}}(g_0,\,[0,0,\kappa]\cdot 0)\,{\tilde {J}}([0,0,\kappa],\,0)\,
f(g_0\cdot 0)\\
&=e^{\pi i\,\s (S\kappa) }\,{\tilde {J}}(g_0,0)\,f(g_0\cdot 0)\\
&=e^{\pi \s(S\kappa)}\,\varphi_f(g_0).
\end{aligned}$$
We introduce a system of complex coordinates on $\BC^{(m,n)}$ with
respect to $\Omega$\,:
$$W=\la\Omega+\mu,\ \ \ {\overline {W}}=\la{\overline {\Omega}}+\mu,\ \ \
\la,\mu\  \textrm{real}.$$
We set
$$dW=\begin{pmatrix} dW_{11}&dW_{12}&\ldots&dW_{1n}\\
dW_{21}&dW_{22}&\ldots&dW_{2n}\\
\vdots&\vdots&\ddots&\vdots\\
dW_{m1}&dW_{m2}&\ldots&dW_{mn}\end{pmatrix},\ \ \
{{\partial}\over {\partial W}}=
\begin{pmatrix}  {\partial}\over {\partial W_{11}}&{\partial}\over
{\partial W_{21}}&\ldots &{\partial}\over {\partial W_{m1}}\\
{\partial}\over {\partial W_{12}}& {\partial}\over {\partial W_{22}}&
\ldots&{\partial}\over {\partial W_{m2}}\\
\vdots&\vdots&\ddots&\vdots\\
{\partial}\over {\partial W_{1n}}&{\partial}\over {\partial W_{2n}}&
\ldots&{\partial}\over {\partial W_{mn}} \end{pmatrix}.$$
Then an easy computation yields
$$\begin{aligned}
 {{\partial}\over {\partial \la}}=&\Omega{{\partial}\over
{\partial W}}+{\overline {\Omega}}{{\partial}\over
{\partial {\overline {W} }}},\\
{{\partial}\over {\partial \mu}}=&
{{\partial}\over {\partial W}}+{{\partial}\over {\partial {\overline W}}}.
\end{aligned}$$
Thus we obtain the following
\begin{equation}
{{\partial}\over {\partial {\overline W}}}\,=\,{\frac i2}\,\left(\,
\text{Im}\,\Omega\,
\right)^{-1}\,\biggl(\,{{\partial}\over {\partial \la}}-\Omega
{{\partial}\over {\partial \mu}}\,\biggr).
\end{equation}
Since $f$ is holomorphic, according to (3.19), $f$ satisfies the conditions
\begin{equation}
\biggl(\,{{\partial}\over {\partial\la_{ka}}}\,-\,\sum_{b=1}^n\,\Omega_{ab}
{{\partial}\over {\partial\mu_{kb}}}\,\biggl)\,f(W)=0,\ \ \ 1\leq k\leq m,\
1\leq a\leq n.
\end{equation}
Conversely, if a smooth function on $\BC^{(m,n)}$ satisfies the condition
(3.20), it is holomorphic.\vskip2mm

In order to prove that $\varphi_f$ satisfies the condition (c), we first
compute ${\mathcal  L}_{X_{ka}}\varphi_f$ and ${\mathcal  L}_{{\widehat X}_{lb}}\varphi_f$
for $1\leq k,l\leq m$ and $1\leq a,b\leq n.$ If $\,g=[\la,\mu,\kappa]\in
H_{\BR}^{(n,m)}$ and $S=(s_{kl}),$
$$\begin{aligned}
\bigl(\,{\mathcal  L}_{X_{ka}}\varphi_f\,\bigl)
&=\dt \varphi_f(g\diamond \exp\,tX_{ka})\\
&=\dt \varphi_f([\la,\mu,\kappa]\diamond [tE_{ka},0,0])\\
&=\dt {\tilde {J}}([\la+tE_{ka},\mu,\kappa],\,0)\,f((\la+tE_{ka})\Omega+\mu)\\
&=\dt e^{\pi i\,\s\{S(\la+tE_{ka})\Omega\,^t(\la+tE_{ka})\} }\,
e^{\pi i\,\s(S\kappa)}\,f((\la+tE_{ka})\Omega+\mu)\\
&=e^{\pi i\,\s\{S(\kappa+\la\Omega\,^t\la)\} }\,\Biggl\{\,2\pi i\left(
\,\sum_{b=1}^n\sum_{l=1}^m\,s_{kl}\Omega_{ab}\la_{ab}\right)\,+\,
{{\partial}\over {\partial\la_{ka}}}\Biggr\}\,f(W).\\
\end{aligned}$$
On the other hand,
$$\begin{aligned}
\bigl(\,{\mathcal  L}_{{\widehat X}_{lb}}\varphi_f\,\bigl)(g)
=&\dt\varphi_f(g\diamond \exp\,t{\widehat X}_{lb})\\
=&\dt \varphi_f([\la,\mu,\kappa]\diamond [0,tE_{lb},0])\\
=&\dt \varphi_f([\la,\mu+tE_{lb},\kappa+t\la\,^tE_{kb}+tE_{kb}\,^t\!\la])\\
=&e^{\pi i\, \s\{ S(\kappa+\la\Omega\,^t\!\la)\} }\,\dt
e^{2\pi it\, \s(S\la\,^tE_{lb})}\,f(\la\Omega+(\mu+tE_{lb}))\\
=&e^{\pi i\, \s\{ S(\kappa+\la\Omega\,^t\!\la)\} }\,\Biggl\{\,2\pi i\left(
\,\sum_{p=1}^m\,s_{lp}\la_{pb}\,\right)\,+\,{{\partial}\over {\partial\mu_{lb}}}
\,\Biggr\}\,f(W).
\end{aligned}$$
Thus
$$\begin{aligned}
 &\bigl(\,{\mathcal  L}_{X_{ka}}\,-\,\sum_{b=1}^n\,\Omega_{ab}
{\mathcal  L}_{{\widehat X}_{kb}}\,\bigl)\varphi_f(g)\\
=&e^{\pi i\, \s\{ S(\kappa+\la\Omega\,^t\!\la)\} }\,\Biggl\{\,
{{\partial}\over {\partial\la_{ka}}}-\sum_{b=1}^n\,\Omega_{ab}
{{\partial}\over {\partial\mu_{kb}}}\Biggr\}\,f(W)=0.\\
\end{aligned}$$
This completes the proof.
\hfill $\square$

\end{section}

\newpage


\begin{section}{{\large\bf Induced Representations}}
\setcounter{equation}{0}

\ \ \ Let $G$ be a locally compact, separable topological group and $K$
be a closed subgroup of $G$. Let $\sigma$ be an irreducible unitary
representation of $K$ in a separable Hilbert space ${\mathcal  H}$. Let $\mu$
be a $G$-invariant measure in the homogeneous space $X:=K{\backslash}G=
\{ Kg\,|\ g\in G\,\}$ of the right $K$-cosets in $G$. We denote the induced
representation of $G$ from $\sigma$ by
$$U_\sigma:=\text{ Ind}^G_K\,\sigma.$$
 Let ${\mathcal  H}^{\sigma}$ be the Hilbert space
consisting of all functions
$\phi:G\lrt {\mathcal  H}$ which satisfy the following conditions:\vskip2mm

 (1) $(\phi(g),v)_{\mathcal  H}$ is measurable with respect
to $dg$ for all $v\in {\mathcal  H}$.
\vskip2mm
 (2) $\phi(kg)=\sigma(k)(\phi(g))$ for all $k\in K$ and $g\in G.$
\vskip2mm
 (3) $\parallel\phi\parallel^2=\int_X\parallel\phi(g)\parallel^2d\mu
({\dot{g}}) < \infty,\ \ {\dot{g}}=Kg,$\vskip2mm
\noindent
where $dg$ is a $G$-invariant measure on $G$ and
$(\,\,\,,\,\,)_{{\mathcal  H}}$ is an inner product in ${\mathcal  H}$ and
$\parallel \phi(g)\parallel$ is the norm in ${\mathcal  H}$. The inner product
$(\,\,,\,\,)$ in ${\mathcal  H}^{\sigma}$ is given by
$$(\phi_1,\phi_2)=\int_X\,(\phi_1(g),\phi_2(g))_{{\mathcal  H}}\,
d\mu({\dot{g}}),\ \ \
\phi_1,\phi_2\in {\mathcal  H}^{\sigma}.$$
Then $U_\sigma=\text{ Ind}_K^G\,\sigma$ is realized in the Hilbert space
${\mathcal  H}^{\sigma}$ as follows:
\begin{equation}
\left(U_\sigma(g_0)\phi\right)(g)=\phi(gg_0),\ \ g,g_0\in G,\
\phi\in {\mathcal  H}^{\sigma}.
\end{equation}
It is easy to see that ${\mathcal  H}^{\sigma}$ is isomorphic to the Hilbert
space ${\mathcal  H}_{\sigma}:=L^2(X,\mu,{\mathcal  H})$ of square integrable functions
$f:X\lrt {\mathcal  H}$ with values in ${\mathcal  H}$ via the formula
\begin{equation}
\phi_f(g)=\sigma(k_g)\left(f({\dot{g}})\right),\ \ f\in {\mathcal  H}_{\sigma},
\ g\in G,
\end{equation}
where ${\dot{g}}=Kg$ and $k_g$ is the $K$-component of $g$ in the Mackey
decomposition $g=k_gs_g.$\vskip2mm

We can show easily that $U_\sigma$ is realized in ${\mathcal  H}_{\sigma}$ by
\begin{equation}
\left(U_\sigma(g_0)f\right)({\dot{g}})=\sigma(k_{s_gg_0})\left(
f({\dot{g}}g_0)\right),\ \ g_0\in G,\ f\in {\mathcal  H}_{\sigma},\ {\dot{g}}=Kg
\in X,
\end{equation}
where $k_{s_gg_0}$ denotes the $K$-component of $s_gg_0$ in the Mackey
decomposition of $s_gg_0.$

If $\sigma$ is a one-dimensional representation of $K,\ U_\sigma$ is called
a $\textsf{monomial}$ representation. \vskip2mm

\begin{remark}
It is interesting to find out
irreducible closed subspaces
of ${\mathcal  H}^{\sigma}$ or ${\mathcal  H}_{\sigma}$ invariant under $G.$
\end{remark}
\vskip2mm

We recall $\MA,\ \MS,\ G_H,\ \MS_{\hat {\kappa}}$ and etc in Section 2. Mackey's
method teaches us that an irreducible unitary representation of $G_H\cong
H_{\BR}^{(n,m)}$ is of the following form

\newpage

$$T_{\hat {\kappa}}:=\text{ Ind}_{\MS_{\hat {\kappa}}\ltimes \MA}^{G_H}\,
\chi_{\hat {\kappa}}\cdot {\hat {a}}({\hat {\kappa}})=\text{Ind}_{\MA}^{G_H}\,
{\hat {a}}({\hat {\kappa}})$$
or
$$T_{{\hat {x}},{\hat {y}}}:=\text{ Ind}_{\MS_{\hat {y}}\ltimes \MA}^{G_H}\,
\chi_{\hat {x}}\cdot {\hat {a}}({\hat {y}})=\text{ Ind}_{\MS\ltimes \MA}^{G_H}\,
\chi_{\hat {x}}\cdot {\hat {a}}({\hat {y}}),$$
where $\chi_{\hat {x}}$ is the character of $\MS$ defined by $\chi_{\hat {x}}
(\l):=e^{2\pi i\s({\hat {x}}\,^t\!\l)}$ for $\l\in \MS.$ Therefore the unitary
dual ${\widehat {G}}_H$ of $G_H$ or $H_{\BR}^{(n,m)}$ is determined completely by

\vskip2mm

{\sc Type} {\sc I.}\ \ \ ${\hat {\kappa}}\in  \textrm{Sym}\,(m,\BR),\ \
{\hat {\kappa}}\neq 0.$

{\sc Type} {\sc II.} \ \ $({\hat {x}},{\hat {y}})\in \BR^{(m,n)}\times
\BR^{(m,n)}$ with ${\hat x},{\hat y}\in \BR^{(m,n)}.$
\vskip2mm\noindent
The representation $\rho\in {\widehat {G}}_H$ of {\it type} {\sc I}
acts nontrivially on the
center ${\mathcal  Z}\cong  \textrm{Sym}\,(m,\BR)$ of $G_H$.
On the other hand, the representation
$\rho\in {\widehat {G}}_H$ of {\it type} {\sc II} acts trivially on the center
${\mathcal  Z}$ of $G_H.$
\end{section}

\newpage


\begin{section}{{\large\bf The Schr\"{o}dinger Representation}}
\setcounter{equation}{0}

For two fixed positive integers $m$ and $n$, we put
$G:=H_{\BR}^{(n,m)}$ and
\begin{equation}
K:=\left\{(0,\mu,\kappa)\in G\,\vert\ \mu\in \BR^{(m,n)},\ \kappa=
\,^t\!\kappa\in \BR^{(m,m)}\ \right\}.
\end{equation}
We note that $K=A$\,(cf.\,Section 2,\,(2.4)) and that $K$ is a closed,
commutative normal subgroup of $G$. Since $(\la,\mu,\kappa)=(0,\mu,
\kappa+\mu\,^t\!\la)\circ (\la,0,0)$ for $(\la,\mu,\kappa)\in G,$ the
homogeneous space $X:=K\backslash G$ is identified with $\BR^{(m,n)}$
via
$$Kg=K\circ (\la,0,0)\longmapsto \la,\ \ \ g=(\la,\mu,\kappa)\in G.$$
We observe that $G$ acts on $X$ by
\begin{equation}
(Kg)\cdot g_0=K\,(\la+\la_0,0,0),
\end{equation}
where $g=(\la,\mu,\kappa)\in G$ and $g_0=(\la_0,\mu_0,\kappa_0)\in G.$

If $g=(\la,\mu,\kappa)\in G,$ we have
\begin{equation}
k_g=(0,\mu,\kappa+\mu\,^t\!\la),\ \ \ s_g=(\la,0,0)
\end{equation}
in the Mackey decomposition of $g=k_g\circ s_g.$ Thus if $g_0=(\la_0,\mu_0,
\kappa_0)\in G,$ then we have
\begin{equation}
s_g\circ g_0=(\la,0,0)\circ (\la_0,\mu_0,\kappa_0)=(\la+\la_0,\mu_0,
\kappa_0+\la\,^t\!\mu_0)
\end{equation}
and so
\begin{equation}
k_{s_g\circ g_0}=(0,\mu_0,\kappa_0+\mu_0\,^t\!\la_0+\la\,^t\!\mu_0+
\mu_0\,^t\!\la).
\end{equation}

For a real symmetric matrix $c=\,^tc\in \BR^{(m,m)}$ with $c\neq 0$, we
consider the one-dimensional unitary representation $\sigma_c$ of $K$
defined by
\begin{equation}
\sigma_c\left((0,\mu,\kappa)\right):=e^{2\pi i\,\s(c\kappa)}\,I,\ \ \
(0,\mu,\kappa)\in K,
\end{equation}
where $I$ denotes the identity mapping. Then the induced representation
$U_{\s_c}:=\text{ Ind}_K^G\,\s_c$ of $G$ induced from $\s_c$ is realized in the
Hilbert space ${\mathcal  H}_{\s_c}=L^2(X,d{\dot {g}},\BC)
\cong L^2\left(\BR^{(m,n)},
d\xi\right)$ as follows. If $g_0=(\la_0,\mu_0,\kappa_0)\in G$ and
$x=Kg\in X$ with $g=(\la,\mu,\kappa)\in G,$ then according to (4.3), we have
\begin{equation}
\left(U_{\s_c}(g_0)f\right)(x)=\s_c\left(k_{s_g\circ g_0}\right)\left(
f(xg_0)\right),\ \ f\in {\mathcal  H}_{\s_c}.
\end{equation}
It follows from (5.5) that
\begin{equation}
\left(U_{\s_c}(g_0) f\right)(\la)=e^{2\pi i\s\{c(\kappa_0+\mu_0\,^t\!\la_0+
2\la\,^t\!\mu_0)\}}\,f(\la+\la_0).
\end{equation}
Here we identified $x=Kg$\,(resp.\,$xg_0=Kgg_0$) with $\la$\,(resp.\,
$\la+\la_0$). The induced representation $U_{\s_c}$ is called the
$\textsf{Schr{\" {o}}dinger\ representation}$ of $G$ associated with $\s_c.$
$U_{\s_c}$ is a monomial representation.

In the previous section, we denoted by ${\mathcal  H}^{\s_c}$ the Hilbert space
consisting of all functions $\phi:G\lrt \BC$ which satisfy the following
conditions:\vskip2mm
 (1) $\phi(g)$ is measurable with respect to $dg.$\vskip2mm
 (2) $\phi\left( (0,\mu,\kappa)\circ g)\right)=
e^{2\pi i\,\s(c\kappa)}\phi(g)$ \ \ for\ all $g\in G.$\vskip2mm
 (3) $\parallel\phi\parallel^2:=\int_X\,\vert \phi(g)\vert^2\,
d{\dot {g}} < \infty,\ \ \ {\dot {g}}=Kg,$\vskip2mm
\noindent
where $dg$\,(resp.\,$d{\dot {g}}$) is a $G$-invariant measure on $G$
(resp.\,$X=K\backslash G$). The inner product $(\ ,\ )$ on ${\mathcal  H}^{\s_c}$
is given by
$$(\phi_1,\phi_2)=\int_G\,\phi_1(g)\,{\overline {\phi_2(g)}}\,dg,\ \ \
\phi_1,\ \phi_2\in {\mathcal  H}^{\s_c}.$$
We observe that ${\mathcal  H}_{\s_c}=L^2(\BR^{(m,n)},d\xi)$ and that the mapping
$\Phi_c:{\mathcal  H}_{\s_c}\lrt {\mathcal  H}^{\s_c}$ defined by
\begin{equation}
\left( \Phi_c(f)\right)(g):=\phi_f(g):=e^{2\pi i\s
\{c(\kappa+\mu\,^t\!\la)\}}\,f(\la)
\end{equation}
$(f\in {\mathcal  H}_{\s_c},\ g=(\la,\mu,
\kappa)\in G)$
is an isomorphism of Hilbert spaces. The inverse $\Psi_c:{\mathcal  H}^{\s_c}
\lrt {\mathcal  H}_{\s_c}$ of $\Phi_c$ is given by
\begin{equation}
\left( \Psi_c(\phi)\right)(\la):=f_{\phi}(\la):=\phi((\la,0,0)),\ \ \
\phi\in {\mathcal  H}^{\s_c},\ \la\in \BR^{(m,n)}.
\end{equation}

From now on, for brevity we put
$$U_c=U_{\sigma_c},\ \ {\mathcal H}_c ={\mathcal  H}_{\s_c}\quad \textrm{and}\quad {\mathcal H}^c ={\mathcal  H}^{\s_c}.$$
The Schr{\" {o}}dinger representation $U_c$ of $G$ on ${\mathcal  H}^c$
is given by
\begin{equation}
\left( U_c(g_0)\phi\right)(g)=
e^{2\pi i\s\{ c(\kappa_0+\mu_0\,^t\!\la_0+\la\,^t\!\mu_0-\la_0\,^t\!\mu)\}}\,
\phi\left( (\la_0,0,0)\circ g\right),
\end{equation}
where $g_0=(\la_0,\mu_0,\kappa_0),\ g=(\la,\mu,\kappa)\in G$ and
$\phi\in {\mathcal  H}^c.$ (5.11) can be expressed as follows.
\begin{equation}
\left(\,U_c(g_0)\phi\,\right)(g)=e^{2\pi i\s\{c(\kappa_0+\kappa+
\mu_0\,^t\!\la_0+\mu\,^t\!\la+2\la\,^t\mu_0)\}}\,\phi((\la_0+\la,0,0)).
\end{equation}
\vskip2mm

\begin{theorem}
Let $c$ be a positive symmetric
half-integral matrix of degree $m$. Then the Schr{\" {o}}dinger
representation $U_{c}$ of $G$ is irreducible.
\end{theorem}
\noindent
{\it Proof.}\ The proof can be found in \cite{Y1},\ Theorem 3.
\hfill \Box

\vskip2mm
We let $dU_c$ be the infinitesimal representation associated to
the Schr{\"{o}}dinger
representation $U_c$. If $X$ is an element of the Lie algebra of $G$,
then
$$dU_c(X)f=\dt\,U_c(\exp\,tX)f,\ \ \ f\in {\mathcal  H}_c\ or\
{\mathcal  H}^c.$$
We fix an element $\Omega\in {\mathbb H}_n$ once and for all. We let $c$ be a positive
symmetric real matrix of degree $m$. For each $\J,$ we put
\begin{equation}
f_{c,J}(\xi):=e^{2\pi i\,\s(c\xi\Omega\,^t\!\xi)}\,\xi^J,\ \ \ \xi\in
\BR^{(m,n)}.
\end{equation}
Then the set $\left\{\,f_{c,J}\,\mid\  \J \,\right\}$ forms a basis of
$L^2\big(\BR^{(m,n)},d\xi\big)\cong {\mathcal  H}_c.$

\begin{proposition}
Let $c=(c_{kl})$ be a positive symmetric real matrix
of degree $m$. For each $\J \,$ we have
\begin{eqnarray}
\ & & dU_c(D_{kl}^0)f_{c,J}(\xi)= 2\pi i\,c_{kl}\,
f_{c,J}(\xi),\ \ \
1\leq k\leq l\leq m,\\
\ & & dU_c(D_{ka})f_{c,J}(\xi)= 4\pi i\sum_{l=1}^m\sum_{b=1}^n\,
c_{kl}\Omega_{ab}f_{c,J+\epsilon_{lb}}(\xi)
\, +\,J_{ka}\,f_{c,J-\epsilon_{ka}}(\xi),\\
\ & & dU_c({\widehat D}_{lb})f_{c,J}(\xi)=4\pi i\sum_{p=1}^m\,
c_{lp}\,f_{c,J+\epsilon_{pb}}(\xi).
\end{eqnarray}
Here $1\leq k,l\leq m$ and $1\leq a,b\leq n.$
\end{proposition}
\noindent
{\it Proof.} We put $E_{kl}^0={\frac 12}(E_{kl}+E_{lk}),$ where
$1\leq k,l\leq m.$
$$\begin{aligned}
dU_c(D_{kl}^0)f_{c,J}(\xi)&=\dt\,U_c\big(\exp\,tX_{kl}^0\big)
f_{c,J}(\xi)\\
&=\dt\,U_c\big((0,0,tE_{kl}^0)\big)f_{c,J}(\xi)\\
&=\lt\,{ {e^{2\pi i\s(tcE_{kl}^0)}-I}\over t}\,f_{c,J}(\xi)\\
&=\lt\,{{e^{2\pi itc_{kl}}-I}\over t}\,f_{c,J}(\xi)\\
&=2\pi i\,c_{kl}\,f_{c,J}(\xi).
\end{aligned}$$
$$\begin{aligned}
dU_c(D_{ka})f_{c,J}(\xi)&=\dt\,U_c(\exp\,tX_{ka})
f_{c,J}(\xi)\\
&=\dt U_c((tE_{ka},0,0))f_{c,J}(\xi)\\
&=\dt\,e^{2\pi i\, \s\{ c(\xi+tE_{ka})\Omega\,^t(\xi+tE_{ka})\}}\,
(\xi+tE_{ka})^J\\
&=4\pi i\sum_{l=1}^m\sum_{b=1}^n\,c_{kl}\Omega_{ab}
f_{c,J+\epsilon_{lb}}(\xi)+J_{ka}\,f_{c,J-\epsilon_{ka}}(\xi).
\end{aligned}$$
Finally,
$$\begin{aligned}
dU_c({\widehat D}_{lb})f_{c,J}(\xi)&=\dt U_c(\exp\,t{\widehat X}_{lb})
f_{c,J}(\xi)\\
&=\dt U_c((0,tE_{lb},0))f_{c,J}(\xi)\\
&=\lt\,{ {e^{4\pi i\, \s(tc\xi\,^tE_{lb})}-I}\over t }\,f_{c,J}(\xi)\\
&=4\pi i\sum_{p=1}^m\,c_{lp}f_{c,J+\epsilon_{pb}}(\xi).
\end{aligned}$$
\hfill \Box\vskip2mm

For each $\J,$ we put
\begin{equation}
\phi_{c,J}(g)=e^{2\pi i\, \s\{ c(\kappa+\mu\,^t\!\la)\} }
\,f_{c,J}(\la),
\end{equation}
where $g=(\la,\mu,\kappa)\in G.$ Then the set $\{\,\phi_{c,J}\, \mid\ \J\,
\}$ is a basis of ${\mathcal  H}^c.$
\vskip2mm
\begin{proposition}
For each $\J$ and $g=(\la,\mu,\kappa)\in G,$
we have
\begin{eqnarray}
\ \ \ & & dU_c(D_{kl}^0)\,\phi_{c,J}(g)= 2\pi i\,c_{kl}\,
\phi_{c,J}(g),\ \ \
1\leq k\leq l\leq m,\\
\ \ & & dU_c(D_{ka})\phi_{c,J}(g)=
4\pi i\sum_{l=1}^m\sum_{b=1}^n\,
c_{kl}\,\Omega_{ab}\,\phi_{c,J+\epsilon_{lb}}(g)
\,+\,J_{ka}\,\phi_{c,J-\epsilon_{ka}}(g),\\
\ \ & & dU_c({\widehat U}_{lb})\phi_{c,J}(g)= \,2\pi i\sum_{p=1}^m\,
c_{lp}\phi_{c,J+\epsilon_{pb}}(g).
\end{eqnarray}
Here $1\leq k,l\leq m$ and $1\leq a,b\leq n.$
\end{proposition}
\noindent
{\it Proof.} We put $E_{kl}^0={\frac 12}(E_{kl}+E_{lk}),$ where
$1\leq k,l\leq m.$ Then we have
$$\begin{aligned}
dU_c(D_{kl}^0)\phi_{c,J}(g)&=\dt\,U_c(\exp\,tX_{kl}^0)
\phi_{c,J}(g)\\
&=\dt\,U_c\big((0,0,tE_{kl}^0)\big)\phi_{c,J}(g)\\
&=\lt\,{ {e^{2\pi i\, \s(tcE_{kl}^0)}-I}\over t }\,\phi_{c,J}(g)\\
&=2\pi i\,c_{kl}\,\phi_{c,J}(g).
\end{aligned}$$
And we have
$$\begin{aligned}
dU_c(D_{ka})\phi_{c,J}(g)&=\dt\,U_c(\exp\,tX_{ka})
\phi_{c,J}(g)\\
&=\dt\,U_c((tE_{ka},0,0))\phi_{c,J}(g)\\
&=\dt\,e^{-2\pi it \,\s(cE_{ka}\,^t\!\mu)}\,
\phi_{c,J}((tE_{ka},0,0)\circ g)\\
&=\dt\,e^{-2\pi it\, \s(cE_{ka}\,^t\!\mu)}\cdot
e^{2\pi i\, \s\{ c(\kappa+tE_{ka} {}^t\!\mu+\mu\,^t\!\la+t\mu\,^tE_{ka})\} }\\
& \quad\ \times e^{2\pi i\, \s\{ c(\la+tE_{ka})\Omega\,^t(\la+tE_{ka})\} }\,
(\la+tE_{ka})^J\\
&= e^{ 2\pi i\, \s\{c(\kappa+\mu\,^t\!\la+\la\Omega\,^t\!\la)\} }\\
& \quad\ \times \dt e^{ 4\pi it\, \s(c\la\Omega\,^tE_{ka})
+2\pi it^2 \,\s(cE_{ka}\Omega\,^tE_{ka}) }\,(\la+tE_{ka})^J\\
&=4\pi i\sum_{l=1}^m\sum_{b=1}^n\,c_{kl}\,\Omega_{ab}\,
\phi_{c,J+\epsilon_{lb}}(g)\,+\,J_{ka}\,\phi_{c,J-\epsilon_{ka}}(g).
\end{aligned}$$
Finally,
$$\begin{aligned}
dU_{{\widehat D}_{lb}}(\s_c)\,\phi_{c,J}(g)&=\dt U_c(\exp\,t{\widehat X}_{lb})\phi_{c,J}(g)\\
&=\dt U_c((0,tE_{lb},0))\phi_{c,J}(g)\\
&=\dt\,e^{ 2\pi it \,\s(c\la\,^tE_{lb})}\phi_{c,J}(g)\\
&=\lt\,{ {e^{2\pi it(\sum_{p=1}^m\,c_{lp}\la_{pb})}-I}\over t }\,
\phi_{c,J}(g)\\
&=2\pi i\,\left(\,\sum_{p=1}^m\,c_{lp}\la_{pb}\,\right)\,\phi_{c,J}(g)\\
&=2\pi i\,\sum_{p=1}^m\,c_{lp}\phi_{c,J+\epsilon_{pb}}(g).
\end{aligned}$$
$\hfill \square$

\end{section}

\newpage


\begin{section}{{\large\bf Fock Representations}}
\setcounter{equation}{0}

We consider the vector space $V^{(m,n)}:=\BR^{(m,n)}\times \BR^{(m,n)}.$
We put
\begin{equation}
P_{ka}=(E_{ka},0),\ \ \ Q_{lb}=(0,E_{lb}),
\end{equation}
where $1\leq k,l\leq m$ and $1\leq a,b\leq n.$ Then the set
$\{ P_{ka},\,Q_{ka}\}$ forms a basis for $V^{(m,n)}.$ We define the
alternating bilinear form ${\bA} :V^{(m,n)}\times V^{(m,n)}\lrt \BR$ by
\begin{equation}
{\bA}\left( (\la_0,\mu_0),(\la,\mu)\right)=\s(\la_0\,^t\!\mu-\mu_0\,^t\!\la),
\ \ (\la_0,\mu_0),\,(\la,\mu)\in V^{(m,n)}.
\end{equation}
Then we have
\begin{equation}
{\bA}(P_{ka},P_{lb})={\bA}(Q_{ka},Q_{lb})=0,\ \ {\bA}(P_{ka},Q_{lb})=
\delta_{ab}\,\delta_{kl},
\end{equation}
where $1\leq k,l\leq m$ and $1\leq a,b\leq n.$ Any element $v\in V^{(m,n)}$
can be written uniquely as
\begin{equation}
v=\sum_{k,a}\,x_{ka}P_{ka}+\sum_{l,b}\,y_{lb}Q_{lb},\ \ \ x_{ka},\,
y_{lb}\in \BR.
\end{equation}
From now on, for brevity, we write $V:=V^{(m,n)}$ and $v=xP+yQ$ instead of
(6.4). Then it is easy to see that the endomorphism $J:V\lrt V$ defined by
\begin{equation}
J(xP+yQ):=-yP+xQ,\ \ \ xP+yQ\in V
\end{equation}
is a complex structure on $V$ which is compatible with the alternating
bilinear form $\bA$. This means that $J$ is an endomorphism of $V$
satisfying the following conditions:\vskip2mm
 (J1)\ \ $J^2=-I$\ \ on $V$.\vskip2mm
 (J2)\ ${\bA}(Jv_0,Jv)={\bA}(v_0,v)$ for all $v_0,v\in V.$
\vskip2mm
 (J3)\ ${\bA}(v,Jv) > 0$ for all $v\in V$ with $v\neq 0.$
\vskip2mm

Now we let $V_{\BC}=V+i\, V$ be the complexification of $V$, where $i=
\sqrt{-1}.$ For an element $w=v_1+i\, v_2\in V_{\BC}$ with $v_1,v_2\in V$,
we put
\begin{equation}
{\overline w}:=v_1-i\,v_2.
\end{equation}
Let $\bA_{\BC}$ be the complex bilinear form on $V_{\BC}$ extending $\bA$
and let $J_{\BC}$ be the complex linear map of $V_{\BC}$ extending $J.$
Since $J_{\BC}^2=-I,\ J_{\BC}$ has the only eigenvalues $\pm \,i.$ We
denote by $V^+$\,(resp.\,$V^-$) the eigenspace of $V_{\BC}$ corresponding to
the eigenvalues $i$\,(resp.\,$-i$). Thus $V_{\BC}=V^+ +V^-.$ Since
$$J_{\BC}(P_{ka}\pm i\,Q_{ka})=\mp i \,(P_{ka}\pm i\,Q_{ka}),$$
we have
\begin{equation}
V^+=\sum_{k,a}\,\BC\,(P_{ka}-i\,Q_{ka}), \ \ \ V^-=\sum_{k,a}\BC\,
(P_{ka}+i\,Q_{ka}).
\end{equation}
Let
\begin{equation}
V_*:=\sum_{k,a}\,\BC\,P_{ka},\ \ \ 1\leq k\leq m,\ \ 1\leq a\leq n
\end{equation}
be the subspace of $V_{\BC}$ as a $\BC$-vector space. It is easy to see
that $V_*$ is isomorphic to $V$ as $\BR$-vector spaces via the isomorphism
$T:V\lrt V_*$ defined by
\begin{equation}
T(P_{ka})=P_{ka},\ \ \ T(Q_{lb})=i\,P_{lb}.
\end{equation}
We define the complex linear map $J_*:V_*\lrt V_*$ by $J_*(P_{ka})=
i\,P_{ka}$ for $1\leq k\leq m,\ 1\leq a\leq n.$ Then $J_*$ is compatible with
$J$, that is, $T\circ J=J_*\circ T.$ It is easily seen that there exists a
unique hermitian form $\bH$ on $V_*$ with ${\text{Im}}\,\bH=\bA.$ Indeed,
$\bH$ is given by
\begin{equation}
{\bH}(v,w)={\bA}(v,J_*w)+i\,{\bA}(v,w),\ \ v,w\in V_*.
\end{equation}
For $v=\sum_{k,a}z_{ka}P_{ka}\in V_*$ with $z_{ka}=x_{ka}+iy_{ka}\
(x_{ka},y_{ka}\in \BR),$ for brevity we write $v=zP.$ For two elements
$v=zP$ and $v'=z'P$ in $V_*,\ {\bH}(v,v')=\sum_{k,a}{\overline {z_{ka}}}
\,z'_{ka}.$
We observe that
$$V_{\BC}=\sum_{k,a}\BC\,P_{ka}+\sum_{l,b}\BC\,Q_{lb}=V^++V^-\supset
V^{\pm}.$$
For $w=z^0P+z^1Q\in V_{\BC},$ we put
$$w=w^++w^-,\ \ w^+:=z^+(P-i\,Q),\ \ w^-:=z^-(P+i\, Q).$$
The relations among $z^0,z^1,z^+,z^-$ are given by
\begin{equation}
z^{\pm}={1\over 2}(z^0\pm i\,z^1),\ \ z^0=z^++z^-,\ \ z^1=i\,(z^--z^+).
\end{equation}
Precisely, (6.11) implies that
$$z^{\pm}_{ka}={1\over 2}\,(z^0_{ka}\pm i\,z^1_{ka}),\ \
z^0_{ka}=z^+_{ka}+z^-_{ka},\ \ z^1_{ka}=i\,(z^-_{ka}-z^+_{ka}),$$
where $1\leq k\leq m$ and $1\leq a\leq n.$ It is easy to see that
\begin{equation}
{\bA}_{\BC}(w^-,w^+)=-2i\sum_{k,a}z^-_{ka}z^+_{ka}=-{i\over 2}
\sum_{k,a}\left\{ (z^0_{ka})^2+(z^1_{ka})^2\right\}.
\end{equation}

Let
$$G_{\BC}:=\left\{\,(z^0,z^1,a)\ \big|\ z^0,z^1\in \BC,\ \ a\in \BC^{(m,m)},
\ \ a+z^1\,^tz^0\ \textrm{symmetric}\,\right\}$$
be the complexification of the real Heisenberg group $G:=H_{\BR}^{(n,m)}.$
Analogously in the real case, the multiplication on $G_{\BC}$ is given by
(2.1). If $w=z^0P+z^1Q:=\sum_{k,a}z^0_{ka}P_{ka}+\sum_{l,b}z^1_{lb}Q_{lb},$
we identify $z^0,z^1$ with the $m\times n$ matrices respectively\,:
$$z^0:=\begin{pmatrix} z^0_{11}&z^0_{12}&\ldots&z^0_{1n}\\
z^0_{21}&z^0_{22}&\ldots&z^0_{2n}\\
\vdots&\vdots&\ddots&\vdots\\
z^0_{m1}&z^0_{m2}&\ldots&z^0_{mn}\end{pmatrix},\ \ \
z^1:=\begin{pmatrix} z^1_{11}&z^1_{12}&\ldots&z^1_{1n}\\
z^1_{21}&z^1_{22}&\ldots&z^1_{2n}\\
\vdots&\vdots&\ddots&\vdots\\
z^1_{m1}&z^1_{m2}&\ldots&z^1_{mn}\end{pmatrix}.$$
That is, we identify $w=z^0P+z^1Q\in V_{\BC}$ with $(z^0,z^1)\in
\BC^{(m,n)}\times \BC^{(m,n)}.$ If $w=z^0P+z^1Q,\ {\hat {w}}=
{\hat z}^0 P+{\hat z}^1Q\in V_{\BC},$ then
\begin{equation}
(w,a)\circ ({\hat {w}},{\hat {a}})=(w+{\hat {w}},a+{\hat {a}}+
z^0\,^t{\hat z}^1-z^1\,^t{\hat z}^0),\ \ a,{\hat a}\in \BC^{(m,m)}.
\end{equation}
From now on, for brevity we put
\begin{equation}
R^+:=P-i\,Q,\ \ \ \ \ R^-:=P+i\,Q.
\end{equation}
If $w=z^+R^++z^-R^-,\ {\hat w}={\hat {z}}^+R^++{\hat {z}}^-R^-\in V_{\BC},$
by an easy computation, we have
\begin{equation}
(w,a)\circ ({\hat {w}},{\hat {a}})=({\tilde {w}},a+{\hat {a}}+2\,i\,
(z^+\,^t{\hat {z}}^--z^-\,^t{\hat {z}}^+))
\end{equation}
with
$${\tilde {w}}=(z^++{\hat {z}}^+)R^++(z^-+{\hat {z}}^-)R^-.$$
Here we identified $z^+,z^-$ with $m\times n$ matrices
$$z^+:=\begin{pmatrix} z_{11}^+&z^+_{12}&\ldots&z^+_{1n}\\
z^+_{21}&z^+_{22}&\ldots&z^+_{2n}\\
\vdots&\vdots&\ddots&\vdots\\
z^+_{m1}&z^+_{m2}&\ldots&z^+_{mn}\end{pmatrix},\ \ \
z^-:=\begin{pmatrix} z^-_{11}&z^-_{12}&\ldots&z^-_{1n}\\
z^-_{21}&z^-_{22}&\ldots&z^-_{2n}\\
\vdots&\vdots&\ddots&\vdots\\
z^-_{m1}&z^-_{m2}&\ldots&z^-_{mn}\end{pmatrix}.$$
It is easy to see that
\begin{equation}
P_{\BC}:=\left\{\,(w^-,a)\in G_{\BC}\,\vert\ w^-\in V^-,\ \
a\in \BC^{(m,m)}\,\right\}
\end{equation}
is a commutative subgroup of $G_{\BC}$ and
$$G\cap P_{\BC}={\mathcal  Z},\ \ \ G_{\BC}=G\circ P_{\BC},$$
where ${\mathcal  Z}:=\left\{\,(0,0,\k)\in G\,\vert\,\k=\,^t\k\in \BR^{(m,m)}\,
\right\}\cong \textrm{Sym}(m,\BR)$ is the center of $G$.
Moreover,
\begin{equation}
P_{\BC}\backslash G_{\BC}\cong V^+\cong \BR^{(m,n)}\times \BR^{(m,n)}
\cong {\mathcal  Z}\ba G.
\end{equation}

For $c=\,^tc\in \textrm{Sym}(m,\BR)$ with $c>0,$ we let $\delta_c:P_{\BC}\lrt
\BC^{\times}$ be a quasi-character of $P_{\BC}$ defined by
\begin{equation}
\delta_c((w^-,a))=e^{2\pi i\s(ca)},\ \ \ (w^-,a)\in P_{\BC}.
\end{equation}
Let
$$U^{F,c}={\text{Ind}}_{P_{\BC}}^{G_{\BC}}\,\delta_c$$
be the representation of $G_{\BC}$ induced from a quasi-character
$\delta_c$ of $P_{\BC}.$ Then $U^{F,c}$ is realized in the Hilbert space
${\mathcal  H}^{F,c}$ consisting of all holomorphic functions $\psi:G_{\BC}\lrt
\BC$ satisfying the following conditions:\vskip2mm

(F1)\ $\psi((w^-,a)\circ g)=\delta_c((w^-,a))\psi(g)=e^{2\pi i\,\s(ca)}\,
\psi(g)$
\ \ for all $(w^-,a)\in P_{\BC}$\\ $\qquad \textrm{and}\ g\in G_{\BC}$.\vskip2mm

(F2)\ \ $\int_{{\mathcal  Z}\ba G}\,\vert\psi({\dot {g}})\vert^2\,d{\dot {g}} < \infty.$

\vskip2mm

The inner product $\langle\,\,\, ,\ \rangle_{F,c}$ on ${\mathcal  H}^{F,c}$ is given by
$$\langle \psi_1,\psi_2 \rangle_{F,c}:=\int_{{\mathcal  Z}\ba G}\,\psi_1({\dot {g}})\,
{\overline {\psi_2({\dot {g}})}}\,d{\dot {g}},\ \ \psi_1,\psi_2\in
{\mathcal  H}^{F,c},\ {\dot {g}}={\mathcal  Z}g.$$
$U^{F,c}$ is realized by the right regular representation of $G_{\BC}$ on
${\mathcal  H}^{F,c}:$
\begin{equation}
\left( U^{F,c}(g_0)\psi\right)(g)=\psi(gg_0),\ \ \psi\in {\mathcal  H}^{F,c},
\ g_0,g\in G_{\BC}.
\end{equation}
\vskip 0.3cm
Now we will show that $U^{F,c}$ is realized as a representation of $G$ in
the Fock space. The Fock space ${\mathcal  H}_{F,c}$ is the Hilbert space
consisting of all holomorphic functions $f:\BC^{(m,n)}\cong V_*\lrt
\BC$ satisfying the condition
$$\parallel f\parallel^2_{F,c}=\int_{\BC^{(m,n)}}\,\vert f(W)\vert^2\,
e^{-2\pi \s(c\,W\,^t{\overline {W}})}\,dW  < \infty.$$
The inner product $(\,\,\,,\,\,)_{F,c}$ on ${\mathcal  H}_{F,c}$ is given by
$$(f_1,f_2)_{F,c}=\int_{\BC^{(m,n)}}\,f_1(W)\,{\overline {f_2(W)}}\,
e^{-2\pi \s(c\,W\,^t{\overline W})}\,dW,\ \ f_1,f_2\in {\mathcal  H}_{F,c}.$$
\vskip 0.2cm

\begin{lemma}
The mapping $\Lambda:{\mathcal  H}_{F,c}\lrt {\mathcal  H}^{F,c},
\ \Lambda_f:=\Lambda(f)\,(\,f\in {\mathcal  H}_{F,c}\,)$ defined by
\begin{equation}
\Lambda_f((z^0P+z^1Q,a))=e^{2\pi i\,\s\{c(a\,+\,2\,i\,z^-\,^tz^+)\}}\,f(2z^+)
\end{equation}
is an isometry of ${\mathcal  H}_{F,c}$ onto ${\mathcal  H}^{F,c},$ where $2z^{\pm}=
z^0\pm i\,z^1$\,(cf.\,(6.11)). The inverse $\Delta:{\mathcal  H}^{F,c}\lrt
{\mathcal  H}_{F,c},\ \Delta_{\psi}:=\Delta(\psi)\,(\psi\in {\mathcal  H}^{F,c})$ is
given by
\begin{equation}
\Delta_{\psi}(W)=\psi\left({1\over 2}WR^+ \right),\ \ W\in \BC^{(m,n)},
\end{equation}
where $R^{\pm}=P\mp i\,Q$\,(cf.\,(6.14)).
\end{lemma}
\noindent
{\it Proof.} First we observe that for $w=z^0P+z^1Q=z^+R^++z^-R^-\in
V_{\BC},$
$$(w,a)=(z^-R^-,a+2\,i\,z^-\,^tz^+)\circ (z^+R^+,0).$$
Thus if $\psi\in {\mathcal  H}^{F,c}$ and $w=z^0P+z^1Q=z^+R^++z^-R^-,$ by (F1),
\begin{equation}
\psi((w,a))=e^{2\pi i\, \s\{ c(a\,+\,2\,i\,z^-\,^tz^+)\}}\,\psi((z^+R^+,0)).
\end{equation}
Let $W=x+i\,y\in \BC^{(m,n)}$ with $x,y\in \BR^{(m,n)}.$ Then
$$xP+yQ=z^+R^++z^-R^-,\ \ 2z^{\pm}=x\pm i\,y.$$
So $z^-\,^tz^+={1\over 4}\,W\,{}^t{\overline {W}}.$ According to (6.22), if
$\psi\in {\mathcal  H}^{F,c},$ we have
$$\psi((xP+yQ,0))=e^{-\pi\, \s(c\,W\,^t{\overline {W}})}\,\psi\left(\left({1\over 2}WR^+,0\right)\right).$$
Thus we get
$$\vert\psi((xP+yQ,0))\vert^2=e^{-2\pi\, \s(c\,W\,^t{\overline {W}})}\,
\biggl|\psi
\left(\left( {1\over 2}WR^+,0\right)\right) \biggr|^2.$$
Therefore
$$\int_{{\mathcal  Z}\ba G}\,\vert\psi({\dot {g}})\vert\,d{\dot {g}}=\int_{\BC^{(m,n)}}
\,e^{-2\pi \s(c\,W\,^t{\overline {W}})}\,
\bigl|\Delta_{\psi}(W)\bigr|^2dW < \infty.$$
It is easy to see that $\Delta$ is the inverse of $\Lambda$. Hence we obtain
the desired results.
\hfill \Box
\vskip2mm
\begin{lemma}
The representation $U^{F,c}$ is realized as a
representation of $G$ in the Fock space ${\mathcal  H}_{F,c}$ as follows. If
$g=(\la P+\mu Q,\k)=(\la,\mu,\k)\in G$ and $f\in {\mathcal  H}_{F,c},$ then
\begin{equation}
\left( U^{F,c}(g)f\right)(W)=e^{2\pi i\, \s(c\,\k)}\,
e^{-\pi\, \s\{ c\,(\zeta\,^t\!{\bar {\zeta}}\,+ \,2\,W\,^t\!{\bar {\zeta}})\}}\,
f(W+\zeta),
\end{equation}
where $W\in \BC^{(m,n)}$ and $\zeta=\la+i\,\mu.$
\end{lemma}
\noindent
{\it Proof.}
$$\begin{aligned}
\left( U^{F,c}(g)f\right)(W)&=\left( \Delta(U^{F,c}(g)(\Lambda_f))\right)
(W)\\
&=\left( U^{F,c}(g)(\Lambda_f)\right)\left({\frac 12}WR^+\right)\\
&=\Lambda_f\biggl( \left({\frac 12}WR^+,0\right)\circ g\biggr)\\
&=\Lambda_f\left( {\left({\frac 12}W,-{\frac i2}W,0\right)}
\circ (\la,\mu,\k)\right)\\
&=\Lambda_f\left( \left( \la+{\frac 12}W \right)P+\left(\mu-{\frac i2}W\right)Q,\k+{\frac 12}\,
W\,^t\!\mu+{\frac i2}\,W\,^t\!\la\right)\\
&=e^{2\pi i\, \s\{ c(\k\,+\,{\frac i2}\,W\,^t{\bar {\zeta}}\,+\,{\frac i2}\,{\bar {\zeta}}
\,^t\!W\,+\, {i\over 2}\,{\bar {\zeta}}\,^t\!{\zeta})\}}\,f(W+\zeta)
\ \ \ \ \ (**)\\
&=e^{2\pi i\, \s(c\k)}\cdot e^{-\pi\, \s\{\M(\zeta\,^t{\bar {\zeta}}\,+\,
W\,^t\!{\bar {\zeta}})\}}\,f(W+\zeta),
\end{aligned}$$
where $\zeta=\la+i\mu.$ In (**), we used (6.20) and the facts that
$2iz^-\,^tz^+={\frac i2}({\overline W}\,^t\!{\zeta}+{\overline W}\,^t\!W)$ and
$2z^+=W+\zeta.$\hfill \Box\vskip2mm

\begin{definition} The induced representation $U^{F,c}$ of $G$ in the
Fock space ${\mathcal  H}_{F,c}$ is called the $\textsf{Fock\ representation}$ of
$G$.
\end{definition}
\vskip2mm

Let $W=U+iV\in \BC^{(m,n)}$ with $U,V\in \BR^{(m,n)}.$ If $U=(u_{ka}),\,
V=(v_{lb})$ are coordinates in $\BC^{(m,n)},$ we put
$$dU=du_{11}du_{12}\cdots du_{mn},\ \ \
dV=dv_{11}dv_{12}\cdots dv_{mn}$$
and $dW=dUdV.$ And we set
\begin{equation}
d\mu(W)=e^{-\pi\s(W\,^t\!{\overline W})}\,dW.
\end{equation}
Let $f$ be a holomorphic function on $\BC^{(m,n)}$. Then $f(W)$ has the
Taylor expansion
$$f(z)=\sum_{\J}\,a_JW^J,\ \ \ W=(w_{ka})\in \BC^{(m,n)},$$
where $J=(J_{ka})\in \J$ and $W^J:= w_{11}^{J_{11}}w_{12}^{J_{12}}\cdots
w_{mn}^{J_{mn}}.$ \vskip2mm
 We set $\vert W\vert_{\infty}:=\textrm{max}_{k,a}(\vert w_{ka}
\vert ).$ Then by an easy computation, we have
$$\begin{aligned}
\int_{\BC^{(m,n)}}\,\vert f(W)\vert^2\,d\mu(W)
&=\lim_{r\rightarrow \infty}
\int_{\vert W\vert_{\infty}\leq r}\vert f(W)\vert^2
d\mu(W)\\
&=\lim_{r\rightarrow \infty}
\sum_{J,K}a_J{\overline {a_K}}\int_{\vert W\vert_{\infty}\leq
r}W^J{\overline {W^K}}\,d\mu(W)\\
&=\sum_J\vert a_J\vert^2\pi^{-\vert J\vert}J!,
\end{aligned}$$
where $J$ runs over $\J$.

Let ${\mathcal  H}_{m,n}$ be the Hilbert space consisting of all holomorphic
functions $f:\BC^{(m,n)}\lrt \BC$ satisfying the condition
\begin{equation}
\| f\|^2=\int_{\BC^{(m,n)}}\,\vert f(W)\vert^2\,d\mu(W) < \infty.
\end{equation}
The inner product $(\ \,,\ )$ on ${\mathcal  H}_{m,n}$ is given by
$$(f_1,f_2)=\int_{\BC^{(m,n)}}\,f_1(W)\,{\overline {f_2(W)}}\,d\mu(W),\ \
f_1,f_2\in {\mathcal  H}_{m,n}.$$

Thus we have

\vskip2mm

\begin{lemma}
Let $f\in {\mathcal  H}_{m,n}$ and let $f(W)=\sum_J\,a_JW^J$
be the Taylor expansion of $f$. Then
$$\| f\|^2=\sum_{J\in \BZ^{(m,n)}_{\geq 0}}\,
\vert a_J\vert^2\pi^{-\vert J\vert}J!.$$
\end{lemma}

For each $\J,$ we define the holomorphic function $\Phi_J(W)$ on
$\BC^{(m,n)}$ by
\begin{equation}
\Phi_J(W):=(J!)^{-{\frac 12}}\,\left(\pi^{\frac 12}W
\right)^J,\ \ \ W\in \BC^{(m,n)}.
\end{equation}
Then
\begin{equation}
\left( \Phi_J,\Phi_K\right)=\begin{cases} 1\ &\text{if\ $J=K$}\\
                               0\ &\text{otherwise.} \end{cases}
\end{equation}
It is easy to see that the set $\left\{\,\Phi_J\,\big|\,\J\,
\right\}$ forms a complete
orthonormal system in ${\mathcal  H}_{m,n}$. By the Schwarz inequality, for any
$f\in {\mathcal  H}_{m,n},$ we have
\begin{equation}
\vert f(W)\vert \leq e^{{{\pi}\over 2}\s(W\,^t\!{\overline W})}\,\|f\|,\ \
W\in \BC^{(m,n)}.
\end{equation}
Consequently, the norm convergence in ${\mathcal  H}_{m,n}$ implies the uniform
convergence on any bounded subset of $\BC^{(m,n)}.$ We observe that for a
fixed $W'\in \BC^{(m,n)},$ the holomorphic function $W\lrt e^{\pi \s
(W\,^t{\overline {W'}})}$ admits the following Taylor expansion
\begin{equation}
e^{\pi \s(W\,^t{\overline {W'}})}=\sum_{\J}\,\Phi_J(W)\,
\Phi_J({\overline {W'}}).
\end{equation}
From (6.29), we obtain
\begin{equation}
\Phi_J({\overline {W'}})=(J!)^{-{\frac 12}}\int_{\BC^{(m,n)}}\,
e^{\pi\s(W\,^t{\overline {W'}})}\,\left(\pi^{\frac 12}{\overline W}
\right)^J
\,d\mu(W).
\end{equation}
Thus if $f\in {\mathcal  H}_{m,n},$ we get
$$\begin{aligned}
\left( f(W),\,e^{\pi\s(W\,^t\!{\overline {W'}})}\right)&=
\left( f,\,\sum_J\Phi_J({\overline {W'}})\,\Phi_J(\cdot)\right)\\
&=\sum_J\Phi_J(W')\,(f,\Phi_J)\\
&=f(W').
\end{aligned}$$
Hence $e^{\pi\s(W\,^t\!{\overline {W'}})}$ is the reproducing kernel for
${\mathcal  H}_{m,n}$ in the sense that for any $f\in {\mathcal  H}_{m,n},$
\begin{equation}
f(W)=\int_{\BC^{(m,n)}}\,e^{\pi \s(W\,^t{\overline {W'}})}\,f(W')\,
d\mu(W').
\end{equation}
We set
\begin{equation}
\k(W,W'):=e^{\pi \s(W\,^t{\overline {W'}})},\ \ W,W'\in \BC^{(m,n)}.
\end{equation}
Obviously $\k(W,W')={\overline {\k(W',W)}}.$ (6.31) may be written as
\begin{equation}
f(W)=\int_{\BC^{(m,n)}}\,\k(W,W')\,f(W')\,d\mu(W'),\ \ f\in
{\mathcal  H}_{m,n}.
\end{equation}

Let $\M$ be a positive definite, symmetric half-integral matrix of degree
$m$. We define the measure
\begin{equation}
d\mu_{\M}(W)=e^{-2\pi \,\s(\M W\,^t{\overline W})}\,dW.
\end{equation}
We recall the $\textsf{Fock space}$ ${\mathcal  H}_{F,\M}$ consisting of all
holomorphic functions $f:\BC^{(m,n)}\lrt \BC$ that satisfy the condition
\begin{equation}
\| f\|^2_{\M}:=\| f\|_{F,\M}^2:=\int_{\BC^{(m,n)}}\,
\vert f(W)\vert^2\,d\mu_{\M}(W) < \infty.
\end{equation}
The inner product $(\ \,,\ )_{\M}:=(\ \,,\ )_{F,\M}$ on ${\mathcal  H}_{F,\M}$ is
given by
$$(f_1,f_2)_{\M}=\int_{\BC^{(m,n)}}\,f_1(W)\,{\overline {f_2(W)}}\,
d\mu_{\M}(W),\ \ f_1,f_2\in {\mathcal  H}_{F,\M}.$$
\vskip 0.2cm
\begin{lemma}
Let $f\in {\mathcal  H}_{F,\M}$ and let $g(W)=f\left(
(2\M)^{-{\frac 12}}W\right)$ be the holomorphic function on $\BC^{(m,n)}.$
We let
$$g(W)=\sum_{\J}\,a_{\M,J}\,W^J$$
be the Taylor expansion of $g(W).$ Then we have
$$\| f\|_{\M}^2=(f,f)_{\M}=2^{-n}(\det \M)^{-n}\sum_{\J}\vert a_{\M,J}
\vert^2\,\pi^{-\vert J\vert}J!.$$
\end{lemma}
\noindent
{\it Proof.} Let $\M^{\frac 12}$ be the unique positive definite symmetric
matrix of degree $m$ such that $\left( \M^{\frac 12}\right)^2=\M.$ We put
${\widetilde W}:={\sqrt {2}}\M^{\frac 12}W.$ Obviously $d{\widetilde W}=2^n\,
(det\,\M)^n dW.$ Thus for $f\in {\mathcal  H}_{F,\M},$ we have
$$\begin{aligned}
(f,f)_{\M}&=\int_{\BC^{(m,n)}}\vert f(W)\vert^2\,d\mu_{\M}(W)\\
&=2^{-n}(\det \M)^{-n}\int_{\BC^{(m,n)}}\,\vert g(W)\vert^2\,d\mu(W)\\
&=2^{-n}(\det \M)^{-n}\sum_{\J}\vert a_{\M,J}\vert^2\pi^{-\vert J\vert}
J!\ \ \ (\textrm{by\ Lemma\ 6.4})
\end{aligned}$$
\hfill \Box

For each $\J,$ we put
\begin{equation}
\Phi_{\M,J}(W):=2^{\frac n2}\left(\,\det\M\,
\right)^{\frac n2}(J!)^{-{\frac 12}}
\left( (2\pi\M)^{\frac 12}W\right)^J,\ \ W\in \BC^{(m,n)}.
\end{equation}
\vskip 0.2cm

\begin{lemma}
The set $\left\{\, \Phi_{\M,J}\,\big|\,\J\,\right\}$ is a
complete orthonormal system in ${\mathcal  H}_{F,\M}.$
\end{lemma}
\noindent
{\it Proof.} For $J,K\in \BZ^{(m,n)}_{\geq 0},$ we have
$$\begin{aligned}
\left(\Phi_{\M,J},\Phi_{\M,K}\right)_{\M}&=
2^n(\det\M)^n(J!)^{-{\frac 12}}(K!)^{-{\frac 12}}\\
&  \ \times\int_{\BC^{(m,n)}}
\,\left((2\pi\M)^{\frac 12}W\right)^J\left((2\pi\M)^{\frac 12}
{\overline W}
\right)^K\,d\mu_{\M}(W)\\
&=(J!)^{-{\frac 12}}(K!)^{-{\frac 12}}\int_{\BC^{(m,n)}}\,
(\pi^{\frac 12}W)^J\,{\overline {(\pi^{\frac 12}W)^K}}\,d\mu(W)\\
&=(\Phi_J,\Phi_K).
\end{aligned}$$
By (6.27), we have
\begin{equation}
(\Phi_{\M,J},\Phi_{\M,K})_{\M}=\begin{cases} 1\ &\text{if $J=K$}\\
0\ & \text{otherwise.}\end{cases}
\end{equation}
We leave the proof of the completeness to the reader.\hfill \Box

\vskip 0.2cm
We observe that for a fixed $W'\in \BC^{(m,n)},$ the holomorphic function
$W\lrt e^{\pi\s(\M W\,^t{\overline {W'}})}$ admits
the following Taylor expansion
\begin{equation}
e^{\pi\s(\M W\,^t{\overline {W'}})}=\sum_{\J}\,\Phi_{\M,J}(W)\,
\Phi_{\M,J}({\overline {W'}}).
\end{equation}
If $f\in {\mathcal  H}_{F,\M},$ we have
$$\begin{aligned}
\left( f(W),\,e^{\pi \s(\M W\,^t{\overline {W'}})}\right)_{\M}&=
\sum_{\J}\,(f,\Phi_{\M,J})_{\M}\,\Phi_{\M,J}(W')\\
&=f(W').
\end{aligned}$$
Hence $e^{\pi \s(\M W\,^t{\overline {W'}})}$ is the reproducing kernel for
${\mathcal  H}_{F,\M}$ in the sense that
\begin{equation}
f(W)=\int_{\BC^{(m,n)}}\,f(W')\,e^{\pi \,\s(\M W\,^t{\overline {W'}})}\,
d\mu_{\M}(W').
\end{equation}
For $U\in \BR^{(m,n)}$ and $W\in \BC^{(m,n)},$ we put
\begin{equation}
k(U,W):=e^{2\pi \,\s(-U\,^t\!U\,+\,{\frac 12}W\,^t\!W\,+\,2\,i\,U\,^t\!W)}.
\end{equation}
Then we have the following lemma.\vskip2mm
\begin{lemma}
$$\int_{\BR^{(m,n)}}\,k(U,W)\,{\overline {k(U,W')}}\,dU=
e^{2\pi \s(W\,^t\!W')}.$$
\end{lemma}
\noindent
{\it Proof.} We put
$${\mathcal  I}(W,W'):=\int_{\BR^{(m,n)}}\,k(U,W)\,{\overline {k(U,W')}}\,dU.$$
Then we have
$$\begin{aligned}
{\mathcal  I}(W,W')&=e^{\pi\, \s(W\,^tW+{\overline {W'}}\,^t
{\overline {W'}})}\,
\int_{\BR^{(m,n)}}e^{-4\pi \,\s(U\,^tU)}\,e^{4\pi i
\s\{ U\,^t(W-{\overline {W'}})
\}}\,dU\\
&=e^{\pi \,\s(W\,^tW+{\overline {W'}}\,^t{\overline {W'}})}\cdot
\prod_{k,a}\,\int_{\BR}\,
e^{-4\pi \{u_{ka}^2-iu_{ka}(w_{ka}-{\overline {w'_{ka}}})\}}\,du_{ka},
\end{aligned}$$
where $W=(w_{ka}),\,W'=(w'_{ka})\in \BC^{(m,n)}$ and $U=(u_{ka})\in
\BR^{(m,n)}.$ It is easy to show that
$$\int_{\BR}\,e^{-4\pi \{ u_{ka}^2-iu_{ka}(w_{ka}-
{\overline {w'_{ka}} })\}}\,
du_{ka}=e^{-\pi (w_{ka}-{\overline {w'_{ka}} })^2}.$$
Thus we get
$$\begin{aligned}
{\mathcal  I}(W,W')&=e^{\pi\, \s(W\,^tW+{\overline {W'}}\,^t
{\overline {W'}})}\cdot
e^{-\pi \sum_{k,a}(w_{ka}-{\overline {w'_{ka}} })^2}\\
&=e^{2\pi \sum_{k,a}w_{ka}{\overline {w'_{ka}} } }\\
&=e^{2\pi \,\s(W\,^t{\overline {W'}})}.
\end{aligned}$$
\hfill \Box\vskip2mm

For $U\in \BR^{(m,n)}$ and $W\in \BC^{(m,n)},$ we put
\begin{equation}
k_{\M}(U,W):=e^{2\pi\, \s\{\M(-U\,^t\!U-{\frac 12}W\,^t\!W+2U\,^t\!W)\}}.
\end{equation}
\vskip 0.2cm
\begin{lemma}
Let $\M$ be a positive definite, symmetric half-integral
matrix of degree $m$. Then we have
\begin{equation}
k_{\M}(U,W)=k(\M^{\frac 12}U,\,-i\,\M^{\frac 12}W)
\end{equation}
and
\begin{equation}
\int_{\BR^{(m,n)}}k_{\M}(U,W)\,{\overline {k_{\M}(U,W')}}\,dU=
(\det\M)^{-{\frac n2}}\cdot e^{2\pi \s(\M W\,
^t{\overline {W'}})}.
\end{equation}
\end{lemma}
\noindent
{\it Proof.} The formula (6.42) follows immediately from a straightforward
computation. We put
$${\mathcal  I}_{\M}(W,W'):=\int_{\BR^{(m,n)}}\,k_{\M}(U,W)\,
{\overline {k_{\M}(U,W')}}\,dU.$$
Using (6.42), we have
$$\begin{aligned}
{\mathcal  I}_{\M}(W,W')&=\int_{\BR^{(m,n)}}\,k\left( \M^{\frac 12}U,\,
-i\M^{\frac 12}W\right)\, {\overline {k\left( \M^{\frac 12}U,-i
\M^{\frac 12}W'\right)}}\,dU\\
&=(\det\M)^{-{\frac n2}}\int_{\BR^{(m,n)}}\,k\left( U,-i\M^{\frac 12}W
\right)\,
{\overline {k\left( U,-i\M^{\frac 12}W'\right)}}\,dU\\
&=(\det\M)^{-{\frac n2}}\cdot e^{2\pi\, \s(\M W\,^t{\overline {W'}})}\ \ \
\ \ \ (\textrm{by\ Lemma}\ 6.7)
\end{aligned}$$
\hfill \Box
\vskip2mm

We recall that the Fock representation $U^{F,\M}$ of the real Heisenberg
group $G$ in
${\mathcal  H}_{F,\M}$(cf.\,(6.23)) is given by
\begin{equation}
\left( U^{F,\M}(g)f\right)(W)=e^{2\pi i\,\s(\M \k)}\cdot
e^{-\pi\, \s\{ \M(\zeta\,^t{\bar {\zeta}}\,+\,2\,W\,^t\!{\bar {\zeta}})\}}\,
f(W+\zeta),
\end{equation}
where $g=(\la,\mu,\k)\in G,\ f\in {\mathcal  H}_{F,\M}$ and $\zeta=\la+i\,\mu  \in
\BC^{(m,n)}.$\vskip2mm
\begin{lemma}
The Fock representation $U^{F,\M}$ of $G$ in
${\mathcal  H}_{F,\M}$ is unitary.
\end{lemma}
\noindent
{\it Proof.} For brevity, we put $U_{g,f}(W):=\left( U^{F,\M}(g)f\right)(W)$
for $g=(\la,\mu,\k)\in G$ and $f\in {\mathcal  H}_{F,\M}.$ Then we have
$$\begin{aligned}
(U_{g,f},\,U_{g,f})_{\M}&=\| U_{g,f}\|_{\M}^2\\
&=\int_{\BC^{(m,n)}}\,U_{g,f}(W)\,{\overline {U_{g,f}(W)}}\,d\mu_{\M}(W)\\
&=\int_{\BC^{(m,n)}}\,e^{-\pi \, \s\{ \M (\zeta\,{}^t\!{\bar {\zeta}}\,+\,
2W\,{}^t{\bar {\zeta}}\,+\, {\bar {\zeta}}\,{}^t\!\zeta\,+\,2\,{\overline W}\,{}^tW\,+\,
2\,W\,{}^t{\overline W})\}}\,\vert f(W+\zeta)\vert^2\,dW\\
&=\int_{\BC^{(m,n)}}\,\vert f(W)\vert^2\,d\mu_{\M}(W)\\
&=(f,f)_{\M}=\| f\|_{\M}^2.
\end{aligned}$$
\hfill \Box\vskip2mm

We recall that the Schr{\" {o}}dinger representation $U^{S,\M}:=U_{\s_{\M}}$
of the real Heisenberg group $G$ in the Hilbert space ${\mathcal  H}_{S,\M}\cong
L^2\left( \BR^{(m,n)},\,d\xi\right)$\,(cf.\,(5.8)) is given by
\begin{equation}
\left( U^{S,\M}(g)f\right) (\xi)=
e^{2\pi i\,\s \{ \M(\k\,+\,\mu\,{}^t\!\la\,+\,2\,\mu\,{}^t\!\xi)\}}\,f(\xi+\la),
\end{equation}
where $g=(\la,\mu,\k)\in G,\ f\in {\mathcal  H}_{S,\M}$ and $\xi\in
\BR^{(m,n)}.\ U^{S,\M}$ is called the $\textsf{Schr{\" {o}}dinger representation of}$
$G$ $\textsf{of index}$ $\M$. The inner product $(\ ,\ )_{S,\M}$ on
${\mathcal  H}_{S,\M}$ is given by
$$(f_1,\,f_2)_{S,\M}=\int_{\BR^{(m,n)}}\,f_1(U)\,
{\overline {f_2(U)}}\,dU,\ \ \ f_1,f_2\in {\mathcal  H}_{S,\M}.$$
And we define the norm $\|\ \,\|_{S,\M}$ on ${\mathcal  H}_{S,\M}$ by
$$\| f\|_{S,\M}^2=\int_{\BR^{(m,n)}}\,\vert f(U)\vert^2\,dU,\ \ \
f\in {\mathcal  H}_{S,\M}.$$
\vskip 0.2cm
\begin{theorem}
The Fock representation $\big(U^{F,\M},{\mathcal  H}_{F,\M} \big)$ of
$G$ is untarily equivalent to the Schr{\" {o}}dinger representation
$\big(U^{S,\M},\,{\mathcal  H}_{S,\M} \big)$ of $G$ of index $\M$. Therefore the Fock
representation $U^{F,\M}$ is irreducible. The intertwining unitary isometry
$I_{\M}:{\mathcal  H}_{S,\M}\lrt {\mathcal  H}^{F,\M}$ is given by
\begin{equation}
\left(\,I_{\M}f\,\right)(W)=\int_{\BR^{(m,n)}}\,k_{\M}(\xi,W)\,
f(\xi)\,d\xi,
\end{equation}
where $f\in {\mathcal  H}_{S,\M}=L^2\left( \BR^{(m,n)},\,d\xi\right),\ W\in
\BC^{(m,n)}$ and $k_{\M}(\xi,W)$ is a function on $\BR^{(m,n)}\times
\BC^{(m,n)}$ defined by (6.41).
\end{theorem}
\noindent
{\it Proof.} For any $f\in {\mathcal  H}_{S,\M}=L^2\left(\BR^{(m,n)},\,d\xi
\right),$ we define
$$\left( I_{\M}f\right)(W)=\int_{\BR^{(m,n)}}\,k_{\M}(\xi,W)\,f(\xi)\,
d\xi,\ \ W\in \BC^{(m,n)}.$$
Now we will show the following (I1),\,(I2) and (I3):
 (I1) The image of ${\mathcal  H}_{S,\M}$ under $I_{\M}$ is contained in
${\mathcal  H}_{F,\M}.$\vskip2mm
 (I2) $I_{\M}$ preserves the norms, i.e., $\|f\|_{S,\M}=
\|I_{\M}f\|_{\M}.$\vskip2mm
 (I3) $I_{\M}$ is a bijective operator of ${\mathcal  H}_{S,\M}$ onto
${\mathcal  H}_{F,\M}.$

Before we prove (I1),\,(I2) and (I3), we prove the following lemma.
\vskip2mm
\begin{lemma}
For a fixed $U\in \BR^{(m,n)},$ we consider the Taylor
expansion
\begin{equation}
k_{\M}(U,W)=\sum_{\J}\,h_J(U)\,\Phi_{\M,J}(W),\ \ \ W\in \BC^{(m,n)}
\end{equation}
of the holomorphic function $k_{\M}(U,\,\cdot\,)$ on $\BC^{(m,n)}.$ Then
the set $\left\{\,h_J\,\big|\ \J\,\right\}$ forms a complete orthonormal
system in $L^2\left(\BR^{(m,n)},\,d\xi\right).$ Moreover, for a fixed
$W\in \BC^{(m,n)},$ (6.47) is the Fourier expansion of $k_{\M}(\,\cdot\,,
W)$ with respect to this orthonormal system $\left\{\,h_J\,\big|\ \J\,
\right\}.$
\end{lemma}
\noindent
{\it Proof.} Following Igusa \cite{I},\,pp.\,33-34, we can prove it. The detail
will be left to the reader.
\hfill \Box\vskip2mm

If $f\in {\mathcal  H}_{S,\M},$ then by the Schwarz inequality and Lemma 6.8,
(6.43), we have
$$\begin{aligned}
\vert\, \left(\,I_{\M}f\,\right)(W)\vert &\leq
\left(\,\int_{\BR^{(m,n)}}\,\vert k_{\M}(U,W)\vert^2\,dU\,\right)^{\frac 12}
\cdot \left(\,\int_{\BR^{(m,n)}}\,\vert f(U)\vert^2\,dU\,\right)^{\frac 12}\\
&=\left( \det\M\right)^{-{\frac n4}}\cdot e^{\pi\, \s(\M W\,
^t{\overline W})}\,
\| f\|_{S,\M}.
\end{aligned}$$
Thus the above integral $(I_{\M}f)(W)$ converges uniformly on any compact
subset of $\BC^{(m,n)}$ and hence $(I_{\M}f)(W)$ is holomorphic in
$\BC^{(m,n)}.$ And according to Lemma 6.11, we get
$$\begin{aligned}
\left(\,I_{\M}f\,\right)(W)&=\sum_{\J}\,\int_{\BR^{(m,n)}}\,h_J(U)\,
f(U)\,\Phi_{\M,J}(W)\,dU\\
&=\sum_{\J}\,(h_J,\,{\bar f}\,)_{S,\M}\,\Phi_{\M,J}(W).
\end{aligned}$$
Therefore we get
$$\begin{aligned}
\| I_{\M}f\|_{F,\M}^2&=\int_{\BC^{(m,n)}}\vert I_{\M}f (W)\vert^2\,
d\mu_{\M}(W)\\
&=\sum_{J,\,K\in \BZ^{(m,n)}_{\geq 0}}\,(h_J,\,{\bar f}\,)_{S,\M}\cdot
{\overline {(h_K,\,{\bar f}\,)}}\,\int_{\BC^{(m,n)}}\,\Phi_{\M,J}(W)\,
{\overline {\Phi_{\M,K}(W)}}\,d\mu_{\M}(W)\\
&=\sum_{\J}\,\vert (h_j,\,{\bar f}\,)_{S,\M}\vert^2\ \ \ \ \ (\textrm{by}\ (6.37))\\
&=\| f\|^2_{S,\M} < \infty.
\end{aligned}$$
This proves (I1) and (I2). It is easy to see that $I_{\M}{\overline {h_J}}=
\Phi_{\M,J}$ for all $\J.$ Since the set $\left\{\,\Phi_{\M,J}\,\big|\
\J\,\right\}$ forms a complete orthonormal system of ${\mathcal  H}_{F,\M},\
I_{\M}$ is surjective. Obviously the injectivity of $I_{\M}$ follows
immediately from the fact that $I_{\M}{\overline {h_J}}=\Phi_{\M,J}$ for all
$\J.$ This proves (I3).
\vskip  0.2cm
On the other hand, we let $f\in {\mathcal  H}_{S,\M}$ and
$g=(\la,\mu,\k)\in G.$ We put $\zeta=\la+i\,\mu.$ Then we get
$$\begin{aligned}
& \left(\,U^{F,\M}(g)(I_{\M}f)\right)(W)\\
&=e^{2\pi i\,\s(\M\k)}\cdot e^{-\pi \,\s\{\M(\zeta\,^t\!{\bar {\zeta}}\,+\,
2\, W\,^t{\bar {\zeta}})\}}\,(I_{\M}f)(W+\zeta)\ \ \ (\,\textrm{by}\ (6.44)\,)\\
&=e^{2\pi i\, \s (\M\k)}\cdot e^{-\pi \s\{ \M(\zeta\,^t\!{\bar {\zeta}}\,+\,
2\,W\,^t{\bar {\zeta}})\}}\,\int_{\BR^{(m,n)}}\,k_{\M}(U,W+\zeta)\,f(U)\,dU.
\end{aligned}$$
We define the function $A_\M:\BR^{(m,n)}\times \BR^{(m,n)}\lrt \BC$ by
\begin{equation}
A_{\M}(U,W):=\s\left\{\,\M \left(-U\,^tU-{{W\,^t\!W}\over 2}+2U\,^t\!W
\right)\right\}.
\end{equation}
Obviously $\k_{\M}(U,W)=e^{2\pi A_{\M}(U,W)}$ for $U\in \BR^{(m,n)}$ and
$W\in \BC^{(m,n)}.$
\vskip 0.2cm
By an easy computation, we get
$$A_{\M}(U,W+\zeta)-A_\M(U-\la,W)=\s \left\{ \M \left(\,{{\zeta\,^t\!
{\bar {\zeta}}}\over 2}+W\,^t\!{\bar {\zeta}}-i\,\la\,^t\mu\,+\,2\,i\, U\,^t\!\mu
\right)\right\}.$$
Therefore we get
$$\begin{aligned}
&  k_{\M}(U,W+\zeta)\\
&=e^{2\pi A_{\M}(U-\la,W)}\cdot
   e^{2\pi \,\s\left\{ \M\left( {\frac 12}\,\zeta\,^t{\bar {\zeta}}\,+\,
   W\,^t\!{\bar {\zeta}}-i\,\la\,^t\!\mu\,+\,2\,i\,U\,{}^t\!\mu\right)\right\} }\\
&=k_{\M}(U-\la,W)\cdot e^{2\pi\, \s\left\{ \M\left( {\frac 12}\, \zeta
\,^t\!{\bar {\zeta}}\,+\,W\,{}^t{\bar {\zeta}}-i\,\la\,^t\!\mu\,+\,2\,i\,U\,^t\!\mu\right)
\right\} }.
\end{aligned}$$
Hence we have
$$\begin{aligned}
& \left( U^{F,\M}(g)(I_{\M}f)\right)(W)\\
&=\int_{\BR^{(m,n)}}\,e^{2\pi i\, \s \left\{ \M(\k\,+\, 2\,U\,^t\!\mu-\la\,{}^t\!\mu)
\right\} }\,k_{\M}(U-\la,W)\,f(U)\,dU\\
&=\int_{\BR^{(m,n)}}\,e^{2\pi i\,\s\left\{ \M(\k\,+\,2\,\la\,{}^t\!\mu\,+\, 2\,U\,{}^t\!\mu-
\la\,^t\!\mu)\right\} }\,k_{\M}(U,W)\,f(U+\la)\,dU\\
&=\int_{\BR^{(m,n)}}\,e^{2\pi i\,\s\left\{ \M(\k\, +\, 2\, U\,{}^t\!\mu\,+\,\la\,{}^t\!\mu)
\right\} }\,k_{\M}(U,W)\,f(U+\la)\,dU\\
&=\int_{\BR^{(m,n)}}\,k_{\M}(U,W)\,\left(\,U^{S,\M}(g)f\,\right)(U)\,dU
\ \ \ \ \ (\textrm{by}\ (6.45))\\
&=\left(\,I_{\M}\left( U^{S,\M}(g)f\right)\,\right)(W).
\end{aligned}$$
So far we proved that $U^{F,\M}\circ I_{\M}=I_{\M}\circ U^{S,\M}(g)$ for all
$g\in G.$ That is, the unitary isometry $I_{\M}$ of ${\mathcal  H}_{S,\M}$ onto
${\mathcal  H}_{F,\M}$ is the intertwining operator. This completes the proof.
\vskip2mm\hfill \Box

The infinitesimal representation $dU^{F,\M}$ associated to
the Fock representation
$U^{F,\M}$ is given as follows. \vskip2mm
\begin{proposition}
Let $\M$ be as before. We put
$$\M=(\,\M_{kl}\,),\ \ \ (2\pi\M)^{\frac 12}=(\tau_{kl}),$$
where $\tau_{kl}\in \BR$ and $1\leq k,l\leq m.$ For each $J=(J_{ka})\in
\BZ^{(m,n)}_{\geq 0}$ and $W=(W_{ka})\in \BC^{(m,n)},$ we have
\begin{equation}
 dU^{F,\M}(D_{kl}^0)\,\Phi_{\M,J}(W)=2\,\pi\, i\,\M_{kl}\,
\Phi_{\M,J}(W),\ \ \ 1\leq k\leq l\leq m.
\end{equation}
\begin{equation}
\begin{aligned}
 dU^{F,\M}(D_{ka})\,\Phi_{\M,J}(W)&=-2\,\pi\,\left(
\,\sum_{p=1}^m\,\M_{pk}W_{pa}\,\right)\,\Phi_{\M,J}(W) \\
& \ \ \ \ \ \ +\,
\sum_{p=1}^m\,\tau_{pk}J_{pa}^{\frac 12}\,\Phi_{\M,J-\epsilon_{pa}}(W).
\end{aligned}
\end{equation}
\begin{equation}
\begin{aligned}
 dU^{F,\M}({\widehat D}_{lb}\,\Phi_{\M,J}(W)&=\,2\,\pi \,i\,
\left(\,\sum_{p=1}^m\,\M_{pl}W_{pb}\,\right)\,\Phi_{\M,J}(W)  \\
& \  \ \ \ \ \ +\,
i\sum_{p=1}^m\,\tau_{pl}J_{pb}^{\frac 12}\,\Phi_{\M,J-\epsilon_{lb}}(W).
\end{aligned}
\end{equation}
\end{proposition}
\noindent
{\it Proof.} We put $\,E_{kl}^0={\frac 12}(E_{kl}+E_{lk}),$ where
$1\leq k\leq l\leq m.$
$$\begin{aligned}
dU^{F,\M}(D_{kl}^0)\,\Phi_{\M,J}(W)&=\dt U^{F,\M}(\exp\,tX_{kl}^0)\,
\Phi_{\M,J}(W)\\
&=\dt U^{F,\M}\left( (0,0,tE_{kl}^0)\right)\,\Phi_{\M,J}(W)\\
&=\lt\,{ { e^{2\pi i\, \s(t\M E_{kl}^0)}-I}\over t }\,\Phi_{\M,J}(W)\\
&=\lt\,{ { e^{2\pi it\M_{kl}}-I }\over t }\,\Phi_{\M,J}(W)\\
&=2\,\pi\, i\,\M_{kl}\,\Phi_{\M,J}(W).
\end{aligned}$$
And we have
$$\begin{aligned}
dU^{F,\M}(D_{ka})\,\Phi_{\M,J}(W)&=\dt U^{F,\M}(\exp\,tX_{ka})\,
\Phi_{\M,J}(W)\\
&=\dt U^{F,\M}\left( (tE_{ka},0,0) \right)\Phi_{\M,J}(W)\\
&=\dt e^{-\pi t^2\,\s(\M E_{ka}\,^tE_{ka})-2\pi t\, \s(\M W\,^tE_{ka})}\,
\Phi_{\M,J}(W+tE_{ka})\\
&=-2\,\pi\,\left(\,\sum_{p=1}^m\,\M_{pk}W_{pa}\,\right)\,\Phi_{\M,J}(W)\\
& \ \ \ \ +\,
\dt \Phi_{\M,J}(W+tE_{ka})\\
&=-2\,\pi\,\left(\,\sum_{p=1}^m\,\M_{pk}W_{pa}\,\right)\,\Phi_{\M,J}(W)\\
& \ \ \ \ +\,\sum_{p=1}^m\,\tau_{pk}\,J_{pa}^{\frac 12}\,
\Phi_{\M,J-\epsilon_{pa}}(W).
\end{aligned}$$
Finally,
$$\begin{aligned}
dU^{F,\M}({\widehat D}_{lb})\,\Phi_{\M,J}(W)&=\dt U^{F,\M}(\exp\,t{\widehat X}_{lb})
\Phi_{\M,J}(W)\\
&=\dt U^{F,\M}\left( (0,tE_{lb},0)\right)\Phi_{\M,J}(W)\\
&=\dt e^{-\pi t^2\,\s(\M E_{lb}\,{}^tE_{lb})
\,+\,2\pi it\,\s(\M W\,^tE_{lb}) }
\,\Phi_{\M,J}(W+i\,tE_{lb})\\
&=2\pi i\,\left( \,\sum_{p=1}^m\,\M_{pl}W_{pb}\,\right)\,\Phi_{\M,J}(W)\\
& \ \ \ \ +\,\dt \Phi_{\M,J}(W+i\,tE_{lb})\\
&=2\pi i\,\left(\,\sum_{p=1}^m\,\M_{pl}W_{pb}\,\right)\,\Phi_{\M,J}(W)\\
& \ \ \ \ +\,
i\,\sum_{p=1}^m\,\tau_{pl}\,J_{pb}^{\frac 12}\,\Phi_{\M,J-\epsilon_{pb}}(W).
\end{aligned}$$
$\hfill \square$
\end{section}

\newpage


\begin{section}{{\large\bf Lattice Representations}}
\setcounter{equation}{0}

Let $L:=\BZ^{(m,n)}\times \BZ^{(m,n)}$ be the lattice in the vector space
$V\cong \BC^{(m,n)}.$ Let $B$ be an alternating bilinear form on $V$ such
that $B(L,L)\subset \BZ,$ that is, $\BZ$-valued on $L\times L.$ The dual
$L_B^{*}$ of $L$ with respect to $B$ is defined by
$$L_B^{*}:=\left\{\,v\in V\,\vert\ B(v,L)\in \BZ\ for\ all\ l\in L\,
\right\}.$$
Then $L\subset L_B^{*}.$ If $B$ is nondegenerate, $L_B^{*}$ is also a
lattice in $V,$ called the {\it dual\ lattice} of $L$. In case $B$ is
nondegenerate, there exist a $\BZ$-basis $\{\,\xi_{11},\xi_{12},\cdots,
\xi_{mn},\eta_{11},\eta_{12},\cdots,\eta_{mn}\,\}$ of $L$ and a set
$\{\,e_{11},e_{12},\cdots,e_{mn}\,\}$ of positive integers with
$e_{11}\vert e_{12},\,e_{12}\vert e_{13},\cdots,e_{m,n-1}\vert e_{mn}$
such that
$$\begin{pmatrix} B(\xi_{ka},\xi_{lb})& B(\xi_{ka},\eta_{lb})\\
B(\eta_{ka},\xi_{lb}) & B(\eta_{ka},\eta_{lb}\end{pmatrix}=
\begin{pmatrix} 0&e\\ -e&0\end{pmatrix},$$
where $1\leq k,l\leq m,\,1\leq a,b\leq n$ and $e:=\textrm{diag}\,(e_{11},e_{12},
\cdots,e_{mn})$ is the diagonal matrix of degree $mn$ with entries
$e_{11},e_{12},\cdots,e_{mn}.$ It is well known that $[L_B^{*}:L]=
(\,\det\,e\,)^2=(e_{11}e_{12}\cdots e_{mn})^2$\,(cf.\,\cite{I}\,p.\,72). The
number $\,\det\,e\,$ is called the $\textsf{Pfaffian}$ of $B.$

Now we consider the following subgroups of $G$:
\begin{equation}
\G_L=\left\{\,(\la,\mu,\k)\in G\,\vert\ (\la,\mu)\in L,\ \k\in
\BR^{(m,m)}\,\right\}
\end{equation}
and
\begin{equation}\G_{L_B^{*}}=\left\{\,(\la,\mu,\k)\in G\,\vert\ (\la,\mu)\in L_B^{*},\
\k\in \BR^{(m,m)}\,\right\}.
\end{equation}
Then both $\G_L$ and $\G_{L_B^{*}}$ are the normal subgroups of $G.$
\vskip2mm
We put
\begin{equation}
{\mathcal  Z}_0=\left\{\,(0,0,\k)\in {\mathcal  Z}\,\vert\ \k=\,^t\k\in
\BZ^{(m,m)}\ \textrm{integral}\ \right\}.
\end{equation}
It is easy to show that
$$\G_{L_B^{*}}=\left\{\,g\in G\,\vert\ g\gamma g^{-1}\gamma^{-1}\in
{\mathcal  Z}_0\ \textrm{for\ all}\ \gamma\in \G_L\,\right\}.$$
We define
\begin{equation}
Y_L=\left\{\,\phi\in {\text{ Hom}}\,(\G_L,\BC_1^{\times})\,\vert\
\phi\ \textrm{is\ trivial\ on}\ {\mathcal  Z}_0\,\right\}
\end{equation}
and
\begin{equation}
Y_{L,S}=\left\{\,\phi\in Y_L\,\vert\ \phi(\k)=e^{2\pi i\s(S\k)}\ \textrm{for
\ all}\ \k=\,^t\k\in \BR^{(m,m)}\,\right\}
\end{equation}
for each symmetric real matrix $S$ of degree $m.$ We observe that if $S$
is not half-integral, then $Y_L=\emptyset$ and so $Y_{L,S}=\emptyset.$ It
is clear that if $S$ is symmetric half-integral, then $Y_{L,S}$ is
not empty.
Thus we have
\begin{equation}
Y_L=\cup_{\M}\,Y_{L,\M},
\end{equation}
where $\M$ runs through the set of all symmetric half-integral matrices of
degree $m$.\vskip2mm
\begin{lemma}
Let $\M$ be a symmetric half-integral matrix of degree
$m$ with $\M\neq 0.$ Then any element $\phi$ of $Y_{L,\M}$ is of the form
$\phi_{\M,q}.$ Here $\phi_{\M,q}$ is the character of $\G_L$ defined by
\begin{equation}
\phi_{\M,q}((l,\k)):=e^{2\pi i\,\s(\M \k)}\cdot e^{\pi i\,q(l)},\ \
(l,\k)\in \G_L,
\end{equation}
where $q:L\lrt \BR/2\BZ\cong [0,2)$ is a function on $L$ satisfying the
following condition:
\begin{equation}
q(l_0+l_1)\equiv q(l_0)\,+\,
q(l_1)-2\s\{ \M(\la_0\,^t\!\mu_1-\mu_0\,^t\!\la_1)\}
\ \ {\rm mod}\ 2
\end{equation}
for all $l_0=(\la_0,\mu_0)\in L$ and $l_1=(\la_1,\mu_1)\in L.$
\end{lemma}
\noindent
{\it Proof.} (7.8) follows immediately from the fact that $\phi_{\M,q}$ is
a character of $\G_L.$ It is obvious that any element of $Y_{L,\M}$ is of
the form $\phi_{\M,q}.$ \hfill \Box\vskip2mm

\begin{lemma}
An element of $Y_{L,0}$ is of the form $\phi_{k,l}\,
(k,l\in \BR^{(m,n)}).$ Here $\phi_{k,l}$ is the character of $\G_L$ defined
by
\begin{equation}
\phi_{k,l}(\gamma):=e^{2\pi i\, \s(k\,^t\!\la\,+\,l\,^t\!\mu)},\ \
\gamma=(\la,\mu,\k)\in \G_L.
\end{equation}
\end{lemma}
\noindent
{\it Proof.} It is easy to prove it and so we omit the proof.
\hfill \Box\vskip2mm
\begin{lemma}
Let $\M$ be a nonsingular symmetric half-integral matrix
of degree $m$. Let $\phi_{\M,q_1}$ and $\phi_{\M,q_2}$ be the characters of
$\G_L$ defined by (7.7). The character $\phi$ of $\G_L$ defined by
$\phi:= \phi_{\M,q_1}\cdot \phi_{\M,q_2}^{-1}$
is an element of $Y_{L,0}.$
\end{lemma}
\noindent
{\it Proof.} It follows from the fact that there exists an element
$g=(\M^{-1}\la,\M^{-1}\mu,0)\in G$ with $(\la,\mu)\in V$ such that
$$\phi_{\M,q_1}(\gamma)=\phi_{\M,q_2}(g\gamma g^{-1})\ \ \ \textrm{for\ all}\
\gamma\in \G_L.$$
\hfill \Box

We note that the alternating bilinear form ${\bA}$ on $V$ defined by
(6.2) is nondegenerate and $\BZ$-valued on $L\times L.$ According to (6.3),
the elementary divisors $e_{11},e_{12},\cdots,e_{mn}$ of ${\bA}$ are all one
and $L$ is self-dual, i.e., $L=L^{*}_{\bA}.$ The set
$$\left\{\, P_{11},
P_{12},\cdots,
P_{mn},Q_{11},Q_{12},\cdots,Q_{mn}\,\right\}$$ forms a
symplectic basis of $V$ with respect to ${\bA}$. We fix a coordinate
$P_{11},\cdots,P_{mn},$\\ $Q_{11},\cdots,Q_{mn}$ on $V$.

For a unitary character $\varphi_{\M,q}$ of $\G_L$ defined by (7.7), we let
\begin{equation}
\pi_{\M,q}={\text{ Ind}}_{\G_L}^G\,\varphi_{\M,q}
\end{equation}
be the representation of $G$ induced from $\varphi_{\M,q}.$ Let
${\mathcal  H}_{\M,q}$ be the Hilbert space consisting of all measurable functions
$\phi:G\lrt \BC$ satisfying
\vskip 0.2cm
 (L1) $\phi(\gamma g)=\varphi_{\M,q}(\gamma)\,\phi(g)$ for all
$\gamma\in \G_L$ and $g\in G.$\vskip2mm
 (L2)\ \ $\|\phi\|^2_{\M,q}=\int_{\G_L\ba G}\,\vert \phi({\bar {g}})
\vert\,d{\bar {g}} l\infty,\ \ {\bar {g}}=\G_L g.$
\vskip 0.5cm
The induced representation $\pi_{\M,q}$ is realized in ${\mathcal  H}_{\M,q}$
as follows:
\begin{equation}
\biggl(\,\pi_{\M,q}(g_0)\phi\,\biggr)(g)=\phi(gg_0),\ \ g_0,g\in G,\
\phi\in {\mathcal  H}_{\M,q}.
\end{equation}
$\pi_{\M,q}$ is called the $\textsf{lattice\ representation}$ of $G$ associated
with the lattice $L$.\vskip2mm

\begin{theorem}
Let $\M$ be a positive definite, symmetric half integral
matrix of degree $m$. Let $\varphi_{\M}$ be the character of $\G_L$ defined
by $\varphi_{\M}((\la,\mu,\k)):=e^{2\pi i\, \s(\M \k)}$ for all $(\la,\mu,\k)\in
\G_L.$ Then the representation
\begin{equation}
\pi_{\M}:={\rm {Ind}}_{\G_L}^G\,\varphi_{\M}
\end{equation}
induced from the character $\varphi_{\M}$ is unitarily equivalent to
the representation
$$\bigoplus\,U_{\M}:=\bigoplus\,{\rm {Ind}}_K^G\,\s_{\M}\ \ \
(\,(\,\hbox{$\det\,2\M\,)^n$-copies}\,),$$
where $K$\,(resp.\,$\s_{\M}$) is defined by (5.1)\,(resp.\,(5.6)).
\end{theorem}
\noindent
{\it Proof.} We first recall that the induced representation $\pi_{\M}$ is
realized in the Hilbert space ${\mathcal  H}_{\M}$ consisting of all
measurable functions $\phi:G\lrt \BC$ satisfying the conditions
\begin{equation}
\phi((\la_0,\mu_0,\k_0)\circ g)=e^{2\pi i\,\s(\M \k_0)}\,\phi(g),\ \ \
(\la_0,\mu_0,\k_0)\in \G_L,\ g\in G
\end{equation}
and
\begin{equation}
\|\phi\|^2_{\pi,\M}:=\int_{\G_L\ba G}\,\vert\phi({\bar g})\vert^2\,
d{\bar g} l\infty,\ \ {\bar g}=\G_L\circ g.
\end{equation}
Now we write
$$g_0=[\la_0,\mu_0,\k_0]\in \G_L\ \ and\ \ g=[\la,\mu,\k]\in G.$$
For $\phi\in {\mathcal  H}_{\M},$ we have
\begin{equation}
\phi(g_0\diamond g)=\phi([\la_0+\la,\mu_0+\mu,\k_0+\k+\la_0\,^t\!\mu+\mu\,
^t\!\la_0]).
\end{equation}
On the other hand, we get
$$\begin{aligned}
\phi(g_0\diamond g)&=\phi((\la_0,\mu_0,\k_0-\mu_0\,^t\!\la_0)\circ g)\\
&=e^{2\pi i\, \s\{ \M (\k_0-\mu_0\,^t\!\la_0)\}}\,\phi(g)\\
&=e^{2\pi i\,  \s(\M \k_0)}\,\phi(g)\ \ \ (\,\textrm{because}\ \s(\M\mu_0\,^t\!\la_0)
\in \BZ\,)
\end{aligned}$$
Thus putting $\k^{\prime}:=\k_0+\la_0\,^t\!\mu+\mu\,^t\!\la_0,$ we get
\begin{equation}
\phi([\la_0+\la,\mu_0+\mu,\k+\k'])=e^{2\pi i\, \s(\M\k')}\cdot
e^{-4\pi i\, \s(\M \la_0\,^t\!\mu)}\,\phi([\la,\mu,\k]).
\end{equation}

Putting $\la_0=\k_0=0$ in (7.16), we have
\begin{equation}
\phi([\la,\mu+\mu_0,\k])=\phi([\la,\mu,\k])\ \ \textrm{for\ all}\ \mu_0\in
\BZ^{(m,n)}\ \textrm{and}\ [\la,\mu,\k]\in G.
\end{equation}
Therefore if we fix $\la$ and $\k,\ \phi$ is periodic in $\mu$ with respect to
the lattice $\BZ^{(m,n)}$ in $\BR^{(m,n)}.$ We note that
$$\phi([\la,\mu,\k])=\phi([0,0,\k]\diamond [\la,\mu,0])=e^{2\pi i\,\s(\M \k)}\,
\phi([\la,\mu,0])$$
for $[\la,\mu,\k]\in G.$ Hence $\phi$ admits a Fourier expansion in $\mu:$
\begin{equation}
\phi([\la,\mu,\k])=e^{2\pi i\,\s(\M \k)}\sum_{\N}\,c_N(\la)\,
e^{2\pi i\,\s(N\,^t\!\mu)}.
\end{equation}
If $\la_0\in \BZ^{(m,n)},$ then we have
$$\begin{aligned}
\phi([\la+\la_0,\mu,\k])&=e^{2\pi i\, \s(\M \k)}\sum_{\N}\,
c_N(\la+\la_0)\,e^{2\pi i\, \s(N\,^t\mu)}\\
&=e^{-4\pi i\, \s(\M \la_0\,^t\mu)}\,\phi([\la,\mu,\k])\ \ \
\ \ \ (\,\textrm{by}\ (7.16)\,)\\
&=e^{-4\pi i\, \s(\M \la_0\,^t\mu)}\,e^{2\pi i\, \s(\M\k)}\sum_{\N}
c_N(\la)\,e^{2\pi i\, \s(N\,^t\mu)},\\
&=e^{2\pi i\, \s(\M\k)}\,\sum_{\N}\,c_N(\la)\,
e^{2\pi i\, \s\{ (N-2\M\la_0)\,^t\mu\}}.\ \ \ (\,\textrm{by}\ (7.18)\,)
\end{aligned}$$
So we get
$$\begin{aligned}
&  \sum_{\N}\,c_N(\la+\la_0)\,e^{2\pi i\, \s(N\,^t\mu)}\\
=&\sum_{\N}\,c_N(\la)\,e^{2\pi i\,\s\{ (N-2\M\la_0)\,^t\mu\}}\\
=&\sum_{\N}\,c_{N+2\M\la_0}(\la)\,e^{2\pi i\, \s(N\,^t\mu)}.
\end{aligned}$$
Hence we get
\begin{equation}
c_N(\la+\la_0)=c_{N+2\M\la_0}(\la)\ \ \textrm{for\ all}\ \la_0\in \BZ^{(m,n)}\ \textrm{and}\
\la\in \BR^{(m,n)}.
\end{equation}
Consequently, it is enough to know only the coefficients $c_{\a}(\la)$
for the representatives $\a$ in $\BZ^{(m,n)}$ modulo $2\M$. It is obvious
that the number of all such representatives $\a$'s is $(\det\,2\M)^n.$
We denote by ${\mathcal  J}$ a complete system of such representatives
$\a$'s in $\BZ^{(m,n)}$ modulo $2\M.$
Then we have
$$\begin{aligned}
\ \ \phi([\la,\mu,\k])\\
=e^{2\pi i\, \s(\M\k)}\,\ &\biggl\{\, \sum_{\N}\,c_{\a+2\M N}(\la)\,
e^{2\pi i\, \s\{ (\a+2\M N)\,^t\mu\}}\\
&+\sum_{\N}\,c_{\be+2\M N}(\la)\,e^{2\pi i\s\{ (\be+2\M N)\,^t\mu\}}\\
&\ \, \vdots\\
&+\sum_{\N}\,c_{\gamma+2\M N}(\la)\,e^{2\pi i\,\s\{ (\gamma+2\M N)\,^t\mu\}}
\,\biggr\} ,
\end{aligned}$$
where $\{\,\a,\be,\cdots,\gamma\,\}$ denotes the complete system
${\mathcal  J}.$

For each $\a\in {\mathcal  J},$ we denote by ${\mathcal  H}_{\M,\a}$ the Hilbert
space consisting of Fourier expansions
$$e^{2\pi i\, \s(\M \k)}\,\sum_{\N}\,c_{\a+2\M N}(\la)\,
e^{2\pi i\, \s\{\,(\a+2\M N)\,^t\mu\}},\ \ \ (\la,\mu,\k)\in G,$$
where $c_N(\la)$ denotes the coefficients of the Fourier expansion (7.18) of
$\phi\in {\mathcal  H}_{\M}$ and $\phi$ runs over the set $\{\,\phi\in
\pi_{\M}\,\}$. It is easy to see that ${\mathcal  H}_{\M,\a}$ is invariant under
$\pi_{\M}.$ We denote the restriction of $\pi_{\M}$ to ${\mathcal  H}_{\M,\a}$
by $\pi_{\M,\a}.$ Then we have
\begin{equation}
\pi_{\M}=\bigoplus_{\a\in {\mathcal  J}}\,\pi_{\M,\a}.
\end{equation}
Let $\phi_{\a}\in \pi_{\M,\a}.$ Then for $[\la,\mu,\k]\in G,$ we get
\begin{equation}
\phi_{\a}([\la,\mu,\k])=e^{2\pi i\,\s(\M \k)}\,\sum_{\N}\,
c_{\a+2\M N}(\la)\,e^{2\pi i\, \s\{ (\a\,+\,2\M N)\,^t\mu\}}.
\end{equation}
We put
$$I_{\la}=\overbrace{[0,1]\times [0,1]\times \cdots \times [0,1]}^{(m\times
n)\text{-times}}
\, \subset \left\{\,[\la,0,0]\,\vert\ \la\in \BR^{(m,n)}\,\right\}$$
and
$$I_{\mu}=\overbrace{[0,1]\times [0,1]\times \cdots \times [0,1]}^{(m\times
n)\text{-times}}
\, \subset\, \left\{\,[0,\mu,0]\,\vert\ \mu\in \BR^{(m,n)}\right\}.$$
Then we obtain
\begin{equation}
\int_{I_{\mu}}\,\phi_{\a}([\la,\mu,\k])\,e^{-2\pi i\, \s(\a\,^t\mu)}\,d\mu
=e^{2\pi i\, \s(\M \k)}\,c_{\a}(\la),\ \ \a\in {\mathcal  J}.
\end{equation}
Since $\G_L\ba G\cong I_{\la}\times I_{\mu},$ we get
$$\begin{aligned}
\|\phi_{\a}\|^2_{\pi,\M,\a}:&=\|\phi_{\a}\|^2_{\pi,\M}=
\int_{\G_L\ba G}\,\vert \phi_{\a}({\bar g})\vert^2\,d{\bar g}\\
&=\int_{I_{\la}}\int_{I_{\mu}}\,\vert \phi_{\a}({\bar g})\vert^2\,d\la d\mu\\
&=\int_{I_{\la}\times I_{\mu}}\biggl|\sum_{\N}\,
c_{\a+2\M N}(\la)\,e^{2\pi i\, \s\{ (\a\,+\,2\M N)\,^t\mu\}}\,\biggr|^2\,
d\la d\mu\\
&=\int_{I_{\la}}\,\sum_{\N}\,\vert c_{\a+2\M N}(\la)\vert^2\, d\la\\
&=\int_{I_{\la}}\,\sum_{\N}\,\vert c_{\a}(\la+N)\vert^2\,d\la\ \ \
\ \ \ (\,by\ (7.19)\,)\\
&=\int_{\BR^{(m,n)}}\,\vert c_{\a}(\la)\vert^2\,d\la.
\end{aligned}$$
Since $\phi_{\a}\in \pi_{\M,\a},\ \|\phi_{\a}\|_{\pi,\M,\a} < \infty$ and
so $c_{\a}(\la)\in L^2\left(\BR^{(m,n)},d\xi\right)$ for all
$\a\in {\mathcal  J}.$
\vskip 0.2cm
For each $\a\in {\mathcal  J},$ we define the mapping $\vth_{\M,\a}$ on
$L^2(\BR^{(m,n)},d\xi)$ by
\begin{eqnarray}
(\vth_{\M,\a}f)([\la,\mu,\k])&=& e^{2\pi i\, \s(\M \k)}\\
 & & \times \sum_{\N}\,
f(\la+N)\,e^{2\pi i\, \s\{ (\a\,+\,2\M N)\,^t\mu\}},\nonumber
\end{eqnarray}
where $f\in L^2\big(\BR^{(m,n)},d\xi \big)$ and $[\la,\mu,\k]\in G.$
\vskip2mm
\begin{lemma}
For each $\a\in {\mathcal  J},$ the image of
$L^2\left( \BR^{(m,n)},d\xi\right)$ under $\vth_{\M,\a}$ is contained in
${\mathcal  H}_{\M,\a}.$ Moreover, the mapping
$\vth_{\M,\a}$ is a one-to-one unitary
operator of $L^2\left(\BR^{(m,n)},d\xi\right)$ onto ${\mathcal  H}_{\M,\a}$
preserving the norms. In other words, the mapping
$$\vth_{\M,\a}:L^2\left( \BR^{(m,n)},d\xi\right)\lrt {\mathcal  H}_{\M,\a}$$
is an isometry.
\end{lemma}
\noindent
{\it Proof.} We already showed that $\vth_{\M,\a}$ preserves the norms.
First we observe that if $(\la_0,\mu_0,\k_0)\in \G_L$ and $g=[\la,\mu,\k]
\in G,$
$$\begin{aligned}
(\la_0,\mu_0,\k_0)\circ g&=[\la_0,\mu_0,\k_0+\mu_0\,^t\!\la_0]\diamond
[\la,\mu,\k]\\
&=[\la_0+\la,\mu_0+\mu,\k+\k_0+\mu_0\,^t\!\la_0+\la_0\,^t\mu+\mu\,^t\!\la_0].
\end{aligned}$$
Thus we get
$$\begin{aligned}
& (\vth_{\M,\a}f)((\la_0,\mu_0,\k_0)\circ g)\\
&=e^{2\pi i\, \s\{ \M(\k+\k_0+\mu_0\,^t\!\la_0+\la_0\,^t\mu+\mu\,^t\!\la_0)\}}\,\\
& \quad\times
\sum_{\N}\,f(\la+\la_0+N)\,e^{2\pi i\, \s\{ (\a\,+\,2\M N)\,^t(\mu_0+\mu)\}}\\
&=e^{2\pi i\, \s (\M \k_0)}\cdot e^{2\pi i\, \s (\M \k)}
\cdot e^{2\pi i\, \s(\a\,^t\!\mu_0)}\,
\sum_{\N}\,f(\la+N)\,e^{2\pi i\, \s\{ (\a+2\M N)\,^t\mu\}}\\
&=e^{2\pi i\, \s (\M \k_0)}\,(\vth_{\M,\a}f)(g).
\end{aligned}$$
Here in the above equalities we used the facts that
$2\s(\M N\,^t\mu_0)\in \BZ$ and $\a\,^t\mu_0\in \BZ.$ It is easy to show
that
$$\int_{\G_L\ba G}\,\vert \vth_{\M,\a}f({\bar g})\vert^2\,d{\bar g}=
\int_{\BR^{(m,n)}}\,\vert f(\la)\vert^2\,d\la=\|f\|^2_2 < \infty.$$
This completes the proof of Lemma 7.5.\ \

Finally it is easy to show that for each $\a\in {\mathcal  J},$ the mapping
$\vth_{\M,\a}$ intertwines the Schr{\" {o}}dinger representation
$\left( U^{S,\M},L^2\big(\BR^{(m,n)},d\xi \big)\right)$ and the representation
$(\pi_{\M,\a},{\mathcal  H}_{\M,\a}).$ Therefore, by Lemma 7.5, for each
$\a\in {\mathcal  J},\ \pi_{\M,\a}$ is unitarily equivalent to $U(\s_{\M})$ and
so $\pi_{\M,\a}$ is an irreducible unitary representation of $G.$
According to (7.20), the induced representation $\pi_{\M}$ is unitarily
equivalent to
$$\bigoplus\,U_{\M}\ \ \ (\,(\,\hbox{$\det\,2\M)^n$-copies}\,).$$
This completes the proof of Theorem 7.4. \hfill \Box
\vskip2mm

Now we state the connection between the lattice representation and theta
functions. As before, we write $\,V=\BR^{(m,n)}\times \BR^{(m,n)}\cong
\BC^{(m,n)},\ L=\BZ^{(m,n)}\times \BZ^{(m,n)}$ and $\M$ is a positive
symmetric half-integral matrix of degree $m$. The function
$q_{\M}:L\lrt \BR / 2\BZ\,=\,[0,2)$ defined by
\begin{equation}
q_{\M}((\xi,\eta))=\,2\,\s(\M \xi\,^t\eta),\ \ \ (\xi,\eta)\in L
\end{equation}
satisfies the condition (7.8). We let $\,\varphi_{\M,q_{\M}}:\G_L\lrt
\BC_1^*$ be the character of $\G_L$ defined by
$$\varphi_{\M,q_{\M}}((l,\kappa))\,=\,
e^{2\pi i\,\s(\M\kappa)}\,e^{\pi i\,q_{\M}(l)}\,,\ \ \ (l,\kappa)\in
\G_L.$$
We denote by ${\mathcal  H}_{\M,q_{\M}}$ the Hilbert space consisting of
measurable functions $\phi:G\lrt \BC$ which satisfy the conditions (7.24)
and (7.25):
\begin{equation}
\phi((l,\kappa)\circ g)\,=\,\varphi_{\M,q_{\M}}((l,\kappa))\,
\phi(g) \ \  \textrm{for\ all}\ (l,\kappa)\in \G_L   \textrm{and}\ g\in G.
\end{equation}
\vskip2mm
\begin{equation}
\int_{{\G_L}\ba G}\,\| \phi({\dot g})\|^2\,d{\dot g}\,l \,\infty,\ \ \
{\dot g}=\G_L\circ g.
\end{equation}
Then the representation
$$\pi_{\M,q_{\M}}=\text{ Ind}_{\G_L}^G\,\varphi_{\M,q_{\M}}$$
of $G$ induced from the character $\varphi_{\M,q_{\M}}$ is realized in
${\mathcal  H}_{\M,q_{\M}}$ as
$$\left(\,\pi_{\M,q_{\M}}(g_0)\,\phi\,\right)(g)\,=\,\phi(gg_0),\ \ \
g_0,g\in G,\ \phi\in {\mathcal  H}_{\M,q_{\M}}.$$
Let ${\bold H}_{\M,q_{\M}}$ be the vector space consisting of measurable functions
$\,F:V\lrt \BC\,$ satisfying the conditions (7.26) and (7.27).
\begin{equation}
F(\la+\xi,\mu+\eta)\,=\,e^{2\pi i\,\s\{ \M(\xi\,^t\eta+\la\,^t\eta-\mu\,
^t\xi)\} }\,F(\la,\mu)
\end{equation}
for all $(\la,\mu)\in V$ and $(\xi,\eta)\in L.$
\begin{equation}
\int_{L\ba V}\,\| F({\dot v})\|^2\,d{\dot g}\,=\,
\int_{I_{\la}\times I_{\mu}}\,\| F(\la,\mu)\|^2\,d\la d\mu\,l\,\infty.
\end{equation}

Given $\,\phi\in {\mathcal  H}_{\M,q_{\M}}$ and a fixed element $\Omega\in {\mathbb H}_n,$
we put\vskip2mm\noindent
\begin{equation}
E_{\phi}(\la,\mu)=\,\phi((\la,\mu,0)),\ \ \ \la,\mu\in
\BR^{(m,n)},
\end{equation}
\begin{equation}
F_{\phi}(\la,\mu)=\,\phi([\la,\mu,0]),\ \ \ \la,\mu\in
\BR^{(m,n)},
\end{equation}
\begin{equation}
F_{\Omega,\phi}(\la,\mu)=\,e^{-2\pi i\,\s(\M \la\Omega\,^t\!\la)}
\,F_{\phi}(\la,\mu),\ \ \ \la,\mu\in \BR^{(m,n)}.
\end{equation}

In addition, we put for $W=\la\Omega+\mu\in \BC^{(m,n)},$
\begin{equation}
\vth_{\Omega,\phi}(W)=\,\vth_{\Omega,\phi}(\la\Omega+\mu):=\,F_{\Omega,\phi}
(\la,\mu).
\end{equation}
We observe that $\,E_{\phi},\,F_{\phi},\,F_{\Omega,\phi}\,$ are functions
defined on $V$ and $\,\vth_{\Omega,\phi}\,$ is a function defined on
$\BC^{(m,n)}.$
\vskip2mm
\begin{proposition}
If $\,\phi\in {\mathcal  H}_{\M,q_{\M}},\ (\xi,\eta)\in L$
and $(\la,\mu)\in V,$ then we have the formulas
\begin{equation}
 E_{\phi}(\la+\xi,\mu+\eta)\,=\,
e^{2\pi i\,\s\{ \M(\xi\,^t\eta\,+\,\la\,^t\eta-\mu\,^t\xi)\} }\,E_{\phi}(\la,\mu).
\end{equation}
 \begin{equation}
 F_{\phi}(\la+\xi,\mu+\eta)\,=\,
e^{ -4\pi i\,\s(\M\xi\,^t\mu) }\,F_{\phi}(\la,\mu).
\end{equation}
\begin{equation}
 F_{\Omega,\phi}(\la+\xi,\mu+\eta)\,=\,
e^{ -2\pi i\,\s\{ \M(\xi\Omega\,^t\xi\,+\,2\,\la\Omega\,^t\xi\,+\,2\,\mu\,^t\xi)\} }\,
F_{\Omega,\phi}(\la,\mu).
\end{equation}

If $\,W=\la\Omega+\eta\in \BC^{(m,n)},$ then we have
\begin{equation}
\vth_{\Omega,\phi}(W+\xi\Omega+\eta)\,=\,
e^{ -2\pi i\,\s\{ \M(\xi\Omega\,{}^t\xi\,+\,2\,W\,^t\xi)\} }\,\vth_{\Omega,\phi}
(W).\end{equation}
Moreover, $\,F_{\phi}\,$ is an element of ${\bold H}_{\M,q_{\M}}.$
\end{proposition}
\noindent
{\it Proof.} We note that
$$(\la+\xi,\mu+\eta,0)=(\xi,\eta,-\xi\,^t\mu+\eta\,^t\!\la)\circ
(\la,\mu,0).$$
Thus we have
$$\begin{aligned}
E_{\phi}(\la+\xi,\mu+\eta)&=\phi((\la+\xi,\mu+\eta,0))\\
&=\phi((\xi,\eta,\,-\xi\,^t\mu+\eta\,^t\!\la)\circ (\la,\mu,0))\\
&=\,e^{2\pi i\, \s\{ \M(\xi\,^t\eta\,+\,\la\,^t\eta-\mu\,^t\xi)\} }\,
\phi((\la,\mu,0))\\
&=\,e^{ 2\pi i\,\s\{ \M(\xi\,^t\eta\,+\, \la\,^t\eta-\mu\,^t\xi)\} }\,E_{\phi}
(\la,\mu).
\end{aligned}$$
This proves the formula (7.33).
\vskip2mm
We observe that
$$[\la+\xi,\mu+\eta,0]=(\xi,\eta,\,-\xi\,^t\mu-\mu\,^t\xi-\eta\,^t\xi)\circ
[\la,\mu,0].$$
Thus we have
$$\begin{aligned}
F_{\phi}(\la+\xi,\mu+\eta)&=\phi([\la+\xi,\mu+\eta,0])\\
&=\,e^{ -2\pi i\,\s\{ \M(\xi\,{}^t\mu\,+\,\mu\,^t\xi\,+\,\eta\,^t\xi)\} }\\
& \ \ \ \ \times e^{2\pi i\,\s(\M\xi\,^t\eta)}\,\phi([\la,\mu,0])\\
&=\,e^{ -4\pi i\,\s(\M\xi\,^t\mu) }\,\phi([\la,\mu,0])\\
&=\,e^{ -4\pi i\,\s(\M\xi\,^t\mu) }\,F_{\phi}(\la,\mu).
\end{aligned}$$
This proves the formula (7.34).\vskip2mm

According to (7.34), we have
$$\begin{aligned}
F_{\Omega,\phi}(\la+\xi,\mu+\eta)&=
\,e^{-2\pi i\,\s\{ \M (\la+\xi)\Omega\,^t(\la+\xi)\} }\,F_{\phi}(\la+\xi,\mu+
\eta)\\
&=\, e^{ -2\pi i\,\s\{ \M (\la+\xi)\Omega\,^t(\la+\xi)\} }\\
&  \ \ \ \ \times e^{ -4\pi i\,\s(\M\xi\,^t\mu) }\,F_{\phi}(\la,\mu)\\
&=\,e^{ -2\pi i\,\s\{ \M (\xi\Omega\,^t\xi\,+\,2\, \la\Omega\,^t\xi\,+\, 2\, \mu\,^t\xi)\} }
\\
& \ \ \ \ \times e^{ -2\pi i\,\s(\M\la\Omega\,^t\!\la) }\,F_{\phi}(\la,\mu)\\
&=\,e^{ -2\pi i\,\s\{ \M(\xi\Omega\,^t\xi\,+\,2\,\la\Omega\,^t\xi\,+\,2\,\mu\,^t\xi)\} }
\, F_{\Omega,\phi}(\la,\mu).
\end{aligned}$$
This proves the formula (7.34). The formula (7.35) follows immediately from
the formula (7.34). \vskip2mm

Indeed, if $\,W=\la\Omega+\mu\,$ with $\la,\mu\in \BR^{(m,n)},$ we have
$$\begin{aligned}
\vth_{\Omega,\phi}(W+\xi\Omega+\eta)&=F_{\Omega,\phi}(\la+\xi,\mu+\eta)\\
&=\,e^{ -2\pi i\,\s\{ \M(\xi\Omega\,{}^t\xi\,+\,2\,(\la\Omega+\mu)\,^t\xi)\} }\,
F_{\Omega,\phi}(\la,\mu)\\
&=\,e^{ -2\pi i\,\s\{ \M(\xi\Omega\,{}^t\xi\,+\,2\,W\,{}^t\xi)\} }\,
\vth_{\Omega,\phi}(W).
\end{aligned}$$
\hfill \Box

\begin{remark}
The function $\,\vth_{\Omega,\phi}(W)\,$ is a theta
function of level $2\M$ with respect to $\Omega$ if $\,\vth_{\Omega,\phi}\,$
is holomorphic. For any $\,\phi\in {\mathcal  H}_{\M,q_{\M}},$ the function
$\,\vth_{\Omega,\phi}\,$ satisfies the transformation law (3.1) of a theta
function. In this sense, the lattice representation
$\,(\,\pi_{\M,q_{\M}},\,{\mathcal  H}_{\M,q_{\M}}\,)$ is closely related to
\end{remark}

\end{section}

\newpage


\begin{section}{{\large\bf The Coadjoint Orbits of Picture }}
\setcounter{equation}{0}

In this section, we find the coadjoint orbits of the Heisenberg group
$H^{(n,m)}_\BR$ and describe the connection between the coadjoint orbits and
the unitary dual of $H^{(n,m)}_\BR$ explicitly.\par
For brevity, we let $G := H^{(n,m)}_\BR$ as before. Let $\frak g$ be the Lie
algebra of $G$ and let ${\frak g}^*$ be the dual space of $\frak g.$ We observe
that $\frak g$ can be regarded as the subalgebra consisting of all $(m+n)\times
(m+n)$ real matrices of the form
$$X(\alpha,\beta,\gamma):= \begin{pmatrix} 0&0&0&{}^t\!\beta\\
                                    \alpha&0&\beta&\gamma\\
                    0&0&0&-{}^t\!\alpha\\
            0&0&0&0 \end{pmatrix}, \;\alpha,\beta \in \BR^{(m,n)},\;
                    \gamma={}^t\!\gamma \in \BR^{(m,m)} $$
of the Lie algebra $\frak {sp} (m+n,\BR)$ of the symplectic group $Sp(m+n,\BR).$
An easy computation yields
$$[X(\alpha,\beta,\gamma), \,X(\delta,\epsilon,\xi)] = X(0,0,\alpha\,{}^t\epsilon+
\epsilon\,{}^t\alpha-\beta\,{}^t\delta-\delta\,{}^t\beta).$$
The dual space ${\frak g}^*$ of $\frak g$ can be identified with the vector
space consisting of all $(m+n)\times (m+n)$ real matrices of the form
$$ F(a,b,c) := \begin{pmatrix} 0& {}^t\!a&\ 0&0\\
                  0&0&\ 0&0\\
          0&{}^t\!b&\ 0& 0\\
          b&c&-a&0\end{pmatrix},\; a,b \in \BR^{(m,n)},\;c ={}^t\!c \in
          \BR^{(m,m)} $$
so that
\begin{eqnarray}
\langle F(a,b,c), X(\alpha,\beta,\gamma)\rangle :&=& \sigma(F(a,b,c)\,X(\alpha,\beta,\gamma))\\
&=&2\, \sigma({}^t\alpha \,a +{}^t\!b\,\beta)+\sigma(c\,\gamma).\nonumber
\end{eqnarray}
The adjoint representation $Ad$ of $G$ is given by $Ad_G(g)X = gX g^{-1}$
for $g\in G$ and $X\in \frak g.$ For $g \in G$ and $F\in {\frak g}^*, \;
gFg^{-1}$ is not of the form $F(a,b,c).$ We denote by $(gFg^{-1})_*$ the
$$\begin{pmatrix} 0&*&0&0 \\
           0&0&0&0 \\
       0&*&0&0\\
       *&*&*&0 \end{pmatrix}-\text{part}$$
of the matrix $gFg^{-1}.$ Then it is easy to see that the coadjoint representation
$Ad_G^* : G \longrightarrow GL({\frak g}^*)$ is given by $Ad_G^*(g)F = (gFg^{-1})_*,$
where $g \in G$ and $F\in {\frak g}^*.$ More precisely,
\begin{equation}
Ad_G^*(g)F(a,b,c) = F(a+c\mu, b-c\lambda,c),
\end{equation}
where $g=(\lambda,\mu,\kappa)\in G.$ Thus the coadjoint orbit $\Omega_{a,b}$
of $G$ at $F(a,b,0) \in {\frak g}^*$ is given by
\begin{equation}
\Omega_{a,b} = Ad_G^*(G)\,F(a,b,0)=\{F(a,b,0)\},\;\text{a single point}
\end{equation}
and the coadjoint orbit $\Omega_c$ of $G$ at $F(0,0,c)\in {\frak g}^*$ with
$c \ne 0$ is given by
\begin{equation}
\Omega_c = Ad_G^*(G)\,F(0,0,c) = \left\{\, F(a,b,c)\, \big|\ a,b \in \BR^{(m,n)}\,\right\}.
\end{equation}
Therefore the coadjoint orbits of $G$ in $\frak g^*$ fall into two classes:
\vskip 0.2cm
\par
({\bf I}) \ \ The single point $\left\{ \,\Omega_{a,b}\,\big|\  a,b \in \BR^{(m,n)}\,\right\}$
 located in the plane $c=0.$ \par
({\bf II}) \ \ The affine planes $\left\{ \, \Omega_c \,\big| \  c={}^t\!c \in \BR^{(m,m)},
\;c\ne 0\, \right\}$ parallel to the homogeneous plane $c=0.$ \par \noindent

\vskip 0.2cm
In other words, the orbit space $\mathcal  O (G)$ of coadjoint orbits is
parametrized by
$$\begin{cases} c\!-\!\text{axis},\;c\ne 0, \; c={}^t\!c \in \BR^{(m,m)} ; \\
(a,b)\!-\!\text{plane} \cong  \BR^{(m,n)} \times \BR^{(m,n)}. \end{cases}$$
The single point coadjoint orbits of the type $\Omega_{a,b}$ are said to be the
$\textsf{degenerate}$ orbits of $G$ in $\frak g^*.$ On the other hand, the flat
coadjoint orbits of the type $\Omega_c$ are said to be the $\textsf{non-degenerate}$
orbits of $G$ in $\frak g^*.$
\vskip 0.3cm



Since $G$ is connected and simply connected 2-step nilpotent Lie group,
according to A. Kirillov (cf. \cite{K1} or \cite{K2} p.249, Theorem 1), the unitary dual
$\widehat{G}$ of $G$ is given by
\begin{equation}
\widehat{G} = \left( \BR^{(m,n)}\times \BR^{(m,n)} \right) \coprod \left\{
z \in \BR^{(m,m)} \;\vert\;  z={}^t\!z ,\; z\ne 0 \right\},
\end{equation}
where $\coprod$ denotes the disjoint union. The topology of $\widehat{G}$ may be
described as follows. The topology on $\{c\!\!-\!\!\text{axis}-(0)\}$ is the usual topology
of the Euclidean space and the topology on $\{F(a,b,0) \vert a,b \in
\BR^{(m,n)} \}$ is the usual Euclidean topology. But a sequence on the $c$-axis
which converges to $0$ in the usual topology converges to the whole Euclidean
space $\BR^{(m,n)}\times \BR^{(m,n)}$ in the topology of $\widehat{G}.$ This is
just the quotient topology on ${\frak g}^*/G$ so that algebraically and
topologically $\widehat{G} = {\frak g}^* \slash G.$
\vskip 0.2cm
\par
It is well known that each coadjoint orbit is a symplectic manifold. We will
state this fact in detail. For the present time being, we fix an element $F$
of $\frak g^*$ once and for all. We consider the alternating $\BR$-bilinear
form ${\bold B}_F$ on $\frak g$ defined by
\begin{equation}
{\bold B}_F(X,Y) := \langle\, F,[X,Y]\,\rangle=\langle\, ad_{\frak g}^*(Y)F,X\,\rangle,\quad X,Y \in {\frak g},
\end{equation}
where $ad^*_{\frak g} : {\frak g} \longrightarrow \text{End}({\frak g}^*)$
denotes the differential of the coadjoint representation
$Ad^*_G : G \lrt GL({\frak g}^*).$ More precisely, if
$F=F(a,b,c),\; X=X(\alpha,\beta,\gamma),\;\text{and}\; Y=X(\delta,\epsilon,\xi),$
then
\begin{equation}
{\bold B}_F(X,Y) = \sigma \{ c\,(\alpha\,{}^t\epsilon + \epsilon\,{}^t\alpha -\beta\,{}^t\delta
-\delta\,{}^t\beta)\}.
\end{equation}
For $F\in {\frak g}^*,$ we let
$$G_F  = \left\{ \,g \in G \,\vert\  Ad_G^*(g)F=F\, \right\} $$
be the stabilizer of the coadjoint action $Ad^*$ of $G$ on ${\frak g}^*$ at $F.$
Since $G_F$ is a closed subgroup of $G,$ $G_F$ is a Lie subgroup of $G.$ We
denote by ${\frak g}_F$ the Lie subalgebra of $\frak g$ corresponding to
$G_F.$ Then it is easy to show that
\begin{equation}
{\frak g}_F = \text{rad}\,{\bold B}_F = \left\{\, X \in {\frak g}\, \vert\ ad^*_{\frak g}(X)F=
0\,\right\}.
\end{equation}
Here $\text{rad}\,{\bold B}_F$ denotes the radical of ${\bold B}_F$ in $\frak g.$
We let ${\dot{\bold B}}_F$ be the non-degenerate alternating $\BR$-bilinear form
on the quotient vector space $\frak g \slash \textrm{rad}\;{\bold B}_F$ induced from
${\bold B}_F.$ Since we may identify the tangent space of the coadjoint orbit
$\Omega_F \cong G \slash G_F $ with $\frak g \slash {\frak g}_F = \frak g
\slash \textrm{rad}\,{\bold B}_F,$ we see that the tangent space of $\Omega_F$ at $F$ is a
symplectic vector space with respect to the symplectic form ${\dot{\bold B}}_F.$
\vskip 0.2cm
\par
Now we are ready to prove that the coadjoint orbit $\Omega_F = Ad^*_G(G)F$ is
a symplectic manifold. We denote by $\widetilde X$ the smooth vector field on
$\frak g^*$ associated to $X\in \frak g.$ That means that for each $\ell \in
{\frak g}^*,$ we have
\begin{equation}
{\widetilde X}(\ell) = ad^*_{\frak g}(X)\; \ell.
\end{equation}
We define the differential 2-form $B_{\Omega_F}$ on $\Omega_F$ by
\begin{equation}
 B_{\Omega_F} (\widetilde X, \widetilde Y) = B_{\Omega_F}( ad^*_{\frak g}(X)F,
ad^*_{\frak g}(Y)F) :={\bold B}_F(X,Y),
\end{equation}
where $X,Y \in {\frak g}.$
\vskip 0.2cm
\begin{lemma}
$B_{\Omega_F}$ is non-degenerate.
\end{lemma}\noindent
{\it Proof.} \ \ Let $\widetilde X$ be the smooth vector field on $\frak g^*$
 associated to $X \in \frak g$ such that $B_{\Omega_F}(\widetilde X, \widetilde
Y)= 0$ for all $\widetilde Y$ with $Y \in \frak g.$ Since
$B_{\Omega_F}(\widetilde X, \widetilde Y) = {\bold B}_F(X,Y) =0$ for all
$Y \in {\frak g},\; X \in {\frak g}_F.$ Thus $\widetilde{X} =0.$ Hence
$B_{\Omega_F}$ is non-degenerate. \hfill\Box

\vskip 0.2cm
\begin{lemma}
$B_{\Omega_F}$ is closed.
\end{lemma}
\noindent{\it Proof.} \ \ If $\wt{ X_1},\, \wt{X_2}, \text{and } \wt{X_3}$ are three smooth
vector fields on $\frak g^*$ associated to $X_1,\,X_2,\, X_3 \in \frak g,$ then
$$\begin{aligned}
dB_{\Omega_F}(\wt{X_1},\wt{X_2},\wt{X_3})
&= \wt{X_1}(B_{\Omega_F}(\wt{X_2},\wt{X_3}))
-\wt{X_2}(B_{\Omega_F}(\wt{X_1},\wt{X_3}))+\wt{X_3}(B_{\Omega_F}
(\wt{X_1},\wt{X_2}))\\
&-B_{\Omega_F}([\wt{X_1},\wt{X_2}],\wt{X_3})+B_{\Omega_F}([\wt{X_1},\wt{X_3}],
\wt{X_2})-B_{\Omega_F}([\wt{X_2},\wt{X_3}],\wt{X_1})\\
&=-\langle F,[[X_1,X_2],X_3]+[[X_2,X_3],X_1]+[[X_3,X_1],X_2]\rangle\\
&=0 \qquad(\text{by the Jacobi identity}).
\end{aligned}$$
Therefore $B_{\Omega_F}$ is closed. \hfill\Box

\vskip 0.2cm
In summary, $(\Omega_F, B_{\Omega_F})$ is a symplectic manifold of dimension
$2mn$ or $0.$
\indent
In order to describe the irreducible unitary representations of $G$
corresponding to the coadjoint orbits under the Kirillov correspondence, we
have to determine the polarizations of $\frak g$ for the linear forms $F\in
{\frak g}^*.$

\vskip 0.2cm
${\sc Case}$ {\bf I}.\ \ $F=F(a,b,0);$ the degenerate case.
\vskip 0.1cm
\par
According to (8.3), $\Omega_F = \Omega_{a,b} =\{F(a,b,0)\}$ is a single point.
It follows from (8.7) that ${\bold B}_F(X,Y) =0$ for all $X,Y \in \frak g.$
Thus $\frak g$ is the unique polarization of $\frak g$ for $F.$ The Kirillov
correspondence says that the irreducible unitary representation $\pi_{a,b}$ of
$G$ corresponding to the coadjoint orbit $\Omega_{a,b}$ is given by
\begin{equation}
\pi_{a,b}( \exp\, X(\alpha,\beta,\gamma))
= e^{2\pi i\, \langle F,X(\alpha,\beta,\gamma)\rangle} =
e^{4 \pi i \,\sigma(\,{}^ta\alpha\, +\,{}^tb\beta)}.
\end{equation}
That is, $\pi_{a,b}$ is a one-dimensional degenerate representation of $G.$

\vskip 0.2cm
${\sc Case}$ {\bf II}. \ \ $F=F(0,0,c),\; 0\ne c={}^t c \in \BR^{(m,m)}:$
the non-degenerate case.
\vskip 0.1cm\par
According to (8.4), $\Omega_F =\Omega_c =\big\{ F(a,b,c)\, \vert \ a,b \in \BR^{(m,n)}\, \big\}.$
By (8.7), we see that
\begin{equation}
\frak k = \big\{\,X(0,\beta,\gamma)\, \big|\ \beta \in \BR^{(m,n)},\;
 \gamma={}^t\gamma \in \BR^{(m,m)}\,\big\}
 \end{equation}
 is a polarization of $\frak g$ for $F,$ i.e.,$\frak k$ is a Lie subalgebra of
 $\frak g$ subordinate to $F\in \frak g^*$ which is maximal among the totally
 isotropic vector subspaces of $\frak g$ relative to the alternating $\BR$-bilinear
 form ${\bold B}_F.$ Let $K$ be the simply connected Lie subgroup of $G$
 corresponding to the Lie subalgebra $\frak k$ of $\frak g.$ We let
 $$ \chi_{c,\frak k} : K \lrt \BC^{\times}_1 $$
 be the unitary character of $K$ defined by
 \begin{equation}
 \chi_{c,\frak k}(\exp X(0,\beta,\gamma)) =
 e^{2\pi i \,\langle F, X(0,\beta,\gamma)\rangle}= e^{2 \pi i\, \sigma(c\gamma)}.
 \end{equation}
 The Kirillov correspondence says that the irreducible unitary representation
 $\pi_{c,\frak k}$ of $G$ corresponding to the coadjoint orbit
 $\Omega_F =\Omega_c$ is given by
 \begin{equation}
 \pi_{c,\frak k} = \text{Ind}_K^G\,\chi_{c,\frak k}.
 \end{equation}
 According to Kirillov's Theorem (cf. \cite{K1}), we know that the induced
 representation $\pi_{c,\frak k}$ is, up to equivalence, independent of the
choice of a polarization of $\frak g$ for $F.$ Thus we denote the
equivalence class of $\pi_{c,\frak k}$ by $\pi_c.$ $\pi_c$ is realized on the
representation space $L^2(\BR^{(m,n)}, d\xi)$ as follows:
\begin{equation}
(\pi_c(g)f)(\xi) = e^{2\pi i \,\sigma\{c(\kappa+\mu^t\lambda +
2 \xi^t\mu)\}}f(\xi+\la),
\end{equation}
where $g=(\la, \mu, \kappa) \in G$ and $\xi \in \BR^{(m,n)}.$ Using the fact
that
$$ \exp X(\alpha,\beta,\gamma) =\left(\alpha, \beta, \gamma+\frac12 \big(\alpha\,{}^t\beta-
\beta\,{}^t\alpha\big) \right),$$
we see that $\pi_c$ is nothing but the Schr\"{o}dinger representation
$U_c=U(\sigma_c)$ of $G$ induced from the one-dimensional unitary representation
$\sigma_c$ of $K$ given by $\sigma_c((0,\mu,\kappa)) =
e^{2 \pi i \,\sigma(c\kappa)}I$ (cf. (5.6) and (5.8)). We note that $\pi_c$ is
the non-degenerate representation of $G$ with central character
$\chi_c : Z \lrt \BC^{\times}_1$ given by
$\chi_c((0,0,\kappa))=e^{2\pi i\, \sigma(c\kappa)}.$
Here $Z=\big\{(0,0,\kappa)\, \vert \ \kappa={}^t\kappa \in \BR^{(m,m)}\,\big\}$ denotes
the center of $G.$ \par
It is well known that the monomial representation $\big( \pi_c,L^2\big(\BR^{(m,n)},d\xi \big) \big)$
of $G$ extends to an operator of trace class
\begin{equation}
\pi_c(\phi) : L^2 \big(\BR^{(m,n)},d\xi\big) \lrt L^2\big(\BR^{(m,n)},d\xi\big)
\end{equation}
for all $\phi \in C_c^{\infty}(G).$ Here $C_c^{\infty}(G)$ is the vector space
of all smooth functions on $G$ with compact support. We let
$C_c^{\infty}(\frak g)$ and $C(\frak g^*)$ the vector space of all smooth
functions on $\frak g$ with compact support and the vector space of all
continuous functions on $\frak g^*$ respectively. If $f\in C_c^{\infty}(\frak g),$
we define the Fourier cotransform
$${\mathcal  C}F_{\frak g} : C_c^\infty(\frak g)\lrt C({\frak g}^*)$$
by
\begin{equation}
\left({\mathcal  C}F_{\frak g}(f)\right)(F') := \int_{\frak g}\, f(X)\,
e^{2\pi i\,\langle F',X\rangle}dX,
\end{equation}
where $F'\in {\frak g}^*$ and $dX$ denotes the usual Lebesgue measure on
$\frak g.$ According to A. Kirillov (cf. \cite{K1}), there exists a measure $\beta$
on the coadjoint orbit $\Omega_c \approx \BR^{(m,n)} \times \BR^{(m,n)}$ which
is invariant under the coadjoint action of $G$ such that
\begin{equation}
\text{tr} \,\pi_c^1(\phi) =
\int_{\Omega_c} {\mathcal  C} F_{\frak g} (\phi\circ \exp)
(F') d\beta(F')
\end{equation}
holds for all test functions $\phi \in C_c^\infty(G),$ where $\exp$ denotes the
exponentional mapping of $\frak g$ onto $G.$ We recall that
$$\pi_c^1(\phi)(f) := \int_G \phi(x) \left(\pi_c (x) f\right) dx,$$
where $\phi \in C_c^\infty(G)$ and $f\in L^2(\BR^{(m,n)},d\xi).$ By the
Plancherel theorem, the mapping
$$S(G\slash Z) \ni \varphi \longmapsto \pi_c^1(\varphi) \in TC(L^2(\BR^{(m,n)},
d\xi))$$ extends to a unitary isometry
\begin{equation}
\pi_c^2 : L^2(G\slash Z, \chi_c) \lrt HS\big( L^2 \big(\BR^{(m,n)},d\xi\big)\big)
\end{equation}
of the representation space $L^2(G\slash Z, \chi_c)$ of $\text{Ind}^G_Z\,\chi_c$
onto the complex Hilbert space
$HS\big(L^2 \big(\BR^{(m,n)},d\xi\big)\big)$ consisting of all
Hilbert-Schmidt operators on $L^2\big( \BR^{(m,n)},d\xi \big),$ where $S(G\slash Z)$ is
the Schwartz space of all infinitely differentiable complex-valued functions
on $G\slash Z \cong \BR^{(m,n)} \times \BR^{(m,n)}$ that are rapidly decreasing
at infinity and $TC \big( L^2 \big(\BR^{(m,n)},d\xi \big)\big)$ denotes the complex vector space of
all continuous $\BC$-linear mappings of $L^2\big( \BR^{(m,n)},d\xi \big)$ into itself which
are of trace class.
\vskip 0.2cm
\par
In summary, we have the following result.
\begin{theorem}
For $F=F(a,b,0) \in {\frak g}^*,$ the irreducible unitary
representation $\pi_{a,b}$ of $G$ corresponding to the coadjoint orbit
$\Omega_F = \Omega_c$ under the Kirillov correspondence is degenerate representation
of $G$ given by
$$\pi_{a,b} \big( \exp X(\alpha,\beta,\gamma) \big) = e^{4\pi i \,\sigma({}^ta\alpha -{}^tb\beta)}.$$
On the other hand, for $F=F(0,0,c)\in {\frak g}^*$ with $0\ne c ={}^t c \in \BR^{(m,m)},$
the irreducible unitary representation $\big( \pi_c,L^2\big( \BR^{(m,n)},d\xi\big)\big)$ of $G$
corresponding to the coadjoint orbit $\Omega_c$ under the Kirillov correspondence
is unitary equivalent to the Schr\"{o}dinger representation
$\big( U_c, L^2\big(\BR^{(m,n)},d\xi\big)\big)$ and this non-degenerate representation
$\pi_c$ is square integrable modulo its center $Z.$ For all test functions
$\phi \in C_c^\infty(G),$ the character formula
$$\text{tr}\, \pi_c^2(\phi) = {\mathcal  C}(\phi, c)\,
 \int_{\BR^{(m,n)}}\,\phi(0,0,\kappa)\,e^{2\pi i \,\sigma(c\kappa)} d\kappa$$
holds for some constant ${\mathcal  C}(\phi,c)$ depending on $\phi$ and $c,$ where
$d\kappa$ is the Lebesgue measure on the Euclidean space $\BR^{(m,m)}.$
\end{theorem}

Now we consider the subgroup $K$ of $G$\,(cf.\,(5.1)) given by
$$K := \big\{ \,(0,0,\kappa) \in G \,\big| \ \mu \in \BR^{(m,n)}, \; \kappa ={}^t\kappa
  \in \BR^{(m,m)}\, \big\}.$$
The Lie algebra $\frak k$ of $K$ is given by (8.12). The dual space $\frak k^*$
of $\frak k$ may be identified with the space
$$\big\{ F(0,b,c) \,\big|\ b \in \BR^{(m,n)},\, c={}^tc \in \BR^{(m,m)} \big\}.$$
We let $Ad_K^* : K \lrt GL({\frak k}^*)$ be the coadjoint
representation of $K$ on
$\frak k^*.$ The coadjoint orbit $\omega_{b,c}$ of $K$ at $F(0,b,c)\in {\frak k}^*$
is given by
\begin{equation}
\omega_{b,c} = Ad^*_K (K) \,F(0,b,c) = \{ F(0,b,c)\},
\;\;\text{a single point}.
\end{equation}
Since $K$ is a commutative group, $[\frak k, \frak k]=0$ and so the alternating
$\BR$-bilinear form ${\bold B}_f$ on $\frak k$ associated to $F:=F(0,b,c)$
identically vanishes on ${\frak k} \times {\frak k}$(cf. (8.6)).
$\frak k$ is the unique polarization of $\frak k$ for $F=F(0,b,c).$ The
Kirillov correspondence says that the irreducible unitary representation
$\chi_{b,c}$ of $K$ corresponding to the coadjoint orbit $\omega_{b,c}$ is
given by
\begin{equation}
\chi_{b,c}\big( \exp X(0,\beta,\gamma) \big) = e^{2\pi i\,\langle F(0,b,c),X(0,\beta,\gamma)\rangle}
=e^{2\pi i\, \sigma(2\,{}^tb\,\beta\, + \,c\,\gamma)}
\end{equation}
or
\begin{equation}
\chi_{b,c}((0,\mu,\kappa))= e^{2 \pi i\, \sigma(2\,{}^tb\,\mu \,+\,c\,\kappa)},\quad
(0,\mu,\kappa) \in K.
\end{equation}

For $0\ne c={}^t c \in \BR^{(m,m)},$
we let $\pi_c$ be the Schr\"o{}dinger representation of $G$ given by (8.15).
We know that the irreducible unitary representation of $G$ corresponding to the
coadjoint orbit
$$\Omega_c =\text{Ad}^*_G(G)\,F(0,0,c) =\left\{ F(a,b,c) \,\vert\, a,b \in
\BR^{(m,n)}\,\right\}.$$
Let $p:{\frak g}^* \lrt {\frak k}^*$ be the natural projection defined by
$p(F(a,b,c))=F(0,b,c).$ Obviously we have
$$p(\Omega_c) = \left \{ F(0,b,c) \,\big|\  b \in \BR^{(m,n)} \right\}=
\bigcup_{b\in \BR^{(m,n)}}\,\omega_{b,c}.$$
According to Kirillov Theorem (cf. \cite{K2} p.249, Theorem 1),
the restriction $\pi_c\vert_K$ of $\pi_c$ to $K$ is the direct integral of all
one-dimensional representations $\chi_{b,c}$ of $K\; (b\in \BR^{(m,n)}).$
Conversely, we let $\chi_{b,c}$ be the element of $\widehat K$ corresponding to
the coadjoint orbit $\omega_{b,c}$
of $K.$ The induced representation $\text{Ind}_K^G\,\chi_{b,c}$ is nothing but
the Schr\"{o}dinger representation $\pi_c.$
The coadjoint orbit $\Omega_c$ of $G$ is the only coadjoint orbit such that
$\Omega_c \cap p^{-1}(\omega_{b,c})$ is nonempty.

\end{section}

\newpage

\begin{section}{{\large\bf Hermite Operators }}
\setcounter{equation}{0}

We recall the Schr\"{o}dinger representation $U_c$ of $G$ induced
from $\sigma_c$ (cf. (5.8)). We consider the special case
when $c=I_m$ is the identity matrix of degree $m$. Then it is easy to see that
$$\begin{aligned}
dU_{I_m}(D^0_{kl})\,f(\xi) &= 2\,\pi\, i\, \delta_{kl}\, f(\xi),\\
dU_{I_m}(D_{ka})\,f(\xi) &= \frac{\partial f(\xi)}{\partial \xi_{ka}},\\
dU_{I_m}({\widehat D}_{lb}) \,f(\xi)&= 4\,\pi\, i \,\xi_{lb} \,f(\xi),
\end{aligned}$$
where $f\in {\mathcal  S}(\BR^{(m,n)})$ or $C^{\infty}(\BR),$
the Schwartz space and
$\xi_{11}, \cdots, \xi_{mn}$ are the coordinates of $\xi.$
In section two, we put
$$\begin{aligned}
Z^0_{kl} :&=  -i\, D^0_{kl}, \qquad  1\leq k\leq l\leq m,\\
Y^+_{ka}:&= \frac12 \,(D_{ka} +\, i\, {\widehat D}_{ka}), \quad 1\leq k\leq m,\;
1\leq a\leq n,\\
Y^-_{lb} :&=  \frac12 \,(D_{lb} - i\, {\widehat D}_{lb}),\quad 1\leq l\leq m,\;
1\leq b\leq n.
\end{aligned}$$
We set
\begin{equation}
A^+_{ka} := dU_{I_m}(Y^+_{ka})=\frac12\, dU_{I_m}(D_{ka})
+\,\frac{i}{2}\, dU_{I_m}({\widehat D}_{ka}),
\end{equation}
\begin{equation}
A^-_{lb} := dU_{I_m}(Y^-_{lb}) =\frac12\, dU_{I_m}(D_{lb})
-\frac{i}{2}\, dU_{I_m}({\widehat D}_{lb}))
\end{equation}
and
\begin{equation}
 C_{kl} :\,= dU_{I_m}(Z^0_{kl}) = -i \,dU_{I_m}(D^0_{kl}).\end{equation}
By Lemma 2.2, we have
$$\begin{aligned}
\,[A^+_{ka}, A^-_{lb}] &= \delta_{ab} \,C_{kl}, \\
[A^+_{ka}, A^+_{lb}]&= [A^-_{ka}, A^-_{lb} ]= 0, \\
[C_{kl}, C_{mn}]&= [C_{kl}, A^+_{ma}] = [C_{kl}, A^-_{ma}] =0.
\end{aligned}$$
In particular, we have
\begin{equation}
[A^+_{ka},A^-_{ka}] = 2\pi \cdot \text{Id}, \qquad 1\leq k\leq m, \quad
1\leq a\leq n.
\end{equation}
We note that $A^+_{ka}$ and $A^-_{lb}$ acts on the Schwartz space
$C^\infty(\BR^{(m,n)})$ or ${\mathcal  S}(\BR^{(m,n)})$ as the following linear
differential operators
\begin{equation}
A_{ka}^+=\, {\frac 12}\,\left( {{\partial\ }\over {\partial\xi_{ka}} }-4\,\pi\,\xi_{ka}\right)
\end{equation}
and
\begin{equation}
A_{lb}^-=\, {\frac 12}\,\left( {{\partial\ }\over {\partial\xi_{lb}} }-4\,\pi\,\xi_{lb}\right),
\end{equation}
where $1\leq k,l \leq m$ and $1 \leq a,b \leq n.$ The differential operators
$A^+_{ka}$ and $A^-_{lb}$ are called the $\textsf{creating
operator}$ of energy quantum and the $\textsf{annihilation operator}$ of energy
quantum respectively.
It is easy to see that the adjoint of $A^-_{ka}$ is $-A^+_{ka}.$

\vskip 3mm

We start with the ground state $f_0(\xi)$ given by
\begin{equation}
f_0(\xi) = (\sqrt{2})^{mn} e^{-2\pi \sum_{k=1}^m \sum_{a=1}^n \xi_{ka}^2}.
\end{equation}
By an easy computation, we have
\begin{equation}
\langle F_0,f_0\rangle = 1 ,\quad A^-_{ka}(f_0) = 0
\end{equation}
for all $1\leq k \leq m$ and $1\leq a \leq n.$ This means that $f_0$ is a
unit vecter in $L^2(\BR^{(m,n)},d\xi)$ which is
annihilated by the annihilation operator $A^-_{ka}  : {\mathcal  S}(\BR^{(m,n)})
\lrt {\mathcal  S}(\BR^{(m,n)}).$
For any $J\in \BZ^{(m,n)}_{\geq 0},$ we define
\begin{equation}
 f_J(\xi) := (A^+)^J f_0(\xi) := (A^+_{11})^{J_{11}} \cdots (A^+_{ka})^{J_{ka}}
\cdots (A^+_{mn})^{J_{mn}} \,f_0(\xi).
\end{equation}
\indent
We give a lexicographic orderring on $\BZ^{(m,n)}_{\geq 0} .$ That is,
for $J,K \in \BZ^{(m,n)}_{\geq 0},\;\; J< K$ if and only if
$J_{11} =K_{11}, \cdots , J_{ij}= K_{ij}, J_{i,j+1} < K_{i,j+1}, \cdots.$

\begin{lemma}
For each $k,a$ with $1 \leq k \leq m$ and
$1 \leq a \leq n,$ we have
\begin{equation}
A^-_{ka} (f_J)=-2\,\pi\, J_{ka}\, f_{J-\epsilon_{ka}}.
\end{equation}
\end{lemma}
\noindent
{\it Proof.} We prove this by induction on $J.$ If $J=(0,\cdots,0),$
(9.10) holds. Suppose (9.10) holds for $J.$ For $\tilde J = J+\epsilon_{ka},$
$$\begin{aligned}
A^-_{ka}\left(f_{J+\epsilon_{ka}}\right) &= A^-_{ka} \circ A^+_{ka}(f_J)\\
&= \left(A^+_{ka} \circ A^-_{ka} - [A^+_{ka}, A^-_{ka}]\right) (f_J) \\
&=A^+_{ka}(-2\,\pi\, J_{ka}\, f_{J-\epsilon_{ka}}) -2\,\pi\, f_J \\
&=-2\,\pi\, J_{ka}\, f_J - 2\,\pi\, f_J \\
&= -2\,\pi\,(J_{ka} + 1)\,f_J.
\end{aligned}$$
This completes the proof. \hfill$\square$

\begin{lemma}
$$\langle f_J, f_K\rangle = \begin{cases} (2\,\pi)^J\, J! & \text{if $J=K$}\\
                          0 & \text{otherwise.} \end{cases} $$
\end{lemma}
\noindent
{\it Proof.} If $J> K,$ we have
$$\begin{aligned}
\langle\, f_J,f_K\,\rangle &= \langle\, (A^+)^J f_0, (A^+)^K f_0 \,\rangle \\
          &=
(-1)^J \langle\, f_0, (A^-)^J\,\circ(A^+)^K\,f_0\, \rangle \\
&=0 \qquad (\text{by Lemma 9.1}).
\end{aligned} $$
In case $J < K$, $\langle \,f_J,f_K\,\rangle = \langle\, f_K,f_J\,\rangle=0.$
In case when $J=K,$ we prove the above identity by induction on $J.$
If $J=(0,0,\cdots,0),$ then $\langle\, f_0,f_0\,\rangle=1.$ Assume that $(f_J,f_J)=(2\,\pi)^J\,J!.$
Then according to (9.4) and Lemma 9.1,
we have,
$$\begin{aligned}
\langle\,f_{J+\epsilon_{ka}},f_{J+\epsilon_{ka}}\,\rangle&=\langle\, A^+_{ka}(f_J),A^+_{ka}(f_J)\,\rangle\\
&=-\langle\, f_J,A^-_{ka}\circ A^+_{ka}(f_J)\,\rangle\\
&=-\langle\, f_J, (A_{ka}^+ \circ A_{ka}^-\, - \, [A_{ka}^+,A_{ka}^-])f_J \,\rangle\\
&=-\langle\, f_J, -2\,\pi\, J_{ka}\, f_J - 2 \,\pi \,f_J\,\rangle \\
&= 2\,\pi\, (J_{ka} +1)\,\langle\, f_J,f_J\,\rangle \\
&=(2\,\pi)^{J+\epsilon_{ka}} (J+ \epsilon_{ka})!.
\end{aligned}  $$
\hfill $\square$

\vskip 5mm

We define the normalized function $h_J \in {\mathcal  S}(\BR^{(m,n)})$ by
\begin{equation}
h_J := \left(\frac {1}{\sqrt{2\pi}}\right)^J (J!)^{-1/2} \,f_J, \quad J\in
\BZ^{(m,n)}_{\geq 0}.
\end{equation}

\begin{lemma}
For each $J\in \BZ^{(m,n)}_{\geq 0}$ and all $k,a \in \BZ$
with $1 \leq k \leq m$ and $ 1\leq a \leq n$, we have
\begin{equation}
A^+_{ka} (h_J)  = \{ 2\, \pi\, (J_{ka} + 1)\}^{1/2}\, h_{J+\epsilon_{ka}}
\end{equation}
and
\begin{equation}
A^-_{ka}(h_J)= -(2\,\pi\, J_{ka})^{1/2}\, h_{J-\epsilon_{ka}}.
\end{equation}
\end{lemma}
\noindent
{\it Proof.} According to (9.9), we have
$$\begin{aligned}
A^+_{ka} (h_J) &= \left(\frac {1}{\sqrt{2\pi}}\right)^J(J!)^{-1/2}\,
f_{J+\epsilon_{ka}} \\
&= (2\,\pi)^{1/2}\,(J_{ka} +1 )^{1/2} \,h_{J+\epsilon_{ka}}\\
&= \{ 2\, \pi (J_{ka} +1 ) \}^{1/2} \, h_{J+\epsilon_{ka}}.
\end{aligned}$$
According to Lemma 9.1, we have
$$\begin{aligned}
A^-_{ka} (h_J) &= \left( \frac{1}{\sqrt{2\pi}}\right)^J
(J!)^{-1/2}A^-_{ka}(f_J) \\
&=\left( \frac {1}{ \sqrt{2\pi}}\right)^J (J!)^{-1/2}(-2\pi) J_{ka}
f_{J-\epsilon_{ka}}\\
&= -(2\,\pi \,J_{ka})^{1/2} \left(\frac{1}{\sqrt{2\pi}}\right)^{J-\epsilon_{ka}}
\{(J- \epsilon_{ka})!\}^{-1/2} f_{J- \epsilon_{ka}}\\
&= -(2\,\pi \,J_{ka} )^{1/2} h_{J-\epsilon_{ka}}.
\end{aligned}  $$
\hfill $\square$

\begin{lemma}
For each $J\in \BZ^{(m,n)}_{\geq 0}$ and $k,a \in \BZ^+$
with $1 \leq k \leq m$ and $1 \leq a \leq n$, we have
\begin{equation}
A^+_{ka}\circ A^-_{ka} (h_J) = -2\,\pi \,J_{ka} h_J,
\end{equation}
\begin{equation}
A^-_{ka} \circ A^+_{ka} (h_J) = -2\,\pi\,(J_{ka} + 1)\, h_J.
\end{equation}
\end{lemma}

\noindent {\it Proof.} It follows immediately from (9.12) and (9.13).

$$\begin{aligned}
A^+_{ka}\circ A^-_{ka}(h_J) &= -(2\,\pi\, J_{ka})^{1/2}\, A^+_{ka}
(h_{J-\epsilon_{ka}})\\
&=-(2\,\pi\, J_{ka})^{1/2}\,(2\pi J_{ka})^{1/2}\, h_J \\
&=-(2\,\pi\, J_{ka})\,h_J, \\
A^-_{ka} \circ A^+_{ka} (h_J) &=\{ 2\,\pi\, (J_{ka} +1)\}^{1/2}\,A^-_{ka}
(h_{J +\epsilon_{ka}})\\
&=\{ 2\,\pi \,(J_{ka} +1)\}^{1/2}\, (-1) \,\{ 2\pi (J_{ka} +1)\}^{1/2}\,h_J\\
&=-2\pi(J_{ka} +1)\, h_J.
\end{aligned} $$
\hfill$\square$

\indent
 The linear differential operators
$$A^+_{ka} \circ A^-_{ka} =\frac14 \left(\frac{\partial^2}{\partial\xi^2_{ka}}
- 16\, \pi^2\, \xi^2_{ka} +4\,\pi \right)$$
and
$$A^-_{ka}\circ A^+_{ka} = \frac14 \left( \frac{\partial^2}{\partial\xi^2_{ka}}
- 16\, \pi^2\, \xi^2_{ka} - 4\,\pi \right)$$
are called the $\textsf{number operators}$ for the family
$\{h_J\,\vert \, J\in \BZ^{(m,n)}_{\geq 0} \}.$ Now we consider the so-called
{\it Hermite differential operator}
$$H_{ka} :\,= -2\,\big( A^+_{ka}\circ A^-_{ka} + A^-_{ka}\circ A^+_{ka}\big)
=-\frac{\partial^2}{\partial\xi^2_{ka}} + 16\, \pi^2\, \xi^2_{ka}.$$
$H_{ka}$ is also called the $\textsf{Schr\"{o}dinger Hamiltonian}$
for the harmonic oscillator system in quantum mechanics. Obviously we have
\begin{equation}
H_{ka}(h_J) = 8\,\pi \, \left(  J_{ka} +\frac12\right)\,h_J,\qquad J\in \BZ^{(m,n)}_{\geq 0}.
\end{equation}
Thus the $\big\{\, h_J \,\vert \, J\in \BZ^{(m,n)}_{\geq 0} \,\big\}$ is the set of
normalized eigenforms of all Hermite operators $H_{ka}$ with
eigenvalues $\big\{ 8\,\pi\,(J_{ka}+ \frac12)\;\vert\; J\in \BZ^{(m,n)} \big\}.$
In other words, each $h_J \;(J\in \BZ^{(m,n)})$ is the harmonic oscillator wave
function with equidistant energies
$\big\{ 8\,\pi\,(J_{ka}+\frac{1}{2})\;\vert\; 1\leq k\leq m,\;1\leq a \leq n \,\big\}$ in natural units.
The Hermite operator $H_{ka}$ acts on the Schwartz space
${\mathcal  S}(\BR^{(m,n)}) \subset L^2\big(\BR^{(m,n)},d\xi\big)$ and is self-adjoint.

\begin{lemma}
For each $J\in \BZ^{(m,n)}_{\geq 0}$ and $k,a\in \BZ$ with
$ 1\leq k \leq m$ and $1 \leq a \leq n,$
\begin{equation}
 h_J (-\xi) =(-1)^J \,h_J(\xi),
\end{equation}
\begin{equation}
\left(\frac{\partial}{\partial\xi_{ka}} -4\,\pi\, \xi_{ka}\right)h_J(\xi)
=2\,\left\{ 2\,\pi\,(J_{ka} +1 )\right\}^{1/2}\,h_{J+\epsilon_{ka}}(\xi),
\end{equation}
\begin{equation}
\widehat{h_J} = (-i)^J \,h_J,
\end{equation}
\begin{equation}
CF(h_J) =i^J\, h_J.
\end{equation}
Thus  $\widehat{h_J}$ and $CF(h_J)$ satisfy the differential equation (9.18).
Here ${\widehat f} (\eta)$ denotes the Fourier transform of $f(\xi)$ on $\BR^{(m,n)}$
defined by
$${\widehat f}(\eta) :\,= \int_{\BR^{(m,n)}} \,f(\xi)\, e^{-2\,\pi\, i\,\langle\xi,\eta\rangle}\, d\xi,
\quad \eta \in \BR^{(m,n)}$$
 and $CF(f)$ denotes the Fourier cotransform of $f$ on $\BR^{(m,n)}$ defined by
$$CF(f)(\xi) :\,= \int_{\BR^{(m,n)}}\, f(\eta)\, e^{2\,\pi \,i\, \langle\eta,\xi\rangle}\, d\eta,
\quad \xi \in \BR^{(m,n)}.$$
\end{lemma}

\noindent {\it Proof.} (9.17) is obvious. (9.18) follows immediately from (9.5) and
(9.12). (9.19) and (9.20) follow from a simple computation.\hfill$\square$

\vskip 0.52cm
For $\xi =(\xi_{ka}) \in \BR^{(m,n)},$ we briefly put $\vert\xi\vert^2 :\,=
\sum_{k=1}^m\sum_{a=1}^n \xi^2_{ka}.$
We define the functions $P_J\;\left( J\in \BZ^{(m,n)}_{\geq 0}\right)$ by
\begin{equation}
h_J(\xi) :\,= P_J (\xi)\, e^{-2\,\pi\, \vert \xi \vert^2},\qquad \xi \in \BR^{(m,n)}.
\end{equation}
Indeed, $P_J(\xi)$ are the Hermite polynomials of degree $J=(J_{11},\cdots,J_{mn})$
normalized in such a way that they form an orthonormal family in
$L^2\big(\BR^{(m,n)},e^{-4\pi\vert\xi\vert^2}d\xi\big)$ (it will proved later).

\begin{lemma}
For each $J\in \BZ^{(m,n)}_{\geq 0}$ and $k,a \in \BZ^+$
with $ 1\leq k \leq m,$ and $1 \leq a \leq n,$ we have
\begin{equation}
\frac{\partial P_J(\xi)}{\partial\xi_{ka}} -8\,\pi\, \xi_{ka}\, P_J(\xi)
- 2\,\{ 2\,\pi\, (J_{ka} +1)\}^{1/2}\, P_{J+\epsilon_{ka}}(\xi) =0
\end{equation}
and
\begin{equation}
\frac{\partial P_{J+\epsilon_{ka}}(\xi)}{\partial\xi_{ka}}\, +\,
2\,\{2\,\pi\,(J_{ka} +1 )\}^{1/2} \,P_{J}(\xi)=0.
\end{equation}
\end{lemma}

\noindent
{\it Proof.} (9.22) follows from (9.18). (9.23) follows from (9.6) and
(9.13). \hfill$\square$

\vskip 0.52cm
Differentiating (9.22) with respect to $\xi_{ka}$, and then using (9.23),
we see that $P_J(\xi)$ satisfies the so-called {\it Hermite equation.}
\begin{equation}
\frac{\partial^2 P_J(\xi)}{\partial \xi_{ka}^2} -8\,\pi\, \xi_{ka}\,\frac{\partial P_J(\xi)}
{\partial\xi_{ka}} +\,8\,\pi\, J_{ka}\, P_J(\xi) =0,
\end{equation}
where $\J$, $ 1\leq k \leq m$ and $1 \leq a \leq n.$
We set $\partial_{ka} :\,= \frac{\partial}{\partial\xi_{ka}}.$ Then (9.24) becomes
$$\partial_{ka}^2 P_J(\xi) -8\,\pi\, \xi_{ka}\, \partial_{ka}\, P_J(\xi)
 +8\pi J_{ka} P_J(\xi) =0.$$
Differentating (9.18) with respect to $\xi_{ka}$, we obtain
\begin{eqnarray}
& & \partial^2_{ka} h_J(\xi) -4\,\pi\, \xi_{ka}\, \partial_{ka} h_J(\xi) -4\,\pi\, h_J(\xi)\\
& & \qquad
-2\,\{2\,\pi\,(J_{ka}+1) \}^{1/2}\,\partial_{ka} h_{J+\e_{ka}}(\xi) =0.\nonumber
\end{eqnarray}
By the way, according to (9.23), we have
$$\begin{aligned}
\partial_{ka} h_{J+\e_{ka}}(\xi) =&\,\,\partial_{ka} P_{J+\e_{ka}}(\xi) \,e^{-2\pi
\vert \xi \vert^2} - 4\,\pi\, \xi_{ka}\, P_{J+\e_{ka}}(\xi)\,e^{-2\pi \vert \xi\vert^2}\\
=&-2\,\{2\,\pi\, (J_{ka} + 1)\}^{1/2}\, h_J(\xi) -4\,\pi\, \xi_{ka}\, h_{J+\e_{ka}}(\xi).
\end{aligned}$$
If we substitute this relation into (9.25), we obtain
\begin{equation}
\partial^2_{ka} h_J(\xi) -16\, \pi^2\, \xi^2_{ka}\, h_J(\xi) = -8\,\pi\,\left(J_{ka} +
\frac12 \right)\, h_J(\xi).
\end{equation}

\begin{theorem}
The set $\left\{ h_J \;\vert \; \J\,\right\}$ of normalized
Hermitian function in ${\mathcal  S}(\bhg)$ forms
an orthonormal basis of $L^2\big(\bhg,d\xi\big).$ These $h_J$ are eigenfunctions of
the Fourier transform
and the Fourier cotransform with eigenvalues $(-i)^J$ and $i^J$ respectively.
\end{theorem}

\noindent
{\it Proof.} If $X$ is a left-invariant vector field on $G,$
we set, for brevity
$$U(X) :\,= dU_{I_m}(X).$$
We will prove that the set
$$\left\{U\left(\exp_G\! \left( \sum_{k,a} x_{ka} D_{ka} +\sum_{l,b} y_{lb}
\hat{D}_{lb}\right)
\right)(f_0) \,\bigg| \ x_{ka},y_{lb} \in \BR,\; 1 \leq k,l \leq m, \;
1\leq a,b \leq n \right\}$$
is contained to the closed vector subspace of $L^2\big(\bhg ,d\xi\big)$ spanned by the
set $\big\{ h_J \,\vert \, \J\,\big\}$
and the subspace generated by the above set is invariant under the action of $U.$
Since the Schr\"{o}dinger representation
$\left( U_{I_m},L^2\big(\bhg,d\xi\big)\right)$ is irreducible, we conclude that
the set $\big\{h_J \,\vert \, \J\,\big\}$
is a {\it complete} orthonormal basis for $L^2\big(\bhg,d\xi\big).$\par
According to the commutation relation among
$D_{kl}^0,\;D_{ka},\; {\widehat D}_{lb}$ (cf. Lemma 2.1) and the fact that
$U(D_{kl}^0)f = 2\,\pi\, i\, \delta_{kl}\, f$ for all $f \in {\mathcal  S}(\bhg),$
it suffices to prove the case $m=1$ and $n=1.$ We put
$D^0 :\,= D_{11}^0,\; D :\,= D_{11}$ and $\widehat D :\,= \widehat{D}_{11}.$
In other words, it remains to prove that the set
$$\left\{U(\exp_G (xD +y\widehat{D})) (f_0)\, \big| \ x,y \in \BR \right\}$$
is contained in the closed vector subspace of $L^2(\BR,d\xi)$ spanned by
the set $\{ h_j \,|\  j=0,1,2,\cdots\}.$\par
First we note that by (9.1) and (9.2)
$$A^+ = \frac12 \left( U(D) +i\,U(\widehat D) \right) \quad \text{and} \quad A^- = \frac12
\left(
U(D)-i\,U(\widehat D) \right).$$
For the present time being, we fix real numbers $x,y\in \BR.$
We put $z= x +iy \in \BC.$ It is obvious that $U(xD+y\widehat{D})=\bar{z}A^+ +zA^-.$
For all integers $k\geq 0,\;\ell\geq 0$ with $0\leq k \leq \ell,$
We define the complex numbers $c_{k\ell}$ by
$$U(xD+ y\hat{D})^\ell (f_0) = \sum_{k=0}^\ell c_{k,\ell} f_k.$$
By the fact that $A^-(f_0) = 0$ and by (9.10), we have
$$\begin{aligned}
U(xD+y\widehat{D})^{\ell+1} (f_0) &= (\bar{z}A^+ +zA^-)\left( \sum_{k=0}^\ell c_{k,\ell} f_k \right) \\
&=\sum_{k=0}^\ell c_{k,\ell}\, ( \bar{z} f_{k+1} - 2\,\pi\, k\, z\,f_{k-1}).
\end{aligned}$$
Thus we get the recurrence formula
$$c_{k,\ell+1} =\bar{z} c_{k-1,\ell} - 2\,\pi\, (k+1)\, z\, c_{k+1,\ell},\qquad 1 \leq k \leq \ell-1.$$
Let $z=\vert z \vert\, e^{2\,\pi\, i\, \varphi}$ with $\varphi \in [0,1)$ for $z \ne 0.$
We put
$$d_{k,\ell} :\,= \left( \vert z \vert^{-1/2}\, e^{\pi\, i\, \varphi} \right)^{\ell+k}
\left( (2\,\pi\, \vert z \vert )^{-1/2}\, e^{-\pi \,i \,(\varphi -\frac12)}\right)^{\ell-k} \,c_{k,\ell}.$$
Then we have the recurrence formula
$$d_{k,\ell+1} = d_{k-1,\ell} + (k+1)\,d_{k+1,\ell},\qquad 1 \leq k \leq \ell-1.$$
For $1\leq k \leq \ell-1,$ we put
$$b_{k,\ell} :\,= d_{\ell-k, \ell}.$$
Then we get the recurrence formula
$$b_{k,\ell} = b_{k,\ell-1} + (\ell-k+1)\,b_{k-2,\ell-1},\qquad 2\leq k \leq \ell-1.$$
If the starting value is $b_{0,0}$ and we define $b_{k,0}=0$ for $k\geq 1,$
then we get
$$b_{2p+1,\ell} = 0 \qquad \text{for} \;\; 0 \leq p \leq \frac12 (\ell-1)$$
and
$$b_{2p,\ell} = \frac{\ell!}{2^p \,p! \,(\ell-2p)!} \quad \text{for} \;\; 0 \leq p \leq
\frac12 \ell .$$
Thus we obtain
$$\begin{aligned}
U(xD+y\widehat{D})^\ell (f_0) =& \sum_{k=0}^\ell  c_{k,\ell} f_k \\
=&\sum_{k=0}^\ell \left( \vert z \vert^{1/2}\, e^{-\pi i \varphi} \right)^{\ell+k}
\left((2\,\pi\, \vert z\vert )^{1/2}\,e^{\pi\, i\,(\varphi -\frac12)}\right)^{\ell-k}
d_{k,\ell} f_k\\
=&\sum_{k=0}^\ell\left( \vert z\vert^{1/2}\, e^{-\pi i\varphi}\right)^{2\ell-k}\left(
(2\,\pi\, \vert z \vert )^{1/2}\,e^{\pi\, i\,(\varphi -\frac12)}\right)^k b_{k,\ell}\, f_{\ell-k} \\
=& \sum_{p=0}^{[\frac{\ell}{2}]}\left( \vert z\vert^{1/2}\,
e^{-\pi\, i\,\varphi}\right)^{2\ell-2p}
\left((2\,\pi\, \vert z\vert)^{1/2}\, e^{\pi\, i\,(\varphi - \frac12)}\right)^{2p}
b_{2p,\ell}\, f_{\ell-2p}\\
=& \sum_{p=0}^{[\frac{\ell}{2}]}{\bar{z}}^{\ell-p}(-2\,\pi\, z)^p\,
\frac{\ell!}{2^p\, p!\,(\ell-2p)!}\, f_{\ell-2p}\\
=& \sum_{p=0}^{[\frac{\ell}{2}]}{\bar{z}}^{\ell-p}(-\pi \,z)^p\,
\frac{\ell !}{p!\,(\ell-2p)!}\, f_{\ell-2p}.
\end{aligned}$$
And so we get
$$\begin{aligned}
e^{U(xD+y\widehat{D})}(f_0) &= \sum_{\ell=0}^\infty \frac{1}{\ell!}\,U(xD+y\widehat{D})^\ell
(f_0)\\
&=\sum_{\ell=0}^\infty\sum_{p=0}^{[\frac{\ell}{2}]}\,\frac{1}{\ell!}\,{\bar{z}}^{\ell-p}\,
(-\pi z)^p \,\frac{\ell!}{p!\,(\ell-2p)!}\, f_{\ell-2p}\\
&=\sum_{\ell=0}^\infty\sum_{p=0}^{[\frac{\ell}{2}]}\,\frac{1}{p!\,(\ell-2p)!}
\,(-\pi \vert z \vert^2)^p\, {\bar{z}}^{\ell-2p}\, f_{\ell-2p} \\
&=\sum_{k=0}^\infty\left\{ \sum_{p=0}^\infty \,\frac{1}{p!}\,(-\pi\,
\vert z\vert^2)^p\,
\right\}\frac{{\bar{z}}^k}{k!}\, f_k \\
&=e^{-\pi \vert z \vert^2} e^{\bar{z}A^+}(f_0) \\
&=e^{-\pi\vert z\vert^2} \sum_{k=0}^\infty \frac{(\sqrt{2\pi}\, \bar{z})^k}
{(k!)^{1/2}}\, h_k.
\end{aligned}$$
Therefore $U(\exp_G(xD+y\widehat{D}))(f_0)$ belongs to the closed subspace of
$L^2(\BR,d\xi)$ spanned by the set
$\{\,h_j\, |\ j=0,1,2,\cdots\,\}.$
The latter part of the theorem follows immediately from (9.19) and (9.20).
This completes the proof. \hfill$\square$

\vskip 0.2cm\noindent
\noindent {\bf Corollary 9.1.} The set $\big\{\,P_J \,\vert\,\J\,\big\}$ of Hermite polynomials
forms an orthonormal basis for the $L^2$-space
$L^2\big(\bhg,e^{-4\pi \vert \xi\vert^2} d\xi\big).$

\noindent {\it Proof.} \ \ The proof follows immediately from Theorem 9.7 and (9.21).
\hfill$\square$

\end{section}

\newpage

\begin{section}{{\large\bf Harmonic Analysis on $\Gamma\backslash G$}}
\setcounter{equation}{0}

We fix an element $\Omega \in {\mathbb H}_n$ once and for all. Let $\mathcal  M$ be a
positive symmetric half-integral matrix of degree $m.$ Let $L^2\big(\bhg,
d\xi_{\Omega,\mathcal  M}\big)$ be the $L^2$-space of $\bhg$ with respect to the measure
$$d\xi_{\Omega,\mathcal  M} = e^{\pi\, i\,\s\{\M\xi(\Omega -\overline{\Omega})\,{}^t\xi\}}d\xi.$$
It is easy to show that the transformation $f(\xi) \lmt e^{\pi\,i\,\s\{\M\xi
\Om\,{}^t\xi\}}f(\xi)$ of $L^2\big(\bhg, d\xi_{\Om,\M} \big)$ into
$L^2\big(\bhg,d\xi \big)$ is an isomorphism. Since the set $\big\{\xi^J \,|\ \J \,\big\}$ is a basis
of $L^2\big(\bhg,d\xi_{\Om,\M} \big),$ the set
$$\big\{\, e^{\pi\,i\,\s\{\M\xi\Om\,{}^t\xi\}}\,\xi^J\,
|\ \J\,\big\}$$
is a basis of $L^2\big(\bhg,d\xi\big).$
We observe that there exists a canonical bijection $A$ from the cosets
$\mathcal  T :\,= \zhg \slash (2\M)\zhg.$ We denote by $A_{\a}$ the image of
$\a \in \mathcal  T$ under the bijection $A.$

\vskip 0.2cm
For each $A_{\a}\in \mathcal  L$ and each $\J,$ we define a function $\Phi^{(\M)}_J
\left[\begin{matrix} A_{\a}\\ 0 \end{matrix}\right](\Om \vert\, \cdot\,)$ on
$G= H^{(n,m)}_\BR$ by

\def\PJ{\Phi^{(\M)}_J\left[\begin{matrix} A_{\a} \\ 0 \end{matrix}\right]}

\begin{equation}
\begin{aligned}
&  \ \ \ \ \PJ \big(\Om \vert(\la,\mu,\k)\big) :\\
&= e^{2\pi i\,\s\{\M(\k-\la\,{}^t\!\mu)\}}\sum_{N\in \zhg}
(\la +N + A_{\a})^J \\
& \ \ \  \,\times\, e^{2\pi i\,\s\{\M((\la+N+A_{\a})\Om\,{}^t(\la+N+A_{\a})
\,+\, 2(\la +N +A_\a)\,{}^t\mu)\}},
\end{aligned}
\end{equation}
where $(\la,\mu,\k) \in G.$ We let $\Gamma_G =H^{(n,m)}_\BZ$ be the discrete
subgroup of $G$ consisting of integral elements. That is,
$$\Gamma_G= \left\{(\la,\mu,\k)\in G\ | \ \la,\mu,\k \text{ integral}\,\right\}.$$
According to \cite{Y1}, Proposition 1.3, the function
$\PJ\big(\Om | (\la,\mu,\k)\big)$ satisfies the transformation behaviour
\begin{equation}
\PJ \big(\Om | \gamma\circ g \big) = \PJ(\Om | g)
\end{equation}
holds for all $\gamma\in \Gamma_G$ and $g\in G.$ Thus the functions
$$\PJ(\Om | (\la,\mu,\k)) \;\left(\J \right)$$
are real analytic functions on the
quotient space $\Gamma_G \backslash G.$ Let $H^{(\M)}_\Om \left[ \begin{matrix} A_\a \\ 0
\end{matrix} \right]$ be the completion of the vector space spanned by
$$\PJ \big(\Om | (\la,\mu,\k) \big) \;\left( \J \right)$$
and let $\overline{H^{(\M)}_\Om \left[ \begin{matrix} A_\a \\ 0
\end{matrix} \right]}$ be the complex conjugate of $H^{(\M)}_\Om\left[ \begin{matrix} A_\a \\ 0
\end{matrix} \right].$
\vskip 0.2cm
Let $L^2(\Gamma_G \backslash G)$ be the $L^2$-space of the quotient space
$\Gamma_G\backslash G$ with respect to the invariant measure
$$d\la_{11}\cdots d\la_{m,n-1}d\la_{mn}d\mu_{11}\cdots d\mu_{m,n-1} d\mu_{mn}
d\k_{11}d\k_{12}\cdots d\k_{mm}.$$
Let $\rho$ be he right regular representation of $G$ on the Hilbert space
$L^2(\Gamma_G \backslash G)$ given by
$$(\rho(g_0)\phi)(g) :\,= \phi(gg_0), \;\; g_0,g \in G,\;\; \phi \in
L^2(\Gamma_G\ba G).$$
\indent
In \cite{Y1}, the author proved that the subspaces $H^{(\M)}_\Om\left[ \begin{matrix} A_\a \\ 0
\end{matrix} \right]$ and $\overline{H^{(\M)}_\Om\left[ \begin{matrix} A_\a \\ 0
\end{matrix} \right]}$ are irreducible invariant subspaces of $L^2(\Gamma_G \ba G)$ with
respect to $\rho$ and the decomposition of the right regular representation
$\rho$ is given by
$$\begin{aligned}
L^2(\Gamma_G\ba G) &= \bigoplus_{\M,\a}H^{(\M)}_\Om\left[ \begin{matrix} A_\a \\ 0
\end{matrix} \right] \oplus\left(\overline{\bigoplus_{\M,\a}H^{(\M)}_\Om\left[ \begin{matrix} A_\a \\ 0
\end{matrix} \right]}\right)\\
&\oplus \left( \bigoplus_c R(c)\right)\oplus
\left( \bigoplus_{k,\ell\in \zhg} \BC\, e^{2\pi i\,\s(k\,{}^t\la\, +\, \ell\,{}^t\mu)}\right),
\end{aligned}$$
where $\M$ (respectively $c$) runs over the set of all positive symmetric half
integral matrices of degree $m$ (respectively the set of all half integral
nonzero matrices of degree $m$ which are neither positive nor negative definite),
$R(c)$ is the sum of irreducible representations $\pi_c$ which occur in $\rho$
and $A_\a$ runs over a complete system of representatives of the cosets
$(2\M)^{-1}\zhg \slash \zhg.$

\begin{lemma}
The transform of $L^2\big(\bhg,d\xi_{\Om,\M}\big)$ onto
$H^{(\M)}_\Om\left[ \begin{matrix} A_\a \\ 0\end{matrix} \right]$ given by
\begin{equation}
\xi^J\lmt \PJ\big(\Om | (\la,\mu,\k)\big)
\end{equation}
is an isomorphism of Hilbert spaces.
\end{lemma}
\noindent
{\it Proof.} For the proof, we refer to \cite{Y1}, Lemma 3.2. \hfill$\square$

We write
\begin{equation}
f_{\Om,J}^{(\M)} (\xi) :\,= e^{2\pi i\,\s (\M\xi\Om\,{}^t\xi)}\xi^J,\;\;
\J.
\end{equation}
We let $\Phi^{(\M)}_{\Om,\a}$ be the transform of $L^2\big(\bhg,d\xi \big)$ onto
$H^{(\M)}_\Om \left[ \begin{matrix} A_\a \\ 0\end{matrix} \right]$ defined by
\begin{equation}
\Phi^{(\M)}_{\Om,\a}\left(f_{\Om,J}^{(\M)} \right) :\,= \PJ\big(\Om | (\la,\mu,\k)\big).
\end{equation}
Then $\Phi^{(\M)}_{\Om,\a}$ is an isometry of $L^2\big(\bhg,d\xi \big)$ onto
$H^{(\M)}_\Om\left[ \begin{matrix} A_\a \\ 0\end{matrix} \right]$ such that
$$U^{S,\M}\left((\la,\mu,\k) \right)\circ \Phi^{(\M)}_{\Om,\a} =\Phi^{(\M)}_{\Om,\a}\circ
U^{S,\M}\big((\la,-\mu,-\k)\big),$$
where $U^{S,\M}$ is the Schr\"{o}dinger representation of $G$ defined by (6.45). \par
Let $\Delta_{\Om,\M}$ be the isometry of $L^2\big(\bhg,d\xi_{\Om,\M}\big)$ onto
$L^2\big(\bhg,d\xi \big)$ defined by
\begin{equation}
(\Delta_{\Om,\M} f)(\xi) :\,= e^{\pi i\,\s\{\M\xi\Om\,{}^t\xi\}} f(\xi).
\end{equation}
We define the unitary representation $U^{S,\M}_{\Om}$ of $G$ on
$L^2\big(\bhg,d\xi_{\Om,\M}\big)$ by
\begin{equation}
\left( U^{S,\M}_{\Om}(g) f\right)(\xi) :\,= \Delta^{-1}_{\Om,\M}\left(
\left( U^{S,\M}_{\Om}(g)(\Delta_{\Omega,\M} f)\right)(\xi)\right),
\end{equation}
where $f\in L^2\big(\bhg,d\xi_{\Om,\M} \big)$ and $\xi\in \bhg.$\par
 Now we write down the image of $f^{(\M)}_{\Om,J} \in L^2\big(\bhg,d\xi \big)$
 under $\vth_{\M,\a}$ (cf. (7.22)) explicitly.
$$\begin{aligned}
& \left(\vth_{\M,\a}f_{\Om,J}^{(\M)} \right)((\la,\mu,\k))\\
&= \left(\vth_{\M,\a}f_{\Om,J}^{(\M)} \right)([\la,\mu,\k+\mu\,{}^t\!\la])\\
&=e^{2\pi i\,\s\{\M(\k\,+\,\mu\,{}^t\la)\}}\,\sum_{N\in \zhg}\,
e^{2\pi i\,\s\{\M((\la+N)\Om\,{}^t(\la+N)\,+\,2\,N\,{}^t\mu)\}}\,(\la+N)^J\\
&=e^{2\pi i\,\s\{\M(\k-\la\,{}^t\mu)\,+\,\a\,{}^t\mu\}}\,
\sum_{N\in\zhg}\,e^{2\pi i\,\s\{\M((\la+N)\Om\,{}^t(\la+N)\,+\,2\,(\la+N)\,{}^t\mu)\}}\,(\la+N)^J.
\end{aligned}$$
In particular, if $\a=0,\;\k=0$ and $J=0,$ then we have
$$\begin{aligned}
& \left( \vth_{\M,0}f_{\Om,0}^{(\M)} \right)((\la,\mu,0))\\
&=e^{-2\pi i\,\s(\M\la\,{}^t\!\mu)}\sum_{N\in\zhg}\,e^{2\pi i\,\s\{\M((\la+N)\Om\,{}^t(\la+N)\,+\,
2\,(\la+N)\,{}^t\!\mu)\}}\\
&=e^{2\pi\,i\,\s\{\M(\la\,\Om\,{}^t\!\la \,+\,\la\,{}^t\!\mu)\}}\,\sum_{N\in \zhg}
\,e^{2\pi i\,\s\{\M(N\Om\,{}^tN\,+\,2\, (\la\Om+\mu)\,{}\!^t\!N)\}}\\
&=e^{2\pi i\,\s\{\M(\la\,\Om\,{}^t\!\la \,+\,\la\,{}^t\!\mu)\}}\,\vth^{(2\M)}\left[\begin{matrix} 0\\0
\end{matrix}\right](\Om,\la\Om+\mu).
\end{aligned}$$
Therefore we obtain

\begin{proposition}
Let $\M$ be a positive symmetric half-integral matrix
of degree $m.$ Let $\a\in \mathcal  T$ and $ \J.$ Then we have
 $$\begin{aligned}
& \left(\vth_{\M,\a}f_{\Om,J}^{(\M)} \right)((\la,\mu,\k))\\
&=e^{2\pi i\,\s\{(\k-\la\,{}^t\!\mu)\, +\,\a\,{}^t\!\mu \}}
\,\sum_{N\in \zhg}\,e^{2\pi i\,\s\{\M((\la+N)\Om\,{}^t\!(\la+N)\,+\,2(\la+N)\,{}^t\!\mu)\}}(\la+N)^J.
\end{aligned}$$
In particular,
$$
\left( \vth_{\M,0}f_{\Om,0}^{(\M)} \right)((\la,\mu,0))=
e^{2\pi i\,\s\{\M(\la\Om\,{}^t\!\la\, +\,\la\,{}^t\!\mu)\}}\,\vth^{(2\M)}\left[\begin{matrix} 0\\0
\end{matrix}\right](\Om,\la\,\Om+\mu).$$
\end{proposition}

\vskip 0.3cm
It is easy to see that the following diagrams are commutative.
$$\begin{CD}
L^2\big(\bhg,d\xi\big) @>U^{S,\M}(g)>> L^2\big(\bhg,d\xi\big) \\
@V\vth_{\M,\a}VV       @VV\vth_{\M,\a}V \\
\begin{matrix} {\mathcal  H}_{\M,\a}\\ \shortparallel\\ \pi_{\M,\a}\end{matrix}
      @>\pi_{\M,\a}(g)>>
\begin{matrix} {\mathcal  H}_{\M,\a}\\ \shortparallel\\ \pi_{\M,\a}\end{matrix}
\end{CD} $$
$$\textsf{  diagram  10.1}$$

\vskip 0.51cm
$$\begin{CD}
L^2\big(\bhg,d\xi_{\Om,\M}\big) @>U^{S,\M}_{\Om}(g)>>  L^2\big(\bhg,d\xi_{\Om,\M}\big) \\
@V\Delta_{\Om,\M}VV       @VV\Delta_{\Om,\M}V  \\
  L^2\big(\bhg,d\xi\big)      @>U^{S,\M}(g)>>   L^2\big(\bhg,d\xi\big)
\end{CD} $$
$$\textsf{  diagram  10.2}$$

\vskip 0.51cm
$$\begin{CD}
L^2\big(\bhg,d\xi\big) @>U^{S,\M}(g)>>  L^2\big(\bhg,d\xi\big) \\
@VI_{\M}VV       @VVI_{\M}V  \\
{\mathcal  H}_{F,\M}        @>U^{F,\M}(g)>>   {\mathcal  H}_{F,\M}
\end{CD} $$
$$\textsf{  diagram  10.3}$$

\vskip 0.51cm
$$\begin{CD}
L^2\big(\bhg,d\xi\big) @>U^{S,\M}(g)>>  L^2\big(\bhg,d\xi\big) \\
@V\Phi^{(\M)}_{\Om,\a}VV       @VV\Phi^{(\M)}_{\Om,\a}V  \\
H^{(\M)}_\Om\left\lk\begin{matrix} A_\a \\ 0 \end{matrix} \right\rk
        @>\rho_{\M}(g)>>
H^{(\M)}_\Om\left\lk\begin{matrix} A_\a \\ 0 \end{matrix} \right\rk
\end{CD} $$
$$\textsf{  diagram  10.4}$$
\vskip 0.51cm
\noindent
Here $g\in G$ and $\rho_\M$ denotes the restriction of the right regular
representation $\rho$ to $H^{(\M)}_\Om\left\lk\begin{matrix} A_\a \\ 0
\end{matrix} \right\rk.$ We know that the mapping $\vth_{\M,\a},\;\Delta_{\Om,\M},
\;I_\M$ and $\Phi^{(\M)}_{\Om,\a}$ are all the isomorphisms preserving the
norms. Hence we have

\begin{theorem}
For each $\a \in {\mathcal  T},\;\Om \in {\mathbb H}_n$ and $\M$ positive
symmetric half-integral matrix of degree $m,$ the Schr\"{o}dinger
representation $\big(U^{S,\M},L^2\big(\bhg,d\xi\big)\big),$ the lattice representation
$(\pi_{\M,\a},{\mathcal  H}_{\M,\a}),$ the Fock representation $\big(U^{F,\M},
{\mathcal  H}_{F,\M}\big),$ the representation
$\left(\rho_\M,H^{(\M)}_\Om\left\lk\begin{matrix} A_\a \\ 0 \end{matrix} \right\rk\right)$
and the representation $\big(U^{S,\M}_{\Om}, L^2\big(\bhg,d\xi_{\Om,\M}\big)\big)$ are unitarily
equivalent to each other via the intertwining operators $\vth_{\M,\a},\;
I_\M,\; \Phi^{(\M)}_{\Om,\a}$ and $\Delta_{\Om,\M}.$
\end{theorem}

\begin{remark}
The multiplicity of the Schr\"{o}dinger representation
$U^{S,\M}$ of $G$ in
\par\noindent
$\big(\rho,L^2(\Gamma_G \backslash G)\big)$ is $(\det 2\M)^n.$
\end{remark}

We refer to \cite{Y1} for detail. Theorem 10.2 may be pictured as follows.

$$\begin{matrix}  &             & {\mathcal  M}_{F,\M}      &           &    \\
           &             &                      &           &    \\
           &             & \uparrow\,I_\M       &           &    \\
           & \Delta_{\M,\a} &               & \vth_{\M,\a}  &    \\
L^2\big(\bhg,d\xi_{\Om,\M}\big)& \lrt    &L^2\big(\bhg,d\xi\big)& \lrt & {\mathcal  H}_{\M,\a} \\
           &          &                         &            &    \\
           &          &   \downarrow\,\Phi^{(\M)}_{\Om,\a}&   &   \\
           &          &                         &             &   \\
       &     &  H^{(\M)}_\Om\left\lk\begin{matrix} A_\a \\ 0 \end{matrix} \right\rk
                   & & \end{matrix}$$
\vskip 0.2cm
$$ \textsf{\qquad  figure 10.5} $$

\vskip 0.5cm
Finally we describe explicitly the orthonormal bases of
$$L^2\big(\bhg,d\xi\big),\; L^2\big(\bhg,d\xi_{\Om,\M}\big),\; {\mathcal  H}_{\M,\a},\;
{\mathcal  H}_{F,\M}
\text{ and } H^{(\M)}_\Om\left\lk\begin{matrix} A_\a \\ 0 \end{matrix} \right\rk$$ respectively.
\par
In the previous section, we proved that the family of the functions
$$h_J(\xi) = \left(\frac{1}{\sqrt{2\pi}}\right)^J\, (J!)^{-1/2}\, f_J(\xi),\quad  \J$$
forms an orthonormal basis of $L^2\big(\bhg,d\xi\big).$ Therefore the set
$$\left\{ e^{-\pi\,i\,\s(\M\xi\Om\,{}^t\xi)}\,h_J(\xi)\, | \ \J \,\right\}$$
forms an orthonormal basis for $L^2\big(\bhg,d\xi_{\Om,\M}\big).$
For each $ \J,$ the set of functions
$$\begin{aligned}
& \ \ \ \ \vth_{\M,\a,J}(\la,\mu,\k) :\\
& =( \vth_{\M,\a} h_J)(\la,\mu,\k) \\
& =
e^{2\pi i\,\s\{\M(\k\,+\,\mu\,{}^t\!\la\,+\,\a\,{}^t\mu)\}}
\,\sum_{N\in \zhg}\,h_J(\la+N)\,e^{4\pi i\,\s(\M N\,{}^t\!\mu)},\;\; \J
\end{aligned}$$
forms an orthonormal basis for ${\mathcal  H}_{\M,\a}.$ For each $ \J,$ we define
the function
$$\begin{aligned}
&H^{(\M)}_J\left\lk\begin{matrix} A_\a \\ 0 \end{matrix} \right\rk \big(\Om |(\la,\mu,\k)\big) :\,
=2^{\frac{mn}{4} -\frac{\vert J \vert}{2} }\, (J!)^{-1/2}\,(\det 2\M)^{n/4}
(\det \textrm{Im}\,\Om)^{m/4} \\
&\times e^{\pi i\, \s(\M(\k-\la\,{}^t\mu))} \sum_{N\in \zhg} H_J\left( \sqrt{2\pi}\,(2\M)^{1/2}\,(\la+N+A_\a)
(\textrm{Im}\, \Om)^{1/2}\right)\\
&\times e^{\pi i\,\s\{\M((\la+N+A_\a)\Om\,{}^t(\la+N+A_\a)\, +\, 2\,(\la+N+A_\a)\,{}^t\mu)\}},
\end{aligned}$$
where $H_J(\xi)$ is the Hermite polynomial on $\bhg$ in several variables
defined by
$$H_J(\xi) :\,= H_{J_{11}}(\xi_{11})H_{J_{12}}(\xi_{12})\cdots H_{J_{mn}}
(\xi_{mn}).$$
It was proved in \cite{Y2} that the functions
$H^{(\M)}_J\left\lk\begin{matrix} A_\a \\ 0 \end{matrix} \right\rk \big(\Om | (\la,\mu,\k)\big)
\quad \left( \J\right) $ form an orthonormal basis for
$H^{(\M)}_\Om\left\lk\begin{matrix} A_\a \\ 0 \end{matrix} \right\rk.$ We have

\begin{theorem}
(1) The set $\left\{\,h_J \,\big|\  \J\,\right\}$ forms an orthonormal
basis for $L^2\big(\bhg,d\xi\big).$ \par
\noindent (2) The set $\left\{ e^{-\pi i\,\s(\M\xi\Om\,{}^t\xi)}\,h_J\, | \ \J\,\right\}$ forms
an orthonormal basis for
\par\noindent
$L^2\big(\bhg,d\xi_{\Om,\M}\big).$\par
\noindent (3) The set $\left\{\,\vth_{\M,\a,J}\, \big|\  \J \,\right\}$ forms an orthonormal basis for
${\mathcal  H}_{\M,\a}.$\par
\noindent (4) The set $\left\{\, \Phi_{\M,J}\, |\  \J\,\right\}$ (cf. (6.36)) forms an orthonormal
basis for ${\mathcal  H}_{F,\M}.$ \par
\noindent (5) The set $H^{(\M)}_J\left\lk\begin{matrix} A_\a \\ 0 \end{matrix} \right\rk
\big(\Om | (\la,\mu,\k)\big)\quad \left( \J \right)$ forms an orthonormal basis for
$H^{(\M)}_\Om\left\lk\begin{matrix} A_\a \\ 0 \end{matrix} \right\rk.$
\end{theorem}

\end{section}

\newpage

\begin{section}{{\large\bf The Symplectic Group}}
\setcounter{equation}{0}
\vskip 0.2cm
We recall that
$$Sp(n,\BR)=\{ M\in \BR^{(2n,2n)}\ \vert \ ^t\!MJ_nM= J_n\ \}$$
is the symplectic group of degree $n$,
where
$$J_n=\begin{pmatrix} 0&I_n\\
                   -I_n&0\end{pmatrix}.$$
If $M= \begin{pmatrix} A& B\\
                   C&D\end{pmatrix}\in Sp(n,\BR)$ with $A,B,C,D\in\BR^{(n,n)},$ then
\begin{equation}
{}^tAD-{}^tCB=I_n,\qquad {}^tAC=\,{}^tC A,\qquad {}^tBD=\,{}^tDB.
\end{equation}
We note that $Sp(1,\BR)=SL(2,\BR).$ The inverse of $M= \begin{pmatrix} A& B\\ C&D\end{pmatrix}\in Sp(n,\BR)$
is
\begin{equation*}
M^{-1}=J_n^{-1}\,{}^tM\,J_n=\,\begin{pmatrix} \,\,{}^tD & -{}^tB\\ -{}^tC& \,\,{}^tA\end{pmatrix}.
\end{equation*}
Since $J_n^{-1}=-J_n$ and ${}^tM^{-1}J_nM^{-1}=J_n$ with $M\in Sp(n,\BR),$ we see that
$${}^tM^{-1}J_n^{-1}M^{-1}=J_n^{-1},\quad \textrm{that\ is},\quad MJ_n\,{}^tM=J_n.$$
Thus if $M\in Sp(n,\BR),$ then ${}^tM\in Sp(n,\BR).$ If $M= \begin{pmatrix} A& B\\ C&D\end{pmatrix}\in Sp(n,\BR)$,
then
\begin{equation}
A\,{}^tD-B\,{}^tC=I_n,\qquad A\,{}^tB=\,B\,{}^tA,\qquad C\,{}^tD=\,D\,{}^tC.
\end{equation}
\begin{lemma}
Let $M= \begin{pmatrix} A& B\\ C&D\end{pmatrix}\in Sp(n,\BR)$ and $\Omega\in \BH_n.$ Then
\vskip 0.1cm
(a) \ $C\Omega+D$ is nonsingular.
\vskip 0.1cm
(b) \ $(A\Om+B)(C\Om+D)^{-1}$ is an element of $\BH_n.$
\end{lemma}
\vskip 0.2cm\noindent
{\it Proof.} Let $\Om=X+iY\in\BH_n$ with $X,Y\in\BR^{(n,n)}$ and $Y>0$. Then
\begin{equation}
{}^{{}^{{}^{{}^\text{\scriptsize $t$}}}}\!\!\!\begin{pmatrix} \Omega\\ I_n \end{pmatrix}\,J_n \begin{pmatrix} \Omega\\ I_n \end{pmatrix}=0
\end{equation}
and
\begin{equation}
{}^{{}^{{}^{{}^\text{\scriptsize $t$}}}}\!\!\!\!\begin{pmatrix} \Omega\\ I_n \end{pmatrix}\,J_n {\overline{\begin{pmatrix} \Omega\\ I_n \end{pmatrix}}}=\,2\,i\,Y>0.
\end{equation}
We set
$$S=A\Om+B\qquad \textrm{and}\qquad T=C\Om+D.$$
By (11.3), we have
\begin{eqnarray*}
 {}^{{}^{{}^{{}^\text{\scriptsize $t$}}}}\!\!\!\begin{pmatrix} S\\ T \end{pmatrix}J_n \begin{pmatrix} S\\ T \end{pmatrix}
&=& {}^{{}^{{}^{{}^\text{\scriptsize $t$}}}}\!\!\!\left\{ M\begin{pmatrix} \Omega\\ I_n \end{pmatrix}\right\}J_n \left\{ M\begin{pmatrix} \Omega\\ I_n \end{pmatrix}\right\}\\
&=& {}^{{}^{{}^{{}^\text{\scriptsize $t$}}}}\!\!\!\begin{pmatrix} \Omega\\ I_n \end{pmatrix}{}^tMJ_nM \begin{pmatrix} \Omega\\ I_n \end{pmatrix}\\
&=& {}^{{}^{{}^{{}^\text{\scriptsize $t$}}}}\!\!\!\begin{pmatrix} \Omega\\ I_n \end{pmatrix}J_n \begin{pmatrix} \Omega\\ I_n \end{pmatrix}=0.
\end{eqnarray*}
By (11.4), we have
$${1\over {2i}}\,{}^{{}^{{}^{{}^\text{\scriptsize $t$}}}}\!\!\!\begin{pmatrix} S\\ T \end{pmatrix}J_n {\overline{\begin{pmatrix} S\\ T \end{pmatrix}}}
={1\over{2i}}\,{}^{{}^{{}^{{}^\text{\scriptsize $t$}}}}\!\!\!\!\begin{pmatrix} \Omega\\ I_n \end{pmatrix}\,J_n {\overline{\begin{pmatrix} \Omega\\ I_n \end{pmatrix}}}=Y>0.$$
Thus we have
\begin{equation}
{}^tST-\,{}^tTS=0\qquad \textrm{and}\qquad {1\over{2i}}\,\big(\, {}^t\!S{\overline T}-\,{}^tT{\overline S\,}\big)=Y>0.
\end{equation}
Assume $Tv=(C\Om+D)v=0$ for some $v\in\BC^n.$ Then ${\overline T}{\overline v}=0,\ {}^tv\,{}^tT=0$ and hence
$$ {1\over{2i}}\,{}^tv\big(\, {}^t\!S{\overline T}-\,{}^tT{\overline S\,}\big){\overline v}=0.$$
By (11.5), $v=0$ and so $T=C\Om+D$ is nonsingular. This proves the statement (a).
\vskip 0.1cm
We set
$$\Omega_*=(A\Om+B)(C\Om+D)^{-1}=ST^{-1}.$$
By (11.5), we have $\Om_*=\,{}^t\Om_*$ and
\begin{eqnarray*}
     \textrm{Im}\,\Om_* &=&    {\frac 1{2\,i}}\big( \Om_*-{\overline{\Om}}_*\big)=\, {\frac 1{2\,i}}\big( \,{}^t\Om_*-{\overline{\Om}}_*\big)  \\
&=&\,{\frac 1{2\,i}}\,{}^tT^{-1} \big(\, {}^t\!S{\overline T}-\,{}^tT{\overline S\,}\big)\,{\overline T}^{-1}\\
&=&\, {}^tT^{-1}Y\,{\overline T}^{-1}>0.
\end{eqnarray*}
Therefore $\Om_*\in \BH_n.$ This completes the proof of the statement (b). $\hfill\square$
\vskip 0.2cm
\begin{lemma}
The symplectic group $Sp(n,\BR)$ acts on the Siegel upper half plane $\BH_n$ transitively by
\begin{equation}
M\cdot \Omega=(A\Omega+B)(C\Omega+D)^{-1},
\end{equation}
where $M= \begin{pmatrix} A& B\\ C&D\end{pmatrix}\in Sp(n,\BR)$ and $\Om\in\BH_n.$
\end{lemma}

\noindent {\it Proof.} Let $\Omega=X+i\,Y\in\BH_n$ with $X,Y\in {\rm{Sym}}(n,\BR)$ and $Y>0.$
It suffices to show that there exists an element $M\in Sp(n,\BR)$ such that $M\!\cdot\!(iI_n)=\Omega.$
We choose
$Q\in GL(n,\BR)$ such that $Q^2=Y.$ We take
$$M= \begin{pmatrix} I_n& X\\ 0&I_n\end{pmatrix} \begin{pmatrix} {}^tQ& 0\\ 0& Q^{-1}\end{pmatrix}.$$
According to (11.2), $M\in Sp(n,\BR).$ Clearly $M\!\cdot\!(iI_n)=X+i\,Y=\Omega.$   $\hfill\square$

\vskip 0.2cm It is known (cf.\,\cite{F}, p.\,322-328, \cite{KV}, p.\,10) that $Sp(n,\BR)$ is generated by the following elements
\begin{eqnarray*}
t_b&=&\begin{pmatrix} I_n & b \\ 0 & I_n \end{pmatrix}\ {\mathrm {with}}\ b=\,{}^tb\in\BR^{(n,n)},\\
d_{a}&=& \begin{pmatrix} {}^ta & 0 \\ 0 & a^{-1} \end{pmatrix}\ {\mathrm {with}}\ a\in GL(n,\BR),\\
\s_n&=& \begin{pmatrix} 0 & -I_n \\ I_n & \ 0 \end{pmatrix}.
\end{eqnarray*}
Thus if $M\in Sp(n,\BR),\ \det M=1.$
\vskip 0.3cm
A subgroup $\Gamma$ is said to be {\it discrete} if $\Gamma\cap K$ is finite for any compact subset $K$ of $Sp(n,\BR).$
\begin{theorem} A discrete subgroup $\Gamma$ of $Sp(n,\BR)$ acts properly discontinuously on $\BH_n$, that is, for any
two compact subsets $C_1,\,C_2$ of $\BH_n$, the set
$$\left\{\,\gamma\in \Gamma\,|\ \gamma\cdot C_1\cap C_2 \neq \emptyset\,\right\}$$
is finite.
\end{theorem}
\vskip 0.1cm\noindent {\it Proof.} We can show that the mapping
$$p:Sp(n,\BR)\lrt \BH_n,\qquad M\lrt M\cdot (iI_n),\quad M\in Sp(n,\BR)$$
is {\it proper}, i.e., for any compact subset $X\subset \BH_n,\ p^{-1}(X)$ is compact in $Sp(n,\BR)$ (cf.\,\cite{F},\,pp.\,28-29).
Suppose $X_1$ and $X_2$ are two compact subsets of $\BH_n$. Then $Z_1=p^{-1}(X_1)$ and $Z_2=p^{-1}(X_2)$ are compact in $Sp(n,\BR)$.
Since the image of $Z_2\times Z_1$ under the continuous mapping $(M_2,M_1)\mapsto M_2M_1^{-1}$ is compact, the set
$$Z_2Z_1^{-1}=\,\{\,M_2M_1^{-1}\,|\ M_1\in Z_1,\ M_2\in Z_2\,\}$$
is compact. It remains to show that $\{ \gamma\in\Gamma\,|\ \gamma\cdot X_1 \cap X_2 \neq \emptyset\,\}$ is finite.
If $\gamma\in\Gamma$ such that $\gamma\cdot X_1 \cap X_2 \neq \emptyset$, then
$$\gamma\cdot \Omega_1=\,M_2\cdot (iI_n)\in X_2\quad \textrm{for\ some}\ \Omega_1\in X_1 \  \ \textrm{and}\ M_2\in Z_2.$$
Since $(\gamma^{-1}M_2)\cdot (iI_n)=\,\Omega_1\in X_1,$ we have $\gamma^{-1}M_2\in Z_1,$ that is, $M_2^{-1}\gamma\in Z_1^{-1}.$
Therefore $\gamma\in M_2Z_1^{-1}\subset Z_2Z_1^{-1}.$ Since $\Gamma\cap Z_2Z_1^{-1}$ is finite, the set
$\{ \gamma\in\Gamma\,|\ \gamma\cdot X_1 \cap X_2 \neq \emptyset\,\}$ is finite.  $\hfill\square$

\vskip 0.2cm By Theorem 11.3, the Siegel modular group $\Gamma_n=Sp(n,\BZ)$ acts properly discontinuously on $\BH_n$.
Therefore the stabilizer $(\Gamma_n)_\Omega$ of $\Om\in\BH_n$ given by
$$ (\Gamma_n)_\Omega=\,\left\{\,\gamma\in\Gamma_n\,|\ \gamma\cdot\Omega=\,\Omega\,\right\}$$
is a finite subgroup of $\Gamma_n.$

\vskip 0.2cm
Let $q$ be a positive integer. The set
$$\Gamma_n (q)=\,\left\{\, M\in \Gamma_n\,|\ M\equiv I_{2n}\ \textrm{mod}\ q\ \right\}$$
is a normal subgroup of $\Gamma_n$ because it is the kernel of the homomorphism $\Gamma_n\lrt Sp(n,\BZ/q\BZ)$ defined by
$\gamma\lrt \gamma\,\textrm{mod}\,q.$ It is called the {\it principal congruence subgroup} of {\it level} $q$. We have
$\Gamma_n(1)=\Gamma_n.$ A subgroup $\Gamma$ of $Sp(n,\BR)$ which contains $\Gamma_n(\ell)$ for some positive integer $\ell$
as a subgroup of finite index is called a {\it modular group}. A subgroup $\Gamma$ of $\Gamma_n$ which contains $\Gamma_n(\ell)$ for some positive integer $\ell$
as a subgroup of finite index is called a {\it congruence subgroup} of $\Gamma_n$. The subset $\Gamma_{\vartheta,n}$ of $\Gamma_n$ consisting of elements $\gamma=\begin{pmatrix} A& B\\ C&D\end{pmatrix}\in \Gamma_n$ such that the diagonal elements of
$A{}^tB$ and $C\,{}^tD$ are even integers is a subgroup of $\Gamma_n$ called the {\it theta group}. For a positive integer $q$,
we let
$$ \Gamma_{n,0}(q)=\,\left\{\, \begin{pmatrix} A& B\\ C&D\end{pmatrix}\in \Gamma_n\,\Big|\ C\equiv 0\ \textrm{mod}\,q\ \right\}.$$
Then $\Gamma_{n,0}(q)$ is a congruence subgroup of $\Gamma_n$ containing the principal congruence subgroup $\Gamma_n(q)$ of level
$q$.

\vskip 0.2cm Let $\Om_1$ and $\Om_2$ be two points of $\BH_n$ and $M=\begin{pmatrix} A& B\\ C&D\end{pmatrix}\in Sp(n,\BR).$
We write $\Om_i^*=M\!\cdot \!\Om_i\,\,(i=1,2)$. Then by the symplectic conditions (11.1) and (11.2), we have
\begin{equation}
\Om_2^*-\Om_1^*=\,{}^t(C\Om_2+D)^{-1} (\Om_2-\Om_1)(C\Om_1+D)^{-1}
\end{equation}
and
\begin{equation}
\Om_2^*-\overline{\Om_1^*}=\,{}^t(C\Om_2+D)^{-1} \big(\Om_2-\overline{\Om_1} \big)(C{\overline{\Om_1}}+D)^{-1}
\end{equation}

Let $\Om=X+i\,Y\in\BH_n$ with $X,Y\in\BR^{(n,n)}.$ If $\Om^*=M\!\cdot\!\Om,$ then we write $\Om^*=X^*+i\,Y^*$ with
$X^*,Y^*\in\BR^{(n,n)}.$ Then by (11.8),
\begin{equation}
\Om^*-\overline{\Om^*}=\,{}^t(C\Om+D)^{-1} (\Om-\overline{\Om}\, ) (C{\overline{\Om}}+D)^{-1}
\end{equation}
and hence
\begin{equation}
Y^*=\,{}^t(C\Om+D)^{-1} Y (C{\overline{\Om}}+D)^{-1}.
\end{equation}
Therefore we obtain
\begin{equation}
\det Y^*=\,\det Y \cdot |\det (C\Om+D)|^{-2}.
\end{equation}
And

\begin{eqnarray*}
d\Om^* &=& d(M\!\cdot\!\Om)=d\big\{(A\Om+B)(C\Om+D)^{-1}\big\}\\
&=&\,A\, d\Om\, (C\Om+D)^{-1}-(A\Om+B)(C\Om+D)^{-1}C \,d\Om\, (C\Om+D)^{-1}\\
&=&\,{}^t(C\Om+D)^{-1} \big\{ \Om (\,{}^tCA-\,{}^tAC)\,+\,(\,{}^tDA-\,{}^tBC) \big\}
d\Om \,(C\Om+D)^{-1}\\
&=& \,{}^t(C\Om+D)^{-1} d\Om \,(C\Om+D)^{-1}.
\end{eqnarray*}

\noindent Thus we have
\begin{equation}
d\Om^*=\,{}^t(C\Om+D)^{-1} \,d\Om\, (C\Om+D)^{-1}.
\end{equation}
By Formulas (11.10) and (11.12),
\begin{equation*}
ds^2=\s (Y^{-1}d\Om\, Y^{-1}d{\overline\Om})
\end{equation*}
is invariant under the action (11.6) of $Sp(n,\BR).$ For $\Om=iI_n$,
$$ds^2=\,\sum_{i=1}\big(dx_{ii}^2+dy_{ii}^2\big)\,+\,2\,\sum_{1\leq i\leq j\leq n}\big(dx_{ij}^2\,+\,dy_{ij}^2\big).$$
Since $Sp(n,\BR)$ acts on $\BH_n$ transitively, $ds^2$ is an $Sp(n,\BR)$-invariant Riemannian metric on $\BH_n.$
\vskip 0.2cm The tangent space $T_\Om(\BH_n)$ of $\BH_n$ at $\Om$ is identified with the vector space ${\rm{Sym}}(n,\BC)$ consisting of all $n\times n$ symmetric complex matrices (cf.\,(12.20) in Section 12). By (11.12), the differential
$$dM_\Om: T_\Om(\BH_n)\lrt T_{M\cdot\Om}(\BH_n)$$ of
the symplectic transformation $M$ at $\Om$ is given by
\begin{equation}
dM_\Om(W)=\,{}^t(C\Om+D)^{-1} \,W\, (C\Om+D)^{-1}, \quad W\in {\rm{Sym}}(n,\BC).
\end{equation}
We can see that the Jacobian of the symplectic transformation $M\in Sp(n,\BC)$ is given by
$${ {\partial (\Om^*)}\over {\partial (\Om)} }=\,\det (C\Om+D)^{-(n+1)},$$
where $\Om^*=\,M\cdot\Om$ with $M= \begin{pmatrix} A& B\\ C&D\end{pmatrix}\in Sp(n,\BR).$

\vskip 0.2cm Finally we describe the universal covering group of $Sp(n,\BR)$ using the so-called 
$\textsf{Maslov\ index}$.
Let $(V,B)$ be a symplectic (real) vector space of
dimension $2\,n$ with a non-degenerate alternating form $B$ on $V$. A subspace of $(V,B)$ such that $B(x,y)=0$ for all $x,y\in L$
is said to be $\textsf{totally\ isotropic}$. For a subspace $L$ of $(V,B)$, we will denote by $L^{\perp}$
the orthogonal complement of $L$ in $V$ relative to $B$, i.e.,
$$L^{\perp}=\,\big\{ x\in V\,\big|\ B(x,y)=0\quad {\rm for\ all}\ y\in V\,\big\}.$$
If $L$ is a subspace of $(V,B)$ such that $L=L^{\perp},$ then $L$ is called a $\textsf{Lagrangian}$ subspace of $(V,B).$ If $L$ is
a totally isotropic subspace of $V$ such that $B(x,L)=0$ implies $x\in L$, then $L$ is said to be
$\textsf{maximally\ totally\ isotropic}$. We note that if $L$ is a Lagrangian subspace of $(V,B)$, then $\dim L=n$ because
$\dim L +\dim L^{\perp}=2\,n.$
\vskip 0.1cm
Let $Sp(B)$ be the symplectic group defined by
\begin{equation*}
Sp(B)=\,\big\{\, g\in GL(V)\,\big|\ B(gx,gy)=B(x,y)\quad {\rm for\ all}\ x,y\in V\,\big\}.
\end{equation*}
\vskip 0.2cm
\begin{definition}
Let $L_1,L_2,L_3$ be three Lagrangian subspaces of $V$.  The integer $\tau(L_1,L_2,L_3)$ is defined to be the signature
of the quadratic form $Q(x_1+x_2+x_3)$ on the $3\,n$-dimensional vector space $L_1\oplus L_2\oplus L_3$ defined by
$$Q(x_1+x_2+x_3):=\,B(x_1,x_2)\,+\,B(x_2,x_3)\,+\,B(x_3,x_1),$$
where $ x_1\in L_1,\ x_2\in L_2$ and $x_3\in L_3.$ The integer $\tau(L_1,L_2,L_3)$ is called the {\it Maslov\ index} of
$(V,B)$.
\end{definition}

\vskip 0.1cm
\begin{lemma}
Let $L_1,L_2,L_3$ be three Lagrangian subspaces of $(V,B).$ Then we have the following properties\,:
\vskip 0.1cm\noindent
(1) The Maslov index is $Sp(n,\BR)$-invariant, i.e., for any $g\in Sp(n,\BR),$ we have $\tau(gL_1,gL_2,gL_3)=
\tau(L_1,L_2,L_3).$
\vskip 0.1cm\noindent
(2) $\tau(L_1,L_2,L_3)=-\tau(L_2,L_1,L_3)=-\tau(L_1,L_3,L_2).$
\end{lemma}
\vskip 0.1cm\noindent {\it Proof.} It follows immediately from the definition. $\hfill\square$

\vskip 0.2cm For a sequence $(L_1,L_2,\cdots,L_k)$ with $k\geq 4$ of Lagrangian subspaces of $(V,B)$,
we define the generalized Maslov index $\tau(L_1,L_2,\cdots,L_k)$ by
$$\tau(L_1,L_2,\cdots,L_k)=\,\tau(L_1,L_2,L_3)+\tau(L_1,L_3,L_4)+\cdots+\tau(L_1,L_{k-1},L_k).$$
\vskip 0.2cm
\begin{proposition}
(1) The Maslov index $\tau(L_1,L_2,\cdots,L_k)$ is invariant under the action of the symplectic group $Sp(B)$, and
its value is unchanged under circular permutation.
\vskip 0.1cm\noindent (2) For any Lagrangian subspace $L_1,L_2,L_3,L_1',L_2',L_3',$ we have
\begin{eqnarray*}
\tau(L_1',L_2',L_3')&\!\!=\!\,\tau(L_1,L_2,L_3)+\tau(L_1',L_2',L_2,L_1)+\tau(L_2',L_3',L_3,L_2)\\
&\!\!\ \!\! & \hskip -8.3cm \ +\,\tau(L_3',L_1',L_1,L_3).
\end{eqnarray*}
\end{proposition}
\vskip 0.1cm\noindent
{\it Proof.} See \cite[pp.\,45-46]{LV}.  $\hfill\square$

\vskip 0.1cm Let $\Lambda$ be the space of all Lagrangian subspaces of $(V,B)$. Then may be regarded as a
closed submanifold of the Grassmannian manifold of all $n$-dimensional subspaces in $\BR^{2n}.$ We define
$$\widetilde{\Lambda}:=\,\Lambda\times\BZ=\,\big\{\,(L,u)\,\big|\ L\in\Lambda,\ u\in\BZ\,\big\}.$$
\noindent We fix a Lagrangian subspace $L_0$ of $(V,B).$ Let $(L_1,u_1)\in \widetilde{\Lambda}$ and let
${\mathscr U}$ be a neighborhood of $L_1$. Let $L_2$ be a Lagrangian subspace of $V$ transverse to $L_1.$
We define
$$U(L_1,u_1;{\mathscr U},L_2):=\big\{\,(L,u)\,\big|\ L\in {\mathscr U},\ u=u_1+\tau(L,L_0,L_1,L_2)\,\big\}.$$
It is proved in Proposition 1.9.5 in \cite{LV} that the set of all such $U(L_1,u_1;{\mathscr U},L_2)$'s form
a neighborhood for a topology on $\widetilde{\Lambda}.$ Let $\pi: \widetilde{\Lambda}\lrt \Lambda$ be the projection
defined by $\pi(L,u)=L.$ Clearly $\pi$ is a continuous map and hence $\widetilde{\Lambda}$ is a covering of $\Lambda.$

\vskip 0.2cm Let $L_*$ be a fixed element of $\Lambda$. We define the group
\begin{equation}
\widetilde{Sp(B)}_*:=\,Sp(B)\times \BZ
\end{equation}
equipped with the multiplication law
\begin{equation*}
(g_1,n_1)\cdot (g_2,n_2)=\,\big(g_1g_2, n_1+n_2+\,\tau(L_*,g_1L_*,g_1g_2L_*)\big),
\end{equation*}
where $g_1,g_2\in Sp(B)$ and $n_1,n_2\in \BZ.$ Then it is easy to see that $\widetilde{Sp(B)}_*$ acts
on $\widetilde{\Lambda}$ by
\begin{equation*}
(g,n)\cdot (L,u)=\,\big(gL, n+u+\tau(L_*,gL_*,gL)\big),
\end{equation*}
where $g\in Sp(B),\ n,u\in\BZ$ and $L\in \Lambda.$
\vskip 0.1cm
Let $L_2$ be a Lagrangian subspace of $(V,B)$ transverse to $L_*$ and ${\mathscr E}$ be a neighborhood of $e$ in
$Sp(B)$, where $e$ is the identity element of $Sp(B).$ We define
$${\mathscr W}({\mathscr E},L_2):=\big\{\,\big(g,-\tau(gL_*,L_2,L_*)\big)\,\big|\ g\in {\mathscr E}\,\big\}.$$
Then the set of all such ${\mathscr W}({\mathscr E},L_2)$'s form a fundamental system of neighborhoods of $(e,0)$ on
$\widetilde{Sp(B)}_*$. Therefore $\widetilde{Sp(B)}_*$ has the structure of a topological group. It is easily seen that
$\widetilde{Sp(B)}_*$ acts on $\widetilde{\Lambda}$ continuously.

\vskip 0.2cm
\begin{definition}
An {\it oriented} vector space of dimension $n$ is a pair $(W,\varepsilon)$, where $W$ is a real vector space of
dimension $n$ and $\varepsilon$ is an orientation of $W$, i.e., a connected component of $\Lambda^nL-\{0\}.$
If $(W_1,\varepsilon_1)$ and $(W_2,\varepsilon_2)$ are two oriented vector spaces of dimension $n$ and $A$ is
a linear invertible map from $W_1$ to $W_2$, we define the sign of the determinant of $A$ denoted by $\delta(A)=\pm 1$,
by the condition
$$\left( \Lambda^nA\right)\varepsilon_1=\,c\,\delta (A)\,\varepsilon_2\quad {\rm with}\ c>0.$$
\end{definition}

\vskip 0.2cm $L$ and $M$ be two Lagrangian subspaces of a symplectic vector space $(V,B).$ We define
$g_{M,L}: L\lrt M^*$ by $\langle\, g_{M,L}(x),y\rangle=\,B(x,y)$ for all $x\in L$ and $y\in M.$ Here $M^*$ denotes the dual
vector space of $M$. Let $(L_1,\varepsilon_1)$ and $(L_2,\varepsilon_2)$ be two oriented Lagrangian subspaces of $(V,B)$
which are transverse. Then $g_{L_2,L_1}:(L_1,\varepsilon_1)\lrt (L_2,\varepsilon_2)$. We define
$$\xi \big( (L_1,\varepsilon_1),(L_2,\varepsilon_2)\big):=\,\delta (g_{L_2,L_1}).$$
This depends only on the relative orientation of $(L_1,\varepsilon_1)$ and $(L_2,\varepsilon_2)$. More generally if $L_1$ and $L_2$ are not transverse, we define $(L_1,\varepsilon_1)$ and $(L_2,\varepsilon_2)$ as follows: Let $\varepsilon$ be an orientation of
$H=L_1\cap L_2.$ Then $\varepsilon$ defines an orientation $\widetilde{\varepsilon}_i\ (i=1,2)$ on $L_i/H$ by
$\widetilde{\varepsilon}_i\wedge \varepsilon=\varepsilon_i.$ Since $L_1/H$ and $L_2/H$ are two transverse subspaces of
$(L_1+L_2)/H=\,H^{\perp}/H,$ we can define
$$\xi \big( (L_1,\varepsilon_1),(L_2,\varepsilon_2)\big):=\,\xi
\big( (L_1/H,\widetilde{\varepsilon}_1),(L_2/H,\widetilde{\varepsilon}_2)\big).$$
We observe that this is independent of the choice of the orientation $\varepsilon$ of $H$ because
$\widetilde{\varepsilon}_1$ and $\widetilde{\varepsilon}_2$ change simultaneously if we change $\varepsilon$ to
$-\varepsilon$.
\vskip 0.2cm If $L_1=L_2$, we define
$$\xi \big( (L_1,\varepsilon_1),(L_2,\varepsilon_2)\big):=\,
\begin{cases} \ \ 1 & \ {\rm if}\ \varepsilon_1=\varepsilon_2,\\
-1 & \ {\rm if}\ \varepsilon_1\not=\varepsilon_2.
\end{cases}$$

\vskip 0.2cm
\begin{definition}
Let $(L_1,\varepsilon_1)$ and $(L_2,\varepsilon_2)$ be two oriented Lagrangian subspaces of
a symplectic vector space $(V,B)$. We define
\begin{equation*}
s\big( (L_1,\varepsilon_1),(L_2,\varepsilon_2)\big):= \big(\sqrt{-1}\big)^{n-\dim (L_1\cap L_2)}\,
\xi \big( (L_1,\varepsilon_1),(L_2,\varepsilon_2)\big).
\end{equation*}
\end{definition}

\begin{definition}
Let $L$ be a Lagrangian subspace of
a symplectic vector space $(V,B)$. We choose an orientation $L^+$ on $L$.
We define the map $s_L:Sp(B)\lrt \BC$ by
\begin{equation*}
s_L(g):=\,s(L^+,gL^+),\quad g\in Sp(B).
\end{equation*}
\end{definition}
This is well-defined because $s_L(g)$ is independent of the choice of the orientation on $L$.
\vskip 0.2cm
We define the map ${\tilde s}_*: \widetilde{Sp(B)}_*\lrt \BC$ by
\begin{equation*}
{\tilde s}_*(g,n):=\,e^{{\pi n i}\over 2}\,s_{L_*}(g),\quad g\in Sp(B),\ n\in\BZ.
\end{equation*}

\vskip 0.2cm
\begin{lemma}
${\tilde s}_*(g,n)$ is a character of $\widetilde{Sp(B)}_*$ with values in $\BZ/4\BZ.$
\end{lemma}
\vskip 0.1cm\noindent
{\it Proof.} The proof can be found in \cite[p.\,72]{LV}. $\hfill\square$

\vskip 0.2cm We see that the kernel of ${\tilde s}_*$ is the universal covering group of $Sp(B)$ and
the fundamental group $\pi_1(Sp(B))$ of $Sp(B)$ is isomorphic to $\BZ$. Therefore $\widetilde{Sp(B)}_*$
is the union of four connected components such that each of them is simply connected.
\vskip 0.2cm
We now consider the group
\begin{equation*}
Sp(B,L_*):=\,Sp(B) \times \BC^*_1
\end{equation*}
equipped with the multiplication law
\begin{equation*}
(g_1,t_1)\cdot (g_2,t_2)=\big(\, g_1g_2,\,t_1t_2\,c_*(g_1,g_2)\,\big),
\end{equation*}
where $g_1,g_2\in Sp(B),\ t_1,t_2\in \BC^*_1$ and
$$c_*(g_1,g_2):=\,e^{- {{\pi i}\over 4} \,\tau(L_*,g_1L_*,g_1g_2L_*) }.$$
It is easily checked that the $\varphi: \widetilde{Sp(B)}_*\lrt Sp(B,L_*)$ defined by
$$\varphi (g,n):=\big( g,e^{{\pi n i}\over 4 }\,\big), \quad g\in Sp(B),\ n\in\BZ$$
is a group homomorphism. We define
\begin{equation}
Mp(B)_*:=\big\{\, (g,t)\in Sp(B,L_*)\,\big|\ t^2=\,s_{L_*}(g)^{-1}\,\big).
\end{equation}

\end{section}

\newpage

\begin{section}{{\large\bf Some Geometry on Siegel Space}}
\setcounter{equation}{0}
\newcommand\POB{ {{\partial}\over {\partial{\overline \Omega}}} }
\newcommand\PZB{ {{\partial}\over {\partial{\overline Z}}} }
\newcommand\PX{ {{\partial}\over{\partial X}} }
\newcommand\PY{ {{\partial}\over {\partial Y}} }
\newcommand\PU{ {{\partial}\over{\partial U}} }
\newcommand\PV{ {{\partial}\over{\partial V}} }
\newcommand\PO{ {{\partial}\over{\partial \Omega}} }
\newcommand\PZ{ {{\partial}\over{\partial Z}} }

\vskip 0.21cm For $\Om=(\omega_{ij})\in\BH_n,$ we write $\Om=X+i\,Y$
with $X=(x_{ij}),\ Y=(y_{ij})$ real and $d\Om=(d\om_{ij})$. We
put
$$\PO=\,\left(\,
{ {1+\delta_{ij}}\over 2}\, { {\partial}\over {\partial \om_{ij} }
} \,\right) \qquad\text{and}\qquad \POB=\,\left(\, {
{1+\delta_{ij}}\over 2}\, { {\partial}\over {\partial {\overline
{\om}}_{ij} } } \,\right).$$ C. L. Siegel \cite{Si1} introduced
the symplectic metric $ds^2$ on $\BH_n$ invariant under the action
(11.5) of $Sp(n,\BR)$ given by
\begin{equation}
ds^2=\s (Y^{-1}d\Om\, Y^{-1}d{\overline\Om})\end{equation} and H.
Maass \cite{M1} proved that its Laplacian is given by
\begin{equation}
\Delta=\,4\,\s \left( Y {}^{{}^{{}^{{}^\text{\scriptsize $t$}}}}\!\!\!
\left(Y\POB\right)\PO\right).\end{equation} And
\begin{equation}
dv_n(\Om)=(\det Y)^{-(n+1)}\prod_{1\leq i\leq j\leq n}dx_{ij}\,
\prod_{1\leq i\leq j\leq n}dy_{ij}\end{equation} is a
$Sp(n,\BR)$-invariant volume element on
$\BH_n$\,(cf.\,\cite{Si3},\,p.\,130).

\vskip 0.2cm\noindent \begin{theorem}\,({\bf Siegel\,\cite{Si1}}).
(1) There exists exactly one geodesic joining two arbitrary points
$\Om_0,\,\Om_1$ in $\BH_n$. Let $R(\Om_0,\Om_1)$ be the
cross-ratio defined by
\begin{equation}
R(\Om_0,\Om_1)=(\Om_0-\Om_1)(\Om_0-{\overline
\Om}_1)^{-1}(\overline{\Om}_0-\overline{\Om}_1)(\overline{\Om}_0-\Om_1)^{-1}.
\end{equation}
For brevity, we put $R_*=R(\Om_0,\Om_1).$ Then the symplectic
length $\rho(\Om_0,\Om_1)$ of the geodesic joining $\Om_0$ and
$\Om_1$ is given by
\begin{equation}
\rho(\Om_0,\Om_1)^2=\s \left( \left( \log { {1+R_*^{\frac 12}
}\over {1-R_*^{\frac 12} } }\right)^2\right),\end{equation} where
\begin{equation*}
\left( \log { {1+R_*^{\frac 12} }\over {1-R_*^{\frac 12} }
}\right)^2=\,4\,R_* \left( \sum_{k=0}^{\infty} { {R_*^k}\over
{2k+1}}\right)^2.
\end{equation*}

\noindent (2) For $M\in Sp(n,\BR)$, we set
$${\widetilde \Om}_0=M\cdot \Om_0\quad \textrm{and}\quad {\widetilde \Om}_1=M\cdot
\Om_1.$$ Then $R(\Om_1,\Om_0)$ and
$R({\widetilde\Om}_1,{\widetilde\Om}_0)$ have the same eigenvalues.

\noindent (3) All geodesics are symplectic images of the special
geodesics
\begin{equation}
\alpha(t)=i\,{\rm diag}(a_1^t,a_2^t,\cdots,a_n^t),
\end{equation}
where $a_1,a_2,\cdots,a_n$ are arbitrary positive real numbers
satisfying the condition
$$\sum_{k=1}^n \left( \log a_k\right)^2=1.$$
\end{theorem}
\noindent The proof of the above theorem can be found in
\cite{Si1}, pp.\,289-293.

\newcommand\OW{\overline{W}}
\newcommand\OP{\overline{P}}
\newcommand\OQ{\overline{Q}}
\newcommand\Dg{{\mathbb D}_n}
\newcommand\Hg{{\mathbb H}_n}
\newcommand\BD{{\mathbb D}}
\newcommand\fk{\frak k}
\newcommand\fp{\frak p}

\vskip 0.1cm Let $$\BD_n=\left\{\,W\in\BC^{(n,n)}\,|\ W=\,{ }^tW,\
I_n-W{\overline W}>0\,\right\}$$ be the generalized unit disk of
degree $g$. The Cayley transform $\Psi:\Dg\lrt\Hg$ defined by
\begin{equation}
\Psi(W)=i\,(I_n+W)(I_n-W)^{-1},\quad W\in\Dg
\end{equation}
is a biholomorphic mapping of $\Dg$ onto $\Hg$ which gives the
bounded realization of $\Hg$ by $\Dg$\,(cf.\,\cite{Si1}). A.
Kor{\'a}nyi and J. Wolf \cite{KW} gave a realization of a bounded
symmetric domain as a Siegel domain of the third kind
investigating a generalized Cayley transform of a bounded
symmetric domain that generalizes the Cayley transform $\Psi$ of
$\BD_n$.

\vskip 0.2cm Let
\begin{equation}
T={1\over {\sqrt{2}} }\,
\begin{pmatrix} \ I_n&\ I_n\\
                   iI_n&-iI_n\end{pmatrix}
\end{equation}
be the $2n\times 2n$ matrix represented by $\Psi.$ Then
\begin{equation}
T^{-1}Sp(n,\BR)\,T=\left\{ \begin{pmatrix} P & Q \\ \OQ & \OP
\end{pmatrix}\,\Big|\ {}^tP\OP-\,{}^t\OQ Q=I_n,\ {}^tP\OQ=\,{}^t\OQ
P\,\right\}.
\end{equation}
Indeed, if $M=\begin{pmatrix} A&B\\
C&D\end{pmatrix}\in Sp(n,\BR)$, then
\begin{equation}
T^{-1}MT=\begin{pmatrix} P & Q \\ \OQ & \OP
\end{pmatrix},
\end{equation}
where
\begin{equation}
P= {\frac 12}\,\Big\{ (A+D)+\,i\,(B-C)\Big\}
\end{equation}
and
\begin{equation}
 Q={\frac
12}\,\Big\{ (A-D)-\,i\,(B+C)\Big\}.
\end{equation}

For brevity, we set
\begin{equation*}
G_*= T^{-1}Sp(n,\BR)T.
\end{equation*}
Then $G_*$ is a subgroup of $SU(n,n),$ where
$$SU(n,n)=\left\{\,h\in\BC^{(n,n)}\,\big|\ {}^th I_{n,n}{\overline
h}=I_{n,n}\,\right\},\quad I_{n,n}=\begin{pmatrix} \ I_n&\ 0\\
0&-I_n\end{pmatrix}.$$ In the case $n=1$, we observe that
$$T^{-1}Sp(1,\BR)T=T^{-1}SL_2(\BR)T=SU(1,1).$$
If $n>1,$ then $G_*$ is a {\it proper} subgroup of $SU(n,n).$ In
fact, since ${}^tTJ_nT=-\,i\,J_n$, we get
\begin{equation}G_*=\Big\{\,h\in SU(n,n)\,\big|\
{}^thJ_nh=J_n\,\Big\}=SU(n,n)\cap Sp(n,\BC),
\end{equation}

\noindent where
$$Sp(n,\BC)=\Big\{\alpha\in \BC^{(2n,2n)}\ \big\vert \ ^t\!\alpha\, J_n\,\alpha= J_n\ \Big\}.$$

Let
\begin{equation*}
P^+=\left\{\begin{pmatrix} I_n & Z\\ 0 & I_n
\end{pmatrix}\,\Big|\ Z=\,{}^tZ\in\BC^{(n,n)}\,\right\}
\end{equation*}
be the $P^+$-part of the complexification of $G_*\subset SU(n,n).$
We note that the Harish-Chandra decomposition of an element
$\begin{pmatrix} P & Q\\ {\overline Q} & {\overline P}
\end{pmatrix}$ in $G_*$ is
\begin{equation*}
\begin{pmatrix} P & Q\\ \OQ & \OP
\end{pmatrix}=\begin{pmatrix} I_n & Q\OP^{-1}\\ 0 & I_n
\end{pmatrix} \begin{pmatrix} P-Q\OP^{-1}\OQ & 0\\ 0 & \OP
\end{pmatrix} \begin{pmatrix} I_n & 0\\ \OP^{-1}\OQ & I_n
\end{pmatrix}.
\end{equation*}
For more detail, we refer to \cite[p.\,155]{Kn}. Thus the
$P^+$-component of the following element
$$\begin{pmatrix} P & Q\\ \OQ & \OP
\end{pmatrix}   \cdot\begin{pmatrix} I_n & W\\ 0 & I_n
\end{pmatrix},\quad W\in \BD_n$$ of the complexification of $G_*^J$ is
given by
\begin{equation}
 \begin{pmatrix} I_n & (PW+Q)(\OQ W+\OP)^{-1}
\\ 0 & I_n
\end{pmatrix}.
\end{equation}

\newcommand\PW{ {{\partial}\over{\partial W}} }
\newcommand\PWB{ {{\partial}\over {\partial{\overline W}}} }
\newcommand\OVW{\overline W}
\newcommand\al{\alpha}

\noindent We note that $Q\OP^{-1}\in\Dg.$ We get the
Harish-Chandra embedding of $\Dg$ into $P^+$\,(cf.
\cite[p.\,155]{Kn} or \cite[pp.\,58-59]{Sa4}). Therefore we see
that $G_*$ acts on $\Dg$ transitively by
\begin{equation}
\begin{pmatrix} P & Q \\ \OQ & \OP
\end{pmatrix}\cdot W=(PW+Q)(\OQ W+\OP)^{-1},\quad \begin{pmatrix} P & Q \\ \OQ & \OP
\end{pmatrix}\in G_*,\ W\in \Dg.
\end{equation}
The isotropy subgroup $K_*$ of $G_*$ at the origin $o$ is given by
$$K_*=\left\{\,\begin{pmatrix} P & 0 \\ 0 & {\overline
P}\end{pmatrix}\,\Big|\ P\in U(n)\ \right\}.$$ Thus $G_*/K_*$ is
biholomorphic to $\Dg$. It is known that the action (11.6) is
compatible with the action (12.15) via the Cayley transform $\Psi$\
(cf.\,(12.7)). In other words, if $M\in Sp(n,\BR)$ and $W\in\BD_n$,
then
\begin{equation}
M\cdot \Psi(W)=\Psi(M_*\cdot W),
\end{equation}

\noindent where $M_*=T^{-1}MT\in G_*.$

\vskip 0.2cm For $W=(w_{ij})\in \Dg,$ we write $dW=(dw_{ij})$ and
$d{\overline W}=(d{\overline{w}}_{ij})$. We put $$\PW=\,\left(\, {
{1+\delta_{ij}}\over 2}\, { {\partial}\over {\partial w_{ij} } }
\,\right) \qquad\text{and}\qquad \PWB=\,\left(\, {
{1+\delta_{ij}}\over 2}\, { {\partial}\over {\partial {\overline
{w}}_{ij} } } \,\right).$$

Using the Cayley transform $\Psi:\Dg\lrt \BH_n$, Siegel showed
(cf.\,\,\cite{Si1}) that
\begin{equation}
ds_*^2=4 \s \Big((I_n-W{\overline W})^{-1}dW\,(I_n-\OVW
W)^{-1}d\OVW\,\Big)\end{equation} is a $G_*$-invariant Riemannian
metric on $\BD_n$ and Maass \cite{M1} showed that its Laplacian is
given by
\begin{equation}
\Delta_*=\,\s \left( (I_n-W\OW)\,{}^{{}^{{}^{{}^\text{\scriptsize $t$}}}}\!\!\!\left( (I_n-W\OVW)
\PWB\right)\PW\right).\end{equation}

\vskip 0.2cm Now we discuss the differential operators on $\BH_n$ invariant under the action (11.6). The isotropy
subgroup $K$ at $iI_n$ for the action (11.6) is a maximal compact
subgroup given by
\begin{equation*}
K=\left\{ \begin{pmatrix} A & -B \\ B & A \end{pmatrix} \Big| \
A\,^t\!A+ B\,^t\!B=I_n,\ A\,^t\!B=B\,^t\!A,\ A,B\in
\BR^{(n,n)}\,\right\}.
\end{equation*}

\noindent Let $\fk$ be the Lie algebra of $K$. Then the Lie
algebra ${\mathfrak s}{\mathfrak p}(n,\BR)$ of $Sp(n,\BR)$ has a Cartan decomposition
${\mathfrak s}{\mathfrak p}(n,\BR)=\fk\oplus
\fp$, where
\begin{equation*}
\fp=\left\{ \begin{pmatrix} X & Y \\ Y & -X \end{pmatrix} \Big| \
X=\,{}^tX,\ Y=\,{}^tY,\ X,Y\in \BR^{(n,n)}\,\right\}.
\end{equation*}

The subspace $\fp$ of ${\mathfrak s}{\mathfrak p}(n,\BR)$ may be regarded as the tangent space
of $\BH_n$ at $iI_n.$ The adjoint representation of $G$ on ${\mathfrak s}{\mathfrak p}(n,\BR)$
induces the action of $K$ on $\fp$ given by
\begin{equation}
k\cdot Z=\,kZ\,^tk,\quad k\in K,\ Z\in \fp.
\end{equation}

Let ${\mathbb T}_n$ be the vector space of $n\times n$ symmetric complex
matrices. We let $\psi: \fp\lrt {\mathbb T}_n$ be the map defined by
\begin{equation}
\psi\left( \begin{pmatrix} X & Y \\ Y & -X \end{pmatrix}
\right)=\,X\,+\,i\,Y, \quad \begin{pmatrix} X & Y \\ Y & -X
\end{pmatrix}\in \fp.
\end{equation}

\noindent We let $\delta:K\lrt U(n)$ be the isomorphism defined by
\begin{equation}
\delta\left( \begin{pmatrix} A & -B \\ B & A \end{pmatrix}
\right)=\,A\,+\,i\,B, \quad \begin{pmatrix} A & -B \\ B & A
\end{pmatrix}\in K,
\end{equation}

\noindent where $U(n)$ denotes the unitary group of degree $n$. We
identify $\fp$ (resp. $K$) with ${\mathbb T}_n$ (resp. $U(n)$) through the
map $\Psi$ (resp. $\delta$). We consider the action of $U(n)$ on
${\mathbb T}_n$ defined by
\begin{equation}
h\cdot Z=\,hZ\,^th,\quad h\in U(n),\ Z\in {\mathbb T}_n.
\end{equation}

\noindent Then the adjoint action (12.19) of $K$ on $\fp$ is
compatible with the action (12.22) of $U(n)$ on ${\mathbb T}_n$ through the
map $\psi.$ Precisely for any $k\in K$ and $\omega\in \fp$, we get
\begin{equation}
\psi(k\,\omega \,^tk)=\delta(k)\,\psi(\omega)\,^t\delta (k).
\end{equation}

\noindent The action (12.22) induces the action of $U(n)$ on the
polynomial algebra $ \textrm{Pol}({\mathbb T}_n)$ and the symmetric algebra
$S({\mathbb T}_n)$ respectively. We denote by $ \textrm{Pol}({\mathbb T}_n)^{U(n)}$
$\Big( \textrm{resp.}\ S({\mathbb T}_n)^{U(n)}\,\Big)$ the subalgebra of $
\textrm{Pol}({\mathbb T}_n)$ $\Big( \textrm{resp.}\ S({\mathbb T}_n)\,\Big)$
consisting of $U(n)$-invariants. The following inner product $(\
,\ )$ on ${\mathbb T}_n$ defined by $$(Z,W)= \, \textrm{tr}
\big(Z\,{\overline W}\,\big),\quad Z,W\in {\mathbb T}_n$$

\noindent gives an isomorphism as vector spaces
\begin{equation}
{\mathbb T}_n\cong {\mathbb T}_n^*,\quad Z\mapsto f_Z,\quad Z\in {\mathbb T}_n,
\end{equation}

\noindent where ${\mathbb T}_n^*$ denotes the dual space of ${\mathbb T}_n$ and $f_Z$
is the linear functional on ${\mathbb T}_n$ defined by
$$f_Z(W)=(W,Z),\quad W\in {\mathbb T}_n.$$

\noindent It is known that there is a canonical linear bijection
of $S({\mathbb T}_n)^{U(n)}$ onto the algebra ${\mathbb D}(\BH_n)$ of
differential operators on $\BH_n$ invariant under the action (11.6)
of $Sp(n,\BR)$. Identifying ${\mathbb T}_n$ with ${\mathbb T}_n^*$ by the above isomorphism
(12.24), we get a canonical linear bijection
\begin{equation}
\Phi:\textrm{Pol}({\mathbb T}_n)^{U(n)} \lrt {\mathbb D}(\BH_n)
\end{equation}

\noindent of $ \textrm{Pol}({\mathbb T}_n)^{U(n)}$ onto ${\mathbb
D}(\BH_n)$. The map $\Phi$ is described explicitly as follows.
Similarly the action (12.19) induces the action of $K$ on the
polynomial algebra $ \textrm{Pol}(\fp)$ and $S(\fp)$ respectively.
Through the map $\psi$, the subalgebra $ \textrm{Pol}(\fp)^K$ of $
\textrm{Pol}(\fp)$ consisting of $K$-invariants is isomorphic to $
\textrm{Pol}({\mathbb T}_n)^{U(n)}$. We put $N=n(n+1)$. Let $\left\{
\xi_{\alpha}\,|\ 1\leq \alpha \leq N\, \right\}$ be a basis of
$\fp$. If $P\in \textrm{Pol}(\fp)^K$, then
\begin{equation}
\Big(\Phi (P)f\Big)(gK)=\left[ P\left( {{\partial}\over {\partial
t_{\al}}}\right)f\left(g\,\text{exp}\, \left(\sum_{\al=1}^N
t_{\al}\xi_{\al}\right) K\right)\right]_{(t_{\al})=0},
\end{equation} where $f\in C^{\infty}({\mathbb H}_n)$. We refer to \cite{He1,He2} for more detail. In
general, it is hard to express $\Phi(P)$ explicitly for a
polynomial $P\in \textrm{Pol}(\fp)^K$.

\vskip 0.3cm According to the work of Harish-Chandra
\cite{HC1,HC2}, the algebra ${\mathbb D}(\BH_n)$ is generated by
$n$ algebraically independent generators and is isomorphic to the
commutative ring $\BC [x_1,\cdots,x_n]$ with $n$ indeterminates.
We note that $n$ is the real rank of $Sp(n,\BR)$. Let ${\mathfrak s}{\mathfrak p}(n,\BC)$ be the
complexification of ${\mathfrak s}{\mathfrak p}(n,\BR)$. It is known that $\BD(\BH_n)$ is
isomorphic to the center of the universal enveloping algebra of
${\mathfrak s}{\mathfrak p}(n,\BC)$\,(cf.\,\cite{Sh}).

\vskip 0.3cm Using a classical invariant theory
(cf.\,\cite{Ho,Wy}), we can show that $\textrm{Pol}({\mathbb T}_n)^{U(n)}$
is generated by the following algebraically independent
polynomials
\begin{equation}
q_j (Z)=\,\sigma\Big( \big(Z {\overline
Z}\big)^j\,\Big),\quad j=1,2,\cdots,n.
\end{equation}

For each $j$ with $1\leq j\leq n,$ the image $\Phi(q_j)$ of $q_j$
is an invariant differential operator on $\BH_n$ of degree $2j$.
The algebra ${\mathbb D}(\BH_n)$ is generated by $n$ algebraically
independent generators $\Phi(q_1),\Phi(q_2),\cdots,\Phi(q_n).$ In
particular,
\begin{equation}
\Phi(q_1)=\,c_1\, \sigma \left( Y\,\,{}^{{}^{{}^{{}^\text{\scriptsize $t$}}}}\!\!\!\!\left(Y\POB\right)\PO\right)\quad  \textrm{for\ some
constant}\ c_1.
\end{equation}

\noindent We observe that if we take $Z=X+i\,Y$ with real $X,Y$,
then $q_1(Z)=q_1(X,Y)=\,\sigma\big( X^2 +Y^2\big)$ and
\begin{equation*}
q_2(Z)=q_2(X,Y)=\, \sigma \Big(
\big(X^2+Y^2\big)^2+\,2\,X\big(XY-YX\big)Y\,\Big).
\end{equation*}

\vskip 0.3cm We propose the following problem.

\vskip 0.2cm \noindent $ \textbf{Problem.}$ Express the images
$\Phi(q_j)$ explicitly for $j=2,3,\cdots,n.$

\vskip 0.3cm We hope that the images $\Phi(q_j)$ for
$j=2,3,\cdots,n$ are expressed in the form of the $\textit{trace}$
as $\Phi(q_1)$.

\vskip 0.3cm\noindent $ \textbf{Example 12.1.}$ We consider the
case $n=1.$ The algebra $ \textrm{Pol}({\mathbb T}_1)^{U(1)}$ is generated
by the polynomial
\begin{equation*}
q(z)=z\,{\overline z},\quad z\in \BC.
\end{equation*}

Using Formula (12.26), we get

\begin{equation*}
\Phi (q)=\,4\,y^2 \left( { {\partial^2}\over {\partial x^2} }+{
{\partial^2}\over {\partial y^2} }\,\right).
\end{equation*}

\noindent Therefore $\BD (\BH_1)=\BC\big[ \Phi(q)\big].$

\vskip 0.3cm\noindent $ \textbf{Example 12.2.}$ We consider the
case $n=2.$ The algebra $ \textrm{Pol}({\mathbb T}_2)^{U(2)}$ is generated
by the polynomial
\begin{equation*}
q_1(Z)=\,\s \big(Z\,{\overline Z}\,\big),\quad q_2(Z)=\,\s \Big(
\big(Z\,{\overline Z}\big)^2\Big), \quad Z\in T_2.
\end{equation*}

Using Formula (12.26), we may express $\Phi(q_1)$ and $\Phi(q_2)$
explicitly. $\Phi (q_1)$ is expressed by Formula (12.28). The
computation of $\Phi(q_2)$ might be quite tedious. We leave the
detail to the reader. In this case, $\Phi (q_2)$ was essentially
computed in \cite{BC}, Proposition 6. Therefore $\BD
(\BH_2)=\BC\big[ \Phi(q_1), \Phi(q_2)\big].$ The authors of
\cite{BC} computed the center of $U({\mathfrak s}{\mathfrak p}(2,\BC)).$

\vskip 0.3cm
\newcommand\CP{{\mathcal P}_n}
\newcommand\CCF{{\mathcal F}_n}

Now we describe the Siegel's fundamental domain for $\Gamma_n\backslash \BH_n$.
We let
$$\CP=\left\{\, Y\in\BR^{(n,n)}\,|\ Y=\,^tY>0\ \right\}$$
be an open cone in $\BR^N$ with $N=n(n+1)/2.$ The general linear
group $GL(n,\BR)$ acts on $\CP$ transitively by
\begin{equation}
g\circ Y:=gY\,^tg,\qquad g\in GL(n,\BR),\ Y\in \CP.\end{equation}
Thus $\CP$ is a symmetric space diffeomorphic to $GL(n,\BR)/O(n).$
\vskip 0.10cm
\newcommand\Mg{{\mathcal M}_n}
\newcommand\Rg{{\mathcal R}_n}
The fundamental domain $\Rg$ for $GL(n,\BZ)\ba \CP$ which was
found by H. Minkowski\,\cite{Min} is defined as a subset of $\CP$
consisting of $Y=(y_{ij})\in \CP$ satisfying the following
conditions (M.1)--(M.2)\ (cf. \cite{I} p.\,191 or \cite{Ma}
p.\,123): \vskip 0.1cm (M.1)\ \ \ $aY\,^ta\geq y_{kk}$\ \ for
every $a=(a_i)\in\BZ^n$ in which $a_k,\cdots,a_n$ are relatively
prime for $k=1,2,\cdots,n$.
\vskip 0.1cm (M.2)\ \ \ \
$y_{k,k+1}\geq 0$ \ for $k=1,\cdots,n-1.$
\vskip 0.1cm We say that
a point of $\Rg$ is $\textsf{Minkowski reduced}$ or simply {\it M}-$\textsf{reduced}$. 
$\Rg$ has the following properties (R1)--(R4): \vskip
0.1cm (R1) \ For any $Y\in\CP,$ there exist a matrix $A\in
GL(n,\BZ)$ and $R\in\Rg$ such that $Y=R[A]$\,(cf. \cite{I}
p.\,191 or \cite{Ma} p.\,139). That is,
$$GL(n,\BZ)\circ \Rg=\CP.$$
\indent (R2)\ \ $\Rg$ is a convex cone through the origin bounded
by a finite number of hyperplanes. $\Rg$ is closed in $\CP$
(cf.\,\cite{Ma} p.\,139).

\vskip 0.1cm (R3) If $Y$ and $Y[A]$ lie in $\Rg$ for $A\in
GL(n,\BZ)$ with $A\neq \pm I_n,$ then $Y$ lies on the boundary
$\partial \Rg$ of $\Rg$. Moreover $\Rg\cap (\Rg [A])\neq
\emptyset$ for only finitely many $A\in GL(n,\BZ)$ (cf.\,\cite{Ma}
p.\,139). \vskip 0.1cm (R4) If $Y=(y_{ij})$ is an element of
$\Rg$, then
$$y_{11}\leq y_{22}\leq \cdots \leq y_{nn}\quad \text{and}\quad
|y_{ij}|<{\frac 12}y_{ii}\quad \text{for}\ 1\leq i< j\leq n.$$
\indent We refer to \cite{I} p.\,192 or \cite{Ma} pp.\,123-124.
\vskip 0.1cm\noindent {\bf Remark.} Grenier\,\cite{Gr} found
another fundamental domain for $GL(n,\BZ)\ba \CP.$

\vskip 0.2cm For $Y=(y_{ij})\in \CP,$ we put
$$dY=(dy_{ij})\qquad\text{and}\qquad \PY\,=\,\left(\,
{ {1+\delta_{ij}}\over 2}\, { {\partial}\over {\partial y_{ij} } }
\,\right).$$ Then we can see easily that
\begin{equation}
ds^2=\s ( (Y^{-1}dY)^2)\end{equation} is a $GL(n,\BR)$-invariant
Riemannian metric on $\CP$ and its Laplacian is given by
$$\Delta=\s \left( \left( Y\PY\right)^2\right).$$
We also can see that
$$d\mu_n(Y)=(\det Y)^{-{ {n+1}\over2 } }\prod_{i\leq j}dy_{ij}$$
is a $GL(n,\BR)$-invariant volume element on $\CP$. The metric
$ds^2$ on $\CP$ induces the metric $ds_{\mathcal R}^2$ on $\Rg.$
Minkowski \cite{Min} calculated the volume of $\Rg$ for the volume
element $[dY]:=\prod_{i\leq j}dy_{ij}$ explicitly. Later Siegel
computed the volume of $\Rg$ for the volume element $[dY] $ by a
simple analytic method and generalized this case to the case of
any algebraic number field.
\vskip 0.1cm Siegel\,\cite{Si1}
determined a fundamental domain $\CCF$ for $\Gamma_n\ba \BH_n.$ We say
that $\Om=X+i\,Y\in \BH_n$ with $X,\,Y$ real is $\textsf{Siegel reduced}$
or {\it S}-$\textsf{reduced}$ if it has the following three properties:
\vskip 0.1cm (S.1)\ \ \ $\det ({\rm{Im}}\,(\g\cdot\Om))\leq \det
(\text{Im}\,(\Om))\qquad\text{for\ all}\ \g\in\Gamma_n$;
\vskip 0.1cm
(S.2)\ \ $Y={\rm{Im}}\,\Om$ is M-reduced, that is, $Y\in \Rg\,;$
\vskip 0.1cm (S.3) \ \ $|x_{ij}|\leq {\frac 12}\quad \text{for}\
1\leq i,j\leq n,\ \text{where}\ X=(x_{ij}).$
\vskip 0.1cm $\CCF$
is defined as the set of all Siegel reduced points in $\BH_n.$
Using the highest point method, Siegel proved the following
(F1)--(F3)\,(cf. \cite{I} pp.\,194-197 or \cite{Ma} p.\,169):
\vskip 0.1cm (F1)\ \ \ $\Gamma_n\cdot \CCF=\BH_n,$ i.e.,
$\BH_n=\cup_{\gamma\in\Gamma_n}\gamma\cdot \CCF.$
\vskip 0.1cm (F2)\ \ $\CCF$
is closed in $\BH_n.$
\vskip 0.1cm (F3)\ \ $\CCF$ is connected and
the boundary of $\CCF$ consists of a finite number of hyperplanes.
\vskip 0.21cm The metric $ds^2$ given by (12.1) induces a metric
$ds_{\mathcal F}^2$ on $\CCF.$ \vskip 0.1cm Siegel\,\cite{Si1}
computed the volume of $\CCF$
\begin{equation}
\text{vol}\,(\CCF)=2\prod_{k=1}^n\pi^{-k}\,\G
(k)\,\zeta(2k),\end{equation} where $\Gamma (s)$ denotes the Gamma
function and $\zeta (s)$ denotes the Riemann zeta function. For
instance,
$$\text{vol}\,({\mathcal F}_1)={{\pi}\over 3},\quad \text{vol}\,({\mathcal F}_2)={{\pi^3}\over {270}},
\quad \text{vol}\,({\mathcal F}_3)={{\pi^6}\over {127575}},\quad
\text{vol}\,({\mathcal F}_4)={{\pi^{10}}\over {200930625}}.$$

\end{section}

\newpage

\begin{section}{{\large\bf The Weil Representation}}
\setcounter{equation}{0}
\vskip 0.2cm We recall that for a real symmetric positive definite matrix $c\in\BR^{(m,m)}$, the Schr{\"o}dinger representation $U_c$ of
$H_\BR^{(n,m)}$ is defined by Formula (5.8)\,(cf.\,(6.45)). For convenience, we rewrite Formula (5.8)
$$(5.8) \hskip 2.5cm\big( U_c(g_0)f\big)(x)=e^{2\pi i\sigma(c(\kappa_0+\mu_0\,{}^t\!\lambda_0+2\,x\,{}^t\!\mu_0))} f (x+\lambda_0)\hskip 4cm$$
for $g_0=(\lambda_0,\mu_0,\kappa_0)\in H_\BR^{(n,m)},\ x\in\BR^{(m,n)}$ and $f\in L^2\big(\BR^{(m,n)}\big).$
\vskip 0.2cm
We let
$$G^J=Sp(n,\BR)\ltimes H_{\BR}^{(n,m)}\quad \ ( \textrm{semi-direct product})$$
be the Jacobi group endowed with the following multiplication law
$$\Big(M,(\lambda,\mu,\kappa)\Big)\Big(M',(\lambda',\mu',\kappa')\Big) =\,
\Big(MM',(\widetilde{\lambda}+\lambda',\widetilde{\mu}+ \mu',
\kappa+\kappa'+\widetilde{\lambda}\,^t\!\mu'
-\widetilde{\mu}\,^t\!\lambda')\Big)$$ with $M,M'\in Sp(n,\BR),
(\lambda,\mu,\kappa),\,(\lambda',\mu',\kappa') \in
H_{\BR}^{(n,m)}$ and
$(\widetilde{\lambda},\widetilde{\mu})=(\lambda,\mu)M'$.
Then $Sp(n,\BR)$ acts on $H_\BR^{(n,m)}$ by conjugation inside $G^J$
\begin{equation}
M\star (\lambda,\mu,\kappa)=M (\lambda,\mu,\kappa) M^{-1}=(\lambda_*,\mu_*,\kappa),
\end{equation}
where $M\in Sp(n,\BR),\ (\lambda,\mu,\kappa)\in H_\BR^{(n,m)}$ and $(\lambda_*,\mu_*)=(\lambda,\mu)M^{-1}.$

\vskip 0.2cm
We fix an element $M\in Sp(n,\BR).$ We consider the mapping $U_c^M$ of $H_\BR^{(n,m)}$ into ${\mathrm{Aut}}\big( L^2\big(\BR^{(m,n)}\big)  \big)$ defined by

\begin{equation}
U_c^M (g)=U_c (M\star g)=U_c (M g M^{-1}),\qquad g\in H_\BR^{(n,m)}.
\end{equation}

\begin{lemma}
$U_c^M$ is an irreducible representation of $H_\BR^{(n,m)}$ on $L^2\big(\BR^{(m,n)}\big)$ such that
$$U_c^M\big( (0,0,\kappa)\big)=U_c \big( (0,0,\kappa)\big)\qquad {\mathrm {for\ all}}\ \kappa=\,{}^t\kappa\in\BR^{(m,m)}.$$
Thus $U_c^M$ is unitarily equivalent to $U_c.$
\end{lemma}

\vskip 0.2cm \noindent{\it Proof.} If $g_1,g_2\in H_\BR^{(n,m)},$ then
\begin{eqnarray*}
U_c^M(g_1g_2)&=& U_c(M\star (g_1g_2))=U_c\big( M(g_1g_2)M^{-1}\big)\\
&=&U_c\big( (Mg_1M^{-1})(Mg_2M^{-1})\big)\\
&=& U_c\big(Mg_1M^{-1}\big)\,U_c\big(Mg_2M^{-1}\big)\\
&=& U_c^M(g_1)\,U_c^M(g_2).
\end{eqnarray*}
Thus $U_c^M$ is a representation of $H_\BR^{(n,m)}$. The irreducibility of $U_c^M$ follows immediately from that of $U_c.$
It is easily seen that
$$U_c^M\big((0,0,\kappa)\big)=U_c\big(M\star (0,0,\kappa)\big)=U_c\big((0,0,\kappa)\big)\qquad {\mathrm {for\ all}}\ \kappa=\,{}^t\kappa\in\BR^{(m,m)}.$$
Therefore it follows from Stone-von Neumann Theorem that $U_c^M$ is unitarily equivalent to $U_c.$ $\hfill \square$

\vskip 0.2cm Since $U_c^M$ is unitarily equivalent to $U_c$, there exists an unitary operator $R_c(M):L^2\big(\BR^{(m,n)}\big)$
$\lrt
L^2\big(\BR^{(m,n)}\big)$ such that $U_c^M(g)\,R_c(M)=R_c(M)\,U_c(g)$ for all $g\in H_\BR^{(n,m)}.$ For convenience, we take
$R_c(I_{2n})=I_c,$ where $I_c$ is the identity operator on $L^2\big(\BR^{(m,n)}\big)$. We observe that $R_c(M)$ is determined uniquely up to a scalar of modulus one. For any two elements $M_1,M_2$ of $Sp(n,\BR)$, the unitary operator $R_c(M_2)^{-1}R_c(M_1)^{-1}R_c(M_1M_2)$
commutes with $U_c$. Indeed, for any element $g\in H_\BR^{(n,m)}$, we have

\begin{eqnarray*}
& & U_c(g)\,R_c(M_2)^{-1}\,R_c(M_1)^{-1}\,R_c(M_1M_2)\\
&=& R_c(M_2)^{-1}U_c^{M_2}(g)\,R_c(M_1)^{-1}\,R_c(M_1M_2)\\
&=& R_c(M_2)^{-1}U_c(M_2\star g)\,R_c(M_1)^{-1}\,R_c(M_1M_2)\\
&=& R_c(M_2)^{-1}\,R_c(M_1)^{-1}\,U_c^{M_1}(M_2\star g)\,R_c(M_1M_2)\\
&=& R_c(M_2)^{-1}\,R_c(M_1)^{-1}\,U_c\big( (M_1M_2)\star g \big)\,R_c(M_1M_2)\\
&=& R_c(M_2)^{-1}\,R_c(M_1)^{-1}\,U_c^{M_1M_2}(g)\,R_c(M_1M_2)\\
&=& R_c(M_2)^{-1}\,R_c(M_1)^{-1}\,R_c(M_1M_2)\,U_c(g).\\
\end{eqnarray*}
\noindent According to Schur's lemma, we obtain a map $\alpha_c:Sp(n,\BR)\times Sp(n,\BR)\lrt \BC^*_1$ satisfying the condition
\begin{equation}
R_c(M_1M_2)=\alpha_c (M_1,M_2)\,R_c(M_1)\,R_c(M_2)
\end{equation}
${\mathrm {for\ all}}\ M_1,M_2\in Sp(n,\BR).$
Thus $R_c$ is a projective representation of $Sp(n,\BR)$ with its multiplier $\alpha_c.$

\begin{lemma}
The map $\alpha_c$ satisfies the cocycle condition
\begin{equation}
\alpha_c(M_1M_2,M_3)\,\alpha_c(M_1,M_2)=\alpha_c(M_1,M_2M_3)\,\alpha_c(M_2,M_3)
\end{equation}
for all $M_1,M_2,M_3\in Sp(n,\BR)$.
\end{lemma}
\noindent {\it Proof.} Let $M_1,M_2,M_3\in Sp(n,\BR).$ Then according to Formula (13.3),
\begin{eqnarray*}
R_c\big( (M_1M_2)M_3\big)&=&\alpha_c (M_1M_2,M_3)\,R_c(M_1M_2)\,R_c(M_3)\\
&=&\alpha_c (M_1M_2,M_3)\,\alpha_c (M_1,M_2)\,R_c(M_1)\,R_c(M_2)\,R_c(M_3)
\end{eqnarray*}
and
\begin{eqnarray*}
R_c\big( M_1(M_2M_3)\big)&=&\alpha_c (M_1,M_2M_3)\,R_c(M_1)\,R_c(M_2M_3)\\
&=&\alpha_c (M_1,M_2M_3)\,\alpha_c (M_2,M_3)\,R_c(M_1)\,R_c(M_2)\,R_c(M_3)
\end{eqnarray*}
\noindent Hence we obtain the cocycle condition (13.4).  $\hfill\square$

\vskip 0.2cm For $M=\begin{pmatrix} A & B \\ C & D \end{pmatrix}\in Sp(n,\BR)$ and $\Omega\in \BH_n$, we put
\begin{equation}
J(M,\Omega)\,=\,\det (C\Omega+D)
\end{equation}
and
\begin{equation}
J^*(M,\Om)\,=\,{ {J(M,\Om)^{1/2}}\over {|J(M,\Om)^{1/2}|} }.
\end{equation}
In fact, if $M_1,M_2\in Sp(n,\BR),$ the cocycle $\alpha_c(M_1,M_2)$ is given by
\begin{equation}
\alpha_c(M_1,M_2)\,=\,{ {J^*(M_1,iI_n)\,J^*(M_2,iI_n)}\over {J^*(M_1M_2,iI_n)} }.
\end{equation}
\vskip 0.2cm The cocycle $\alpha_c$ yields the central extension $Sp(n,\BR)_*$ of $Sp(n,\BR)$ by $\BC_1^*$. The group
$Sp(n,\BR)_*$ is the set $Sp(n,\BR)\times \BC_1^*$ with the following group multiplication

\begin{equation}
(M_1,t_1)\cdot (M_2,t_2)=\big(M_1M_2, t_1t_2\,\alpha_c(M_1,M_2)^{-1}\big)
\end{equation}
for all $M_1,M_2\in Sp(n,\BR)$ and $t_1,t_2\in \BC^*_1.$ We see that the map ${\widetilde R}_c:Sp(n,\BR)_*\lrt
{\mathrm{Aut}}\big(L^2\big(\BR^{(m,n)}\big) \big)$ defined by
\begin{equation}
{\widetilde R}_c(M,t)=t\,R_c(M),\qquad M\in Sp(n,\BR),\ t\in \BC_1^*
\end{equation}
is a true representation of $Sp(n,\BR)_*.$
We define the function $s_c:Sp(n,\BR)\lrt \BC_1^*$ by
\begin{equation}
s_c(M)\,=\,|J(M,iI_n)|\,J(M,iI_n)^{-1},\qquad M\in Sp(n,\BR).
\end{equation}
The following subset
$$Mp(n,\BR)=\left\{ (M,t)\in Sp(n,\BR)_*\,|\ t^2=s_c(M)^{-1}\ \right\}$$
is a subgroup of $Sp(n,\BR)_*$ that is called the $\textsf{metaplectic group}$.
We can show that $Mp(n,\BR)$ is a two-fold covering group of $Sp(n,\BR).$
The restriction $\omega_c$ of ${\widetilde R}_c$ to $Mp(n,\BR)$ is a true representation of $Mp(n,\BR)$ 
which is called the $\textsf{Weil representation}$ of $Sp(n,\BR)$

\vskip 0.2cm Now we describe the action of $\omega_c$ explicitly. It is known that $Sp(n,\BR)$ is generated by
the following elements
\begin{eqnarray*}
t_b&=&\begin{pmatrix} I_n & b \\ 0 & I_n \end{pmatrix}\ {\mathrm {with}}\ b=\,{}^tb\in\BR^{(n,n)},\\
d_{a}&=& \begin{pmatrix} {}^ta & 0 \\ 0 & a^{-1} \end{pmatrix}\ {\mathrm {with}}\ a\in GL(n,\BR),\\
\s_n&=& \begin{pmatrix} 0 & -I_n \\ I_n & \ 0 \end{pmatrix}.
\end{eqnarray*}

\begin{theorem}
The actions of $R_c$ on the generators $t_b,\ d_a$ and $\sigma_n$ of $Sp(n,\BR)$ are given by
\begin{eqnarray}
\ \ \quad \quad \left( R_c(t_b)f\right)(x)&=& e^{2\,\pi \,i\,\sigma(c\,x\,b\,{}^tx) }\,f(x),\\
\ \ \quad \ \quad \left( R_c(d_a)f\right)(x)&=& (\det a)^{\frac m2}\,f(x\,{}^ta),\\
\ \ \quad \ \ \quad\left( R_c(\sigma_n)f\right)(x)&=& \left( {2\over i}\right)^{{mn}\over 2} (\det c)^{\frac n2}\,\int_{\BR^{(m,n)}}f(y)
e^{-4\,\pi\, i\,\,\sigma(c\,y\,{}^tx)}\,dy,
\end{eqnarray}
where $f\in L^2\big(\BR^{(m,n)}\big)$ and $x\in\BR^{(m,n)}.$
\end{theorem}
\noindent {\it Proof.} Let $g=(\lambda,\mu,\kappa)\in H_\BR^{(n,m)},\ x\in\BR^{(m,n)}$ and $f\in L^2\big(\BR^{(m,n)}\big)$. For for
each $t_b\in Sp(n,\BR)$ with $b=\,{}^tb\in\BR^{(n,n)},$ we put
\begin{equation*}
\big( T_c(t_b)f\big)(x)= e^{2\,\pi \,i\,\sigma(c\,x\,b\,{}^tx) }\,f(x)\qquad \textrm{for\ all}\ f\in L^2\big(\BR^{(m,n)}\big).
\end{equation*}
Then

\begin{eqnarray*}
\big( T_c(t_b)U_c(g)f\big)(x) &=& e^{2\,\pi \,i\,\sigma(c\,x\,b\,{}^tx) }\,\big( U_c(g)f\big)(x)\\
&=& e^{2\,\pi \,i\,\sigma(c\,x\,b\,{}^tx) }\cdot e^{2\,\pi\, i \,\sigma( c(\kappa +\mu\,{}^t\lambda+ 2x\,{}^t\mu))}
f(x+\lambda)\\
&=& e^{2\,\pi\, i \,\sigma( c(\kappa +\mu\,{}^t\lambda+ 2x\,{}^t\mu+ x\,b\,{}^tx))} f(x+\lambda).
\end{eqnarray*}

Since $t_b\star (\lambda,\mu,\kappa)=t_b (\lambda,\mu,\kappa)\,t_b^{-1}=(\lambda,-\lambda\, b+\mu,\kappa)$, we obtain

\begin{eqnarray*}
& & \big( U_c^{t_b}(g)\,T_c(t_b)f\big)(x) \\
&=& \big( U_c(t_b\star g)\,T_c(t_b)f\big)(x)\\
&=& \big( U_c(\lambda,-\lambda\, b+\mu,\kappa)\,T_c(t_b)f\big)(x)\\
&=&  e^{2\,\pi\, i \,\sigma( c(\kappa +(-\lambda \,b+\mu)\,{}^t\lambda+ 2x\,{}^t(-\lambda b+\mu)))}
\big( T_c(t_b)f\big)(x+\lambda)\\
&=& e^{2\,\pi\, i \,\sigma( c(\kappa +(-\lambda \,b+\mu)\,{}^t\lambda+ 2x\,{}^t(-\lambda b+\mu)))}
\cdot e^{2\,\pi\, i \,\sigma( c(x+\lambda)b\,{}^t(x+\lambda))} f(x+\lambda).\\
&=& e^{2\,\pi\, i \,\sigma( c(\kappa +\mu\,{}^t\lambda+ 2x\,{}^t\mu+ x\,b\,{}^tx))} f(x+\lambda).
\end{eqnarray*}

\noindent Therefore
\begin{equation*}
T_c(t_b)\,U_c(g)f=U_c^{t_b}(g)\,T_c(t_b)f
\end{equation*}
for all $b={}^tb\in\BR^{(n,n)},\ g\in H_\BR^{(n,m)}\ \textrm{and}\
f\in L^2\big(\BR^{(m,n)}\big).$ \\
Since $T_c(t_0)=T_c(I_{2n})=I_c=R_c(I_{2n}),$ we see that
$$R_c(t_b)=T_c(t_b)\qquad \textrm{for\ all}\ b={}^tb\in\BR^{(n,n)}.$$
We recall that $I_c$ is the identity operator on $L^2\big(\BR^{(m,n)}\big)$.

\vskip 0.2cm
 On the other hand, for each $f\in GL(n,\BR)$ and $f\in L^2\big(\BR^{(m,n)}\big)$, we put
$$\big( A_c(d_a)f\big)(x)=(\det a)^{\frac m2}\,f(x\,{}^ta).$$
Then we have
\begin{eqnarray*}
& & \big( A_c(d_a)U_c(g)f\big)(x)\\
 &=&\big( \det a\big)^{\frac m2}\,\big( U_c(g)f\big)(x\,{}^ta)\\
&=& \big( \det a\big)^{\frac m2}\,
e^{2\,\pi\, i \,\sigma( c(\kappa +\mu\,{}^t\lambda+ 2\,x\,{}^ta\,{}^t\mu))}
f(x\,{}^ta+\lambda).
\end{eqnarray*}

\noindent Since $d_a\star (\lambda,\mu,\kappa)= d_a (\lambda,\mu,\kappa)\,d_a^{-1}=(\lambda\,{}^ta^{-1},\mu \,a,\kappa),$
\begin{eqnarray*}
& & \big( U^{d_a}_c(g)\,A_c(d_a)f\big)(x)\\
&=& \big( U_c(d_a\star g)\,A_c(d_a)f\big)(x) \\
&=&  \big( U_c(\lambda\,{}^ta^{-1},\mu \,a,\kappa)\,A_c(d_a)f\big)(x) \\
&=& e^{2\,\pi\, i \,\sigma( c(\kappa +(\mu\,a)\,{}^t(\lambda\,{}^ta^{-1})+ 2\,x\,{}^t(\mu\,a)))}\,
\big( A_c(d_a)f\big)\big(x+\lambda\,{}^ta^{-1}\big)\\
&=& \big( \det a\big)^{\frac m2}\,e^{2\,\pi\, i \,\sigma( c(\kappa +\mu\,{}^t\lambda+ 2\,x\,{}^ta\,{}^t\mu))}\,
f(x\,{}^ta+\lambda).
\end{eqnarray*}
Thus
$$A_c(d_a)U_c(g)f=U^{d_a}_c(g)\,A_c(d_a)f$$
for all $a\in GL(n,\BR),\ g\in H_\BR^{(n,m)}$ and $f\in L^2\big(\BR^{(m,n)}\big).$\\
Since $A_c(d_{I_n})=I_c=R_c(d_{I_n}),$ we obtain $R_c(d_a)=A_c(d_a)$ for all $a\in GL(n,\BR).$

\vskip 0.2cm Finally we put
\begin{equation*}
\big( B_c(\sigma_n)f\big)(x)=\left( {2\over i}\right)^{{mn}\over 2} (\det c)^{\frac n2}\,\int_{\BR^{(m,n)}}f(y)\,
e^{-4\,\pi\, i\,\,\sigma(c\,y\,{}^tx)}\,dy
\end{equation*}
for all $f\in L^2\big(\BR^{(m,n)}\big).$

\begin{eqnarray*}
& & \big( B_c(\sigma_n)U_c(g)f\big)(x)\\
&=& \left( {2\over i}\right)^{{mn}\over 2} (\det c)^{\frac n2}\,\int_{\BR^{(m,n)}} \big( U_c(g)f\big)(y)\,
e^{-4\,\pi\, i\,\,\sigma(c\,y\,{}^tx)}\,dy\\
&=& \left( {2\over i}\right)^{{mn}\over 2} (\det c)^{\frac n2}\,\int_{\BR^{(m,n)}} e^{2\,\pi\,i\,\sigma( c
(\kappa+\mu\,{}^t\lambda+2\,y\,{}^t\mu))}\cdot
e^{-4\,\pi\, i\,\,\sigma(c\,y\,{}^tx)}\, f(y+\lambda)\,dy \\
&=& \left( {2\over i}\right)^{{mn}\over 2} (\det c)^{\frac n2}\,e^{2\,\pi\,i\,\sigma( c
(\kappa+\mu\,{}^t\lambda))}\,\int_{\BR^{(m,n)}}
e^{4\,\pi\, i\,\,\sigma(c\,y\,{}^t(\mu-x))}\, f(y+\lambda)\,dy \\
&=& \left( {2\over i}\right)^{{mn}\over 2} (\det c)^{\frac n2}\,e^{2\,\pi\,i\,\sigma( c
(\kappa+\mu\,{}^t\lambda))}\,\int_{\BR^{(m,n)}}
e^{4\,\pi\, i\,\,\sigma(c\,({\widetilde y}-\lambda)\,{}^t(\mu-x))}\, f({\widetilde y})\,d{\widetilde y} \\
&=&  \left( {2\over i}\right)^{{mn}\over 2} (\det c)^{\frac n2}\,e^{2\,\pi\,i\,\sigma( c
(\kappa+\mu\,{}^t\lambda))}\cdot e^{-4\,\pi\, i\,\sigma(c\,\lambda\,{}^t(\mu-x))}\\
& & \times \int_{\BR^{(m,n)}} f(y)\,
e^{4\,\pi\, i\,\,\sigma(c\,y\,{}^t(\mu-x))}\,dy \\
&=& \left( {2\over i}\right)^{{mn}\over 2} (\det c)^{\frac n2}\,e^{2\,\pi\,i\,\sigma( c
(\kappa-\lambda\,{}^t\mu+\,2\,x\,{}^t\lambda))}\,
\int_{\BR^{(m,n)}} f(y)\,
e^{-4\,\pi\, i\,\,\sigma(c\,y\,{}^t(x-\mu))}\,dy.
\end{eqnarray*}
Since $\sigma_n\star (\lambda,\mu,\kappa)=\sigma_n (\lambda,\mu,\kappa)\,\s_n=(-\mu,\lambda ,\kappa),$ we obtain

\begin{eqnarray*}
& & \big( U_c^{\sigma_n}(g)\,B_c(\sigma_n)f\big)(x)\\
&=& \big( U_c(\sigma_n\star g)\,B_c(\sigma_n)f\big)(x)\\
&=& \big( U_c(-\mu,\lambda,\kappa)\,B_c(\sigma_n)f\big)(x)\\
&=& e^{2\,\pi\,i\,\sigma( c
(\kappa-\lambda\,{}^t\mu+\,2\,x\,{}^t\lambda))}\,\big( B_c(\sigma_n)f\big)(x-\mu)\\
&=& \left( {2\over i}\right)^{{mn}\over 2} (\det c)^{\frac n2}\,e^{2\,\pi\,i\,\sigma( c
(\kappa-\lambda\,{}^t\mu+\,2\,x\,{}^t\lambda))}\,
\int_{\BR^{(m,n)}} f(y)\,
e^{-4\,\pi\, i\,\,\sigma(c\,y\,{}^t(x-\mu))}\,dy.
\end{eqnarray*}

Therefore
\begin{equation*}
B_c(\sigma_n)U_c(g)f=U_c^{\sigma_n}(g)\,B_c(\sigma_n)f \qquad \textrm{for\ all}\ f\in  L^2\big(\BR^{(m,n)}\big).
\end{equation*}
We note that we can take
$$R_c(\sigma_n)=B_c(\sigma_n).$$
Hence we complete the proof of the above theorem. $\hfill \square$

\vskip 0.5cm

\begin{corollary}
We have the following
\vskip 0.2cm\noindent
(a)\ \ \ $\omega_c((t_b,1))=R_c (t_b)\quad {\rm and}\quad \omega_c((t_b,-1))=-R_c (t_b).$
\vskip 0.2cm\noindent
(b)\ If $\det a > 0,$ then $(d_a,\pm 1)\in Mp(n,\BR)$ and hence we have
$$ \omega_c ((d_a,1))= R_c(d_a)\quad {\rm and}\quad \omega_c ((d_a,-1))= -R_c(d_a).$$
(c)\ If $\det a < 0,$ then $(d_a,\pm i)\in Mp(n,\BR)$ and hence we have
$$ \omega_c ((d_a,i))= i\, R_c(d_a)\quad {\rm and}\quad \omega_c ((d_a,-i))= -i\, R_c(d_a).$$
(d)\ \ \ $\omega_c\big( (\sigma_n,i^{n/2})\big)= i^{n/2} R_c(\sigma_n) \quad {\rm and}\quad
\omega_c\big( (\sigma_n,-i^{n/2})\big)= -i^{n/2} R_c(\sigma_n).$
\end{corollary}
\vskip 0.2cm
\noindent
{\it Proof.} The proof follows immediately from the definition of $Mp(n,\BR)$ and Theorem 13.3.
$\hfill \square$

\vskip 0.2cm Now we review some properties of $\omega_c$. The Weil representation $\omega_c$ is not an irreducible
representation of $Mp(n,\BR).$ In \cite{KV}, Kashiwara and Vergne found an explicit decomposition of $\omega_c$ into
irreducibles. First we observe that the orthogonal group $O(m)$ acts on $L^2\big(\BR^{(m,n)}\big)$ by
$$(\alpha\cdot f)(x)\,=\,f(\alpha^{-1}x),\qquad \alpha\in O(m),\ x\in\BR^{(m,n)},\ f\in L^2\big(\BR^{(m,n)}\big).$$
This action commutes with $\omega_c$. For each irreducible representation $(\sigma,V_\sigma)$ of $O(m),$ we let
$L^2\big(\BR^{(m,n)};\sigma\big)$ be the space of all $V_\sigma$-valued square integrable functions $f:\BR^{(m,n)}\lrt
V_\sigma$ satisfying the condition
$$f(\alpha^{-1}x)=\sigma(\alpha^{-1})f(x)\qquad \textrm{for all}\ \alpha\in O(m),\ x\in\BR^{(m,n)}.$$
We let $\omega_c(\sigma)$ be the representation of $Mp(n,\BR)$ on $L^2\big(\BR^{(m,n)};\sigma\big)$ defined by the
formulas in Corollary 13.4. We denote by $\widehat{O(m)}$ the unitary dual of $O(m).$ In other words,
$\widehat{O(m)}$ is the set of all equivalence classes of irreducible representations of $O(m).$ Let
$$\Sigma_m:=\left\{\, \sigma\in \widehat{O(m)}\ \big|\ L^2\big(\BR^{(m,n)};\sigma\big)\,\neq \,0\ \right\}.$$
Kashiwara and Vergne proved that for any $\sigma\in \Sigma_m,$ the representation $\omega_c(\sigma)$ is an irreducible
unitary representation of $Mp(n,\BR)$ on $L^2\big(\BR^{(m,n)};\sigma\big)$ and that $\omega_c$ is decomposed into
irreducibles as follows\,:
\begin{equation*}
\omega_c\,=\,\bigoplus_{\sigma\in\Sigma_m} ( \dim V_\sigma)\,\omega_c(\sigma).
\end{equation*}
We realize $\omega_c(\sigma)$ in the space of vector valued holomorphic functions on $\BH_n.$ We note that $\BH_n$ is
biholomorphic to the Hermitian complex manifold $Sp(n,\BR)/K$ with $K:=U(n)$ via the map
$$Sp(n,\BR)/K\,\lrt\,\BH_n,\qquad gK \mapsto g\cdot (iI_n),\ \ M\in Sp(n,\BR).$$
Let $\widehat{K}$ be the unitary dual of $K$. For any $(\tau,V_\tau)\in \widehat{K},$ we let ${\mathcal O}(\BH_n,V_\tau)$
be the space of $V_\tau$-valued holomorphic functions on $\BH_n$. Let $T_\tau$ be the representation of $Mp(n,\BR)$ on
${\mathcal O}(\BH_n,V_\tau)$ defined by
\begin{equation}
\big( T_\tau(M)f\big)(\Omega):=\,\tau(\,{}^t(C\Om+D))f(M^{-1}\cdot\Om),
\end{equation}
where $M^{-1}=\begin{pmatrix} A & B \\ C & D \end{pmatrix}\in Sp(n,\BR),\ f\in {\mathcal O}(\BH_n,V_\tau)$ and $\Om\in
\BH_n.$ Here $\tau$ can be extended uniquely to a representation of the complexification $GL(n,\BC)$ of $K$. If $v_\tau$ is
a highest weight vector of $\tau$, then $\Phi_\tau(\Omega):=\,\tau(\Omega+i I_n)v_\tau$ is a a highest weight vector of
$T_\tau.$ It can be shown that $T_\tau$ is an irreducible representation of $Sp(n,\BR)$ with highest weight vector $\Phi_\tau.$

\begin{definition}
A polynomial $f:\BR^{(m,n)}\lrt \BC$ is called $ \textsf{pluriharmonic}$ if
$$\sum_{k=1}^m { {\partial^2 f}\over {\partial x_{ki}\partial x_{kj}} }
 \,=\,0\qquad \textrm{for all}\ 1\leq i,j\leq n.$$
\end{definition}
Let $\mathfrak H$ be the space of all pluriharmonic polynomials on $\BR^{(m,n)}.$ Then $O(m)\times GL(n,\BR)$ acts on
$\mathfrak H$ by
\begin{equation*}
\big( (\alpha,a)\cdot P\big)\,=\,P(\alpha^{-1}xa),\qquad \alpha\in O(m),\ a\in GL(n,\BR),\ P\in {\mathfrak H}.
\end{equation*}

For $(\sigma,V_\sigma)\in \Sigma_m,$ we let ${\mathfrak H}(\sigma)$ be the space of all $V_\sigma$-valued pluriharmonic
polynomials $P:\BR^{(m,n)}\lrt V_\sigma$ such that
$$P(\alpha x)\,=\,\sigma(\alpha^{-1})^{-1}P(x)\qquad \textrm{for all}\ \alpha\in O(m)\ \textrm{and}\ x\in\BR^{(m,n)}.$$
Let $\tau(\sigma)$ be the representation of $GL(n,\BR)$ on ${\mathfrak H}(\sigma)$ defined by
\begin{equation*}
\big( \tau(\sigma)(a)P \big)(x)\,=\,P(xa)\qquad  \ a\in GL(n,\BR),\ P\in {\mathfrak H}(\sigma).
\end{equation*}
For $\sigma\in\Sigma_m,$ we see that ${\mathfrak H}(\sigma)\neq 0$ and $\tau(\sigma)$ is an irreducible finite dimensional
representation of $GL(n,\BR)$ on ${\mathfrak H}(\sigma)$. They proved that the mapping $\sigma\lrt \tau(\sigma)$ is an
injection from $\Sigma_m$ into $\widehat{GL(n,\BR)}$ and
$$ {\mathfrak H} =\bigoplus_{\sigma\in\Sigma_m} \tau(\sigma)\otimes \sigma^*\,=\,
\bigoplus_{\sigma\in\Sigma_m} {\mathfrak H}(\sigma)\otimes \sigma^*$$
as $O(m)\times GL(n,\BR)$-module.

\vskip 0.2cm Let $\sigma\in\Sigma_m.$ We assume that $P:\BR^{(m,n)}\lrt \textrm{Hom}_\BC (V_{ \tau(\sigma) },V_\sigma)$ is a
$\textrm{Hom}_\BC (V_{\tau(\sigma)},V_\sigma)$-valued pluriharmonic polynomial on $\BR^{(m,n)}$ satisfying the conditions
\begin{equation*}
(A)\qquad P(\alpha x)\,=\,\sigma(\alpha^{-1})^{-1}P(x)\qquad \textrm{for all}\ \alpha\in O(m)\ \textrm{and}\
x\in\BR^{(m,n)}\hskip 5cm
\end{equation*}
and
\begin{equation*}
(B)\qquad P(xa)\,=\,P(x)\,\big( \tau(\sigma)\otimes {\det}^{\frac m2}\big)(a)\qquad \textrm{for all}\ a\in GL(n,\BR).
\hskip 5cm
\end{equation*}
The unitary operator
\begin{equation*}
{\mathscr F}_\sigma : L^2\big( \BR^{(m,n)};\sigma\big)\lrt {\mathcal O}\big(\BH_n,V_{\tau(\sigma)} \big)
\end{equation*}
defined by
\begin{equation*}
({\mathscr F}_\sigma f)(\Omega): =\,\int_{\BR^{(m,n)}} e^{\pi\,i\,\sigma (x\,\Om\,{}^tx)}\,P(x)^*f(x)\,dx,\qquad
f\in L^2\big( \BR^{(m,n)};\sigma\big),\ \Omega\in\BH_n
\end{equation*}
intertwines $\omega_c(\sigma)$ with $T_{\tau(\sigma)\otimes \det^{-{\frac m2}} }$.

\end{section}

\newpage
\newcommand\mfc{{\mathscr F}^{(c)} }

\begin{section}{{\large\bf Covariant Maps for the Weil Representation}}
\setcounter{equation}{0}
\vskip 0.2cm
Let $c$ be a symmetric positive definite real matrix of degree $m$. We define the map ${\mathscr F}^{(c)}:\BH_n\lrt
L^2\big(\BR^{(m,n)}\big)$ by

\begin{equation}
{\mathscr F}^{(c)}(\Omega)(x):=e^{2\,\pi\,i\,\sigma(c\,x\,\Omega\,{}^tx)},\qquad \Om\in\BH_n,\ x\in \BR^{(m,n)}.
\end{equation}
We define the automorphic factor $J_m:Sp(n,\BR)\times \BH_n\lrt \BC^*$ for $Sp(n,\BR)$ on $\BH_n$ by
\begin{equation}
J_m(M,\Omega)=\det (C\Omega+D)^{\frac m2},
\end{equation}
where $M=\begin{pmatrix} A & B \\ C & D \end{pmatrix}\in Sp(n,\BR)$ and $\Omega\in \BH_n.$ We see that $Sp(n,\BR)$ acts on $\BH_n$
transitively by

\begin{equation}
M\cdot \Omega=(A\Omega+B)(C\Omega+D)^{-1},
\end{equation}
where $M=\begin{pmatrix} A & B \\ C & D \end{pmatrix}\in Sp(n,\BR)$ and $\Omega\in \BH_n.$

\begin{theorem}
The map ${\mathscr F}^{(c)}:\BH_n\lrt
L^2\big(\BR^{(m,n)}\big)$ defined by Formula (14.1) is a covariant map for the Weil representation $\omega_c$ of $Mp(n,\BR)$
with respect to the automorphic factor $J_m$ defined by Formula (14.2). In other words, ${\mathscr F}^{(c)}$ satisfies the
following covariant relation
\begin{equation}
R_c(M){\mathscr F}^{(c)}(\Omega)= J_m(M,\Om)^{-1}\,{\mathscr F}^{(c)}(M\cdot\Omega)
\end{equation}
for all $M\in Sp(n,\BR)$ and $\Om\in\BH_n.$ We recall that
\begin{equation*}
\omega_c\big( (M,t)\big)=t\,R_c(M)\qquad ({\rm cf.}\,\,(13.9))
\end{equation*}
for all $(M,t)\in Mp(n,\BR)$ with $M\in Sp(n,\BR)$ and $t\in \BC_1^*.$
\end{theorem}
\noindent {\it Proof.} For $M=\begin{pmatrix} A & B \\ C & D \end{pmatrix}\in Sp(n,\BR)$ and $\Omega\in \BH_n$, we put

\begin{equation}
\Omega_*=M\cdot \Omega=(A\Omega+B)(C\Omega+D)^{-1}.
\end{equation}

In this section, we use the notations $t_b,\ d_a$ and $\sigma_n$ in Section 13. It suffices to prove the covariance relation
(14.4) for the generators $t_b\, (b=\,{}^tb\in \BR^{(n,n)}),\ d_a\,(a\in GL(n,\BR))$ and $\sigma_n$ of $Sp(n,\BR).$

\vskip 0.5cm\noindent {\bf Case I.} $M=t_b$ with
$b=\,{}^tb\in\BR^{(n,n)}.$

\vskip 0.1cm In this case, we have
$$\Om_*=\Om+b.$$
By Formula (13.11) in Theorem 13.3,

\begin{eqnarray*}
& & \left( R_c(M){\mathscr F}^{(c)}(\Omega)\right)(x)\\
&=& \left( R_c(t_b){\mathscr F}^{(c)}(\Omega)\right)(x)\\
&=& e^{2\,\pi \,i\,\sigma(c\,x\,b\,{}^tx) }\,{\mathscr F}^{(c)}(\Omega)(x).
\end{eqnarray*}

\noindent On the other hand, according to Formula (14.2),

\begin{eqnarray*}
& & J_m (M,\Omega)^{-1}{\mathscr F}^{(c)}(M\cdot \Omega)(x)\\
&=& {\mathscr F}^{(c)}(\Omega+b)(x)\\
&=& e^{2\,\pi\,i\,\sigma (c\,x(\Om+b)\,{}^tx))}\\
&=& e^{2\,\pi \,i\,\sigma(c\,x\,b\,{}^tx) }\,{\mathscr F}^{(c)}(\Omega)(x).
\end{eqnarray*}

\noindent
Thus
$$ R_c(t_b){\mathscr F}^{(c)}(\Omega)=J_m (t_b,\Omega)^{-1}{\mathscr F}^{(c)}(t_b\cdot \Omega)$$
$\textrm{for\ all}\ b=\,{}^tb\in \BR^{(n,n)}\ \textrm{and}\ \Omega\in\BH_n.$ Therefore we proved the covariance relation
(14.4) in the case $M=t_b$ with $b=\,{}^tb\in\BR^{(n,n)}.$

\vskip 0.5cm\noindent {\bf Case II.} $M=d_a=\begin{pmatrix} {}^ta & 0 \\ 0 & a^{-1} \end{pmatrix}$ with
$a\in GL(n,\BR).$

\vskip 0.1cm In this case, we have
$$\Om_*=\,{}^ta\,\Omega\,a.$$
By Formula (13.12) in Theorem 13.3,

\begin{eqnarray*}
& & \left( R_c(M){\mathscr F}^{(c)}(\Omega)\right)(x)\\
&=& \left( \det a \right)^{\frac m2}\,{\mathscr F}^{(c)}(\Omega)(x\,{}^ta)\\
&=& \left( \det a \right)^{\frac m2}\,e^{2\,\pi\,i\,\sigma (c\,x\,{}^ta\,\Omega\,a\,{}^tx)}.
\end{eqnarray*}

\noindent On the other hand, according to Formula (14.2),

\begin{eqnarray*}
& & J_m (M,\Omega)^{-1}{\mathscr F}^{(c)}(M\cdot \Omega)(x)\\
&=& \left( \det \big(a^{-1}\big)\right)^{-{\frac m2}}\,{\mathscr F}^{(c)}\big({}^ta\,\Omega\,a\big)(x)\\
&=& \left( \det a \right)^{\frac m2}\,e^{2\,\pi\,i\,\sigma (c\,x\,{}^ta\,\Omega\,a\,{}^tx)}.
\end{eqnarray*}

\noindent Thus
$$ R_c(d_a){\mathscr F}^{(c)}(\Omega)=J_m (d_a,\Omega)^{-1}{\mathscr F}^{(c)}(d_a\cdot \Omega)$$
$\textrm{for\ all}\ a\in GL(n,\BR)\ \textrm{and}\ \Omega\in\BH_n.$ Therefore we proved the covariance relation
(14.4) in the case $M=d_a$ with $d_a\in GL(n,\BR).$

\vskip 0.5cm\noindent {\bf Case III.} $M=\sigma_n=\begin{pmatrix} 0 & -I_n \\ I_n & 0 \end{pmatrix}$.

\vskip 0.1cm In this case, we have
$$\Om_*=-\Omega^{-1}\qquad \textrm{and} \qquad J_m(M,\Omega)=\left( \det \Omega\right)^{\frac m2}.$$

In order to prove the covariance relation (14.4), we need the following useful lemma.
\newcommand\rmn{\BR^{(m,n)}}

\begin{lemma} For a fixed element $\Om\in \BH_n$ and a fixed
element $Z\in\BC^{(m,n)},$ we obtain the following property
\begin{eqnarray}
& &\int_{\rmn} e^{\pi\,i\,\sigma
(x\,\Om\,{}^tx+2\,x\,{}^tZ)}dx_{11}\cdots dx_{mn}\\
 &=& \left( \det
{\Omega\over i}\right)^{-{\frac m2}}\,
e^{-\pi\,i\,\sigma(Z\,\Om^{-1}\,{}^tZ)},\nonumber
\end{eqnarray}
where $x=(x_{ij})\in \BR^{(m,n)}.$
\end{lemma}

\noindent {\it Proof of Lemma 14.2.} By a simple computation, we
see that
$$e^{\pi i\, \sigma ( x\Om\, {}^tx +
2x\,{}^tZ )}= e^{-\pi i\,\sigma (Z\Om^{-1}\,{}^tZ )}\cdot e^{\pi
i\,\sigma \{(x+Z\Om^{-1})\Om\,{}^t(x+Z\Om^{-1})\} }.$$ We observe that the
real Jacobi group $Sp(n,\BR)\ltimes H_\BR^{(n,m)}$ acts on
${\mathbb H}_n\times \BC^{(m,n)}$ holomorphically and transitively by
\begin{equation}
\big( M, (\lambda,\mu,\kappa)\big)\cdot (\Omega,Z)=\big( M\cdot\Omega, (Z+\lambda\,\Omega+\mu)(C\Omega+D)^{-1}\big),
\end{equation}
where $M\in Sp(n,\BR),\ (\lambda,\mu,\kappa)\in H_\BR^{(n,m)},\,\Omega\in\BH_n$ and $Z\in\BC^{(m,n)}.$
So we may put
$$\Om=\,i\,A\,{}^t\!A,\quad Z=iV,\quad\  A\in\BR^{(n,n)},\quad
V=(v_{ij})\in\BR^{(m,n)}.$$
Then we obtain
\begin{eqnarray*}
& & \int_{\BR^{(m,n)}}  e^{\pi i\, \sigma ( x\Om\, {}^tx +
2x\,{}^tZ )} dx_{11}\cdots dx_{mn} \\
&=& e^{-\pi i\,\sigma (Z\Omega^{-1}\,{}^tZ)} \int_{\BR^{(m,n)}}
e^{\pi i\,\sigma [\{
x+iV(iA\,{}^t\!A)^{-1}\}(iA\,{}^t\!A)\,{}^t\!\{
x+iV(iA\,{}^t\!A)^{-1}\} ]}\,dx_{11}\cdots dx_{mn}\\
&=&e^{-\pi i\,\sigma (Z\Omega^{-1}\,{}^tZ)} \int_{\BR^{(m,n)}}
e^{\pi i\,\sigma [\{ x+V(A\,{}^t\!A)^{-1}\}A\,{}^t\!A\,{}^t\!\{
x+V(A\,{}^t\!A)^{-1}\} ]}\,dx_{11}\cdots dx_{mn}\\
&=& e^{-\pi i\,\sigma (Z\Omega^{-1}\,{}^tZ)} \int_{\BR^{(m,n)}}
e^{-\pi \,\sigma\{ (uA)\,{}^t\!(uA)\} }\,du_{11}\cdots
du_{mn}\\
& &
\quad \big(\,{\rm Put}\ u=
x+V(A\,{}^t\!A)^{-1}=(u_{ij}) \,\big)\\
&=& e^{-\pi i\,\sigma (Z\Omega^{-1}\,{}^tZ)} \int_{\BR^{(m,n)}}
e^{-\pi \,\sigma (w\,{}^t\!w)} (\det A)^{-m}\,dw_{11}\cdots
dw_{mn}\\
& & \quad \big(\,{\rm Put}\ w=uA=(w_{ij})\,\big)\\
&=& e^{-\pi i\,\sigma (Z\Omega^{-1}\,{}^tZ)} \, (\det A)^{-m}\cdot
\left( \prod_{i=1}^m \prod_{j=1}^n \int_\BR e^{-\pi\,
w_{ij}^2}\,dw_{ij}\right)\\
&=& e^{-\pi i\,\sigma (Z\Omega^{-1}\,{}^tZ)} \, (\det A)^{-m}\quad
\big(\,{\rm because}\ \int_\BR e^{-\pi\,
w_{ij}^2}\,dw_{ij}=1\quad {\rm for\ all}\ i,j\,\big)\\
&=& e^{-\pi i\,\sigma (Z\Omega^{-1}\,{}^tZ)} \, \left( \det \big(
A\, {}^t\!A \big)\right)^{-{\frac m2}}\\
&=& e^{-\pi i\,\sigma (Z\Omega^{-1}\,{}^tZ)} \, \left( \det \left(
{ {\Omega}\over i } \right)\right)^{-{\frac m2}}.
\end{eqnarray*}

\noindent This completes the proof of Lemma 14.2. \hfill $\square$

\vskip 0.2cm
According to Formula (13.13) in Theorem 13.3,

\begin{eqnarray*}
& & \left( R_c(\sigma_n){\mathscr F}^{(c)}(\Omega)\right)(x)\\
&=& \left( {\frac 2i}\right)^{{mn}\over 2} \left( \det c \right)^{\frac n2}\,
\int_{\BR^{(m,n)}} {\mathscr F}^{(c)}(\Omega)(y)\,e^{-4\,\pi\,i\,\sigma(c\,y\,{}^tx)}\,dy\\
&=& \left( {\frac 2i}\right)^{{mn}\over 2} \left( \det c \right)^{\frac n2}\,
\int_{\BR^{(m,n)}} e^{2\,\pi\,i\,\sigma (c\,y\,\Omega\,{}^ty)}\cdot e^{-4\,\pi\,i\,\sigma(c\,y\,{}^tx)}\,dy\\
&=& \left( {\frac 2i}\right)^{{mn}\over 2} \left( \det c \right)^{\frac n2}\,
\int_{\BR^{(m,n)}} e^{\pi\,i\,\sigma \left\{ c\,(y\,(2\,\Omega)\,{}^ty\,+\, 2\,y\,{}^t(-2\,x))\right\} }\,dy\\
\end{eqnarray*}

If we substitute $u=c^{1/2}\,y$, then $du=\big( \det c\big)^{\frac n2}\,dy.$ Therefore according to Lemma 14.2, we obtain

\begin{eqnarray*}
& & \left( R_c(\sigma_n){\mathscr F}^{(c)}(\Omega)\right)(x)\\
&=& \left( {\frac 2i}\right)^{{mn}\over 2} \left( \det c \right)^{\frac n2}\,
\int_{\BR^{(m,n)}} e^{\pi\,i\,\sigma ( u\,(2\,\Omega)\,{}^tu\,+\, 2\,c^{1/2}\,u\,{}^t(-2\,x)) }\,(\det c)^{-{\frac n2}}\,du\\
&=& \left( {\frac 2i}\right)^{{mn}\over 2} \,
\int_{\BR^{(m,n)}} e^{\pi\,i\,\sigma ( u\,(2\,\Omega)\,{}^tu\,+\, 2\,u\,{}^t(-2\,c^{1/2}\,x)) }\,du\\
&=& \left( {\frac 2i}\right)^{{mn}\over 2} \,\left( \det {{2\,\Omega}\over i}\right)^{-{\frac m2}}\,
e^{-\pi\,i\,\sigma( (-2\,c^{1/2}\,x)\,(2\,\Omega)^{-1}\,{}^t(-2\,c^{1/2}\,x))}\\
&=& \big( \det \Omega\big)^{-{\frac m2}}\,e^{-2\,\pi\,i\,\sigma( c\,x\,\Omega^{-1}\,{}^tx)}.
\end{eqnarray*}

\vskip 0.2cm
\noindent On the other hand, according to Formula (14.2),

\begin{eqnarray*}
& & J_m (M,\Omega)^{-1}{\mathscr F}^{(c)}(M\cdot \Omega)(x)\\
&=& J_m(\sigma,\Omega)^{-1}\,{\mathscr F}^{(c)}\big(-\Omega^{-1}\big)(x)\\
&=& \big( \det \Omega\big)^{-{\frac m2}}\, e^{2\,\pi\,i\,\sigma( c\,x\,(-\Omega^{-1})\,{}^tx)}\\
&=& \big( \det \Omega\big)^{-{\frac m2}}\,e^{-2\,\pi\,i\,\sigma( c\,x\,\Omega^{-1}\,{}^tx)}.
\end{eqnarray*}

\noindent So we see that
\begin{equation}
R_c(\sigma_n){\mathscr F}^{(c)}(\Omega)= J_m(\sigma_n,\Om)^{-1}\,{\mathscr F}^{(c)}(\sigma_n\cdot\Omega).
\end{equation}

\noindent Therefore the covariance relation (14.4) holds for the case $\sigma_n=\begin{pmatrix} 0 & -I_n \\ I_n & 0 \end{pmatrix}$.

\vskip 0.2cm Since $J_m$ is an automorphic factor for $Sp(n,\BR)$ on $\BH_n$, we see that if the covariance relation (14.4)
holds for $M_1,\,M_2$ in $Sp(n,\BR),$ then it holds for $M_1M_2.$ Finally we complete the proof. $\hfill \square$


\vskip 0.2cm Now we can give another realization of the metaplectic group $Mp(n,\BR)$ that was dealt with in Section 11
and Section 13.
\vskip 0.1cm
\begin{proposition}
Let $(U_c,{\mathcal H}_c)$ be the Schr{\"o}dinger representation of the Heisenberg group $H_\BR^{(n,m)}$ defined by
Formula (5.8) with the model ${\mathcal H}_c=L^2\big(\BR^{(m,n)},d\xi\big)$. We denote by $U({\mathcal H}_c)$ the group
of all unitary isomorphisms of ${\mathcal H}_c$. Let $\widetilde{Mp(c)}$ be the set of all $R\in U({\mathcal H}_c)$
such that
\begin{equation*}
R\,U_c(g)=U_c (M\star g)R=\,U_c(MgM^{-1})R
\end{equation*}
for all $g\in H_\BR^{(n,m)}$ and for some $M\in Sp(n,\BR).$ Then for a given element $R\in \widetilde{Mp(c)}$, the
corresponding $M\in Sp(n,\BR)$ is determined uniquely, denoted by $M=\nu_c(R).$ Moreover there is an exact
sequence of groups
\begin{equation}
1 \lrt \BC^*_1 \lrt \widetilde{Mp(c)} -\!\!\!\xrightarrow{\nu_c} Sp(n,\BR)\lrt 1.
\end{equation}
\end{proposition}
\vskip 0.1cm\noindent
{\it Proof.} First of all we observe that $\widetilde{Mp(c)}$ is a subgroup of $U({\mathcal H}_c)$.
Let $R\in \widetilde{Mp(c)}$, and $M_1,M_2\in Sp(n,\BR)$ such that
$$R\,U_c(g)=\,U_c(M_1\star g)\,R=\, U_c(M_2\star g)\,R \qquad {\rm{for\ all}}\ g\in H_\BR^{(n,m)}.$$
Then $U_c(M_1\star g)=\, U_c(M_2\star g)$ for all $g\in H_\BR^{(n,m)}.$ According to Formula (5.8),
$(M_1^{-1}M_2)g=\,g(M_1^{-1}M_2)$ for all $g\in H_\BR^{(n,m)}.$ Thus $M_1=M_2.$ It follows that the map
$\nu_c:\widetilde{Mp(c)}\lrt Sp(n,\BR)$ is well defined. it is easily checked that $\nu_c$ is a group homomorphism.
The kernel of $\nu_c$ is given by
\begin{equation*}
{\rm{ker}}\,\nu_c=\left\{\,R\in U({\mathcal H}_c)\,\big|\ R\,U_c(g)=\,U_c(g)\,R\quad
{\rm{for\ all}}\ g\in H_\BR^{(n,m)}\,\right\}.
\end{equation*}
Since $U_c$ is irreducible and unitary, according to Schur's lemma, ${\rm{ker}}\,\nu_c=\BC_1^*.$ The surjectivity of
$\nu_c$ follows from the arguments in Section 13.
$\hfill\square$

\vskip 0.2cm
According to Theorem 13.3, $R_c(t_b),\ R_c(d_a)$ and $R_c(\sigma_n)$ are members of
$\widetilde{Mp(c)}$ sitting above the generators $t_b,\,d_a$ and $\sigma_n$ of $Sp(n,\BR)$ respectively.
That is, $\nu_c(R_c(t_b))=\,t_b,\ \nu_c(R_c(d_a))=\,d_a$ and $\nu_c(R_c(\sigma_n))=\,\sigma_n.$

\vskip 0.2cm
\begin{theorem}
Let $P\in \widetilde{Mp(c)}$ and $\nu_c(P)=M=\,\begin{pmatrix} A & B \\ C & D \end{pmatrix}\in Sp(n,\BR)$.
Then for any $\Om\in\BH_n,$
\begin{equation*}
P{\mathscr F}^{(c)}(\Omega)=\,B_c(P;\Omega){\mathscr F}^{(c)}(M\cdot \Omega),
\end{equation*}
where $B_c(P;\Omega)$ is, up to a scalar of absolute one, a branch of the holomorphic function
$\big\{ \det (C\Om+D)^{\frac 12}\big\}^{-m}$ on $\BH_n.$
\end{theorem}
\vskip 0.1cm\noindent
{\it Proof.} Let ${\mathbb G}_1$ be the subgroup of $\widetilde{Mp(c)}$ consisting of all $P\in \widetilde{Mp(c)}$
such that
\begin{equation*}
P{\mathscr F}^{(c)}(\Omega)=\,c_P\,{\mathscr F}^{(c)}\big(\nu_c (P)\cdot\Omega\big)\quad {\rm{for\ all}}\
\Omega\in\BH_n,
\end{equation*}
where $c_P$ is a constant depending only on $P$. For $P\in {\mathbb G}_1,$ we write
\begin{equation*}
P{\mathscr F}^{(c)}(\Omega)=\,B_c(P;\Om)\,{\mathscr F}^{(c)}\big(\nu_c (P)\cdot\Omega\big)\quad {\rm{for\ all}}\
\Omega\in\BH_n.
\end{equation*}
Let ${\mathbb G}_2$ be the set of all $P\in {\mathbb G}_1$ satisfying the following conditions (${\mathbb G}$1)
and (${\mathbb G}$2)\,:
\vskip 0.1cm
(${\mathbb G}$1) \ $B_c(P;\Om)$ is continuous in $\Om\in \BH_n\,;$
\vskip 0.1cm
(${\mathbb G}$2) \ $\{B_c(P;\Om)\}^2\,|\det (C\Om+D)|^m$ is independent of $\Omega$ with values
in $\BC_1^*$ for\par
\ \ \ \ \ \ \ \ $\nu_c(P)=\,\begin{pmatrix} A & B \\ C & D \end{pmatrix}\in Sp(n,\BR)$.
\vskip 0.2cm\noindent
It is easily checked that for $P,Q\in {\mathbb G}_2$,
\begin{equation}
B_c(QP;\Om) = B_c(P;\Om)\,B_c(Q;\nu_2(P)\cdot \Om) \qquad  {\rm{for\ all}}\
\Omega\in\BH_n.
\end{equation}
Indeed, we get
\begin{eqnarray*}
(QP){\mathscr F}^{(c)}(\Omega)&=&Q\big( P{\mathscr F}^{(c)}(\Omega)\big)\\
&=& B_c(P;\Om)\left( Q\big( {\mathscr F}^{(c)}(\nu_c(P)\cdot \Omega)\big)\right)\\
&=& B_c(P;\Om)\,B_c(Q;\nu_c(P)\cdot\Om)\,{\mathscr F}^{(c)}\big(
\nu_c(Q)\cdot (\nu_c(P)\cdot \Omega)\big)\\
&=& B_c(P;\Om)\,B_c(Q;\nu_c(P)\cdot\Om)\,{\mathscr F}^{(c)}\big(
\nu_c(QP)\cdot  \Omega)\big).
\end{eqnarray*}
By Formula (14.9) together with the fact that $J(M,\Omega):=\det(C\Om+D)$ for $M=\begin{pmatrix} A & B \\ C & D \end{pmatrix}\in Sp(n,\BR)$ and $\Om\in \BH_n$ is automorphic factor, we see that ${\mathbb G}_2$ is a subgroup of ${\mathbb G}_1.$
We observe that $R_c(t_b),\ R_c(d_a),\ R_c(\sigma_n)$ in Theorem 13.3 and $\alpha\in\BC_1^*$ generate the
group $\widetilde{Mp(c)}$. We shall show that $R_c(t_b),\ R_c(d_a),\ R_c(\sigma_n)$ and $\alpha\in\BC_1^*$
belong to ${\mathbb G}_2.$ Then ${\mathbb G}_1={\mathbb G}_2=\widetilde{Mp(c)}.$ This implies the proof of the theorem.
\vskip 0.2cm
Now we shall prove that $R_c(t_b),\ R_c(d_a),\ R_c(\sigma_n)$ and $\alpha\in\BC_1^*$
belong to ${\mathbb G}_2.$ For brevity we put $F_c(P;\Om)=\{B_c(P;\Om)\}^2\,|\det (C\Om+D)|^m$ for
$\nu_c(P)=\,\begin{pmatrix} A & B \\ C & D \end{pmatrix}\in Sp(n,\BR)$ with $P\in \widetilde{Mp(c)}$.
\vskip 0.25cm\noindent
{\bf Case I.} $P=\alpha\in \BC_1^* \subset \widetilde{Mp(c)}.$
\vskip 0.1cm
In this case, we obtain
\begin{equation*}
P{\mathscr F}^{(c)}(\Omega)=\,\alpha\,{\mathscr F}^{(c)}(\Omega).
\end{equation*}
\indent So we get $B_c(P;\Om)=\alpha$ and $F_c(P;\Om)=\alpha^2.$ Thus $\alpha\in {\mathbb G}_2.$

\vskip 0.215cm\noindent
{\bf Case II.} $P=R_c(t_b)$ with $t_b=\begin{pmatrix} I_n & b \\ 0 & I_n \end{pmatrix}\in Sp(n,\BR).$
\vskip 0.1cm
In this case, according to Formula (13.11), we obtain
\begin{eqnarray*}
P{\mathscr F}^{(c)}(\Omega)&=& e^{2\pi i\,\sigma(c\, x b\,{}^tx)}{\mathscr F}^{(c)}(\Omega)(x)\\
&=& e^{2\pi i\,\sigma \{ c\, x (\Om+b)\,{}^tx\} }\\
&=& {\mathscr F}^{(c)}(\Omega+b)(x)={\mathscr F}^{(c)}(t_b\cdot \Omega)(x)\\
&=& {\mathscr F}^{(c)}\big( \nu_c(R_c(t_b))\cdot\Omega\big)(x).
\end{eqnarray*}
\indent We get $B_c(P;\Om)=1$ and $F_c(P;\Om)=1.$ Thus $R_c(t_b)\in {\mathbb G}_2.$

\vskip 0.215cm\noindent
{\bf Case III.} $P=R_c(d_a)$ with $d_a=\begin{pmatrix} {}^ta & 0 \\ 0 & a^{-1} \end{pmatrix}\in Sp(n,\BR).$
\vskip 0.1cm
In this case, according to Formula (13.12), we obtain
\begin{eqnarray*}
P{\mathscr F}^{(c)}(\Omega)&=& (\det a)^{\frac m2}\,{\mathscr F}^{(c)}(\Omega)(x\,{}^ta)\\
&=& (\det a)^{\frac m2}\, e^{2\pi i\,\sigma \{ c\, x ({}^t\!a\,\Omega\,a)\,{}^tx\} }  \\
&=& (\det a)^{\frac m2}\,{\mathscr F}^{(c)}(d_a\cdot\Om)(x)\\
&=& (\det a)^{\frac m2}\,  {\mathscr F}^{(c)}\big( \nu_c(R_c(d_a))\cdot\Omega\big)(x).
\end{eqnarray*}
\indent We get $B_c(P;\Om)=(\det a)^{\frac m2}$ and $F_c(P;\Om)=1.$ Thus $R_c(d_a)\in {\mathbb G}_2.$

\vskip 0.215cm\noindent
{\bf Case IV.} $P=R_c(\sigma_n)$ with $\sigma_n=\begin{pmatrix} 0 & -I_n \\ I_n & \ 0 \end{pmatrix}\in Sp(n,\BR).$
\vskip 0.1cm
In this case, according to Formula (13.13), we obtain
\begin{eqnarray*}
P{\mathscr F}^{(c)}(\Omega)&=& \left( {2\over i}\right)^{{mn}\over 2} (\det c)^{\frac n2}\,\int_{\BR^{(m,n)}}
{\mathscr F}^{(c)}(\Omega)(y)\,
e^{-4 \pi i\,\sigma(c\,y\,{}^tx)}\,dy\\
&=& \left( {2\over i}\right)^{{mn}\over 2} (\det c)^{\frac n2}\,\int_{\BR^{(m,n)}}
e^{2 \pi i\,\sigma\{ c\, (y\,\Om\,{}^t\!y-2\,y {}^t\!x)\}}\,dy \\
&=& (\det \Om)^{-{\frac m2}}\,e^{-2\pi i\,\sigma(c\,x\,\Om^{-1}\,{}^tx)}\qquad ({\rm{by\ Lemma\ 14.2}})     \\
&=& (\det \Om)^{-{\frac m2}}\,{\mathscr F}^{(c)}(-\Omega^{-1})(x)\\
&=& (\det \Om)^{-{\frac m2}}\,{\mathscr F}^{(c)}\big( \nu_c(R_c(\sigma_n))\cdot\Omega\big)(x).
\end{eqnarray*}
\indent We get $B_c(P;\Om)=(\det \Om)^{-{\frac m2}}$ with $B_c(P;iI_n)=i^{-{{mn}\over 2}}$,
and $F_c(P;\Om)=i^{-{{mn}\over 2}}.$
Thus $R_c(\sigma_n)\in {\mathbb G}_2.$ Hence we complete the proof.
$\hfill\square$

\vskip 0.52cm
\begin{definition}
Let $\chi_c: \widetilde{Mp(c)}\lrt \BC$ be the map defined by
\begin{equation*}
\chi_c (P)=\,\det (C\Om+D)^m \big\{ B_c(P;\Om) \big\}^2,\quad P\in \widetilde{Mp(c)},
\end{equation*}
where $\nu_c(P)= \,\begin{pmatrix} A & B \\ C & D \end{pmatrix}\in Sp(n,\BR)$.
According to Theorem 14.4, the image of $\chi_c$ is contained in $\BC_1^*$ and $\chi_c:\widetilde{Mp(c)}\lrt
\BC^*_1$ is a character of $\widetilde{Mp(c)}$. Furthermore we have
\begin{equation*}
\chi_c(\alpha)=\alpha^2 \quad {\rm for\ any}\ \alpha\in \BC^*_1 \subset \widetilde{Mp(c)}.
\end{equation*}
We denote by $Mp(n,\BR)_c$ the kernel of $\chi_c.$ We call $Mp(n,\BR)_c$ the {\it metaplectic group} attached to
$U_c.$
\end{definition}

\vskip 0.2cm We let
\begin{equation*}
m_{\diamond}: \widetilde{Mp(c)}\times \widetilde{Mp(c)}\lrt \widetilde{Mp(c)}
\end{equation*}
be the multiplication map and let
\begin{equation*}
\Phi_{[c]}:\widetilde{Mp(c)}\times \BH_n \lrt \BC^*
\end{equation*}
be the map defined by
\begin{equation*}
\Phi_{[c]}(P,\Omega):=\,B_c(P;\Omega),\quad P\in \widetilde{Mp(c)},\ \Omega\in \BH_n.
\end{equation*}
We provide $\widetilde{Mp(c)}$ with the weakest topology such that the following three maps
\begin{eqnarray*}
& & \nu_c:\widetilde{Mp(c)}\lrt Sp(n,\BR),\quad m_{\diamond}:\widetilde{Mp(c)}\times \widetilde{Mp(c)}\lrt \widetilde{Mp(c)},
\quad \\
& &\Phi_{[c]}:\widetilde{Mp(c)}\times \BH_n \lrt \BC^*
\end{eqnarray*}
are all continuous.

\vskip 0.1cm
Then we have the following properties.

\vskip 0.1cm
\begin{lemma}
$\widetilde{Mp(c)}$ is a Hausdorff space on the above weakest topology.
\end{lemma}
\vskip 0.1cm\noindent
{\it Proof.} Fix an element $\Omega_0\in \BH_n$. Let $\eta: \widetilde{Mp(c)}\lrt Sp(n,\BR)\times \BC^*$ by
\begin{equation}
\eta(P):=\,\big( \nu_c(P), B_c(P;\Om_0)\big),\quad P\in \widetilde{Mp(c)}.
\end{equation}
Then by the weak topology on $\widetilde{Mp(c)},\ \eta$ is continuous. If $P,Q\in \widetilde{Mp(c)}$ such that
$\eta(P)=\eta(Q)$, then $\nu_c(P)=\nu_c(Q)$ and $B_c(P;\Om_0)=B_c(Q;\Om_0).$ $QP^{-1}=\alpha\in \BC^*_1$ because
$\nu_c(QP^{-1})=1.$ Thus $Q=\alpha P.$ By assumption,
$$B_c(Q;\Om_0)=B_c(\alpha P;\Om_0)=\alpha B_c(P;\Om_0)=B_c(P;\Om_0).$$
Therefore $\alpha=1,$ that is, $P=Q.$ This implies that $\eta$ is one-to-one.

\vskip 0.1cm Let $f: Sp(n,\BR)\times \BC^*\lrt \BC^*$ be the map defined by
\begin{equation}
f(M,\alpha):=\,\alpha^2\,\{\det (C\Om_0+D)\}^m,
\end{equation}
where $M=\begin{pmatrix} A & B \\ C & D \end{pmatrix}\in Sp(n,\BR)$ and
$\alpha\in\BC^*.$ By Theorem 14.4, $\eta\big( \widetilde{Mp(c)}\big)=\,f^{-1}(\BC^*_1).$
Since ${{\partial f}\over {\partial\alpha}}\neq 0$ and $\BC^*_1$ is a submanifold of $\BC^*$,
we see that $\eta\big( \widetilde{Mp(c)}\big)$ is a submanifold of $Sp(n,\BR)\times \BC^*.$
Therefore $\eta\big( \widetilde{Mp(c)}\big)$ is Hausdorff because $Sp(n,\BR)$ and $\BC^*$ are Hausdorff.
$\hfill\square$

\vskip 0.2cm
\begin{lemma}
Let $h: \widetilde{Mp(c)}\lrt Sp(n,\BR)\times \BC^*_1$ be the map defined by
$$h(P):=\,\big(\nu_c(P),\chi_c(P)\big),\quad P\in \widetilde{Mp(c)}.$$
Then the map $h$ defines a connected double covering of the Lie group $Sp(n,\BR)\times \BC^*_1$, and
hence gives $\widetilde{Mp(c)}$ the structure of a Lie group.
\end{lemma}
\vskip 0.1cm\noindent
{\it Proof.} We note that $h$ is continuous. We see that
\begin{equation*}
{\rm {ker}}\,h=\,{\rm {ker}}\,\nu_c \cap {\rm {ker}}\,\chi_c=\,\BC^*_1 \cap {\rm {ker}}\,\chi_c
=\,{\rm {ker}}\,\big( \chi_c|_{\BC^*_1} \big)=\,\{ \pm 1\}.
\end{equation*}
Let $h_{\ast}:Sp(n,\BR)\times \BC^*\lrt Sp(n,\BR)\times \BC^*$ be the map defined by
\begin{equation}
h_{\ast}(M,\alpha):=\,(M,f(M,\alpha)),\quad M\in Sp(n,\BR),\ \alpha\in \BC^*,
\end{equation}
where $f$ is the map defined by (14.12). Then $h=h_{\ast}\circ \eta$, where $\eta$ is the map defined
by (14.11). Clearly $h_{\ast}$ is a double covering projection. Since $h_{\ast}^{-1}\big(Sp(n,\BR\times
\BC^*_1\big)=\eta\big( \widetilde{Mp(c)}\big)$, the restriction $h_{\ast,\eta}$ of $h_{\ast}$ to
$\eta\big( \widetilde{Mp(c)}\big)$ is a double covering
$$ h_{\ast,\eta}: \widetilde{Mp(c)} \lrt Sp(n,\BR)\times
\BC^*_1$$
of the manifold $Sp(n,\BR)\times
\BC^*_1$. It only remains to prove that $\widetilde{Mp(c)}$ is connected. Since $Sp(n,\BR)$ and $\BC^*_1$
are connected, according to the exact sequence (14.9), $\widetilde{Mp(c)}$ is connected.
$\hfill\square$

\vskip 0.2cm
\begin{proposition}
$Mp(n,\BR)_c$ is a closed connected subgroup of $\widetilde{Mp(c)}$ and $q_c:Mp(n,\BR)_c\lrt Sp(n,\BR)$ is a double
covering projection with ${\rm ker}\,q_c=\,\{ \pm 1 \},$ where $q_c$ is the restriction of $\nu_c$ to $Mp(n,\BR)_c.$
\end{proposition}
\vskip 0.1cm\noindent
{\it Proof.} By Lemma 14.7, $q_c$ is a double covering projection of $Sp(n,\BR)$. Ir only remains to prove that
$Mp(n,\BR)_c$ is connected. The stabilizer at $i\,I_n$ under the action (11.6) of $Sp(n,\BR)$ is given by
\begin{equation*}
\left\{\,\begin{pmatrix} \  A & B \\ -B & A \end{pmatrix}\in Sp(n,\BR)\ \Big|\
A\,{}^t\!A+\,B\,{}^tB=I_n,\ A\,{}^tB=\,B\,{}^tB\,\right\}
\end{equation*}
that is isomorphic to $U(n)$ via $\begin{pmatrix} \  A & B \\ -B & A \end{pmatrix}.$
The map
\begin{equation*}
Mp(n,\BR)_c\lrt \BH_n,\ \qquad P\longmapsto q_c(P)\cdot (i\,I_n),\quad P\in MP(n,\BR)_c
\end{equation*}
gives the coset space of $Mp(n,\BR)_c$ with respect to $q^{-1}_c(U(n))$, i.e.,
$$Mp(n,\BR)_c/q^{-1}_c(U(n))=\,\BH_n.$$
For $\Om=\,i\,I_n$, the map $B_{c;iI_n}:q^{-1}_c(U(n))\lrt \BC_1^*$ defined by
\begin{equation}
B_{c;iI_n}(P):=\,B_c(P;i\,I_n),\quad P\in q^{-1}_c(U(n))
\end{equation}
is a continuous character. If $P\in q^{-1}_c(U(n))$ with $q_c(P)=\begin{pmatrix} \  A & B \\ -B & A \end{pmatrix}
\in U(n),$ then
\begin{equation*}
\big\{ B_{c;iI_n}(P) \big\}^2=\,\big\{ B_c(P;iI_n) \big\}^2= \,\big\{ \det (A-i\,B)\big\}^{-m}.
\end{equation*}
We define the map $\det_c^*:U(n)\lrt \BC_1^*$ by
\begin{equation}
{\det}_c^* \begin{pmatrix} \  A & B \\ -B & A \end{pmatrix}:=\,\left\{ \det(A-i\,B)\right\}^{-m},
\quad \begin{pmatrix} \  A & B \\ -B & A \end{pmatrix}
\in U(n).
\end{equation}
and the map $Sq:\BC_1^* \lrt \BC_1^*$ by $Sq(\alpha)=\alpha^2$ with $\alpha\in\BC_1^*.$ The we have the
following commutative diagram\,:
\vskip 0.251cm
$$\begin{CD}
q^{-1}_c(U(n)) @> B_{c;iI_n}  >>  \BC_1^* \\
@V q_c VV       @VV Sq V  \\
U(n)       @> \det_c^* >>   \BC_1^*
\end{CD} $$
\vskip 0.2cm
$$\textsf{  diagram  14.1}$$

\vskip 0.5cm\noindent
Thus $q_c^{-1}(U(n))$ along with its topology is the fibre product of $\det_c^*$ and $Sq$.
Since $U(n)$ and $\BC_1^*$ are connected, $q_c^{-1}(U(n))$ is connected.
$\hfill\square$

\vskip 0.2cm
\begin{corollary}
The exact sequence
\begin{equation}
1 \lrt \{\pm 1\} \lrt Mp(n,\BR)_c -\!\!\!\xrightarrow{\!\!\!q_c} Sp(n,\BR)\lrt 1
\end{equation}
\noindent is non-split and $[Mp(n,\BR)_c,Mp(n,\BR)_c]=\,Mp(n,\BR)_c.$
\end{corollary}
\vskip 0.1cm\noindent
{\it Proof.} Embed $U(1)$ into $U(n)$ via $z\mapsto {\rm {diag}}(z,1,1,\cdots,1)$, and embedd $U(n)$ into
$Sp(n,\BR)$ via
\begin{equation*}
U(n)\ni A+iB\,\, \longmapsto \,\,\begin{pmatrix} \  A & B \\ -B & A \end{pmatrix}\in Sp(n,\BR)\quad
{\rm {with}}\ A,B\in\BR^{(n,n)}.
\end{equation*}
So $U(1)\subset U(n)\subset Sp(n,\BR).$ According to the commutative diagram in the proof of Proposition 14.8,
the exact sequence
\begin{equation*}
1\lrt \{\pm 1 \} \lrt q_c^{-1}(U(1))\,\, -\!\!\!\xrightarrow{\!\!\!\! q_c} U(1) \lrt 1
\end{equation*}
can be identified to
\begin{equation}
1\lrt \{\pm 1 \} \lrt \BC_1^*\,\, -\!\!\!\xrightarrow{\!\!\!\! Sq} \BC_1^* \lrt 1.
\end{equation}
If we restrict the exact sequence (14.17) to the torsion subgroups, then we get the non-split exact sequence
\begin{equation}
1\lrt \BZ/2\BZ \lrt \BQ/\BZ  \,\, -\!\!\!\xrightarrow{\!\!\!\! m_2} \BQ/\BZ \lrt 1,
\end{equation}
where $m_2: \BQ/\BZ\lrt \BQ/\BZ$ is the map defined by $m_2(x)=2\,x$ for $x\in \BQ/\BZ.$ Thus the exact sequence (14.16) is
non-split.
\vskip 0.1cm For brevity, we put $Mp_{(c)}:=Mp(n,\BR)_c$. Since $[Sp(n,\BR),Sp(n,\BR)]=Sp(n,\BR),$ $[Mp_{(c)},Mp_{(c)}]$
sits in the exact sequence
\begin{equation}
1\lrt \{\pm 1 \} \cap [Mp_{(c)},Mp_{(c)}] \lrt [Mp_{(c)},Mp_{(c)}] \lrt Sp(n,\BR) \lrt 1.
\end{equation}
Assume $\{\pm 1 \} \cap [Mp_{(c)},Mp_{(c)}]$ is trivial. Then according to the above exact sequence (14.19), we have
an isomorphism $\phi:Sp(n,\BR)\lrt [Mp_{(c)},Mp_{(c)}] \neq Mp_{(c)}.$ Thus the exact sequence (14.16) is split because $q_c\circ \phi$
the identity map. This contradicts the fact that the exact sequence (14.16) is non-split. Hence we obtain
\begin{equation*}
\{\pm 1 \} \cap [Mp_{(c)},Mp_{(c)}]=\{\pm 1 \} \qquad {\rm{and}}\qquad [Mp_{(c)},Mp_{(c)}]=Mp_{(c)}.
\end{equation*}
$\hfill\square$

\vskip 0.2cm
\begin{corollary}
For a fixed element $\Omega\in \BH_n,$ we let
\begin{equation*}
U(\Omega)=\,\big\{\,M\in Sp(n,\BR)\,\big|\ M\cdot \Om=\Om\,\big\}.
\end{equation*}
Let $M_\Om\in Sp(n,\BR)$ such that $\Om=M\cdot (iI_n).$ Then $U(\Om)=M_\Om \,U(n) M_\Om^{-1}.$
If $P\in q_c^{-1}(U(\Om))$ such that $q_c(P)=M_\Om \,q_c(P_0)M_\Om^{-1}$ with $P_0\in q_c^{-1}(U(n)),$
then
\begin{equation*}
\big\{ B_c(P;\Om)\big\}^2=\,{\det}_c^* \big(q_c(P)\big),
\end{equation*}
\vskip 0.2cm\noindent
where ${\det}_c^*:U(n)\lrt \BC^*_1$ is the map defined by Formula (14.15).
\end{corollary}
\vskip 0.1cm\noindent
{\it Proof.} The case $\Om=i\,I_n$ has already been proved before. We note that $U(iI_n)=U(n).$
For $M=\begin{pmatrix} \  A & B \\ C & D \end{pmatrix}$ and $\Om\in\BH_n,$ we put
$J(M,\Om)=\,\det (C\Om+D).$ By definition, if $P\in q_c^{-1}(U(\Om))$ such that
$q_c(P)=M_\Om \,q_c(P_0)M_\Om^{-1}$ with $P_0\in q_c^{-1}(U(n))$ and
$q_c(P_0)=\begin{pmatrix} \  A & B \\ -B & A \end{pmatrix}\in U(n),$
then
\begin{eqnarray*}
\big\{ B_c(P;\Om)\big\}^2 &=& J(q_c(P),\Om)^{-m}=\,\big\{ J\big(M_\Om \,q_c(P_0)M_\Om^{-1},\Om\big)\big\}^{-m}\\
&=& \big\{ J(M_\Om\,q_c(P_0), iI_n)\, J(M_\Om^{-1},\Om)\big\}^{-m}\\
&=& \big\{ J(M_\Om,iI_n)\,J(M_\Om^{-1},\Om)\,J(q_c(P_0),iI_n)\big\}^{-m}\\
&=& \big\{ J(q_c(P_0),iI_n)\big\}^{-m}\\
&=& \big\{ \det (A-iB)\big\}^{-m}={\det}_c^* \big( q_c(P_0)\big).
\end{eqnarray*}
$\hfill\square$

\end{section}

\newpage

\begin{section}{{\large\bf Theta Series with Quadratic Forms}}
\setcounter{equation}{0}
\vskip 0.2cm
In this chapter, we review the theta series of several type.
\begin{definition}
A symmetric integral matrix $S$ of degree $m$ is said to be $\textsf{even}$ if ${}^t\xi\,S\,\xi\equiv 0\ \textrm{mod}\ 2$ for all $\xi\in\BZ^{(m,1)}$. The $\textsf{level}\ q$ of an even symmetric nonsingular matrix $S$ is defined to be the smallest positive integer such that $qS^{-1}$ is even.
\end{definition}
It is well known that if $S$ is positive definite even integral matrix of degree $m$ such that $\det S=1$, then $m$ is divisible by 8.
\begin{definition}
For a symmetric integral matrix $T$ of degree $n$ and
a symmetric integral matrix $S$ of degree $m$, we define
$$A(S,T):=\,\sharp\{ \,\xi\in\BZ^{(m,n)}\,|\ {}^t\xi\,S\,\xi\,=\,T\, \}.$$
\end{definition}
We observe that if $S$ is positive definite, $A(S,T)$ is finite. It is easy to see that $S_1$ and $S_2$ are equivalent, that is,
${}^tUS_1U=S_2$ for some $U\in GL(m,\BZ)$ if and only if $A(S_1,T)=A(S_2,T)$ for all $n$ and symmetric integral matrices $T$ of
degree $n$.

\vskip 0.2cm Let $S$ be a positive definite integral matrix of degree $m$. We define the theta series $\vartheta_S:\BH_n\lrt \BC$
by
\begin{equation}
\vartheta_S(\Omega)=\,\sum_{\xi\in \BZ^{(m,n)}} e^{\pi\,i\,\sigma(S\,\xi\,\Omega\,{}^t\xi)},\quad \Omega\in\BH_n.
\end{equation}
Then $\vartheta_S(\Omega)$ is a holomorphic function on $\BH_n.$ We see that
$$\vartheta_S(\Omega)=\sum_{T=\,{}^tT\geq 0} A(S,T)\,e^{\pi\,i\,\sigma(T\Omega)},$$
where $T$ runs over the set of all semipositive symmetric integral matrices of degree $n$.
\begin{theorem}
Let $S$ be a positive definite symmetric integral matrix of degree $m$. Then $\vartheta_S(\Omega)$ satisfies the
transformation formula
\begin{equation}
\vartheta_{S^{-1}}(-\Omega^{-1})=\,(\det S)^{\frac n2} \left( \det {{\Omega}\over i}\right)^{\frac m2}\,
\vartheta_S(\Omega)\qquad \textrm{for all}\ \Omega\in\BH_n.
\end{equation}
\end{theorem}
Here the function $h:\BH_n\lrt \BC$ given by
$$h(\Omega)=\,\left( \det {{\Omega}\over i}\right)^{\frac 12},\qquad \Omega\in\BH_n$$
is the function determined uniquely by the following properties
\vskip 0.1cm
(a) \ \ $h^2(\Omega)=\,\left( \det {{\Omega}\over i}\right),\quad \Omega\in\BH_n,$
\vskip 0.1cm
(b)\ \ $h(iY)=\,(\det Y)^{\frac 12}$ for any positive definite symmetric real matrix $Y$ of degree $n$.
\vskip 0.2cm
For a positive integer $m$, we define
$$ \left( \det {{\Omega}\over i}\right)^{\frac m2}=\,\left\{ \left( \det {{\Omega}\over i}\right)^{\frac 12}\right\}^m,
\qquad \Omega\in \BH_n.$$
\noindent {\it Proof.} For a fixed element $\Omega\in\BH_n$, we define $f:\BR^{(m,n)}\lrt \BC$ by
\begin{equation}
f(x)=\,\sum_{\xi\in \BZ^{(m,n)}} e^{\pi\,i\,\sigma(S(\xi+x)\,\Omega\,{}^t(\xi+x))},\quad x\in \BR^{(m,n)}.
\end{equation}
We observe that $f$ is well defined because the sum of the right hand side of (15.3) converges absolutely.
It is clear that if $x=(x_{ij})$ is a coordinate in $\BR^{(m,n)}$, then $f$ is periodic in $x_{ij}$ with period 1.
That is,
$$ f(x+\alpha)=f(x)\qquad \textrm{for all}\ \alpha\in \BZ^{(m,n)}.$$
Thus $f$ has the Fourier series
\begin{equation}
f(x)=\,\sum_{\alpha\in \BZ^{(m,n)}} c_\alpha\,e^{2\,\pi\,i\,\sigma(x\,{}^t\alpha)},\quad x\in \BR^{(m,n)},
\end{equation}
where
\begin{eqnarray*}
c_\alpha&=&\,\int_0^1\cdots\int_0^1 f(y)\,e^{-2\,\pi\,i\,\sigma (y\,{}^t\alpha)}\,dy\\
&=& \,\int_0^1\cdots\int_0^1 \sum_{\xi\in \BZ^{(m,n)}} e^{\pi\,i\,\sigma\{ S(\xi+y)\,\Omega\,{}^t(\xi+y)\}}\cdot
e^{-2\,\pi\,i\,\sigma (y\,{}^t\alpha)}\,dy\\
&=&\,\int_{\BR^{(m,n)}} e^{\pi\,i\,\sigma( S\,y\,\Omega\,{}^ty -2\,y\,{}^t\alpha)}\,dy\\
&=&\,\int_{\BR^{(m,n)}} e^{\pi\,i\,\sigma\{ S( y\,\Omega\,{}^ty -2\,y\,{}^t(S^{-1}\,\alpha))\} }\,dy\\
&=&\,(\det S)^{-{\frac n2}}\,\left( \det {{\Omega}\over i}\right)^{-{\frac m2}}\,e^{-\pi\,i\,\sigma (S^{-1}\,\alpha\,
\Omega^{-1}\,{}^t\alpha)} \quad \textrm{(by\ Lemma\ 14.2)}.
\end{eqnarray*}
According to Formulas (15.3) and (15.4),
\begin{eqnarray*}
\vartheta_S(\Omega)=f(0)&=&\,\sum_{\alpha\in \BZ^{(m,n)}} c_\alpha\\
&=&\,(\det S)^{-{\frac n2}}\,\left( \det {{\Omega}\over i}\right)^{-{\frac m2}}\,e^{\pi\,i\,\sigma \{ S^{-1}\alpha\,
(-\Omega^{-1})\,{}^t\alpha\}}\\
&=&\,(\det S)^{-{\frac n2}}\,\left( \det {{\Omega}\over i}\right)^{-{\frac m2}}\,
\vartheta_{S^{-1}}(-\Omega^{-1}).
\end{eqnarray*}
Consequently we obtain the formula (15.2).  $\hfill \square$

\vskip 0.2cm Let $S$ be an positive definite even integral symmetric matrix of degree $m$. Let $A$ and $B$ be $m\times n$ rational
matrices. We define the theta series
$$\vartheta_{S;A,B}(\Omega)=\,\sum_{\xi\in \BZ^{(m,n)}} e^{\pi\,i\,\sigma\{ S(\xi+{\frac 12}A)\,\Omega\,{}^t(\xi+{\frac 12}A)
\,+\,{}^tB\,\xi \} }.$$

\begin{theorem} Let $S$ be an positive definite even integral symmetric matrix of degree $m$. Let $A$ and $B$ be $m\times n$ rational
matrices. Then $\vartheta_{S;A,B}(\Omega)$ satisfies the transformation formula
\begin{equation}
\vartheta_{S^{-1};A,B}(-\Omega^{-1})=\,e^{-{\frac 12}\,\pi\,i\,\sigma(\,{}^tAB)}\,
(\det S)^{{\frac n2}}\,\left( \det {{\Omega}\over i}\right)^{{\frac m2}}\,\vartheta_{S;B,-A}(\Omega)
\end{equation}
\end{theorem}
\noindent for all $\Omega\in\BH_n.$
\vskip 0.1cm\noindent {\it Proof.} Following the argument of the proof of Theorem 15.3, we can obtain the formula (15.5).
We leave the detail to the reader. $\hfill\square$

\vskip 0.2cm
\begin{definition}
A holomorphic function $f:\BH_n\lrt \BC$ is called a Siegel modular form of weight $k\in\BZ$ if it satisfies
the following properties\,:
\vskip 0.1cm
1) \ $f(M\cdot\Omega)=\,\det(C\Omega+D)^k f(\Omega)\quad$ for $\gamma=\begin{pmatrix} A& B\\ C&D\end{pmatrix}\in \Gamma_n$.
\vskip 0.1cm
2) $f$ is bounded in the domain $Y\geq Y_0>0$ with $\Omega=X+i\,Y,\ X,Y$ real.
\end{definition}

We will give some examples of Siegel modular forms using the so-called thetanullwerte. For $a,b\in\BZ^n$, we consider
the thetanullwerte
\begin{equation}
\vartheta(\Om\,;a,b)=\,\sum_{\xi\in\BZ^n} e^{\pi\,i\,\sigma\{\,{}^t(\xi+{\frac 12}a)\,\Om\,(\xi+{\frac 12}a)\,+\,{}^tb\,\xi\,\}},
\quad \Om\in\BH_n.
\end{equation}
\begin{lemma} Let $a,b\in \BZ^n$. Then $\vartheta(\Om\,;a,b)$ satisfies the following properties
\vskip 0.3cm
(a)\ \ $\vartheta(\Om\,;a,b_1)=\vartheta(\Om\,;a,b_2)$ if $b_1\equiv b_2$\ mod 2.
\vskip 0.1cm
(b) If ${\widetilde a}\in\BZ^n,$ then $ \vartheta(\Om\,;a+2\,{\widetilde a},b) =\,(-1)^{{}^tb\,
{\widetilde a}}\, \vartheta(\Om\,;a,b)$.
\vskip 0.1cm
(c) \ \ $\vartheta(\Om\,;a,b)=\,(-1)^{{}^tab}\,\vartheta(\Om\,;a,b).$
\vskip 0.1cm
(d) \ \ $\vartheta(\Om\,;a,b)=0$ \ if\ \ ${}^tab\not\equiv 0$ mod 2.
\end{lemma}
\vskip 0.1cm
\noindent{\it Proof.} $(a)$ follows from  a direct computation. If we put $\xi_*=\xi+{\widetilde a}$,
\begin{eqnarray*}
\vartheta(\Om\,;a+2\,{\widetilde a},b)&=&\,\sum_{\xi\in\BZ^n} e^{\pi\,i\,\sigma\{\,{}^t(\xi+{\frac 12}a
+{\widetilde a})\,\Om\,(\xi+{\frac 12}a +{\widetilde a})\,+\,{}^tb\,\xi\,\}}\\
&=&\,\sum_{\xi_*\in\BZ^n} e^{\pi\,i\,\sigma\{\,{}^t(\xi_*+{\frac 12}a)\,\Om\,(\xi_*+{\frac 12}a)\,+\,{}^tb
(\xi_*-{\widetilde a})\,\}}\\
&=&\,e^{-\pi\,i\,{}^tb{\widetilde a}}\,\vartheta(\Om\,;a,b).
\end{eqnarray*}
\noindent Therefore we get the formula $(b)$.
If we substitute $\xi$ into $-\xi-a$, we obtain the formula $(c)$.
$(d)$ follows immediately from the formula $(c)$.   $\hfill \square$
\vskip 0.2cm A pair $\{a,b\}$ with $a,b\in \{0,1\}^n$ is called a $\textsf{theta\ characteristic}$.
A theta characteristic $\{a,b\}$ is said to be $\textsf{even}$ (resp. $\textsf{odd}$) if ${}^tab$ is even (resp. odd).
By induction on $n$, we can show that the number of even theta characteristics is $(2^n+1)\,2^{n-1}.$
\vskip 0.1cm Let $\gamma=\begin{pmatrix} A& B\\ C&D\end{pmatrix}\in \Gamma_n$ and let $\{a,b\}$ be a theta characteristic.
We define
\begin{equation}
\gamma\diamond \begin{pmatrix} a\\ b \end{pmatrix}:\equiv\,\begin{pmatrix} \ D& -C\\ -B & A\end{pmatrix}
\begin{pmatrix} a\\ b \end{pmatrix}\, +\,\begin{pmatrix} (C\,{}^tD)_0\\ (A\,{}^tB)_0 \end{pmatrix}\
\textrm{mod}\ 2,
\end{equation}
where $T_0$ is the column vector determined by the diagonal entries of an $n\times n$ matrix $T$.
\begin{theorem}
(1) The Siegel modular group $\Gamma_n$ acts on the set ${\mathscr C}$ of theta characteristics by
$$\begin{pmatrix} a\\ b \end{pmatrix} \mapsto \gamma\diamond \begin{pmatrix} a\\ b \end{pmatrix},
\quad \gamma\in\Gamma_n,\ \{a,b\}\in {\mathscr C}.$$
\noindent (2) The sign $(-1)^{{}^tab}$ of the theta characteristic $\{a,b\}$ is invariant under the action
(15.7) of $\Gamma_n$.
\vskip 0.1cm \noindent (3) $\Gamma_n$ acts on the set ${\mathscr C}^e$ of all even theta characteristics
transitively.
\vskip 0.1cm \noindent (4) If \,$\gamma=\begin{pmatrix} A& B\\ C&D\end{pmatrix}\in \Gamma_n,\ \Om\in \BH_n$ and
$\{a,b\}\in {\mathscr C}$, then we have
\begin{equation}
\vartheta^2(\gamma\cdot\Om\,;a,b)=\,\nu(\gamma)\,\det (C\Om+D)\,
\vartheta^2\big(\Om\,;{\widetilde a},{\widetilde b}\,\big),
\end{equation}
where
\vskip 0.1cm \ \ (a)\ \ \ \ $\nu(\gamma)^4=1$,
\vskip 0.1cm \ \ (b)\ \ \ \ $\begin{pmatrix} {\widetilde a}\\ {\widetilde b} \end{pmatrix}=
\gamma\diamond \begin{pmatrix} a\\ b \end{pmatrix}$.
\end{theorem}
\noindent {\it Proof.} By a direct computation, we prove the statement (a). It suffices to show the invariance of
the sign of $(-1)^{{}^tab}$ under the generators $t_S=\begin{pmatrix} I_n& S\\ 0& I_n\end{pmatrix}$ with $S=\,{}^tS
\in \BZ^{(n,n)}$ and $J_n$. By a simple computation,
\begin{eqnarray*}
\begin{pmatrix} I_n& S\\ 0& I_n\end{pmatrix}\diamond \begin{pmatrix} a\\ b\end{pmatrix}&\equiv&
\,\begin{pmatrix} a\\ b-S\,a+S_0\end{pmatrix} \ \textrm{mod}\ 2,\\
\begin{pmatrix} I_n& S\\ 0& I_n\end{pmatrix}\diamond \begin{pmatrix} a\\ b\end{pmatrix}&\equiv&
\begin{pmatrix} \ b\\ -a\end{pmatrix} \ \textrm{mod}\ 2.
\end{eqnarray*}
Therefore it is obvious that the sign of $(-1)^{{}^tab}$ is invariant under the actions of $t_S$ and $J_n.$
In order to prove the transitivity of $\Gamma_n$ on ${\mathscr C}^e$, first of all we have to prove the fact that
given an even characteristic $\{a,b\}\in {\mathscr C}^e$, there exists an element $\gamma=\,\begin{pmatrix} A& B\\
C&D\end{pmatrix}\in \Gamma_n$ such that
$$ \gamma\diamond \begin{pmatrix} 0\\ 0\end{pmatrix}=\,\begin{pmatrix} a\\ b\end{pmatrix},\ \ i.e.,
\ \ a\equiv (C\,{}^tD)_0,\ b\equiv (A\,{}^tB)_0\ \textrm{mod}\ 2.$$
We decompose
\begin{eqnarray*}
a&=&\, \begin{pmatrix} a_1\\ a_2\end{pmatrix},\quad a_1\in\BZ,\ a_2\in \BZ^{n-1},\\
b&=&\, \begin{pmatrix} b_1\\ b_2\end{pmatrix},\quad b_1\in\BZ,\ b_2\in \BZ^{n-1}.
\end{eqnarray*}
\vskip 0.2cm \noindent {\bf Case 1.}\ \ $a_1\, b_1=0$
\vskip 0.2cm Then $\{a_1,b_1\}$ is even and also $\{a_2,b_2\}$ is even. By induction on $n$, we can find $\gamma\in\Gamma_n$
such that $\gamma\diamond \begin{pmatrix} 0\\ 0\end{pmatrix}=\begin{pmatrix} a\\ b\end{pmatrix}.$
\vskip 0.2cm \noindent {\bf Case 2.}\ \ $a_1=b_1=1$
\vskip 0.2cm Since ${}^tab$ is even, there exists an index $\nu$ with $2\leq \nu\leq n$ such that $a_\nu=b_\nu=1.$
Therefore we can find a symmetric integral matrix $S=\,^{}tS\in\BZ^{(n,n)}$ so that
$$\begin{pmatrix} I_n& S\\ 0& I_n\end{pmatrix}\diamond \begin{pmatrix} a\\ b\end{pmatrix}\equiv
\,\begin{pmatrix} a\\ b-S\,a+S_0\end{pmatrix} \ \textrm{mod}\ 2$$
is an even theta characteristic satisfying the assumption of Case 1.
\vskip 0.2cm According to Case 1 and Case 2, we see that $\Gamma_n$ acts on ${\mathscr C}^e$ transitively.
\vskip 0.2cm The transformation formula (15.8) for the generator $J_n$ follows from the formula (15.5) with
$S=1$ and $m=1$. For a generator $t_S$ with $S=\,{}^tS\in\BZ^{(n,n)}$, it is easy to see that
\begin{equation}
\vartheta(\Omega+S\,;a,b)=\,e^{ {{\pi\,i}\over 4}\,{}^taSa}\,\,\vartheta (\Omega\,;a,b+Sa+S_0).
\end{equation}
In fact, (15.9) follows from the following simple fact that ${}^t\xi S\xi\equiv \,{}^tS_0\,\xi$ mod 2 for any
$\xi\in \BZ^n$ and $x^2\equiv x$ mod 2 for any $x\in\BZ.$ We know that $\vartheta(\Omega\,;0,0)\neq 0$ because
$\vartheta(i\,Y\,;0,0)>0.$
$\hfill\square$

\vskip 0.2cm
\begin{theorem}
We set
$$k_n= \begin{cases} 8 \ & \textrm{if}\quad n=1\\ 2 \ & \textrm{if}\quad n=2 \\
1 \ &  \textrm{if}\quad n\geq 3.\end{cases}$$
We define the function $\Delta^{(n)}(\Om)$ on $\BH_n$ by
$$\Delta^{(n)}(\Om):=\prod_{\{a,b\}} \vartheta(\Om\,;a,b)^{k_n},$$
where $\{a,b\}$ runs through even theta characteristics. Then $\Delta^{(n)}(\Om)$ is a nonzero Siegel modular form on
$\BH_n$ of weight 12, 10 and $(2^n+1)\,2^{n-2}$ respectively if $n=1,2$ and $n\geq 3$ respectively.
\end{theorem}
\noindent {\it Proof.} The proof can be found in \cite{F}. $\hfill\square$

\begin{theorem} Let $m$ be an even positive integer. Let $S$ be a positive definite even integral symmetric matrix of
degree $m$ and of level $q$. Then for all $\gamma=\begin{pmatrix} A& B\\ C&D\end{pmatrix}\in \Gamma_{n,0}(q)$ with
$\det D>0,$
\begin{equation*}
\vartheta_S(\gamma\cdot\Omega)=\,\nu_S (\gamma)\,\det (C\Omega+D)^{\frac m2}\,\vartheta_S(\Omega),\qquad \Omega\in\BH_n,
\end{equation*}
where
\begin{eqnarray*}
\nu_S(\gamma)&=&\,(\det D)^{{\frac m2}-mn}\,\sum_{\xi\in \BZ^{(m,n)}} e^{\pi\,i\,\sigma(BD^{-1}\,{}^t\xi \,S\,\xi)}\\
&=&\,( \textrm{sgn}\,\det D)^{\frac m2}\,\left( {(-1)^{\frac m2}\,\det S}\over {|\det D|}\right).
\end{eqnarray*}
Here $\left( {\frac ab}\right)$ denotes the generalized Legendre symbol.
\end{theorem}
\noindent {\it Proof.} The proof can be found in \cite{F},\,pp.\,302\,-303.  $\hfill \square$

\begin{theorem}
Let $m$ be an even positive integer. Let $S$ be a positive definite even integral symmetric matrix of
degree $m$ and of level $q$. Then $\vartheta_S(\Omega)$ is a modular form with respect to the principal congruence subgroup
$\Gamma_n(q)$ of $\Gamma_n.$
\end{theorem}
\noindent {\it Proof.} The proof follows from Theorem 15.9. $\hfill \square$

\begin{theorem}
Let $S$ be an positive definite even integral symmetric matrix of degree $m$. Let $A$ and $B$ be $m\times n$ rational
matrices. Then the theta series
$$\vartheta_{S;A,B}(\Omega)=\,\sum_{\xi\in \BZ^{(m,n)}} e^{\pi\,i\,\sigma \{ S(\xi+{\frac 12}A)\,\Omega\,{}^t(\xi+{\frac 12}A)
\,+\,{}^tB\,\xi\} }$$
is a modular form of weight ${\frac m2}$ with respect to a certain congruence subgroup $\Gamma_n(\ell)$ of $\Gamma_n$.
\end{theorem}
\noindent {\it Proof.} The proof can be found in \cite{AM}. $\hfill \square$

\end{section}

\newpage

\begin{section}{{\large\bf Theta Series in Spherical Harmonics}}
\setcounter{equation}{0}
\vskip 0.23cm
Let $S$ be a positive definite symmetric $m\times m$ rational matrix, and let $\alpha$ and $\beta$ be
an $m\times n$ rational matrix. We define the theta series
\begin{equation*}
\vartheta_S \!\left[ \begin{array}{c} \alpha \\ \beta \end{array} \right]:\,\BH_n\times\BC^{(m,n)}\,\lrt \,\BC
\end{equation*}
by
\begin{equation}
\vartheta_S \!\left[ \begin{array}{c} \alpha \\ \beta \end{array} \right]\!(\Omega,Z):=\,
\sum_{N\in\BQ^{(m,n)}} \chi \!\left[ \begin{array}{c} \alpha \\ \beta \end{array} \right]\!(N)\,
e^{\pi\,i\,\sigma(\,{}^t\!NSN\Omega\,+\,2\,{}^t\!NZ)},
\end{equation}
where $\Omega\in\BH_n,\ Z\in \BC^{(m,n)}$ and
$$\chi\! \left[ \begin{array}{c} \alpha \\ \beta \end{array} \right]\!(N)=\,\begin{cases} \ \ \ \ \ 1 \ & \textrm{if}\
N-\alpha\not\in \BZ^{(m,n)} \\ e^{2\,\pi\,i\,\sigma (\,{}^t\!NB)}\ & \textrm{otherwise}.\end{cases}$$

Let $ {\mathfrak P}_{m,n}$ be the algebra of complex valued polynomial functions on $\BC^{(m,n)}$. We take a
coordinate $Z=(z_{kj})$ in $\BC^{(m,n)}$.

\begin{definition} Let $S,\ \alpha$ and $\beta$ be as above. For a homogeneous polynomial $P\in {\mathfrak P}_{m,n}$,
we define
\begin{eqnarray}
\quad \quad\quad \vartheta_{S,P}\! \left[ \begin{array}{c} \alpha \\ \beta \end{array} \right]\!(\Omega,Z)&=&\,
\sum_{N\in\BQ^{(m,n)}} \chi\! \left[ \begin{array}{c} \alpha \\ \beta \end{array} \right]\!(N)\,P(N)\,
e^{\pi\,i\,\sigma(\,{}^t\!NSN\Omega\,+\,2\,{}^t\!NZ)},\\
\vartheta_{S,P}(\Omega,Z)&=&\,\vartheta_{S,P}\! \left[ \begin{array}{c} 0 \\ 0 \end{array} \right]\!(\Omega,Z),\\
\vartheta_{S,P}\! \left[ \begin{array}{c} \alpha \\ \beta \end{array} \right]\!(\Omega)&=&\,
\vartheta_{S,P}\! \left[ \begin{array}{c} \alpha \\ \beta \end{array} \right]\!(\Omega,0).
\end{eqnarray}
\end{definition}

For any homogeneous polynomial $P$ in ${\mathfrak P}_{m,n}$, we put
$$ P(\partial)\,=\,P\left( {{\partial}\over {\partial z_{kj}}} \right),\quad 1 \leq k\leq m,\ 1\leq j\leq n.$$
Then we get
\begin{eqnarray}
& &  P(\partial)\vartheta_S \!\left[ \begin{array}{c} \alpha \\ \beta \end{array} \right]\!(\Omega,Z)\,\\
&=&\,
\sum_{N\in\BQ^{(m,n)}} \chi\! \left[ \begin{array}{c} \alpha \\ \beta \end{array} \right]\!(N)\,P(2\,\pi i N)\,
e^{\pi\,i\,\sigma(\,{}^t\!NSN\Omega\,+\,2\,{}^t\!NZ)}.\nonumber
\end{eqnarray}
\begin{definition}
Let $T=(t_{kj})$ be the inverse matrix of $S$. Then a polynomial $P$ in ${\mathfrak P}_{m,n}$ is said to be
$ \textsf{pluriharmonic}$ with respect to $S$ if it satisfies the equations
$$ \sum_{k,l=1}^m { {\partial^2 P}\over {\partial z_{ki} \partial z_{lj}} }\,t_{kl}\,=\,0
\qquad \textrm{for all}\ i,j=1,2,\cdots,n.$$
\end{definition}

\begin{theorem}
Let $S,\ \alpha$ and $\beta$ be as above. Then for all $\begin{pmatrix} A & B\\ C & D \end{pmatrix}$ in a suitable
subgroup $\Gamma$ of $\Gamma_n$, we have
\begin{eqnarray}
& & \vartheta_S \!\left[ \begin{array}{c} \alpha \\ \beta \end{array} \right]\!\big( (A\Omega+B)(C\Omega+D)^{-1},
Z(C\Omega+D)^{-1}\big)\\
&=&\,\det (C\Omega+D)^{\frac m2}\,e^{\pi\,i\,\sigma (Z(C\Omega+D)^{-1}C\,{}^t\!Z S^{-1})}\,
\vartheta_S\! \left[ \begin{array}{c} \alpha \\ \beta \end{array} \right]\!(\Omega,Z)\nonumber
\end{eqnarray}
\end{theorem}
\noindent {\it Proof.} $Sp(n,\BR)$ acts on the homogeneous space $\BH_n\times \BC^{(m,n)}$ by
$$M\cdot (\Omega,Z)\,=\,\big( (A\Om+B)(C\Om+D)^{-1},\,Z(C\Om+D)^{-1}\big),$$
where $M=\begin{pmatrix} A & B \\ C & D \end{pmatrix}\in Sp(n,\BR),\ \Om\in \BH_n$ and $Z\in \BC^{(m,n)}. $
It is known that $Sp(n,\BR)$ is generated by the translations $t_b$ with $b\,=\,{}^tb$ and the inversion $\sigma_n$.
Thus it suffices to prove the functional equation (16.6) for the generators $t_b$ and $\sigma_n$ in a suitable congruence
subgroup $\Gamma$ of $\Gamma_n.$
\vskip 0.1cm For $t_b\,=\,\begin{pmatrix} I_n & b \\ 0 & I_n \end{pmatrix}\in Sp(n,\BR)$,
\begin{eqnarray*}
\vartheta_S \!\left[ \begin{array}{c} \alpha \\ \beta \end{array} \right]\!(\Omega+b,Z)\,&=&\,
\sum_{N\in\BQ^{(m,n)}} \chi \!\left[ \begin{array}{c} \alpha \\ \beta \end{array} \right]\!(N)\,
e^{\pi\,i\,\sigma(\,{}^t\!NSN(\Omega+b)\,+\,2\,{}^t\!NZ)}\\
&=&\,\vartheta_S \!\left[ \begin{array}{c} \alpha \\ \beta \end{array} \right]\!(\Omega,Z)
\end{eqnarray*}
if we choose suitable $b$'s so that $e^{\pi\,i\,\sigma({}^t\!NSNb)}=1.$ This is possible because $S,\ \alpha$ and $\beta$ are
$ \textit{rational}$ matrices.
\vskip 0.1cm For the inversion $\sigma_n=\begin{pmatrix} 0 & -I_n \\ I_n & 0 \end{pmatrix}$, we can prove the functional
equation (16.6) following the argument in the proof of Theorem 15.3. We leave the details to a reader. Another representation
theoretic proof can be found in \cite{Mum2}
$\hfill \square$

\vskip 0.5cm
\begin{lemma} Let ${\mathfrak P}_N:=\,\BC [X_1,\cdots,X_N].$ For $P\in {\mathfrak P}_N,$ we let $P(\partial)$ denote the differential
operator $P\left( { {\partial}\over {\partial X_1} },\cdots,{{\partial}\over {\partial X_N}}\right)$.
For $P,\,Q\in {\mathfrak P}_N,$ we define
$$ \langle P,Q\rangle=(P(\partial)Q)(0).$$
Then $\langle \ ,\ \rangle$ is a symmetric nondegenerate bilinear form on ${\mathfrak P}_N$ which satisfies the property
$\langle P, QR \rangle= \langle Q(\partial)P, R \rangle =\langle R(\partial)P, Q \rangle$ for all $P,Q,R\in
{\mathfrak P}_N.$
\end{lemma}
\noindent {\it Proof.} We first observe that
$$\big\langle X_1^{a_1}\cdots X_N^{a_N}, X_1^{b_1}\cdots X_N^{b_N} \big\rangle\,=\,
\begin{cases} a_1 !\cdots a_N!\ & \textrm{if}\ (a_1,\cdots,a_N)=(b_1,\cdots,b_N),\\
\ \ \ \ \ \ \ \ 0\ \ \ \ \ & \textrm{otherwise}.\end{cases}$$
Thus $\langle P,Q\rangle$ is a symmetric nondegenerate bilinear form on ${\mathfrak P}_N$. Similarly
$\langle P, QR \rangle= \langle Q(\partial)P, R \rangle =\langle R(\partial)P, Q \rangle$ is easily shown for monomials
$P,\,Q,\,R.$ Hence we complete the proof. $\hfill \square$

\vskip 0.5cm
\begin{lemma} Let ${\mathfrak H}(S)\subset {\mathfrak P}_{m,n}$ be the space of pluriharmonic polynomials with respect
to $S$, and $I\subset {\mathfrak P}_{m,n}$ be the ideal generated by the $h_{ij}=\sum_{k,l=1}^m t_{kl}z_{ki}z_{lj}$
for all $i,\,j=1,\cdots,n$, where $T=(t_{kl})=S^{-1}$ as before in Definition 16.2. Then ${\mathfrak H}(S)=I^{\perp}$
with respect to the pairing $\langle \ ,\ \rangle$ introduced in Lemma 16.4, and
$$ {\mathfrak P}_{m,n}= {\mathfrak H}(S) \oplus I \qquad ( \textrm{orthogonal sum}).$$
\end{lemma}
\noindent {\it Proof.} Let $P\in {\mathfrak P}_{m,n}$. Then $\langle f h_{ij},P \rangle =\big( f(\partial)h_{ij}
(\partial)P\big)(0)=0$ for all $f\in {\mathfrak P}_{m,n}$ if and only if $h_{ij}
(\partial)P=0$ for all $i,j$ if and only if $P$ is pluriharmonic with respect to $S.$ Thus ${\mathfrak H}(S)=I^{\perp}$.
Let ${\mathfrak P}_{m,n}(\BR)=\BR[Z_{11},Z_{12},\cdots,Z_{mn}].$
By the same argument, we have ${\mathfrak H}(S)_\BR=I_\BR^{\perp},$
where ${\mathfrak H}(S)_\BR={\mathfrak H}(S)\cap {\mathfrak P}_{m,n}(\BR)$ and $I_\BR=I \cap {\mathfrak P}_{m,n}(\BR)$.
It is easy to see that $\langle \ ,\ \rangle$ is positive definite on ${\mathfrak P}_{m,n}(\BR)$. So
${\mathfrak P}_{m,n}(\BR)={\mathfrak H}(S)_\BR \oplus I_\BR.$ Therefore we have ${\mathfrak P}_{m,n}= {\mathfrak H}(S)
\oplus I.$ $\hfill\square$

\vskip 0.5cm
\begin{lemma}
If $P$ is a pluriharmonic polynomial in ${\mathfrak H}(S)\subset {\mathfrak P}_{m,n}$, then
\begin{equation}
\left( P(\partial)\left[ g(Z)\,e^{\sigma(ZC\,{}^t\!ZS^{-1})}\right]\right)(0)\,=\,\big( P(\partial)g(Z)\big)(0)
\end{equation}
for any $C\in \BC^{(n,n)}$ and any analytic function $g$ defined in a neighborhood of $0$.
\end{lemma}
\noindent {\it Proof.} We put $h(Z)=\sigma(ZC\,{}^t\!ZS^{-1})$ and $T=(t_{kl})=S^{-1}.$ It suffices to prove the
formula (16.7) for any polynomials $g(Z).$
\begin{eqnarray*}
\left( P(\partial)\left[ g(Z)\,e^{h(Z)}\right]\right)(0)
&=&\,\sum_{n=0}^\infty {1\over {n!}}\,\left( P(\partial)\big[ g(Z)h(Z)^n\big]\right)(0)\\
&=&\,\sum_{n=0}^\infty {1\over {n!}}\,\langle P, gh^n\rangle\\
&=&\,\sum_{n=0}^\infty {1\over {n!}}\,\langle h(\partial)^nP, g\rangle\qquad ( \textrm{by Lemma 16.4}).
\end{eqnarray*}
By the way, $h(\partial)P=0$ because $P$ is pluriharmonic. Indeed, if we put $C=(c_{ij})$ and $Z=(z_{ki})$,
then we have
\begin{eqnarray*}
h(Z)&=&\,\sigma (ZC\,{}^t\!ZS^{-1})\\
&=&\,\sum_{i,j=1}^n c_{ij}\,\left( \sum_{k,l=1}^m t_{kl}\,z_{ki}\,z_{lj}\right).
\end{eqnarray*}
We put $ f_{ij}(Z)= \sum_{k,l=1}^m t_{kl}\,z_{ki}\,z_{lj}$. Then $h(\partial)P=\,\sum_{i,j=1}^n c_{ij}\,
\big( f_{ij}(\partial)P\big)=0$ because $P$ is pluriharmonic. Therefore we get
\begin{equation*}
\left( P(\partial)\left[ g(Z)\,e^{h(Z)}\right]\right)(0)=\langle P,g\,\rangle=\,
\big( P(\partial)g(Z)\big)(0).
\end{equation*}
$\hfill\square$
\vskip 0.5cm
\noindent {\bf Corollary.} If $P$ is a pluriharmonic polynomial in ${\mathfrak H}(S)\subset {\mathfrak P}_{m,n}$ and
$C$ is an $n\times n$ $\textsf{symmetric}$ complex matrix, then
\begin{equation*}
P(\partial)e^{\sigma(ZC\,{}^t\!ZS^{-1})}=\,P(2\,C\,{}^t\!ZS^{-1})\,e^{\sigma(ZC\,{}^t\!ZS^{-1})}.
\end{equation*}
\noindent {\it Proof.} We put $h(Z)=e^{\sigma(ZC\,{}^t\!ZS^{-1})}.$ For any $A\in\BC^{(m,n)}$, we let
$$f(Z)=h(Z+A)=\,h(Z)\,h(A)\,g(Z),$$
where $g(Z)=\,e^{2\,\sigma(AC\,{}^t\!ZS^{-1})}.$ Then
\begin{eqnarray*}
\big( P(\partial)h(Z)\big)(A)\,&=&\,\big( P(\partial)f(Z)\big)(0)\\
&=&\, h(A)\,\big( P(\partial)[h(Z)\,g(Z)]\big)(0)\\
&=&\,h(A)\,\big( P(\partial)g(Z)\big)(0)\qquad ( \textrm{Lemma 16.6}).
\end{eqnarray*}
But
\begin{equation*}
{ {\partial g}\over {\partial z_{ki}} } =\,(2\,C\,{}^t\!A\,S^{-1})_{ki}\,g(Z).
\end{equation*}
By a repeated application of this, we have
$$ P(\partial)g(Z)=\,P(2\,C\,{}^t\!A\,S^{-1})\,g(Z).$$
Therefore
\begin{eqnarray*}
\big( P(\partial)h(Z)\big)(A)\,&=&\,h(A)\,\big( P(\partial)g(Z)\big)(0)\\
&=&\,h(A)\,P(2\,C\,{}^t\!A\,S^{-1})\,g(0)\\
&=&\,h(A)\,P(2\,C\,{}^t\!A\,S^{-1}).
\end{eqnarray*}
Hence $P(\partial)h(Z)\,=\,P(2\,C\,{}^t\!A\,S^{-1})\,h(Z).$
$\hfill\square$

\vskip 0.5cm
\begin{lemma} Let $f$ be an analytic function on $\BC^{(m,n)}$ and let $P\in {\mathfrak P}_{m,n}$.
For $A\in\BC^{(n,n)}$ and $B\in\BC^{(m,m)}$, we let
$$ f_{A,B}(Z)=f(BZA) \qquad \textrm{and}\qquad P_{A,B}(Z)=P(\,{}^tBZA^{-1}). $$
Then
\begin{equation*}
P(\partial)f_{A,B}(Z)=\big( P_{A,B}(\partial)f\big)(BZA).
\end{equation*}
In particular, $\langle P, f_{A,B}\rangle =\langle P_{A,B},f\rangle.$
\end{lemma}
\noindent {\it Proof.} We let $A=(a_{ij})\in \BC^{(n,n)},\ b=(b_{kl})\in \BC^{(m,m)}$ and $Z=(z_{lp}).$
By an easy computation, we get
\begin{equation*}
{ {\partial f_{A,B}}\over {\partial z_{lp}} }(Z)\,=\,\sum_{k=1}^m\sum_{i=1}^n b_{kl}\,a_{pi}\,
{ {\partial f}\over {\partial z_{ki}} }(BZA),\quad 1\leq l\leq m,\ 1\leq p\leq n.
\end{equation*}
We put
$$ {\widetilde Z}=\,^tBZA^{-1}\quad \ \textrm{with}\ \ {\widetilde Z}=\big( {\widetilde z}_{lp}\big).$$
Since
$$ { {\partial}\over {\partial {\widetilde z}_{lp}} }\,=\,
\sum_{k=1}^m\sum_{i=1}^n b_{kl}\,a_{pi}\,{ {\partial}\over {\partial z_{ki}} },
\quad 1\leq l\leq m,\ 1\leq p\leq n,$$
we have
$$ { {\partial f_{A,B}}\over {\partial z_{lp}} }(Z)\,=\,{ {\partial f}\over {\partial {\widetilde z}_{lp}} }
(BZA)  \qquad \textrm{for all}\ l,p.$$
Therefore we have $P(\partial)f_{A,B}(Z)=\big( P_{A,B}(\partial)f\big)(BZA).$
$\hfill\square$.

\vskip 0.5cm
\begin{lemma}
$GL(n,\BC)\times O(S)$ acts on ${\mathfrak H}_{m,n}$ by
\begin{equation}
(A,B)P(Z)\,=\,P(B^{-1}ZA),
\end{equation}
where $A\in GL(n,\BC),\ B\in O(S)$ and $P\in {\mathfrak P}_{m,n}.$
The space ${\mathfrak H}(S)$ of pluriharmonic polynomials in ${\mathfrak P}_{m,n}$ is invariant under the action
(16.8).
\end{lemma}
\noindent {\it Proof.} According to Lemma 16.5, $I={\mathfrak H}(S)^{\perp}$ is the ideal of ${\mathfrak P}_{m,n}$
generated by $h_{ij}(Z)=\sum_{k,l=1}^m t_{kl}\,z_{ki}\,z_{lj}$ for all $i,\,j.$ So by Lemma 16.7, it suffices to
show that $h_{ij}(ZA)$ and $h_{ij}(BZ)$ belong to $I$ for all $A\in GL(n,\BC)$ and $B\in O(S).$ If $A=(a_{ij})\in
GL(n,\BC),\ Z=(z_{ki})\in\BC^{(m,n)}$ and $T=(t_{kl})=S^{-1},$ then
\begin{eqnarray*}
h_{ij}(ZA)&=&\,\sum_{p,q=1}^n a_{pi}\,a_{qj}\,\left( \sum_{k,l=1}^m t_{kl}z_{kp}z_{lq}\right)\\
&=&\,\sum_{p,q=1}^n a_{pi}\,a_{qj}\,h_{pq}(Z)\in I.
\end{eqnarray*}
If $B=(b_{kl})\in O(S),$ then
\begin{eqnarray*}
h_{ij}(BZ)&=&\,\sum_{p,q=1}^n z_{pi}\,z_{qj}\,\left( \sum_{k,l=1}^m t_{kl}b_{kp}b_{lq}\right)\\
&=&\,\sum_{p,q=1}^n z_{pi}\,z_{qj}\,\big( {}^tBTB\big)_{pq}.
\end{eqnarray*}
Since $B\in O(S)$, we have $T=\,{}^tBTB.$ Indeed $BS\,{}^tB=$ and hence ${}^tB\,S^{-1}\,B^{-1}=S^{-1}.$ Thus
$ {}^tBTB=T.$ Hence we have $h_{ij}(BZ)=\,\sum_{p,q=1}^n t_{pq}\,z_{pi}\,z_{qj}=\,h_{ij}(Z)\in I.$ Therefore we
complete the proof. $\hfill\square$

\vskip 0.5cm
\begin{theorem} Let $S,\ \alpha$ and $\beta$ be as above. Let $P$ be a pluriharmonic polynomial in ${\mathfrak P}_{m,n}$
with respect to $S$. Then
\begin{equation}
\vartheta_{S,P}\! \left[ \begin{array}{c} \alpha \\ \beta \end{array} \right]\!(\Omega)\,=\,
\det(C\Omega+D)^{-\frac m2}\,\vartheta_{S,{\widetilde P}}\! \left[ \begin{array}{c} \alpha \\ \beta \end{array} \right]\!\big( (A\Omega+B)(C\Omega+D)^{-1}\big),
\end{equation}
where ${\widetilde P}(Z)\,=\,P(Z(C\Omega+D)),$ for all $\begin{pmatrix} A & B\\ C & D \end{pmatrix}$ in a suitable
subgroup $\Gamma$ of $\Gamma_n.$
\end{theorem}
\noindent {\it Proof.} Let $P$ be a homogeneous pluriharmonic polynomial of degree $k$. Then according to Formula (16.5),
we get
\begin{eqnarray*}
& & \,(2\,\pi\,i)^{-k}P(\partial)\vartheta_S \!\left[ \begin{array}{c} \alpha \\ \beta \end{array} \right]\!(\Omega,Z)\\
&=&\, (2\,\pi\,i)^{-k}\,\sum_{N\in\BQ^{(m,n)}} \chi\! \left[ \begin{array}{c} \alpha \\ \beta \end{array} \right]\!
(N)\,P(2\,\pi i N)\,
e^{\pi\,i\,\sigma(\,{}^t\!NSN\Omega\,+\,2\,{}^t\!NZ)}\\
&=& \,\sum_{N\in\BQ^{(m,n)}} \chi\! \left[ \begin{array}{c} \alpha \\ \beta \end{array} \right]\!(N)\,P(N)\,
e^{\pi\,i\,\sigma(\,{}^t\!NSN\Omega\,+\,2\,{}^t\!NZ)}\\
&=&\,\vartheta_{S,P}\! \left[ \begin{array}{c} \alpha \\ \beta \end{array} \right]\!(\Omega,Z).
\end{eqnarray*}
Here the fact that $P$ is homogeneous of degree $k$ was used in the second equality. Putting $Z=0$, we get
\begin{equation}
(2\,\pi\,i)^{-k}\left( P(\partial)\vartheta_S \!\left[ \begin{array}{c} \alpha \\ \beta \end{array} \right]\right)\!(\Omega,0)\,
=\,\vartheta_{S,P}\! \left[ \begin{array}{c} \alpha \\ \beta \end{array} \right]\!(\Omega).
\end{equation}
By Theorem 16.3,
\begin{eqnarray}
\quad\qquad\vartheta_S \!\left[ \begin{array}{c} \alpha \\ \beta \end{array} \right]\!(\Omega,Z)\!\!\!&=&\!\!\!
\, \det (C\Om+D)^{-{\frac m2}} \nonumber\\
& & \,\times
e^{-\pi\,i\,\sigma (Z(C\Omega+D)^{-1}C\,{}^t\!Z S^{-1})}  \\
& &\,\times\, \vartheta_S\! \left[ \begin{array}{c} \alpha \\ \beta \end{array} \right]\!\big( (A\Omega+B)(C\Omega+D)^{-1},
Z(C\Omega+D)^{-1}\big).\nonumber
\end{eqnarray}
If we apply the differential operator $(2\,\pi\,i)^{-k}P(\partial)$ to both sides of Formula (16.11) and put $Z=0$,
according to Formula (16.10), Lemma 16.6 and Lemma 16.7, we obtain
\begin{eqnarray*}
& &\,\vartheta_{S,P}\! \left[ \begin{array}{c} \alpha \\ \beta \end{array} \right]\!(\Omega)\\
&=&\, (2\,\pi\,i)^{-k}\, \det (C\Om+D)^{-{\frac m2}}\,
\left[ P(\partial) \vartheta_S\! \left[ \begin{array}{c} \alpha \\ \beta \end{array} \right]\!\big( (A\Omega+B)(C\Omega+D)^{-1},
Z(C\Omega+D)^{-1}\big)     \right]_{Z=0} \\
&=&\,(2\,\pi\,i)^{-k}\, \det (C\Om+D)^{-{\frac m2}}\,
\left[ {\widetilde P}(\partial)  \vartheta_S\!\left[ \begin{array}{c} \alpha \\ \beta \end{array} \right]\!\big( (A\Omega+B)(C\Omega+D)^{-1},
Z\big)     \right]_{Z=0} \\
&=&\, \det (C\Om+D)^{-{\frac m2}}\,\vartheta_{S,{\widetilde P}}\! \left[ \begin{array}{c} \alpha \\ \beta \end{array} \right]
\!\big( (A\Omega+B)(C\Omega+D)^{-1},0\big)\\
&=&\, \det (C\Om+D)^{-{\frac m2}}\,\vartheta_{S,{\widetilde P}}\! \left[ \begin{array}{c} \alpha \\ \beta \end{array} \right]
\!\big( (A\Omega+B)(C\Omega+D)^{-1}\big),
\end{eqnarray*}
where $ {\widetilde P}(Z)\,=\,P(Z(C\Omega+D)).$ We note that we used Formula (16.6) and Lemma 16.6 in the first equality,
and Lemma 16.7 in the second equality. In the third equality we used the fact that $ {\widetilde P}$ is homogeneous of
degree $k$. Consequently we complete the proof.
$\hfill \square$

\vskip 0.5cm
\begin{definition} Let $(\rho,V_\rho)$ be a finite dimensional rational representation of $GL(n,\BC)$. A vector valued
function $f:\BH_n\lrt V_\rho$ is called a {\it modular form} with respect to $\rho$ if it is a holomorphic function on
$\BH_n$ such that
$$ f((A\Omega+B)(C\Omega+D)^{-1})\,=\,\rho(C\Omega+D)f(\Omega),\quad \Omega\in\BH_n$$
for all $\begin{pmatrix} A & B\\ C & D \end{pmatrix}$ in a suitable congruence subgroup of $ \Gamma_n.$
\end{definition}

\vskip 0.1cm We recall that ${\mathfrak H}(S)$ denotes the space of all pluriharmonic polynomials in ${\mathfrak P}_{m,n}$
with respect to $S$. Let $W$ be some $GL(n,\BC)$-stable subspace of ${\mathfrak H}(S)$. We define the $W^*$-valued
function $\vartheta_W\! \left[ \begin{array}{c} \alpha \\ \beta \end{array} \right]:\,\BH_n\lrt W^*$ by
\begin{equation}
\left( \vartheta_{W}\! \left[ \begin{array}{c} \alpha \\ \beta \end{array} \right]\!(\Omega)\right)(P):=\,
\vartheta_{S,P}\! \left[ \begin{array}{c} \alpha \\ \beta \end{array} \right]\!(\Omega)
\end{equation}
for all $\Omega\in \BH_n$ and $P\in W\subset {\mathfrak H}(S).$ Here $W^*$ denotes the dual space of $W$.

\vskip 0.2cm Now we introduce the homogeneous line bundle ${\mathcal L}^{\frac 12}$ over $\BH_n$. First of all we consider the double
covering $\widetilde{GL(n,\BC)}$ of $GL(n,\BC)$ defined by
$$ \widetilde{GL(n,\BC)}\,=\,\left\{\, (g,\alpha)\,|\ \alpha^2\,=\,\det (g),\ \,g\in GL(n,\BC),\
\alpha\in\BC^*\,\right\}$$
equipped with the multiplication
$$ (g_1,\alpha_1) (g_2,\alpha_2)\,=\,(g_1g_2,\alpha_1\alpha_2),\quad g_1,g_2\in GL(n,\BC),\ \alpha_1,\alpha_2\in\BC^*.$$
Let $\rho$ be a one-dimensional representation of $GL(n,\BC)$ defined by
$$\rho(g,\alpha)\,=\,\alpha\,=\,(\det (g))^{\frac 12},\quad g\in GL(n,\BC),\ \alpha\in \BC^*.$$
Then $\rho$ yields the homogeneous line bundle on $\BH_n$, denoted by ${\mathcal L}^{\frac 12}$. The complex manifold
$$ {\mathcal L}^{\frac 12}\,=\,\BH_n\times\BC$$
is a holomorphic line bundle over $\BH_n$ with the action of the metaplectic group $Mp(n,\BR)$ given by
$$ \widetilde{M}\cdot (\Omega,z)\,=\,\big( (A\Om+B)(C\Om+D)^{-1},\,\det(C\Om+D)^{1/2}\,z\,\big),\quad \widetilde{M}\in Mp(n,\BR),$$
where $\begin{pmatrix} A & B\\ C& D \end{pmatrix}\in Sp(n,\BR)$ is the image of $\widetilde{M}$ under the surjective homomorphism of
$Mp(n,\BR)$ onto $Sp(n,\BR).$ For a positive integer $k$, we define
$$ {\mathcal L}^{\frac k2}\,=\,\big( {\mathcal L}^{\frac 12}\big)^{\otimes k}\,=\,{\mathcal L}^{\frac 12}\otimes\cdots
\otimes {\mathcal L}^{\frac 12}\quad (\,k\!-\!\textrm{times}\,).$$
Let $\tau$ be the representation of $GL(n,\BC)$ on $W$ defined by
\begin{equation*}
\big( \tau(g)P\big)(Z):\,=\,P(Zg),\quad g\in GL(n,\BC),\ P\in W,\ Z\in \BC^{(m,n)}.
\end{equation*}
We observe that if $\widetilde{P}$ is a homogeneous pluriharmonic polynomial given by Theorem 16.9, then
$\widetilde{P}\,=\,\tau(C\Om+D)P.$ Let $\tau^*$ be the contragredient of $\tau$. That is,
\begin{equation*}
\big( \tau^*(g)\ell\big)(P)\,=\,\ell \big( \tau(g)^{-1} P\big),\quad g\in GL(n,\BC),\ \ell\in W^*,\ P\in W.
\end{equation*}

\begin{theorem} Let $\alpha$ and $\beta$ as above. Then the function $\vartheta_{W}\! \left[ \begin{array}{c}
\alpha \\ \beta \end{array} \right]\!(\Omega)$ defined in (16.12) is a modular form with values in $W^*\otimes
{\mathcal L}^{\frac m2}$ with respect to the representation $\tau^*\otimes \det^{\frac m2}$ for
a suitable congruence subgroup $\Gamma$. For any $W$ and $\Omega$, it is non-zero for
suitable $\alpha$ and $\beta.$
\end{theorem}
\noindent {\it Proof.} By Theorem 16.9, for all $P\in W\subset {\mathfrak H}(S)$ and for all
$\begin{pmatrix} A & B\\ C& D \end{pmatrix}$ in a suitable congruence subgroup $\Gamma$ of $\Gamma_n$, we get
\begin{eqnarray*}
& &\,\left(\vartheta_{W}\! \left[ \begin{array}{c} \alpha \\ \beta \end{array} \right]\!(\Omega)\right)(P)\\
&=&\, \det (C\Om+D)^{-{\frac m2}}\cdot
 \vartheta_{S,\widetilde{P}}  \! \left[ \begin{array}{c} \alpha \\ \beta \end{array} \right]\!\big( (A\Omega+B)(C\Omega+D)^{-1}
\big) \\
&=& \, \det (C\Om+D)^{-{\frac m2}}\cdot
 \left( \vartheta_{W}  \! \left[ \begin{array}{c} \alpha \\ \beta \end{array} \right]\!\big( (A\Omega+B)(C\Omega+D)^{-1}
\big) \right) (\widetilde{P})    \\
&=&\, \det (C\Om+D)^{-{\frac m2}}\cdot
 \left( \vartheta_{W}  \! \left[ \begin{array}{c} \alpha \\ \beta \end{array} \right]\!\big( (A\Omega+B)(C\Omega+D)^{-1}
\big) \right) (\tau(C\Om+D)P)    \\
&=&\, \det (C\Om+D)^{-{\frac m2}}\cdot
 \left( \tau^*(C\Om+D)^{-1}\vartheta_{W}  \! \left[ \begin{array}{c} \alpha \\ \beta \end{array} \right]\!\big( (A\Omega+B)(C\Omega+D)^{-1}
\big) \right) (P),
\end{eqnarray*}
where $\widetilde{P}$ is a homogeneous pluriharmonic polynomial defined by $\widetilde{P} (Z)\,=\,P(Z(C\Om+D)).$ Therefore
$$\vartheta_W\! \left[ \begin{array}{c} \alpha \\ \beta \end{array} \right]\!\big( (A\Omega+B)(C\Omega+D)^{-1}\big)\,=\,
\det (C\Om+D)^{\frac m2}\cdot
 \tau^*(C\Om+D)\vartheta_W\! \left[ \begin{array}{c} \alpha \\ \beta \end{array} \right]\!(\Omega).$$
Hence $\vartheta_W\! \left[ \begin{array}{c} \alpha \\ \beta \end{array} \right]\!(\Omega)$ is a modular form on $\BH_n$
with values in $W^*\otimes {\mathcal L}^{\frac m2}$ with respect to a suitable congruence subgroup $\G$ of $\G_n.$
$\hfill \square$

\vskip 0.5cm
\begin{remark} Using Theorem 16.11, we can prove that for all $n\geq 2$ and $1\leq r \leq n-1,$ there are congruence subgroups $\Gamma\subset \Gamma_n$
and $\Gamma$-invariant non-vanishing holomorphic $k$-forms on $\BH_n$, where $k= {{n(n+1)}\over 2}-
{{r(r+1)}\over 2}.$ The proof can be found in \cite{Mum2}. This fact was proved by Freitag and Stillman.
\end{remark}

\begin{definition} Let $(\rho,V_\rho)$ be a finite dimensional rational representation of $GL(n,\BC)$.
A $ \textsf{pluriharmonic form}$ with respect to $\rho$ is a polynomial $P$ from $\BC^{(m,n)}$ to $V_\rho$
if it satisfies the following conditions\,:
\begin{equation}
\sum_{k=1}^m { {\partial^2P}\over {\partial z_{ki}\partial z_{kj}} }\,=\,0\qquad \textrm{for all}\ i,j=1,2,\cdots,n
\end{equation}
and
\begin{equation}
P(ZA)\,=\,\rho(\,{}^t\!A)P(Z)\qquad \textrm{for all}\ A\in GL(n,\BC).
\end{equation}
We denote by ${\mathfrak H}_{m,n}(\rho)$ the space of all pluriharmonic forms with respect to $\rho$.
\end{definition}

\vskip 0.2cm
  Freitag proved the following.

\begin{theorem} Let $S$ be a positive definite even unimodular matrix of degree $m$ and let $(\rho,V_\rho)$ be
a finite dimensional rational representation of $GL(n,\BC)$. Let $P\in {\mathfrak H}_{m,n}(\rho)$ be a pluriharmonic form
with respect to $\rho.$ Then the theta series
\begin{equation}
\Theta_{S,P}(\Omega):=\,\sum_{N\in \BZ^{(m,n)} } P(S^{1/2}N)\,e^{\pi\,i\,\sigma(\,{}^t\!NSN\Omega)}
\end{equation}
is a modular form with respect to the representation $\rho_*$ of $GL(n,\BC)$ defined by
$$\rho_* (A)\,=\,\rho(A)\,(\det A)^{\frac m2},\qquad A\in GL(n,\BC)$$
for the the Siegel modular group $\Gamma_n.$
\end{theorem}
\noindent {\it Proof.} We will omit the proof. The proof can be found in \cite{F}.   $\hfill \square$

\end{section}

\newpage

\begin{section}{{\large\bf Relation between Theta series and the Weil Representation
}}
\setcounter{equation}{0}

\vskip 0.3cm Let $(\pi,V_\pi)$ be a unitary projective representation of
$Sp(n,\BR)$ on the representation space $V_\pi$. We assume that
$(\pi,V_\pi)$ satisfies the following conditions (A) and (B):
\vskip 0.1cm \noindent {\bf (A)} There exists a vector valued map
\begin{equation*}
{\mathscr F}:\BH_n\lrt V_\pi,\qquad\ \Omega\mapsto {\mathscr
F}_\Omega:={\mathscr F}(\Om)
\end{equation*}

\noindent satisfying the following covariance relation
\begin{equation}
\pi( M) {\mathscr F}_\Omega=\psi(M)\,J(M ,\Omega)^{-1}\,{\mathscr F}_{ M\cdot \Omega}
\end{equation}
$\textrm{for all}\ M\in Sp(n,\BR)$ and
$\Omega\in\BH_n.$
Here $\psi$ is a character of $Sp(n,\BR)$ and $J:Sp(n,\BR)\times
\BH_n\lrt GL(1,\BC)$ is a certain automorphic factor for $Sp(n,\BR)$
on $\BH_n.$

\vskip 0.1cm \noindent {\bf (B)} Let $\Gamma$ be an
arithmetic subgroup of the Siegel modular group $\Gamma_n$. There exists a linear functional
$\theta:V_\pi\lrt \BC$ which is $\textsf{semi-invariant}$ under the action of
$\Gamma$, in other words, for all $\gamma\in\Gamma$ and $\Omega\in\BH_n,$
\begin{equation}
\langle\, \pi^*(\gamma )\theta,\,{\mathscr
F}_\Omega\,\rangle=\langle\, \theta,\pi (
\gamma)^{-1}{\mathscr
F}_\Omega\,\rangle=\chi (\gamma)\,\langle\,
\theta,\,{\mathscr F}_\Omega\,\rangle ,
\end{equation}

\noindent where $\pi^*$ is the contragredient of $\pi$ and
$\chi:\Gamma\lrt \BC_1^*$ is a unitary character of
$\Gamma$.

\vskip 0.2cm Under the assumptions (A) and (B) on a unitary projective
representation $(\pi,V_\pi)$, we define the function $\Theta$ on
$\BH_n$ by
\begin{equation}
\Theta(\Om):=\,\langle\,\theta,{\mathscr
F}_\Omega\,\rangle=\theta\big({\mathscr F}_\Omega\big),\quad\
\Omega\in\BH_n.
\end{equation}

We now shall see that $\Theta$ is an automorphic form on
$\BH_n$ with respect to $\Gamma$ for the automorphic
factor $J$.

\begin{lemma} Let $(\pi,V_\pi)$ be a unitary projective representation of
$Sp(n,\BR)$ satisfying the above assumptions (A) and (B). Then the
function $\Theta$ on $\BH_n$ defined by (17.3) satisfies the
following modular transformation behavior

\begin{equation}
\Theta(\gamma\cdot \Omega)=\,\psi (
\gamma)^{-1}\,\chi(\gamma)^{-1}\,J(
 \gamma,\Omega)\,\Theta(\Om)
\end{equation}

\noindent for all $\gamma\in \Gamma$ and
$\Omega\in\BH_n.$
\end{lemma}

\noindent {\it Proof.} For any $\gamma\in \Gamma$
and $\Omega\in\BH_n,$ according to the assumptions (17.1) and
(17.2), we obtain
\begin{eqnarray*}
& & \Theta(\gamma\cdot \Omega)=\big\langle\,\theta,
{\mathscr F}_{ \gamma\cdot \Omega} \big\rangle \hskip 5cm\\
&=&\big\langle\,\theta,\psi(\gamma)^{-1}\,J(
 \gamma,\Omega)\,\pi(\gamma){\mathscr
F}_{ \Omega}\,\big\rangle\\
&=&\psi(\gamma  )^{-1} J(\gamma,\Omega)\,\big\langle\,\theta,\pi(\gamma){\mathscr
F}_\Omega\,\big\rangle\\
&=&\,\psi(\gamma)^{-1}\,\chi(\gamma)^{-1}\,J(
\gamma,\Omega)\,\big\langle\,\theta,{\mathscr
F}_{\Omega}\,\big\rangle\\
&=&\,\psi(\gamma)^{-1}\,\chi(\gamma)^{-1}\,J(\gamma,\Omega)\,\Theta(\Om).
\end{eqnarray*}
\hfill$\square$

\newcommand\mfm{{\mathscr F}^{(\CM)} }
\newcommand\mfoz{{\mathscr F}^{(\CM)}_{\Om} }
\newcommand\wg{{\widetilde g} }
\newcommand\wgm{{\widetilde \gamma} }
\newcommand\Tm{\Theta^{(\CM)} }

\newcommand\rmn{\BR^{(m,n)} }
\newcommand\zmn{\BZ^{(m,n)} }
\newcommand\wgam{\widetilde\gamma}

\vskip 0.2cm Now for a positive definite real symmetric matrix
$\CM$ of degree $m$, we define the holomorphic function
$\Theta_\CM:\BH_{n}\lrt\BC$ by

\begin{equation}
\Theta_\CM (\Om):=\sum_{\xi\in \BZ^{(m,n)}}
e^{2\,\pi\,i\,\sigma\left( \CM
\xi\,\Om\,{}^t\xi\,\right) },\qquad \Omega\in
\BH_n.
\end{equation}

\begin{theorem} Let $2\,{\mathcal M}$ be a symmetric positive definite, unimodular even integral matrix of degree $m$.
Then for any $\gamma\in\Gamma_n$, the function
$\Theta_{\mathcal M}$ satisfies the functional equation
\begin{equation}
\Theta_{\mathcal M}( \gamma\cdot \Omega)=\rho_{\mathcal M} (\gamma)\,
J_m( \gamma,\Omega) \Theta_{\mathcal M}(\Omega),\qquad \Omega\in
\BH_{n},
\end{equation}
\end{theorem}

\noindent where $\rho_{\mathcal M}$ is a character of $\Gamma$ with
$|\rho_\CM(\gamma)|^8=1$ for all $\gamma\in\Gamma_n$ and $J_m:Sp(n,\BR)\times \BH_n\lrt
\BC^*_1$ is the automorphic factor for $Sp(n,\BR)$ on $\BH_n$
defined by the formula (14.2) in Section 14.

\noindent {\it Proof.}
For an element
$\gamma=\begin{pmatrix}
A & B
\\ C & D \end{pmatrix}\in \Gamma_{n}$ and $\Omega\in\BH_n,$ we put
\begin{equation*}
\Om_*=\g\cdot \Om=(A\Om+B)(C\Om+D)^{-1}.
\end{equation*}

\noindent We define the linear functional $\vartheta$ on $L^2\big(
\rmn\big)$ by

\begin{equation*}
\vartheta (f)=\langle \vartheta,f \rangle:=\sum_{\xi\in
\zmn}f(\xi),\quad\ f\in L^2\big( \rmn\big).
\end{equation*}

\noindent We note that $\Theta_\CM(\Om)=\vartheta\big(
\mfoz\big).$ Since $\mfm$ is a covariant map for the
Weil representation $\omega_\CM$ with respect to the automorphic factor $J_m$
by Theorem 14.1,
according to Lemma 17.1, it suffices to prove that $\vartheta$ is
semi-invariant for $\omega_\CM$ under the action of $\Gamma_n$, in other
words, $\vartheta$ satisfies the following semi-invariance
relation
\begin{equation}
\Big\langle\, \vartheta,R_\CM ( \gamma)\mfoz \,\Big\rangle
=\,\rho_\CM (\gamma)^{-1} \,\Big\langle\, \vartheta,\mfoz
\,\Big\rangle
\end{equation}

\noindent for all $\gamma\in \Gamma_n$ and $\Om\in \BH_n.$

\vskip 0.12cm We see that the following elements
\begin{eqnarray*}
t_{\beta}&=&\begin{pmatrix} I_n & \beta \\ 0 & I_n \end{pmatrix}
\textrm{with any}\ \beta=\,{}^t\beta\in \BZ^{(n,n)},\\
 d_{\alpha}&=& \begin{pmatrix} {}^t\alpha & 0 \\ 0 & \alpha^{-1} \end{pmatrix}
\textrm{with any}\ \alpha\in GL(n,\BZ),\\
 \s_n&=&\begin{pmatrix} 0 & -I_n \\ I_n  & 0 \end{pmatrix}
\end{eqnarray*}
generate the Siegel modular group $\Gamma_n.$ Therefore it suffices to
prove the semi-invariance relation (17.7) for the above generators
of $\Gamma_n.$

\vskip 0.5cm\noindent {\bf Case I.} $\gamma=t_\beta$ with
$\beta=\,{}^t\beta\in\BZ^{(n,n)}.$

\vskip 0.1cm In this case, we have
$$\Om_*=\Om+\beta\qquad \textrm{and}\qquad J_m(\gamma,\Omega)=1.$$

\noindent According to the covariance relation (14.4) in Section 14, we obtain
\begin{eqnarray*}
& &\big\langle \,\vartheta, R_\CM (\gamma)
\mfoz\,\big\rangle\\
&=&\,\big\langle \,\vartheta,
J_m(\gamma,\Omega)^{-1}\mfm_{\gamma\cdot\Omega}\,\big\rangle\\
&=&\,\big\langle \,\vartheta,
\mfm_{\Om+\beta}\,\big\rangle\\
&=&\,\sum_{\xi\in\zmn} \mfm_{\Om+\beta}(\xi)\\
&=&\, \sum_{\xi\in\zmn} e^{2\,\pi\,i\,\s ( \CM \,\xi\,(\Om+\beta)\,{}^t\!\xi))}\\
&=&\, \sum_{\xi\in\zmn} e^{2\,\pi\,i\,\s(\CM\,\xi\,\Omega\,{}^t\!\xi)}\cdot
e^{2\,\pi\,i\,\sigma(\CM\, \xi\,\beta\,{}^t\!\xi)}\\
&=&\, \sum_{\xi\in\zmn} e^{2\,\pi\,i\,\sigma ( \CM \,
\xi\,\Om\,{}^t\!\xi) }\\
&=&\,\big\langle \,\vartheta, \mfm_{\Om}\,\big\rangle.
\end{eqnarray*}

\noindent Here we used the fact that $2\,\sigma (\CM \,\xi\,\beta\,{}^t\!\xi)$ is an
even integer because $2\,\CM$ is even integral. We put $\rho_\CM(\gamma)=\rho_\CM\big(
t_\beta\big)=1$ for all $\beta=\,{}^t\beta\in\BZ^{(n,n)}.$ Therefore $\vartheta$ satisfies the semi-invariance
relation (17.7) in the case $\gamma=t_\beta$ with $\beta=\,{}^t\beta\in\BZ^{(n,n)}.$

\vskip 0.5cm\noindent {\bf Case II.} $\gamma=d_\alpha$ with
$\alpha\in GL(n,\BZ).$

\vskip 0.1cm In this case, we have
$$\Om_*=\,{}^t\alpha\,\Om\,\alpha\qquad \textrm{and}\qquad
J_m(d_\alpha,\Omega)=(\det\alpha)^{-{\frac m2}}.$$

\noindent According to the covariance relation (14.4) in Section 14, we obtain
\begin{eqnarray*}
& &\big\langle \,\vartheta, R_\CM (\gamma)
\mfoz\,\big\rangle \ \textrm{with}\ \gamma=d_\alpha\\
&=&\,\big\langle \,\vartheta,
J_m(\gamma,\Omega)^{-1}\mfm_{\gamma\cdot\Omega}\,\big\rangle\\
&=&\,\left( \det \alpha \right)^{\frac m2}\,\big\langle
\,\vartheta,
\mfm_{{}^t\alpha\,\Om\,\alpha}\,\big\rangle\\
&=&\,\left( \det \alpha \right)^{\frac m2} \sum_{\xi\in\BZ^{(m,n)} }
\mfm_{{}^t\alpha\,\Om\,\alpha}(\xi)\\
&=&\,\left( \det \alpha \right)^{\frac m2} \sum_{\xi\in\BZ^{(m,n)} } e^{2\,\pi\,i\,\sigma
(\CM\,\xi\,{}^t\alpha\,\Omega\,\alpha\,{}^t\xi)}\\
&=&\,\left( \det \alpha \right)^{\frac m2}\sum_{\xi\in\BZ^{(m,n)} }
e^{2\,\pi i\,\sigma\{
\CM( (\xi\,{}^t\alpha)\,\Om\,{}^t(\xi\,{}^t\alpha))\}} \\
&=&\,\left( \det \alpha \right)^{\frac m2}\,\big\langle
\,\vartheta, \mfm_{\Om}\,\big\rangle.
\end{eqnarray*}

\noindent Here we put $\rho_\CM (d_\alpha)=(\det\alpha)^{-{\frac m2}}.$ Therefore $\vartheta$
satisfies the semi-invariance relation (17.7) in the case
$\gamma=d_\alpha$ with $\alpha\in GL(n,\BZ).$

\vskip 0.52cm\noindent {\bf Case III.} $\gamma=\sigma_n=
\begin{pmatrix} 0 & -I_n \\ I_n & \ 0
\end{pmatrix}.$

\vskip 0.1cm In this case, we have
$$\Om_*=-\Om^{-1}\qquad \textrm{and}\qquad J_m(\s_n,\Omega)=
\,\big( \det\Om\big)^{\frac m2}.$$

\noindent In the process of the proof of Theorem 14.1, using Lemma
14.2, we already showed that

\begin{eqnarray}
& & \int_{\rmn} e^{2\,\pi\,i\,\sigma( \CM(y\,\Om\,{}^ty\,+\,2\,y\,{}^tx)
) } dy \\
&=&\,\big(\det \CM\big)^{-{\frac n2}}\left( \det {{2\,\Om} \over
i}\right)^{-{\frac
m2}}\,e^{-2\,\pi\,i\,\sigma(\CM\,x\,\Om^{-1}\,{}^tx)}.\nonumber
\end{eqnarray}

\noindent By Formula (17.8), we obtain
\begin{eqnarray*}
\widehat {\mfoz} (2\,{\mathcal M}x)&=&\, \int_{\rmn} \mfoz (y)\,e^{-2\,\pi\,i\,\sigma(y\,{}^t(2\,\CM\,x))}\,dy\\
&=&\, \int_{\rmn} e^{2\,\pi\,i\,\sigma (\CM\,y\,\Omega\,{}^ty)}\cdot e^{-4\,\pi\,i\,\sigma(\CM\,y\,{}^tx)}\,dy\\
&=&\, \int_{\rmn} e^{2\,\pi\,i\,\sigma \{ \CM (y\,\Omega\,{}^ty\,+\,2\,y\,{}^t(-x)) \}}\,dy\\
&=&\,(\det \CM)^{-{\frac n2}}\,\left( \det {{2\,\Omega}\over i}\right)^{-{\frac m2}}\,e^{-2\,\pi\,i\,\sigma(\CM\,(-x)\,\Omega^{-1}\,{}^t(-x))}\\
&=&\,(\det \CM)^{-{\frac n2}}\,\left( \det {{2\,\Omega}\over i}\right)^{-{\frac m2}}\,e^{-2\,\pi\,i\,\sigma(\CM\,x\,\Omega^{-1}\,{}^tx)}.
\end{eqnarray*}

Thus we obtain
\begin{equation}
\widehat {\mfoz} (2\,{\mathcal M}x)=\,\big(\det \CM\big)^{-{\frac n2}}\left( \det
{{2\,\Om} \over i}\right)^{-{\frac
m2}}\,e^{-2\,\pi\,i\,\sigma(\CM\,x\,\Om^{-1}\,{}^tx)},
\end{equation}
where $\widehat {f}$ is the Fourier transform of $f$ defined by
$$ {\widehat f}(x)=\,\int_{\rmn} f(y)\,e^{-2\,\pi\,i\,\sigma (y\,{}^tx)}\,dy,\quad x\in \BR^{(m,n)}.$$

\vskip 0.1cm We prove the Poisson summation formula in our setting.

\begin{lemma} Let $f$ be an element in $L^2\big( \BR^{(m,n)}\big).$ Then
\begin{equation}
\sum_{\xi\in \BZ^{(m,n)}}{\widehat f}(\xi)=\sum_{\xi\in \BZ^{(m,n)}}f(\xi).
\end{equation}
\end{lemma}

\vskip 0.1cm \noindent $\textit{Proof.}$ We define
\begin{equation}
h(x)=\sum_{\xi\in \BZ^{(m,n)}} f(x+\xi),\qquad x\in \BR^{(m,n)}.
\end{equation}
We see that $h(x)$ is periodic in $x_{ij}$ with period 1, where $x=(x_{ij})$ is a coordinate in $\BR^{(m,n)}.$ Thus $h(x)$ has
the following Fourier series
\begin{equation}
h(x)=\sum_{\xi\in \BZ^{(m,n)}} c_\xi \,e^{2\,\pi\,i\,\sigma (x\,{}^t\xi)},
\end{equation}
where
\begin{eqnarray*}
c_\xi&=&\,\int_0^1\cdots \int_0^1 h(x)\,e^{-2\,\pi\,i\,\sigma (x\,{}^t\xi)}\,dx\\
&=&\,\int_0^1\cdots \int_0^1 \sum_{\xi\in \BZ^{(m,n)}} f(x+\xi)\,e^{-2\,\pi\,i\,\sigma (x\,{}^t\xi)}\,dx\\
&=&\, \int_{\rmn} f(x)\,e^{-2\,\pi\,i\,\sigma (x\,{}^t\xi)}\,dx\, =\, {\widehat f}(\xi).
\end{eqnarray*}
Here we interchanged summation and integration, and made a change of variables replacing $x+\xi$ by $x$ to obtain the above equality.

\vskip 0.1cm\noindent By the definition (17.11), we have
$$h(0)=\sum_{\xi\in \BZ^{(m,n)}} f(\xi).$$
On the other hand, from Formula (17.12), we get
$$h(0)=\sum_{\xi\in \BZ^{(m,n)}} c_{\xi}\,=\,\sum_{\xi\in \BZ^{(m,n)}} {\widehat f}(\xi).$$
Therefore we obtain the Poisson summation formula (17.10). $\hfill \square$

\vskip 0.2cm
According to the covariance relation (14.4) in Section 14, Formula
(17.9) and Poisson summation formula, we obtain
\begin{eqnarray*}
& &\big\langle \,\vartheta, R_\CM (\gamma)
\mfoz\,\big\rangle\quad  \textrm{with}\ \gamma=\sigma_n\\
&=&\,\big\langle \,\vartheta,
J_m(\gamma,\Omega)^{-1}\mfm_{\gamma\cdot\Omega}\,\big\rangle\\
&=&\,J_m(\gamma,\Omega)^{-1} \big\langle \,\vartheta,
\mfm_{-\Om^{-1}}\,\big\rangle\\
&=&\,(\det \Om)^{-{\frac
m2}}\,
\sum_{\xi\in\BZ^{(m,n)}} e^{-2\,\pi\,i\,\sigma( \CM\,
\xi\,\Om^{-1}\,{}^t\!\xi) }\\
&=&\,(\det \Om)^{-{\frac m2}}\big(\det \CM\big)^{{\frac n2}}\left(
\det {{2\,\Om} \over i}\right)^{{\frac m2}}\,\sum_{\xi\in\BZ^{(m,n)}}
\widehat {\mfoz} (2\,{\mathcal M}\,\xi) \\
& &\quad (\,\textrm{by\ Formula}\ (17 .9))\\
&=&\,\big(\det 2\,\CM\big)^{{\frac n2}}\left( \det {{I_n} \over
i}\right)^{{\frac m2}}\,
\sum_{\xi\in\BZ^{(m,n)}}
\widehat {\mfoz} (\xi) \\
& & \quad (\,\textrm{because}\ 2\,{\mathcal M}\ \textrm{is\ unimodular}))\\
&=&\,\left( \det {{I_n} \over
i}\right)^{{\frac m2}}\,\sum_{\xi\in\BZ^{(m,n)}}
{\mfoz} (\xi) \quad (\,\textrm{by\ Poisson summation formula}) \\
&=&\,(-i)^{{mn}\over
2}\,\big\langle \,\vartheta, \mfm_{\Om}\,\big\rangle \\
&=&\,(-i)^{{mn}\over 2}\,\big\langle \,\vartheta,
\mfm_{\Om}\,\big\rangle.
\end{eqnarray*}

\noindent We put $\rho_\CM (\sigma_n)
=(-i)^{-{{mn}\over 2}}.$ Therefore $\vartheta$ satisfies the
semi-invariance relation (17.7) in the case $\gamma=\sigma_n.$ The
proof of Case III is completed. Since $J_m$ is an automorphic
factor for $Sp(n,\BR)$ on $\BH_{n}$, we see that if the formula (17.6)
holds for two elements $\gamma_1,\gamma_2$ in $\Gamma$, then it holds
for $\gamma_1 \gamma_2.$ Finally we complete the proof of Theorem
17.2. \hfill $\square$

\vskip 0.2cm\noindent {\bf Remark.} For a symmetric positive
definite integral matrix $\CM$ such that $2\CM$ is not unimodular even
integral, we obtain a similar transformation formula like (17.6).
If $m$ is odd, $\Theta_\CM(\Om)$ is a modular form of a
half-integral weight ${\frac m2}$ and index ${\CM}\over 2$ with
respect to a suitable arithmetic subgroup $\Gamma_{\Theta,\CM}$ of
$\Gamma_n$ and a suitable character $\rho_\CM$ of $\Gamma_{\Theta,\CM}$.

\end{section}

\newpage

\newcommand\Imm{\text{Im}\,}
\newcommand\MCM{\mathcal M}
\newcommand\Fgh{{\mathcal F}_{n,m}}

\begin{section}{{\large\bf Spectral Theory on the Abelian Variety}}
\setcounter{equation}{0}
\vskip 0.3cm
\newcommand\CCF{{\mathcal F}_n}
\newcommand\POB{ {{\partial}\over {\partial{\overline \Omega}}} }
\newcommand\PZB{ {{\partial}\over {\partial{\overline Z}}} }
\newcommand\PX{ {{\partial}\over{\partial X}} }
\newcommand\PY{ {{\partial}\over {\partial Y}} }
\newcommand\PU{ {{\partial}\over{\partial U}} }
\newcommand\PV{ {{\partial}\over{\partial V}} }
\newcommand\PO{ {{\partial}\over{\partial \Omega}} }
\newcommand\PZ{ {{\partial}\over{\partial Z}} }
\newcommand\ka{\kappa}
\newcommand\Chg{{\mathbb C}^{(m,n)}}

We recall the Jacobi group (cf.\,Section 12)
$$G^J=Sp(n,\BR)\ltimes H_{\BR}^{(n,m)}$$
which is the semidirect product of $Sp(n,\BR)$
and $H_{\BR}^{(n,m)}$
endowed with the following multiplication law
$$\Big(M,(\lambda,\mu,\kappa)\Big)\Big(M',(\lambda',\mu',\kappa')\Big) =\,
\Big(MM',(\widetilde{\lambda}+\lambda',\widetilde{\mu}+ \mu',
\kappa+\kappa'+\widetilde{\lambda}\,^t\!\mu'
-\widetilde{\mu}\,^t\!\lambda')\Big)$$ with $M,M'\in Sp(n,\BR),
(\lambda,\mu,\kappa),\,(\lambda',\mu',\kappa') \in
H_{\BR}^{(n,m)}$ and
$(\widetilde{\lambda},\widetilde{\mu})=(\lambda,\mu)M'$.
Then $G^J$ acts
on $\BH_n\times \BC^{(m,n)}$ transitively by
\begin{eqnarray}
& & \big(M,(\lambda,\mu,\kappa)\big)\cdot (\Om,Z)\\
&=&\big(
(A\Omega+B)(C\Omega+D)^{-1},(Z+\lambda
\Om+\mu)(C\Omega+D)^{-1}\big), \nonumber
\end{eqnarray}
where $M=\begin{pmatrix} A&B\\
C&D\end{pmatrix} \in Sp(n,\BR),\ (\lambda,\mu, \kappa)\in
H_{\BR}^{(n,m)}$ and $(\Om,Z)\in \BH_n\times \BC^{(m,n)}.$ We note
that the Jacobi group $G^J$ is {\it not} a reductive Lie group and
also that the space ${\mathbb H}_n\times \BC^{(m,n)}$ is not a
symmetric space. We refer to \cite{YJH0}-\cite{YJH6} and
\cite{Z} about automorphic forms on $G^J$ and topics related to
the content of this book.

\vskip 0.3cm
From now on, for brevity, we write
$$\BH_{n,m}:=\BH_n\times \BC^{(m,n)}.$$
$\BH_{n,m}$ is called the $\textsf{Siegel-Jacobi space of degree}$ $n$ and $\textsf{index}$ $m$.

\vskip 0.2cm We let
$$\Gamma_{n,m}:=\Gamma_n\ltimes H_{\BZ}^{(n,m)}$$
be the
discrete subgroup of $G^J$, where
$$H_{\BZ}^{(n,m)}=\left\{\,(\la,\mu,\ka)\in H_{\BR}^{(n,m)}\,\big|\ \la,\mu\in \BZ^{(m,n)},\ \
\ka\in \BZ^{(m,m)}\, \right\}.$$

\vskip 0.2cm
Let $E_{kj}$ be the $m\times n$ matrix with entry 1
where the $k$-th row and the $j$-th colume meet, and all other
entries 0. For an element $\Om\in \BH_n$, we set for brevity
\begin{equation}
F_{kj}(\Om):=E_{kj}\Om,\qquad 1\leq k\leq m,\ 1\leq j\leq
n.\end{equation}
 \indent For each $\Om\in {\mathcal F}_n,$ we define a
subset $P_{\Omega}$ of $\BC^{(m,n)}$ by
\begin{equation*}
P_{\Om}=\left\{ \,\sum_{k=1}^m\sum_{j=1}^n \la_{kj}E_{kj}+
\sum_{k=1}^m\sum_{j=1}^n \mu_{kj}F_{kj}(\Om)\,\Big|\ 0\leq
\la_{kj},\mu_{kj}\leq 1\,\right\}. \end{equation*} \indent For
each $\Om\in \CCF,$ we define the subset $D_{\Om}$ of $\BH_{n,m}$
by
\begin{equation*} D_{\Om}:=\left\{\,(\Om,Z)\in\BH_{n,m}\,\vert\ Z\in
P_{\Om}\,\right\}.\end{equation*}

We define
\begin{equation*} \Fgh:=\cup_{\Om\in\CCF}D_{\Omega}.\end{equation*}
\begin{theorem}
$\Fgh$ is a fundamental domain for $\Gamma_{n,m}\ba \BH_{n,m}.$
\end{theorem}
\begin{proof} Let $({\tilde{\Om}},{\tilde{Z}})$ be an arbitrary
element of $\BH_{n,m}.$ We must find an element $(\Om,Z)$ of
$\Fgh$ and an element $\gamma^J=(\gamma,(\la,\mu;\ka))\in\Gamma_{n,m}$ with
$\gamma\in\Gamma_n$ such that $\g^J\cdot
(\Om,Z)=({\tilde{\Om}},{\tilde{Z}}).$ Since $\CCF$ is a
fundamental domain for $\G_n\ba \BH_n,$ there exists an element
$\g$ of $\G_n$ and an element $\Om$ of $\CCF$ such that
$\g\cdot\Om={\tilde {\Om}}.$ Here $\Om$ is unique up to the
boundary of $\CCF$. \vskip 0.1cm We write
$$\g=\begin{pmatrix} A & B\\ C & D \end{pmatrix} \in \G_n.$$
It is easy to see that we can find $\la,\mu\in \BZ^{(m,n)}$ and
$Z\in P_{\Om}$ satisfying the equation
$$Z+\la \Om +\mu={\tilde Z}(C\Om+D).$$
If we take $\g^J=(\g,(\la,\mu;0))\in \G_{n,m},$ we see that
$\g^J\cdot (\Om,Z)=({\tilde{\Om}},{\tilde{Z}}).$ Therefore we
obtain
$$\BH_{n,m}=\cup_{\g^J\in \G_{n,m}}\g^J\cdot\Fgh.$$
Let $(\Om,Z)$ and $\g^J\cdot (\Om,Z)$ be two elements of $\Fgh$
with $\g^J=(\g,(\la,\mu;\ka))\in \G_{n,m}.$ Then both $\Om$ and
$\g\cdot\Om$ lie in $\CCF$. Therefore both of them either lie in
the boundary of $\CCF$ or $\g=\pm I_{2n}.$ In the case that both
$\Om$ and $\g\cdot\Om$ lie in the boundary of $\CCF$, both
$(\Om,Z)$ and $\g^J\cdot (\Om,Z)$ lie in the boundary of $\Fgh$.
If $\g=\pm I_{2n},$ we have
\begin{equation}
Z\in P_{\Om}\quad \text{and}\quad \pm (Z+\la \Om+\mu)\in
P_{\Om},\quad \la,\mu\in \BZ^{(m,n)}.\end{equation} From the
definition of $P_{\Om}$ and (18.3), we see that either $\la=\mu=0,\
\g\neq -I_{2n}$ or both $Z$ and $\pm (Z+\la \Om+\mu)$ lie on the
boundary of the parallelepiped $P_{\Om}$. Hence either
both$(\Om,Z)$ and $\g^J\cdot (\Om,Z)$ lie in the boundary of
$\Fgh$ or $\g^J=(I_{2n},(0,0;\ka))\in\G_{n,m}$. Consequently
$\Fgh$ is a fundamental domain for $\G_{n,m}\ba \BH_{n,m}.$
\end{proof}

For a coordinate $(\Om,Z)\in\BH_{n,m}$ with
$\Om=(\om_{\mu\nu})\in {\mathbb H}_n$ and $Z=(z_{kl})\in \Chg,$ we
put
\begin{align*}
\Om\,=&\,X\,+\,iY,\quad\ \ X\,=\,(x_{\mu\nu}),\quad\ \
Y\,=\,(y_{\mu\nu})
\ \ \text{real},\\
Z\,=&U\,+\,iV,\quad\ \ U\,=\,(u_{kl}),\quad\ \ V\,=\,(v_{kl})\ \
\text{real},\\
d\Om\,=&\,(d\om_{\mu\nu}),\quad\ \ dX\,=\,(dx_{\mu\nu}),\quad\ \
dY\,=\,(dy_{\mu\nu}),\\
dZ\,=&\,(dz_{kl}),\quad\ \ dU\,=\,(du_{kl}),\quad\ \
dV\,=\,(dv_{kl}),\\
d{\overline{\Om}}=&\,(d{\overline{\om}}_{\mu\nu}),\quad
d{\overline Z}=(d{\bar z}_{kl}),
\end{align*}

$$ {{\partial}\over{\partial \Omega}}\,=\,\left(\, { {1+\delta_{\mu\nu}} \over 2}\, {
{\partial}\over {\partial \om_{\mu\nu}} } \,\right),\quad
\POB\,=\,\left(\, { {1+\delta_{\mu\nu}}\over 2} \, {
{\partial}\over {\partial {\overline \om}_{\mu\nu} }  }
\,\right),$$
$$\PZ=\begin{pmatrix} { {\partial}\over{\partial z_{11}} } & \hdots &
{ {\partial}\over{\partial z_{m1}} }\\
\vdots&\ddots&\vdots\\
{ {\partial}\over{\partial z_{1n}} }&\hdots &{ {\partial}\over
{\partial z_{mn}} } \end{pmatrix},\quad \PZB=\begin{pmatrix} {
{\partial}\over{\partial {\overline z}_{11} }   }&
\hdots&{ {\partial}\over{\partial {\overline z}_{m1} }  }\\
\vdots&\ddots&\vdots\\
{ {\partial}\over{\partial{\overline z}_{1n} }  }&\hdots & {
{\partial}\over{\partial{\overline z}_{mn} }  }
\end{pmatrix}.$$

\begin{remark}
The following metric \begin{align*}
ds_{n,m}^2=&\,\s\left(Y^{-1}d\Om\,Y^{-1}d{\overline \Om}\right)\,+
\,\s\left(Y^{-1}\,{}^tV\,V\,Y^{-1}d\Om\,Y^{-1}
d{\overline{\Om}}\right) \notag\\
&\ \ \ \ +\,\s\left(Y^{-1}\,^t(dZ)\,d{\overline Z}\right)\\
&\ \ -\s\left(\,V\,Y^{-1}d\Om\,Y^{-1}\,^t( d{\overline{\Om}} )\,
+\,V\,Y^{-1} d{\overline {\Om}} \, Y^{-1}\,^t(dZ)\,\right)\notag
\end{align*}
is a K{\"a}hler metric on $\BH_{n,m}$ which is invariant under the
action (18.1) of the Jacobi group $G^J$. Its Laplacian is given by
\begin{align*}
\Delta_{n,m}\,=\,& 4\,\s\left(\,Y\,\,
{}^{{}^{{}^{{}^\text{\scriptsize $t$}}}}\!\!\!\left(Y\POB\right)\PO\,\right)\,+\,
4\,\s\left(\, Y\,\PZ {}^{{}^{{}^{{}^\text{\scriptsize $t$}}}}\!\!\!\left( \PZB\right)\,\right) \nonumber \\
&\ \ \ \ +\,4\,\s\left(\,VY^{-1}\,^tV\,\,{}^{{}^{{}^{{}^\text{\scriptsize $t$}}}}\!\!\!\left(Y\PZB\right)\,\PZ\,\right)\\
&\ \
+\,4\,\s\left(V\,\,{}^{{}^{{}^{{}^\text{\scriptsize $t$}}}}\!\!\!\left(Y\POB\right)\PZ\,\right)+\,4\s\left(\,^tV\,\,{}^{{}^{{}^{{}^\text{\scriptsize $t$}}}}\!\!\!\left(Y\PZB\right)\PO\,\right).\nonumber
\end{align*}

The following differential form
$$dv_{n,m}=\,\left(\,\det Y\,\right)^{-(n+m+1)}[dX]\wedge [dY]\w
[dU]\w [dV]$$ is a $G^J$-invariant volume element on $\BH_{n,m}$,
where
$$[dX]=\w_{\mu\leq\nu}dx_{\mu\nu},\quad [dY]=\w_{\mu\leq\nu}
dy_{\mu\nu},\quad [dU]=\w_{k,l}du_{kl}\quad \text{and} \quad
[dV]=\w_{k,l}dv_{kl}.$$ The point is that the invariant metric
$ds_{n,m}^2$ and its Laplacian are beautifully expressed in terms
of the {\it trace} form. The proof of the above facts can be found
in \cite{Y8}. We also refer to \cite{Y9} for the action of the Jacobi group $G^J$ on the Siegel-Jacobi disk
${\mathbb D}_n\times \BC^{(m,n)}.$
\end{remark}

\vskip 0.2cm
We fix two positive integers $m$ and $n$ throughout this section.
\vskip 0.1cm For an element $\Om\in \BH_n,$ we set
\begin{equation*}
L_{\Om}:=\BZ^{(m,n)}+\BZ^{(m,n)}\Om\end{equation*} We use the
notation (18.2). It follows from the positivity of $\text{Im}\,\Om$
that the elements $E_{kj},\,F_{kj}(\Om)\,
(1\leq k\leq m,\ 1\leq
j\leq n)$ of $L_{\Om}$ are linearly independent over $\BR$.
Therefore $L_{\Om}$ is a lattice in $\BC^{(m,n)}$ and the set
$$\left\{\,E_{kj},\,F_{kj}(\Om)\,|\ 1\leq k\leq m,\
 1\leq j\leq n\,
\right\}$$ forms an integral basis of $L_{\Om}$. We see easily that
if $\Om$ is an element of $\BH_n$, the period matrix
$\Om_\flat:=(I_n,\Om)$ satisfies the Riemann conditions (RC.1) and
(RC.2)\,: \vskip 0.1cm (RC.1) \ \ \ $\Om_\flat\, J_n\,^t\Om_\flat=0\,$;
\vskip 0.1cm (RC.2) \ \ \ $-{1 \over
{i}}\,\Om_\flat\, J_n\,^t{\overline{\Om}}_\flat
>0$.

\vskip 0.2cm \noindent Thus the complex torus
$A_{\Om}:=\BC^{(m,n)}/L_{\Omega}$ is an abelian variety. For more
details on $A_{\Om}$, we refer to \cite{I} and \cite{Mum1}.
\vskip 0.2cm It might be interesting to investigate the spectral
theory of the Laplacian $\Delta_{n,m}$ on a fundamental domain
$\Fgh$. But this work is very complicated and difficult at this
moment. It may be that the first step is to develop the spectral
theory of the Laplacian $\Delta_{\Omega}$ on the abelian variety
$A_{\Omega}.$ The second step will be to study the spectral theory
of the Laplacian $\Delta_*$\,(see (12.2) in Section 12) on the moduli space
$\Gamma_n\backslash \BH_n$ of principally polarized abelian
varieties of dimension $g$. The final step would be to combine the
above steps and more works to develop the spectral theory of the
Lapalcian $\Delta_{n,m}$ on $\Fgh.$ In this section, we deal only
with the spectral theory $\Delta_{\Omega}$ on $L^2(A_{\Omega}).$

 \vskip 0.1cm We fix
an element $\Om=X+i\,Y$ of $\BH_n$ with $X=\text{Re}\,\Om$ and
$Y=\text{Im}\, \Om.$ For a pair $(A,B)$ with $A,B\in\BZ^{(m,n)},$
we define the function $E_{\Om;A,B}: \BC^{(m,n)}\lrt \BC$ by
\begin{equation*}
E_{\Om;A,B}(Z)=e^{2\pi i\left( \s\,(\,^tAU\,)+\,\s\,
((B-AX)Y^{-1}\,^tV)\right)},\end{equation*} where $Z=U+iV$ is a
variable in $\BC^{(m,n)}$ with real $U,V$. \vskip 0.1cm\noindent
\begin{lemma} For any $A,B\in \BZ^{(m,n)},$ the function
$E_{\Om;A,B}$ satisfies the following functional equation
\begin{equation*}
E_{\Om;A,B}(Z+\la \Om+\mu)=E_{\Om;A,B}(Z),\quad
Z\in\BC^{(m,n)}\end{equation*} for all $\la,\mu\in\BZ^{(m,n)}.$ Thus
$E_{\Om;A,B}$ can be regarded as a function on $A_{\Om}.$ \vskip
0.1cm \end{lemma}
\begin{proof}
We write $\Om=X+iY$ with real $X,Y.$ For any
$\la,\mu\in\BZ^{(m,n)},$ we have
\begin{align*}
E_{\Om;A,B}(Z+\la\Om+\mu)&=E_{\Om;A,B}((U+\la X+\mu)+i(V+\la Y))\\
&=e^{ 2\pi i \left\{\,\s\,(\,^t\!A(U+\la
X+\mu))+\,\s\,((B-AX)Y^{-1}\,^t\!(V+\la Y))\,\right\} }\\
&=e^{ 2\pi i \left\{\,\s\,(\,^t\!AU+\,^t\!A\la X+\,^t\!A\mu)
+\,\s\,((B-AX)Y^{-1}\,^tV+B\,^t\la-AX\,^t\la) \right\} }\\
&=e^{2\pi i \left\{\,\s\,(\,^t\!AU)\,+\,\s\,((B-AX)Y^{-1}\,^tV)\right\} }\\
&=E_{\Om;A,B}(Z).\end{align*} Here we used the fact that
$^t\!A\mu$ and $B\,^t\la$ are integral. \end{proof}
\newcommand\AO{A_{\Omega}}
\begin{lemma}
The metric
$$ds_{\Om}^2=\s\left(({\rm{Im}}\,\Om)^{-1}\,\,^t(dZ)\,d{\overline Z})\,\right)$$
is a K{\"a}hler metric on $A_{\Om}$ invariant under the action
(18.1) of $\G^J=Sp(n,\BZ)\ltimes H_{\BZ}^{(n,m)}$ on $(\Om,Z)$ with
$\Om$ fixed. Its Laplacian $\Delta_{\Om}$ of $ds_{\Om}^2$ is given
by
\begin{equation*}
\Delta_{\Om}=\,\s\left( ({\rm{Im}}\,\Omega)\,{ {\partial}\over
{\partial {Z}} }    {}^{{}^{{}^{{}^\text{\scriptsize $t$}}}}\!\!\!\left(  {{\partial}\over {\partial
{\overline Z}}} \right)\,
 \right). \end{equation*}
\end{lemma}
\begin{proof} Let ${\tilde \gamma}=(\gamma,(\la,\mu;\kappa))\in
\Gamma^J$ with $\gamma=\begin{pmatrix} A & B\\ C & D
\end{pmatrix}\in Sp(n,\BZ)$ and $({\tilde \Omega},{\tilde
Z})={\tilde \gamma}\cdot (\Omega,Z)$ with $\Omega\in {\mathbb
H}_n$ fixed. Then according to \cite[p.\,33]{Ma},
$$\Imm\,\g \cdot\Omega=\,^t(C{\overline
\Om}+D)^{-1}\,\Imm\,\Om\,(C\Om+D)^{-1}$$ and
$$d{\tilde Z}=dZ\,(C\Om+D)^{-1}.$$
Therefore
\begin{eqnarray*}
& & (\Imm\,{\tilde\Om})^{-1}\,^t(d{\tilde Z})\,d{\overline{\tilde Z}}  \\
&=&(C{\overline\Om}+D)\, (\Imm\,\Om)^{-1}\,^t(C
{\Om}+D)\,^t(C{ \Om}+D)^{-1}\,^t(d{ Z})\,d{\overline Z}\,(C{\overline \Om}+D)^{-1} \\
&=& (C{\overline\Om}+D)\,(\Imm\,\Om)^{-1}\,^t(dZ)\, d{\overline
Z}\,(C{\overline \Om}+D)^{-1}.
\end{eqnarray*}
The metric $ds_{iI_n}=\s (dZ\,^t(d{\overline Z}))$ at $Z=0$ is
positive definite. Since $G^J$ acts on ${\mathbb H}_{n,m}$
transitively, $ds^2_{\Om}$ is a Riemannian metric for any $\Om\in
{\mathbb H}_n.$ We note that the differential operator
$\Delta_{\Om}$ is invariant under the action of $\Gamma^J.$ In
fact,
$${ {\partial}\over {\partial {\tilde Z}} }=(C\Om +D)\,{
{\partial}\over {\partial Z} }.$$
Hence if $f$ is a differentiable
function on $A_{\Om}$, then
\begin{eqnarray*}
& & \Imm\,{\tilde \Omega}\, { {\partial}\over {\partial {\tilde
Z}} }
  {}^{{}^{{}^{{}^\text{\scriptsize $t$}}}}\!\!\!\left( { {\partial f}\over {\partial \overline{\tilde Z} }  } \right) \\
&=&\,^t(C{\overline\Om}+D)^{-1}\,(\Imm\,\Om)\,(C{\Om}+D)^{-1} (C{
\Om}+D)\, { {\partial}\over {\partial Z} } {}^{{}^{{}^{{}^\text{\scriptsize $t$}}}}\!\!\!\left(
(C{\overline\Om}+D){ {\partial f}\over {\partial \overline Z} }
\right)\\
&=&\,^t(C{\overline\Om}+D)^{-1} \, \Imm\,\Omega\,{ {\partial}\over
{\partial Z} } {}^{{}^{{}^{{}^\text{\scriptsize $t$}}}}\!\!\!\left( { {\partial f}\over {\partial \overline
Z} } \right)\,{}^t (C{\overline\Om}+D).
\end{eqnarray*}
Therefore
\begin{equation*}
\s\left(\Imm\,{\tilde \Omega}\, { {\partial}\over {\partial
{\tilde Z}} } {}^{{}^{{}^{{}^\text{\scriptsize $t$}}}}\!\!\!\left( { {\partial }\over {\partial
\overline{\tilde Z} }  } \right) \right)=\,\sigma\left(\,
\Imm\,\Omega\,{ {\partial}\over {\partial Z} }  {}^{{}^{{}^{{}^\text{\scriptsize $t$}}}}\!\!\!\left( {
{\partial f}\over {\partial \overline Z} } \right) \right).
\end{equation*}

 By the induction on $m$, we can compute the Laplacian
$\Delta_{\Om}.$

\end{proof}
  \vskip 0.1cm We
let $L^2(\AO)$ be the space of all functions $f:\AO\lrt\BC$ such
that
$$||f||_{\Om}:=\int_{\AO}|f(Z)|^2dv_{\Om},$$
where $dv_{\Om}$ is the volume element on $\AO$ normalized so that
$\int_{\AO}dv_{\Om}=1.$ The inner product $(\,\,,\,\,)_{\Om}$ on
the Hilbert space $L^2(\AO)$ is given by
\begin{equation}
(f,g)_{\Om}:=\int_{\AO}f(Z)\,{\overline{g(Z)} }\,dv_{\Om},\quad
f,g\in L^2(\AO).\end{equation}
\begin{theorem}
The set $\left\{\,E_{\Om;A,B}\,|\ A,B\in\BZ^{(m,n)}\,\right\}$ is
a complete orthonormal basis for $L^2(\AO)$. Moreover we have the
following spectral decomposition of $\Delta_{\Om}$:
$$L^2(\AO)=\oplus_{A,B\in \BZ^{(m,n)}}\BC\cdot E_{\Om;A,B}.$$
\end{theorem}
\begin{proof} Let
\begin{equation*}
T=\BC^{(m,n)}\big/\big(\BZ^{(m,n)}\times \BZ^{(m,n)}\big)=\big(\BR^{(m,n)}\times
\BR^{(m,n)}\big)\big/ \big(\BZ^{(m,n)}\times \BZ^{(m,n)}\big)
\end{equation*}
be the torus of real dimension $2mn$. The Hilbert space $L^2(T)$
is isomorphic to the $2\,m\,n$ tensor product of $L^2(\BR/\BZ)$, where
$\BR/\BZ$ is the one-dimensional real torus. Since
$L^2(\BR/\BZ)=\oplus_{k\in\BZ}\BC\cdot e^{2\pi ikx},$ the Hilbert
space $L^2(T)$ is
\begin{equation*}
L^2(T)=\oplus_{A,B\in \BZ^{(m,n)}}\BC\cdot E_{A,B}(W),
\end{equation*}
where $W=P+i\,Q,\ P,Q\in \BR^{(m,n)}$ and
\begin{equation*}
E_{A,B}(W):=e^{2\pi i\,\sigma(\,^t\!AP+\,^tBQ)},\quad
A,B\in\BZ^{(m,n)}.
\end{equation*}
The inner product on $L^2(T)$ is defined by
\begin{equation}
(f,g):=\int_0^1\cdots \int_0^1 f(W)\,{\overline{g(W)}
}\,dp_{11}\cdots dp_{mn}dq_{11}\cdots dq_{mn},\textbf{}\end{equation}
where $f,g\in
L^2(T),\ W=P+iQ\in T,\ P=(p_{kl})$ and
$Q=(q_{kl}).$ Then we see that the set
$$\left\{ E_{A,B}(W)\,|\
A,B\in\BZ^{(m,n)}\,\right\}$$ is a complete orthonormal basis for
$L^2(T)$, and each $E_{A,B}(W)$ is an eigenfunction of the
standard Laplacian
\begin{equation*}
\Delta_T=\sum_{k=1}^m\sum_{l=1}^n \left( {{\partial^2}\over
{\partial p_{kl}^2} }+{{\partial^2}\over {\partial q_{kl}^2}
}\right).
\end{equation*}
We define the mapping $\Phi_{\Omega}:T\lrt A_{\Omega}$ by
\begin{equation}
\Phi_{\Omega}(P+iQ)=(P+QX)+i\,QY,
\end{equation}
where $P+i\,Q\in T,\ P,Q\in
\BR^{(m,n)}.$
This is well defined. We can see that $\Phi_{\Omega}$ is a
diffeomorphism and that the inverse $\Phi_{\Omega}^{-1}$ of
$\Phi_{\Omega}$ is given by
\begin{equation}
\Phi_{\Omega}^{-1}(U+i\,V)=(U-VY^{-1}X)\,+\,i\,VY^{-1},
\end{equation}
where $U+i\,V\in A_{\Omega},\ U,V\in \BR^{(m,n)}.$
Using (18.7), we can show that for $A,B\in\BZ^{(m,n)}$, the
function $E_{A,B}(W)$ on $T$ is transformed to the function
$E_{\Omega;A,B}$ on $A_{\Omega}$ via the diffeomorphism
$\Phi_{\Omega}$. Using (18.5) and the diffeomorphism
$\Phi_{\Omega}$, we can choose a normalized volume element
$dv_{\Omega}$ on $A_{\Omega}$ and then we get the inner product on
$L^2(A_{\Omega})$ defined by (18.4). This completes the proof.

\end{proof}

\end{section}

\vskip 1cm

\newpage
%
%

\bibliographystyle{amsalpha}

\end{document}